\documentclass{amsart}
\usepackage{comment}
 \usepackage{float}
\usepackage[colorlinks=true,linkcolor=black]{hyperref}

\numberwithin{equation}{section}

\newcommand{\Z}{\mathbb{Z}}

\newcommand{\cA}{{\mathscr A}}

\newcommand{\GG}{\mathbb{G}}

\newcommand{\cH}{{\mathcal H}}

\newcommand{\cC}{{\mathscr C}}

\newcommand{\cF}{{\mathscr F}}

\newcommand{\cI}{{\mathscr I}}
\newcommand{\cJ}{{\mathscr J}}
\newcommand{\cK}{{\mathcal K}}

\newcommand{\cont}{\mathsf{c}}

\newcommand{\cO}{{\mathscr O}}

\newcommand{\cU}{{\mathscr U}}

\newcommand{\cV}{{\mathcal V}}

\newcommand{\direct}{\mathsf{dir}}

\newcommand{\tw}{{\mathsf{tw}}}
\newcommand{\sgn}{{\mathsf{sgn}}}

\newcommand{\sEnd}{\underline{\mathsf{End}}}

\newcommand{\res}{{\mathsf{res}}}
\newcommand{\hres}{{\mathsf{hres}}}

\newcommand{\bbC}{\mathbb{C}}
\newcommand{\bbZ}{\mathbb{Z}}

\newcommand{\bbL}{\mathbb{L}}

\newcommand{\bone}{{\mathbf 1}}
\newcommand{\triv}{{\mathsf{triv}}}

\newcommand{\Muk}{{\sf Muk}}
\newcommand{\fr}{{\sf fr}}
\newcommand{\hS}{{\sf hS}}
\newcommand{\comb}{{\sf comb}}

\newcommand{\hg}{\widehat{\mathfrak g}}
\newcommand{\hn}{\widehat{\mathfrak{n}}}
\newcommand{\hh}{\widehat{\mathfrak h}}
\newcommand{\hcV}{\widehat{\mathcal V}}
\newcommand{\lbA}{\overline{\beta}^{A,s}}
\newcommand{\hbA}{\widehat{\beta}^A}

\newcommand{\sym}{{\sf Sym}}
\newcommand{\mc}{{\sf mc}}
\newcommand{\ad}{{\sf ad}}
\newcommand{\db}{\overline{\partial}}
\newcommand{\hbcA}{\widehat{\beta}^\cA}

\newcommand{\ra}{\rightarrow}

\newcommand{\lra}{\longrightarrow}

\DeclareMathOperator{\Aut}{Aut}

\DeclareMathOperator{\Giv}{Giv}

\DeclareMathOperator{\wt}{wt}

\DeclareMathOperator{\Hom}{Hom}
\DeclareMathOperator{\End}{End}

\DeclareMathOperator{\imag}{Im}
\DeclareMathOperator{\id}{id}

\DeclareMathOperator{\Ext}{Ext}

\DeclareMathOperator{\KS}{KS}

\renewcommand{\phi}{\varphi}

\newsavebox{\sembox}
\newlength{\semwidth}
\newlength{\boxwidth}

\newsavebox{\lbox}
\newlength{\lwidth}
\newlength{\lboxwidth}

\newtheoremstyle{bigskipstyle}
  {\bigskipamount} % Space above
  {\bigskipamount} % Space below
  {\itshape}       % Body font (e.g., italics)
  {}               % Indent amount
  {\bfseries}      % Theorem head font (e.g., bold)
  {.}              % Punctuation after theorem head
  {.5em}           % Space after theorem head
  {}               % Theorem head spec

\theoremstyle{bigskipstyle}

\newtheorem{Theorem}{Theorem}[section]
\newtheorem{Lemma}{Lemma}[section]

\newtheorem{Proposition}{Proposition}[section]
\newtheorem{Corollary}[Theorem]{Corollary}
\theoremstyle{definition}
\newtheorem{Definition}{Definition}[section]
\theoremstyle{remark}
\newtheorem{remark}[Theorem]{Remark}
\newtheorem{example}[Theorem]{Example}%[section]
\makeatletter
\let\c@Lemma\c@Theorem
\let\c@Conjecture\c@Theorem
\let\c@Proposition\c@Theorem
\let\c@Definition\c@Theorem
\let\c@Remark\c@Theorem
\makeatother

\usepackage{tikz-cd} 
\usepackage{float}
\usepackage{pdfpages}
\usepackage{amscd}
\usepackage{amsmath, amsthm, amssymb, mathrsfs}%, eufrak}
\usepackage{graphicx}
\usepackage{tikz, dsfont}
\usepackage{marginnote}
\usepackage[colorinlistoftodos]{todonotes}
\usepackage{mathtools}

\usepackage[margin=3cm]{geometry}

\usepackage{multicol}
\usetikzlibrary{arrows,shapes,positioning}
\usetikzlibrary{decorations.markings}
\tikzstyle arrowstyle=[scale=1]
\tikzset{->-/.style={decoration={
  markings,
  mark=at position .5 with {\arrow{>}}},postaction={decorate}}}
\tikzset{cross/.style={cross out, draw=black, minimum size=2*(#1-\pgflinewidth), inner sep=0pt, outer sep=0pt},
default radius will be 1pt. 
cross/.default={1pt}}
\tikzstyle directed=[postaction={decorate,decoration={markings,
    mark=at position .65 with {\arrow[arrowstyle]{stealth}}}}]
\tikzstyle reverse directed=[postaction={decorate,decoration={markings,
    mark=at position .65 with {\arrowreversed[arrowstyle]{stealth};}}}]
    \tikzset{cross/.style={cross out, draw=black, minimum size=2*(#1-\pgflinewidth), inner sep=0pt, outer sep=0pt},
cross/.default={1pt}}
\begin{document}

\title[B-model Categorical Enumerative Invariants and HAE]
{%On the 
B-model Categorical Enumerative Invariants\\and holomorphic anomaly equations
}

\author{% 
Yefeng Shen\; and \;Junwu Tu}

\address{Yefeng Shen, University of Oregon, Eugene, OR 97403, USA.}
\address{Junwu Tu, Institute of Mathematical Sciences, ShanghaiTech University,
Shanghai, 201210, China.}

\begin{abstract} 
In this paper, we study the B-model categorical enumerative invariants (CEI) associated with derived categories of coherent sheaves on smooth projective Calabi-Yau 3-folds. 
We first prove the analogs of the dilaton, string, and divisor equations of CEI in a general context. 
Then we use these equations and the Givental quantization formula to prove that the B-model CEI for any miniversal family of smooth projective Calabi-Yau $3$-folds satisfies the holomorphic anomaly equations introduced by Bershadsky-Cecotti-Ooguri-Vafa. This provides strong evidence that CEI may be taken as a rigorous mathematical definition of the B-model topological string partition function.
\end{abstract}

\maketitle

\setcounter{tocdepth}{1}
\tableofcontents

\section{Introduction}
Categorical Enumerative Invariants (CEI) were introduced by Costello~\cite{Cos2} and in a more explicit form by C\u ad\u araru-Tu~\cite{CalTu1} and Costello-C\u ad\u araru-Tu~\cite{CCT}. When applied to the derived category of coherent sheaves on a smooth projective Calabi-Yau threefold, Costello~\cite[Section 13.1]{Cos2} proposed that these categorically defined invariants could be taken as a mathematical definition of the B-model partition function in topological string theory. In their groundbreaking paper~\cite{BCOV}, Bershadsky-Cecotti-Ooguri-Vafa derived a system of differential equations satisfied by the B-model partition function, known as the {\sl holomorphic anomaly equations}. The purpose of this paper is to prove that categorical enumerative invariants indeed satisfy Bershadsky-Cecotti-Ooguri-Vafa's holomorphic anomaly equation, for any miniversal family of smooth projective Calabi-Yau threefolds. This provides a first evidence to Costello's proposal. In the following, after a brief review of the construction of CEI, we shall present our main result in a more precise form. Throughout the paper, we shall work over the field $\bbC$ unless otherwise stated.

\subsection{B-model CEI of Calabi-Yau threefolds}\label{para:intro-1.1} 

Recall from~\cite{Cos2,CalTu1} that CEI are certain invariants associated with a triple $(\cC,\Omega,s)$ consisting of
\begin{itemize}
\item a smooth and proper $A_\infty$-category $\cC$ over $\mathbb{C}$, 
\item a Calabi-Yau structure $\Omega\in HC^-_d(\cC)$ in the negative cyclic homology of $\cC$ of degree $d$,
\item a choice of splitting of the nc-Hodge filtration of $\cC$ (see Definition~\ref{def:nc-splitting}) denoted by $s$.
\end{itemize}
Then these invariants take the following form:
\[\langle \alpha_1 \psi^{k_1},\ldots, \alpha_n \psi^{k_n}\rangle_{g, n}^{\cC,\Omega,s} \in \mathbb{C},\]
with $\alpha_1,\ldots,\alpha_n$ elements in the shifted Hochschild homology group $HH_\bullet(\cC)[d]$, 
 $k_1,\ldots,k_n\in \mathbb{N}$, and the pair of integers $(g,n)$ be stable, i.e. $2g-2+n>0$. The variable $\psi$ is a formal symbol, and is chosen to match the $\psi$-classes in Gromov-Witten theory. Due to their apparent similarity with Gromov-Witten invariants, these invariants were called categorical Gromov-Witten invariants by Costello~\cite{Cos2}. Following~\cite{CalTu1} we prefer to use the term Categorical Enumerative Invariants (CEI) due to their Morita invariance proved in~\cite{AmoTu2}. Another reason is that CEI are only expected to match with Gromov-Witten invariants when applied to Fukaya categories for a particular choice of the splitting data. In other contexts, for example, when $\cC$ is the derived category of coherent sheaves on a smooth and projective Calabi-Yau variety, the natural choice splitting is the complex-conjugate splitting. In this case, the resulting CEI yields genuine new invariants. 

Strictly speaking, the CEI defined in~\cite{CalTu1}  requires $n>0$. However, by forcing the dilaton equation (see Theorem~\ref{thm:main2}), we may extend the definition of CEI to also include the cases when $g\geq 2$ and $n=0$. More precisely, let us denote by $[\Omega]\in HH_d(\cC)[d]$ the Hochschild homology class of $\Omega$ under the natural projection map $HC^-_d(\cC)[d]\to HH_d(\cC)[d]$, and define 
\begin{equation}\label{eq:cei-no-ins}
 F_g^{\cC,\Omega,s} :=
\frac{1}{2g-2} \langle\Omega\psi\rangle_{g, 1}^{\cC,\Omega,s}.
\end{equation}

In this paper, we shall be interested in the following setup. Let $X$ be a smooth projective Calabi-Yau $3$-fold. %By a Calabi-Yau $3$-fold, 
We shall also assume that $\pi_1(X)=0$. To obtain CEI, we use the following triplet.
\begin{itemize}
    \item Take $\cC$ to be the category $\cC_X:=D_{dg}^b({\sf Coh}(X))$, a dg-enhancement of the derived category of coherent sheaves on $X$. By Lunts-Orlov~\cite{LunOrl}, such a dg-enhancement is unique. 
    \item Fix a Calabi-Yau structure $\Omega\in HC^-_3(\cC_X)\cong H^0(X,\omega_X)$. 
    \item There exists an intrinsic splitting to the dg category $\cC_X$ which we denote by $s^{\sf BT}$ (named after Blanc's work~\cite{Bla} which is based on To\"en's proposal~\cite{Toe}). Geometrically, under the comparison isomorphism in Theorem~\ref{conj:comparison}, the Blanc-To\"en splitting corresponds to the complex conjugate splitting of the classical Hodge filtration on $X$. 
\end{itemize}
We refer to the CEI associated with the triplet $(\cC_X,\Omega,s^{\sf BT})$ as the {\sl B-model categorical enumerative invariants}. In particular, by setting $n=0$, we obtain a complex number $F_g^{\cC_X,\Omega,s^{\sf BT}}$ for each $g\geq 2$. Furthermore, by the rescaling property~\cite[Section 9.1]{CalTu1} of CEI, we have
\[F_g^{\cC_X,\lambda \Omega,s^{\sf BT}} = \lambda^{2-2g} F_g^{\cC_X,\Omega,s^{\sf BT}}.\]
This implies that the following vector
\[ F_g^X:= F_g^{\cC_X,\Omega,s^{\sf BT}} \Omega^{2g-2}
\in H^0(X,\omega_X)^{\otimes (2g-2)}\] 
is independent of the choice of $\Omega$. Furthermore, by the Morita invariance of CEI~\cite{AmoTu2}, the vectors $F_g^X (g\geq 2)$ are derived invariants of $X$!

To formulate Bershadsky-Cecotti-Ooguri-Vafa's holomorphic anomaly equation, we consider the family version of the constructions above. Let $p:\mathfrak{X}\ra M$ be a smooth and projective family of Calabi-Yau $3$-folds over a smooth base $M$. In this case, one can show that the B-model CEI may be defined for families (see Section~\ref{subsec:cei-family}), which yields a smooth section
\[F_g^{\mathfrak{X}/M} \in C^\infty(\mathbb{L}^{2g-2})\]
in the tensor product of the so-called {\em vacuum line bundle}: the push-forward of the relative canonical bundle $\mathbb{L}:=p_*\omega_{\mathfrak{X}/M}$. The line bundle $\mathbb{L}$ is a holomorphic line bundle over $M$ by the Calabi-Yau assumption. However, the section $F_g^{\mathfrak{X}/M}$ is only a $C^\infty$-section because the intrinsic splitting $s^{\sf BT}$ is only a $C^\infty$-splitting relative to $M$. Geometrically, this is a well-known fact: the complex-conjugate splitting is only a $C^\infty$-splitting of the Hodge filtration over $M$. More precisely, the splitting data $s^{\sf BT}$ is given by an isomorphism of $C^\infty$ vector bundles over $M$:
\begin{equation}\label{eq:hodge-decomposition}
\mathcal{H}^3 \cong \mathcal{H}^{3,0}\oplus \mathcal{H}^{2,1}\oplus \mathcal{H}^{1,2}\oplus \mathcal{H}^{0,3},
\end{equation}
where $\mathcal{H}^3$ denotes the $C^\infty$ middle cohomology bundle of the family $\mathfrak{X}\to M$, and each $\mathcal{H}^{p,q}$ denotes the $C^\infty$ sub-bundle of $\mathcal{H}^3$ from the classical Hodge decomposition theorem. 

\subsection{Bershadsky-Cecotti-Ooguri-Vafa's holomorphic anomaly equation}\label{sec-intro-1.2}

Consider the Gauss-Manin connection $\nabla^{\sf GM}$ on the bundle of middle cohomology $\mathcal{H}^3$. We may write it in the Hodge decomposition~\eqref{eq:hodge-decomposition} as
\begin{equation}
\label{eq:gm-components} 
\nabla^{\sf GM}= D+\overline{D}+C+\overline{C}.
\end{equation}
As in~\cite[Section 2.2]{KZ}, $D$ and $\overline{D}$ are the diagonal components in the $(1,0)$ and $(0,1)$ directions (over the base space $M$) respectively. The operators $C$ and $\overline{C}$ are non-diagonal components in the $(1,0)$ and $(0,1)$ directions, respectively,  as illustrated in the following picture.
\[\begin{tikzpicture}
\node (A) at (0,0) {$\mathcal{H}^{3,0}$};
\node (B) at (3,0) {$\mathcal{H}^{2,1}$};
\node (C) at (6,0) {$\mathcal{H}^{1,2}$};
\node (D) at (9,0) {$\mathcal{H}^{0,3}$.};
\draw[->,thick] (A) to[out=20,in=160] node[fill=white,inner sep=2pt,midway] {$C$} (B);
\draw[<-,thick] (A) to[out=-20,in=-160] node[fill=white,inner sep=2pt,midway] {$\overline{C}$} (B);
\draw[->,thick] (B) to[out=20,in=160] node[fill=white,inner sep=2pt,midway] {$C$} (C);
\draw[<-,thick] (B) to[out=-20,in=-160] node[fill=white,inner sep=2pt,midway] {$\overline{C}$} (C);
\draw[->,thick] (C) to[out=20,in=160] node[fill=white,inner sep=2pt,midway] {$C$} (D);
\draw[<-,thick] (C) to[out=-20,in=-160] node[fill=white,inner sep=2pt,midway] {$\overline{C}$} (D);
\end{tikzpicture}\]
The fact that $C$ and $\overline{C}$ are of this particular form is due to Griffiths' transversality~\cite{Gri}.

Now, let us also assume that the family  $p:\mathfrak{X}\ra M$ is miniversal, i.e., the Kodaira-Spencer map 
\[\KS: T_M \stackrel{\cong}{\longrightarrow} R^1p_*T_{\mathfrak{X}/M}\]
is an isomorphism. Under this assumption, we shall use the components in the decomposition~\eqref{eq:gm-components}  to write down the holomorphic anomaly equation. Our exposition closely follows the coordinate-free treatment in~\cite{Liu}. We need to introduce some notation. Since $\mathbb{L}= \mathbb{H}^{3,0}$, we may use the connection $D$ to differentiate the section $F_g^{\mathfrak{X}/M}$. For each $g\geq 2$, this yields a section $$DF_g^{\mathfrak{X}/M}\in C^\infty\big( \Omega_M\otimes \mathbb{L}^{2g-2}\big).$$ 
Observe that by the miniversality assumption, the contraction of vector fields against volume forms induces an isomorphism 
$$R^1p_*T_{\mathfrak{X}/M}\otimes \mathbb{H}^{3,0}\cong \mathbb{H}^{2,1}.$$ 
Combined with the Kodaira-Spencer isomorphism, we obtain an isomorphism 
\begin{equation}\label{eq:intro-iso}
(\mathbb{H}^{2,1})^\vee\otimes \mathbb{H}^{3,0} \cong \Omega_M.
\end{equation}
Since the left hand side has a $(1,0)$-type connection $D$, the cotangent bundle $\Omega_M$ has an induced connection which we still denote by $D$. Using this connection, we obtain the second order covariant derivative $DDF_g^{\mathfrak{X}/M}$ as a smooth section of $\Omega_M^{\otimes 2}\otimes \mathbb{L}^{2g-2}$. 

For $g=1$, due to the stability condition $2g-2+n>0$, $F_1^{\mathfrak{X}/M}$ is not defined. However, we shall slightly abuse the notation $DF_1^{\mathfrak{X}/M}$ by defining it as
\[ DF_1^{\mathfrak{X}/M}:= F_{1,1}^{\mathfrak{X}/M}.\]
By construction, $F_{1,1}^{\mathfrak{X}/M}$ is a smooth section of $\Omega_M$ (using the isomorphism~\eqref{eq:intro-iso} above). With this definition, we arrive at the conclusion that $DF_g^{\mathfrak{X}/M}\in C^\infty\big(\Omega_M\otimes \mathbb{L}^{2g-2}\big)$ for all $g\geq 1$.

\begin{Theorem}
\label{thm:main}
Assume that the family $p:\mathfrak{X}\ra M$ is miniversal. 
For each $g\geq 2$, the B-model CEI $F_g^{\mathfrak{X}/M}\in C^\infty(\mathbb{L}^{2g-2})$ satisfies Bershadsky-Cecotti-Ooguri-Vafa's holomorphic anomaly equation:
\begin{equation}\label{eq:intro-hae-identity} \overline{\partial} F_g^{\mathfrak{X}/M} = \frac{1}{2} \overline{C} \diamond  \left( \sum_{r=1}^{g-1} D F_r^{\mathfrak{X}/M} \otimes DF_{g-r}^{\mathfrak{X}/M} + DD F_{g-1}^{\mathfrak{X}/M} \right).\end{equation}
The notations used here are standard in the literature~\cite[Section 4.4]{Liu}, see also Equation~~\eqref{eq:tensors}.
%and Equation~\eqref{eq:C-tensors}.
\end{Theorem}

\begin{remark}
We also prove a holomorphic anomaly equation in genus one; see Theorem~\ref{thm:genus-one-hae}. 
In this case, it is interesting to compare the B-model CEI with the BCOV torsion~\cite{BCOV,FLY, EFM}. From the B-model perspective, Theorem~\ref{thm:main} suggests that CEI may be taken as a mathematical definition of the B-model non-perturbative topological string partition function. The ``holomorphic limit" of these invariants is conjecturally mirror to higher genus Gromov-Witten invariants. We should also remark that there is an alternative differential geometric approach to the B-model topological string partition function by Costello and Li in ~\cite{CosLi,Li}. Our approach is purely algebraic and categorical. 
\end{remark}

\begin{remark}
On the A-model side, the holomorphic anomaly equation plays an important role in understanding and computing higher genus Gromov-Witten invariants of Calabi-Yau $3$-folds. This was evident in the work of physicists~\cite{BCOV,HKQ,YY,ABK}, as well as for some local Calabi-Yau $3$-folds in~\cite{CoaIri,LP}, and the recent mathematical breakthroughs for compact Calabi-Yau $3$-folds in~\cite{CGL,CGL2,CGLL,CLLL,CJR,CJR2,GJR}. 
%Formal structures of the holomorphic anomaly equation and its applications to Gromov-Witten theory were also studied for local Calabi-Yau's in~\cite{ABK,CoaIri,LP}. 
\end{remark}

\subsection{Dilaton, string and divisor equations} 

To prove Theorem~\ref{thm:main}, the idea is to compare $F_g^{\mathfrak{X}/M}$ with $F_g^{\mathfrak{X}/M,s^{\sf hol}}$ with $s^{\sf hol}$ a local holomorphic splitting of the nc-Hodge filtration. Then, by the compatibility between CEI and Givental's group action~\cite{CalTu1}, one can almost deduce the following equation
\begin{equation}\label{eq:pre-hae}
 \overline{\partial} F_g^{\mathfrak{X}/M} = \frac{1}{2} \overline{C}\diamond \left( \sum_{r=1}^{g-1} F_{r,1}^{\mathfrak{X}/M} \otimes F_{g-r,1}^{\mathfrak{X}/M} + F_{g-1,2}^{\mathfrak{X}/M} \right)
 \end{equation}
where $F_{g,n}^{\mathfrak{X}/M}$ denotes the generating function of CEI with $n$ insertions. This equation is already close to the holomorphic anomaly equation. However, we need to deal with two issues:
\begin{itemize}
\item[(a.)] In~\cite{CalTu1}, the compatibility with Givental's group action is only proved for CEI $F_{g,n}$'s with $n$ strictly positive, while we need to use it in the case $n=0$. 
\item[(b.)] By the Givental's action formula, there are {\sl a priori} extra terms on the right hand side of the above equation involving ``unit" insertions.
\end{itemize}
Furthermore, to deduce Theorem~\ref{thm:main} from Equation~\eqref{eq:pre-hae},  we also need to
\begin{itemize}
\item[(c.)] express the CEI $F_{r,1}^{\mathfrak{X}/M}$ and $F_{g-1,2}^{\mathfrak{X}/M}$ as covariant derivatives of $F^{\mathfrak{X}/M}$. 
\end{itemize}

It turns out that the issues listed in $(a.)$, $(b.)$, and $(c.)$ above can be solved by proving the CEI analogs of dilaton, string and divisor equations (also called translation equation) in Gromov-Witten theory. Indeed, the majority of this paper is devoted to prove the following theorem.
\begin{Theorem}\label{thm:main2}
The B-model CEI of $\mathfrak{X}/M$ satisfies the following equations. 
\begin{itemize}
\item%[(a.)] 
{\bf Dilaton equation [Theorem~\ref{thm:dilaton}]}
\begin{equation*}
\langle [\Omega]\psi,\alpha_1\psi^{k_1},\ldots,\alpha_n\psi^{k_n}\rangle_{g,n+1}^{\mathfrak{X}/M,\Omega,s^{\sf BT}} = (2g-2+n) \langle \alpha_1 \psi^{k_1},\ldots,\alpha_n \psi^{k_n}\rangle_{g,n}^{\mathfrak{X}/M,\Omega,s^{\sf BT}}.
\end{equation*}
\item%[(b.)] 
{\bf String equation [Theorem~\ref{thm:string}]}
\begin{equation*}
\langle [\Omega],\alpha_1 \psi^{k_1},\ldots,\alpha_n \psi^{k_n}\rangle_{g,n+1}^{\mathfrak{X}/M,\Omega,s^{\sf BT}}
 =  \sum_{j=1}^{n} \langle \alpha_1 \psi^{k_1},\ldots,\alpha_j \psi^{k_j-1},\ldots,\alpha_n \psi^{k_n}\rangle_{g,n}^{\mathfrak{X}/M,\Omega,s^{\sf BT}}.
 \end{equation*}
\item%[(c.)] 
{\bf Divisor equation [Theorem~\ref{thm:divisor}].}
In the setup of Theorem~\ref{thm:main}, we have
\begin{equation*}
F_{r,1}^{\mathfrak{X}/M} = D F_{r}^{\mathfrak{X}/M},\quad\quad F_{g-1,2}^{\mathfrak{X}/M} = D D F_{g-1}^{\mathfrak{X}/M}.
\end{equation*}
\end{itemize}
\end{Theorem}

%\vspace{.5cm}
\begin{remark}
%We refer to Theorem~\ref{thm:dilaton}, Theorem~\ref{thm:string} and Theorem~\ref{thm:divisor} for the details.
These equations are proved in a slightly more general setting, namely for CEI of $\Z/2\Z$-graded cyclic $A_\infty$-algebras satisfying certain conditions. Thus, the above theorem is also applicable to various Fukaya categories that are only $\Z/2\Z$-graded. 
\end{remark}

\subsection{Partial recursion of combinatorial string vertices} 

Let us briefly describe the strategy to prove Theorem~\ref{thm:main2}. Recall the construction of CEI~\cite{Cos2, CalTu1} relies fundamentally on the notion of string vertices introduced by Sen-Zwiebach~\cite{SenZwi}. Geometrically, string vertices are a collection of chains $\{\mathcal{V}_{g,n}\}$ with rational coefficients in the coarse moduli space $M_{g,n}/S_n$ of genus $g$ smooth Riemann surfaces with $n$ marked points. String vertices are essentially defined by the Maurer-Cartan equation in the Sen-Zwiebach DGLA (see Section~\ref{subsec:cei-form}). Algebraically, their importance lies at the fact that they provide a canonical and universal solution to the so-called {\em quantum master equation} in any $2d$ Topological Conformal Field Theory (TCFT), see~\cite{Cos2,CalTu1}. Following~\cite{CalTu1}, in this paper we work exclusively with the combinatorial version of the Sen-Zwiebach DGLA. This combinatorial model makes heavy use of the so-called {\em black-and-white graphs} and its PROP-structure studied in~\cite{KonSoi,Cos1,WahWes}. In this model,  the string vertices are given by a collection of chains
$\{ \hcV_{g,k,l} \}$ which are linear combinations (with appropriate circle parameters) of black-and-white graphs of genus $g$, with $k\geq 1$ cycles and $l\geq 0$ white vertices. Again, combinatorial string vertices are essentially defined by the Maurer-Cartan equation in the combinatorial Sen-Zwiebach DGLA. We refer to Section~\ref{subsec:cei-form} for more details.

Turning back to the proof of Theorem~\ref{thm:main2}, the key is to show that the combinatorial string vertices can be chosen to be compatible with certain combinatorial forgetful maps. Geometrically, this compatibility may be thought of as describing the fiber of the string vertex $\cV_{g,n+1}$ over $\cV_{g,n}$ under the forgetful map from $M_{g,n+1}/S_{n+1} \ra M_{g,n}/S_n$. Combinatorially, we obtain a partial recursion formula producing the string vertices $\{\mathcal{V}_{g,k',l'}\}_{k'+l'=n+1}$ from $\{\mathcal{V}_{g,k,l}\}_{k+l=n}$. We refer to Theorem~\ref{thm:mc-recursion} for a precise statement of the partial recursion formula. Remarkably, via the graph calculus developed in~\cite{KonSoi,Cos1,WahWes}, our recursion formula is essentially obtained by sewing with the black-and-white graphs used in  Getzler's explicit formula of the Gauss-Manin connection on periodic cyclic homology~\cite{Get}.  Using the recursion property, we may deduce Theorem~\ref{thm:main2} from the explicit formula of CEI obtained in~\cite{CalTu1}.

To our point of view, it is also extremely interesting to explore other types of recursions such as topological recursion, Virasoro recursion in the context of string vertices. It seems the uniqueness of string vertices could be a really useful tool to understand recursion properties of Gromov-Witten type invariants. Recently, a progress in this direction is obtained in~\cite{FVM} in connection with Mirzakhani's recursion relation~\cite{Mir}.

\subsection*{Organization of the paper} 

We have made efforts to make the paper self-contained. In particular, we start with Section~\ref{sec:cei} where all necessary ingredients in the construction of CEI are recalled, following~\cite{CalTu1}. %Section~\ref{sec:family} discusses the construction of CEI over a base space. 
In Section~\ref{sec:dilaton} we prove the dilaton equation using string vertices that are compatible with forgetful maps. In Section~\ref{sec:recursion}, we prove a partial recursion of string vertices under the framing forgetful maps. This section is at the technical heart of the paper. Building on this recursion property, we proceed to prove the genus zero string equation in Section~\ref{sec:string}, the genus zero divisor equation in Section~\ref{sec:divisor}, and the higher genus version of these equations in Section~\ref{app:proof-string}. Using these equations and Givental's quantization formula, we prove the holomorphic anomaly equations in Section~\ref{sec:hae}.

\subsection*{Acknowledgments} 
We would like to warmly thank Andrei C\u ald\u araru and Alex Menzia for reading an earlier draft of the paper, for pointing out a mistake there concerning the recursion identity in Theorem~\ref{thm:mc-recursion}, and for their generous help in fixing the signs in Equation~\eqref{construction-T}. We also thank Kevin Costello and Yongbin Ruan for their continuous encouragements to explore the B-model CEI, as well as for proposing several research problems in this direction. The second named author is also grateful to Maxim Kontsevich for valuable discussions on various aspects of CEI during his visit to IHES in 2023. 
We both thank IHES for providing excellent research conditions at different stages of this project. 
%He also thanks the institute for providing excellent research conditions. 
YS is partially supported by Simons Grant MPS-TSM-00007929.
JT was partially supported by the National Key Research and Development Program of China No. 2023YFA1009803.

\subsection*{AI disclosure} AI was not used. All the mathematical results and proofs are due to the authors. 

\section{Recollections of CEI}\label{sec:cei}

In this section, we briefly recall the construction of CEI, following~\cite{CalTu1}. Throughout the section, we shall let $A$ be a cyclic, unital $A_\infty$-algebra that is smooth and finite dimensional over $\mathbb{C}$. This is related to the setup of Section~\ref{para:intro-1.1} as follows. Given a Calabi-Yau $A_\infty$-category $(\cC, \Omega)$, assume furthermore that it is compactly generated. Then we can choose a compact generator $E$ and obtain (if needed using homological perturbation) a minimal $A_\infty$-algebra $A:=\End_\cC(E)$.  Since $\cC$ is proper, the $A_\infty$-algebra is finite dimensional. Then we apply a formal Darboux theorem (following Kontsevich-Soibelman~\cite{KonSoi,Cho}, and also~\cite{AmoTu2} for its unital version) to obtain a cyclic structure on $A$. As shown in~\cite{AmoTu2}, the CEI (defined using $A$) does not depend on the choice of $E$ and the cyclic model.

\subsection{Black-and-white graphs}\label{subsec:graphs} 

For an integer $k\geq 1$, $l\geq 0$, let $M_{g,k,l}^\fr$ denotes the moduli space of Riemann surfaces with $k$ framed incoming marked points and $l$ framed outgoing marked points.  Here a framing of a marked point is given by a choice of coordinate centered at the marked point. Such a framed Riemann surface may be denoted as
\[(\Sigma, p_1,\ldots,p_k,q_1,\ldots,q_l,\phi_1,\ldots,\phi_k,\psi_1,\ldots,\psi_l),\]
where $\Sigma$ is a smooth Riemann surface of genus $g$, the $p$'s and $q$'s are incoming marked points and outgoing marked points respectively, and the $\phi$'s and $\psi$'s are a choice of coordinates centered at the corresponding marked points. 
We also require the domains of the framings be disjoint. Using the framings at marked points, we can define a ``composition" structure called the sewing operations
\[ M_{g,k,l}^\fr\times (M_{g_1,k_1,l_1}^\fr \times \cdots \times M_{g_r,k_r,l_r}^\fr)  \ra M_{g+g_1+\cdots+g_r+k-r,k_1+\cdots+k_r,l}^\fr,\]
with $l_1+\cdots+l_r=k$. Algebraically, this is usually referred to as a ``PROP" structure, although here we have the additional requirement that the number of incomings is strictly positive.

These topological spaces are actually complex manifolds. Furthermore, observe that the map $M_{g,k,l}^\fr \ra M_{g,k,l}$ which forgets all the framings at marked points has fibers homotopy equivalent to $k+l$ copies of $S^1$. In~\cite{KonSoi,Cos1,WahWes}, an explicit chain complex denoted by $C^{\sf comb}_\bullet(M_{g,k,l}^\fr)$ (with $\mathbb{Q}$ coefficients) is defined using the so-called black-and-white ribbon graphs which would computes the rational homology of $M_{g,k,l}^\fr$. Furthermore, the geometric ``PROP" structure described above also admits an explicit combinatorial description:
\begin{equation}\label{eq:prop} 
 C^{\sf comb}_\bullet(M_{g,k,l}^\fr) \otimes C^{\sf comb}_\bullet(M_{g_1,k_1,l_1}^\fr) \otimes \cdots \otimes C^{\sf comb}_\bullet(M_{g_r,k_r,l_r}^\fr)  \ra C^{\sf comb}_\bullet(M_{g+g_1+\cdots+g_r+k-r,k_1+\cdots+k_r,l}^\fr).
\end{equation}
In the following, we shall recall the notion of black-and-white graphs and the combinatorial sewing operations in more detail. Our exposition follows closely the excellent treatment by Wahl-Westerland~\cite{WahWes}. 

\subsubsection{Definition of black-and-white graphs}

By a graph we mean a tuple $G=(V,H,s,i)$ where $V$ is the set of vertices, $H$ the set of half-edges, $s: H\ra V$ the source map, and $i: H \ra H$ is an involution. Fixed elements of the involution are called leaves. A non-trivial orbit of $i$, i.e., a pair of elements $\{ h,i(h)\}$ with $i(h)\neq h$ is called an edge. Denote by $E$ the set of edges of $G$. The valence of a vertex $v\in V$ is defined by ${\sf val}(v)=|s^{-1}(v)|$. 

A ribbon graph is a graph $G=(V,H,s,i)$ together with a cyclic ordering of each of the set $s^{-1}(v)$ for $v\in V$. These cyclic ordering may be viewed as a map $\tau: H \ra H$ which sends a half-edge to its next half-edge in the cyclic ordering. Then we define cycles of a ribbon graph by the orbits of the map $i\circ \tau: H \ra H$.

A black-and-white graph is a ribbon graph whose vertices is decomposed as $V=V_b \coprod V_w$, with $V_b$ denote the set of black vertices and $V_w$ the set of white vertices. The black vertices must be at least of valence three, while white vertices have valence at least one.  Furthermore, we require that every cycle has a leaf as its starting half-edge, and every white vertex has a starting half-edge, i.e., a choice of an element in $s^{-1}(v)$ for $v\in V_w$. These starting half-edges play the role of framing in the geometric setting. We also require that there are no other types of leaves. This last condition is not in the definition of black-and-white graphs in~\cite{WahWes} since they also deal with open sewing operations, while in this paper we only consider closed theory.

Given a black-and-white graph $G$, we may define its combinatorial type by setting
\begin{align*}
\begin{dcases}
    k &:= \mbox{number of cycles in $G$};\\
    l &:= \mbox{number of white vertices in $G$};\\
    g & := \frac{2+|E|-k-|V|}{2}.
\end{dcases}
\end{align*}
Observe that since every graph must have at least one cycle, we always have $k\geq 1$. For later usage, we label the white vertices by the set $\{1, 2, \ldots,l\}$ and the black vertices by the set  $\{1, 2, \ldots, k\}$. As the notation already suggests, the cycles in a graph $G$ correspond to incoming marked points, while the white vertices correspond to outgoing marked points.

In the following, we give a few examples of basic black-and-white graphs. 
\begin{enumerate}
\item The so-called {\em Mukai graph} denoted by $M$ is depicted as 
\begin{equation}\label{graph:mukai}
M:=\begin{tikzpicture}[baseline={([yshift=-.4ex]current bounding box.center)},scale=.35]
	\draw (2,0) node[cross=2pt,label=above:{}] {};
	\draw (6.2,0) node[cross=2pt,label=above:{}] {};
	\draw (2,-1) node {$1$};
	\draw (5.2,0) node {$2$};
	\draw [thick] (6.2,0) to (7.4,0);
	\draw [thick] (2,0) to (3.4,0);
	\draw [thick] (5.2,0) + (-85:2) arc(-90:270:2);
	\end{tikzpicture}\end{equation}
The cyclic ordering at each vertex is by the counterclockwise orientation of the plane. It has two cycles. The two leaves with a cross tale are the starting half-edges of the cycles. One can easily compute its combinatorial type as $k=2$, $l=0$ and $g=0$.
\item The following graph which we shall refer to as the T-graph
\begin{equation}\label{graph:T}
	T:=\begin{tikzpicture}[baseline={([yshift=-1.2ex]current bounding
      box.center)},scale=0.35] 
\draw [thick] (0,0) to (0,2);
\draw [thick] (-0.2, 1.8) to (0.2, 2.2);
\draw [thick] (0.2, 1.8) to (-0.2, 2.2);
\draw [thick] (0,0) to (-2,0);
\draw [thick] (0,0) to (2,0);
\draw [thick] (-2.2,0) circle [radius=0.2];
\draw [thick] (2.2,0) circle [radius=0.2];
\end{tikzpicture}\end{equation}
has one black vertex, two white vertices and one cycle. Thus we obtain $k=1$, $l=2$ and $g=0$. The unique leaf with a cross tale is the starting half-edge of the cycle. And each white vertex is of valence one, which implies a unique choice of its starting half-edge. 
\item Finally, we give an example of positive genus. Let us consider the following graph:
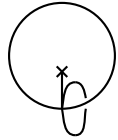
\[ \begin{tikzpicture}[baseline={([yshift=-.4ex]current bounding box.center)},scale=0.35]
\draw [thick] (0,2) circle [radius=2];
\draw [thick] (0,0) to (0,1.4);
\draw [thick] (-0.2, 1.2) to (0.2, 1.6);
\draw [thick] (-0.2, 1.6) to (0.2, 1.2);
\draw [thick] (0,0) to [out=80, in=180] (0.5, 1);
\draw [thick] (0.5,1) to [out=0, in=100] (0.9, 0.4);
\draw [thick] (0,0) to [out=-80, in=180] (0.5, -1);
\draw [thick] (0.5, -1) to [out=0, in=-100] (0.9, 0);
\end{tikzpicture}\]
One checks that it has a unique cycle, and the leaf with a cross tale is its starting half-edge. So we have $k=1$ and $l=0$. It has two loop edges, one black vertex, and $g=1$.
\end{enumerate}

An {{\em orientation}} of a graph $G=(V,H,s,i)$ is a unit vector in ${\sf det}\big(\mathbb{R}(V\coprod H)\big)$. With this in mind we set $C^{\sf comb}_\bullet(M_{g,k,l}^\fr)$ to be the vector space over $\mathbb{Q}$ generated by isomorphism classes of oriented black-and-white graphs of combinatorial type $(g,k,l)$ together with the relation that $(G,-\sigma)=-(G,\sigma)$ if $\sigma$ is an orientation of $G$. In~\cite{KonSoi,Cos1,WahWes}, a chain complex structure is constructed on $C^{\sf comb}_\bullet(M_{g,k,l}^\fr)$. Indeed, the degree of a black-and-white graph is 
\[ \deg(G):= \sum_{v\in V_b} ({\sf val}(v)-3) + \sum_{v\in V_w} ({\sf val}(v)-1).\]
For a black-and-white graph $G$, a blow-up of $G$ is a black-and-white graph $\widetilde{G}$ obtained from $G$ by performing a local operation at a vertex $v$ of $G$ according to the following rules.
\begin{enumerate}
	\item If $v$ is a black vertex, we may expand it into two black vertices connected by an edge and still requiring  both vertices have valence at least $3$. For example, if ${\sf val}(v)=4$, there are possibly two such expansions:
	\begin{align*}
		&
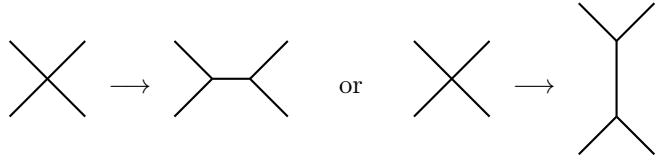
\begin{tikzpicture}[baseline={([yshift=-1.2ex]current bounding
			box.center)},scale=0.5] 
		\draw [thick] (0,0) to (1,1);
		\draw [thick] (0,0) to (-1,-1);
		\draw [thick] (0,0) to (-1,1);
		\draw [thick] (0,0) to (1,-1);
	\end{tikzpicture} \;\;\stackrel{\;}{\longrightarrow}\;\; 
	\begin{tikzpicture}[baseline={([yshift=-1.2ex]current bounding
			box.center)},scale=0.5] 
		\draw [thick] (0,0) to (-1,-1);
		\draw [thick] (0,0) to (-1,1);
		\draw [thick] (0,0) to (1,0);
		\draw [thick] (1,0) to (2,1);
		\draw [thick] (1,0) to (2,-1);
	\end{tikzpicture} \mbox{\;\; \quad or \quad \;\;}	 
 \begin{tikzpicture}[baseline={([yshift=-1.2ex]current bounding
			box.center)},scale=0.5] 
		\draw [thick] (0,0) to (1,1);
		\draw [thick] (0,0) to (-1,-1);
		\draw [thick] (0,0) to (-1,1);
		\draw [thick] (0,0) to (1,-1);
	\end{tikzpicture} \;\;\stackrel{\;}{\longrightarrow}\;\; 
	\begin{tikzpicture}[baseline={([yshift=-1.2ex]current bounding
			box.center)},scale=0.5] 
		\draw [thick] (0,1) to (-1,2);
		\draw [thick] (0,1) to (1,2);
		\draw [thick] (0,1) to (0,-1);
		\draw [thick] (0,-1) to (-1,-2);
		\draw [thick] (0,-1) to (1,-2);
	\end{tikzpicture}
	\end{align*}
	\item If $v$ is a white vertex, we may expand it into one black vertex (with valence at least $3$) and one white vertex connected by an edge. If the starting half-edge at $v$ after expansion is still at the white vertex, then it is the starting half-edge of the white vertex. Otherwise, the starting half-edge marking at the white vertex is by the new half-edge attached to the white vertex. The latter situation is illustrated in the following example:
	
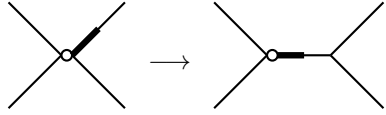
\begin{align*}
		\begin{tikzpicture}[baseline={([yshift=-1.2ex]current bounding
				box.center)},scale=0.35]
		\draw [thick] (9,0) circle [radius=0.2];
		\draw [thick] (6.8,2) to (8.8,0);
		\draw [thick] (6.8,-2) to (8.8,0);
		\draw [thick] (11.2,2) to (9.2,0);
		\draw [thick] (11.2,-2) to (9.2,0);
		\draw [line width=2.4pt] (9.2,0) to (10.2,1);
		\end{tikzpicture} \;\;\stackrel{\;}{\longrightarrow}\;\; 
			\begin{tikzpicture}[baseline={([yshift=-1.2ex]current bounding
				box.center)},scale=0.35]
			\draw [thick] (9,0) circle [radius=0.2];
			\draw [thick] (6.8,2) to (8.8,0);
			\draw [thick] (6.8,-2) to (8.8,0);
			\draw [thick] (9.2,0) to (11.2,0);
			\draw [line width=2.4pt] (9.2,0) to (10.2,0);
			\draw [thick] (11.2,0) to (13.2,2);
			\draw [thick] (11.2,0) to (13.2,-2);
		\end{tikzpicture}
	\end{align*}
	Note that in these pictures, the thickened half-edge or leaf at a white vertex represents its starting half-edge or leaf. For a white vertex $v\in V_w$ with $|{\sf val}(v)|=1$, since its unique half edge must be its starting half-edge, we omit the thickening in this case. 
\end{enumerate}
Fix a vertex $v\in V$ of $G$, denote the set of black-and-white graphs obtained from expanding $G$ at $v$ by $(G,v)$. Then the boundary map in the chain complex  $C^{\sf comb}_\bullet(M_{g,k,l}^\fr)$ is defined by 
\[ \partial G:= \sum_{\substack{v\in V,\\ \widetilde{G}\in (G,v)}}  \widetilde{G},\]
i.e., the boundary is summing over all possible ways to expand a vertex in $G$. We refer the details (including how orientations are defined in the boundary map) to Wahl-Westerland's careful treatment~\cite[Section 2.5]{WahWes}. The following pictures illustrate a few cases of the boundary map. 
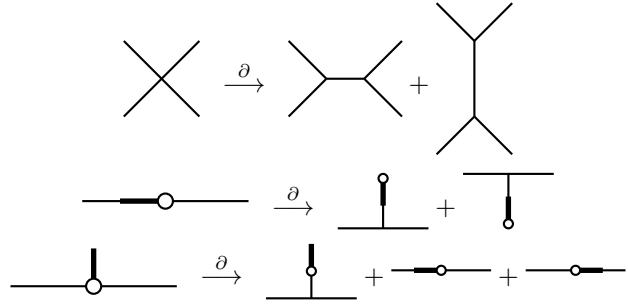
\[\begin{tikzpicture}[baseline={([yshift=-1.2ex]current bounding
      box.center)},scale=0.5] 
\draw [thick] (0,0) to (1,1);
\draw [thick] (0,0) to (-1,-1);
\draw [thick] (0,0) to (-1,1);
\draw [thick] (0,0) to (1,-1);
\end{tikzpicture} \;\;\stackrel{\partial}{\longrightarrow}\;\; 
\begin{tikzpicture}[baseline={([yshift=-1.2ex]current bounding
      box.center)},scale=0.5] 
\draw [thick] (0,0) to (-1,-1);
\draw [thick] (0,0) to (-1,1);
\draw [thick] (0,0) to (1,0);
\draw [thick] (1,0) to (2,1);
\draw [thick] (1,0) to (2,-1);
\end{tikzpicture} +
\begin{tikzpicture}[baseline={([yshift=-1.2ex]current bounding
      box.center)},scale=0.5] 
\draw [thick] (0,1) to (-1,2);
\draw [thick] (0,1) to (1,2);
\draw [thick] (0,1) to (0,-1);
\draw [thick] (0,-1) to (-1,-2);
\draw [thick] (0,-1) to (1,-2);
\end{tikzpicture}
\]
 \[\begin{tikzpicture}[baseline={([yshift=-1.2ex]current bounding
      box.center)},scale=0.5] 
\draw [thick] (0,0) to (-2,0);
\draw [thick] (-2.2,0) circle [radius=0.2];
\draw [line width=2pt] (-3.4,0) to (-2.4,0);
\draw [thick] (-4.4,0) to (-2.4,0);
\end{tikzpicture} \;\;\stackrel{\partial}{\longrightarrow}\;\; 
\begin{tikzpicture}[baseline={([yshift=-1.2ex]current bounding
      box.center)},scale=0.3] 
\draw [thick] (2,0) to (-2,0);
\draw [thick] (0,2.2) circle [radius=0.2];
\draw [line width=2pt] (0,1) to (0,2);
\draw [thick] (0,0) to (0,2);
\end{tikzpicture} + \begin{tikzpicture}[baseline={([yshift=-1.2ex]current bounding
      box.center)},scale=0.3] 
\draw [thick] (2,0) to (-2,0);
\draw [thick] (0,-2.2) circle [radius=0.2];
\draw [line width=2pt] (0,-1) to (0,-2);
\draw [thick] (0,0) to (0,-2);
\end{tikzpicture}
\]
 \[\begin{tikzpicture}[baseline={([yshift=-1.2ex]current bounding
      box.center)},scale=0.5] 
\draw [thick] (0,0) to (-2,0);
\draw [thick] (-2.2,0) circle [radius=0.2];
\draw [line width=2pt] (-2.2,0.2) to (-2.2,1);
\draw [thick] (-4.4,0) to (-2.4,0);
\end{tikzpicture} \;\;\stackrel{\partial}{\longrightarrow}\;\; 
\begin{tikzpicture}[baseline={([yshift=-1.2ex]current bounding
      box.center)},scale=0.3] 
\draw [thick] (2,0) to (-2,0);
\draw [thick] (0,1.2) circle [radius=0.2];
\draw [line width=2pt] (0,2.4) to (0,1.4);
\draw [thick] (0,0) to (0,1);
\end{tikzpicture} + \begin{tikzpicture}[baseline={([yshift=-1.2ex]current bounding
      box.center)},scale=0.3] 
\draw [thick] (0,0) to (-2,0);
\draw [thick] (-2.2,0) circle [radius=0.2];
\draw [line width=2pt] (-3.4,0) to (-2.4,0);
\draw [thick] (-4.4,0) to (-2.4,0);
\end{tikzpicture} + \begin{tikzpicture}[baseline={([yshift=-1.2ex]current bounding
      box.center)},scale=0.3] 
\draw [thick] (0,0) to (-2,0);
\draw [thick] (-2.2,0) circle [radius=0.2];
\draw [line width=2pt] (-1,0) to (-2,0);
\draw [thick] (-4.4,0) to (-2.4,0);
\end{tikzpicture}
\]

\subsubsection{Sewing of black-and-white graphs}

To define the sewing operations combinatorially, let us consider the following simplified case when $r=1$ in Equation~\eqref{eq:prop}. The case with a general $r$ in Equation~\eqref{eq:prop} is defined in a similar way. We will define a map of the form
\[ C^{\sf comb}_\bullet(M_{g,k,l}^\fr) \otimes C^{\sf comb}_\bullet(M_{g_1,k_1,k}^\fr)  \ra C^{\sf comb}_\bullet(M_{g+g_1+k-1,k_1,l}^\fr).\]
Let us take $G_1\in C^{\sf comb}_\bullet(M_{g_1,k_1,k}^\fr)$ and $G\in C^{\sf comb}_\bullet(M_{g,k,l}^\fr)$. Note that both white vertices of $G_1$ and cycles of $G$ are labeled by $1,2,\ldots, k$. We define the result of sewing $G\circ G_1$ by the sum over all possible black-and-white graphs that can be obtained from $G_1$ and $G$ by
\begin{itemize}
    \item[(a)] removing the $k$ white vertices of $G_1$,
    \item[(b1)] if the $i$-th white vertex $v_i$ of $G_1$ is a starting half-edge, we identify it with the $i$-th starting leaf $\lambda_i$ of the $i$-th cycle in $G$,
    \item[(b2)] if the $i$-th white vertex $v_i$ has a starting leaf, and the starting leaf $\lambda_i$ of the $i$-th cycle is attached to a trivalent vertex (otherwise the result $G\circ G_1$ is defined to be zero), we remove both starting leaves,
    \item[(c)] attaching the remaining half-edges in $s^{-1}(v_i)$ to vertices of the $i$-th cycle of $G$, respecting the cyclic ordering of the half-edges.
\end{itemize}
We refer to~\cite{WahWes} for more details and how orientations work out in the sewing constructions described above. It was also proved in {\sl Loc. Cit.} that the gluing map $\circ$ is a chain map and is an associative operation~\cite[Lemma 2.5]{WahWes}.

More generally, there are gluing maps defined between inputs and outputs. More precisely, fixing a subset $I\subset\{1,\ldots,l''\}$, $J\subset \{1,\ldots,k'\}$ such that $|I|=|J|=r\geq 1$, there is a gluing map defined by sewing the outputs in $I$ with the inputs in $J$ in all possible $r!$ ways :
\begin{equation}\label{eq:gluing-maps}
	\prescript{}{J}\circ_I : C_\bullet^\comb(M_{g',k',l'}^\fr)\otimes C_\bullet^\comb(M_{g'',k'',l''}^\fr) \ra C_\bullet^\comb(M_{g,k,l}^\fr),\end{equation}
where $(g',k',l',g'',k'',l'')$ satisfies
%$g=g'+g''+r-1$, $k=k'+k''-r$ and $l=l'+l''-r$. 
\begin{equation}
\label{indices-mc-equation}
\begin{dcases}
g'+g'' =g-r+1,\\
k'+k''= k+r,\\
l'+l''= l+r.
\end{dcases}
\end{equation}
Furthermore, when the subset $I=\{1,\ldots,l''\}$ or $J=\{1,\ldots,k'\}$, we shall drop the corresponding subcript when writing the composition $\prescript{}{J}\circ_I$.

In the following, we give a few examples of sewing operations.

\begin{enumerate}
\item There is a special black-and-white graph of combinatorial type $(g=0,k=1,l=1)$ that plays the role of identity in the graph composition defined above. It is depicted as 
\[\begin{tikzpicture}[baseline={([yshift=-1.2ex]current bounding
      box.center)},scale=0.5] 
\draw [thick] (0,0) to (-2,0);
\draw [thick] (-0.2, -0.2) to (0.2, 0.2);
\draw [thick] (0.2, -0.2) to (-0.2, 0.2);
\draw [thick] (-2.2,0) circle [radius=0.2];
\end{tikzpicture}\]
Using the above sewing algorithm, one may verify sewing a graph $G$ with the above graph at a chosen cycle or a chosen white vertex yields the same graph $G$. We can also perform sewing at a fixed index:
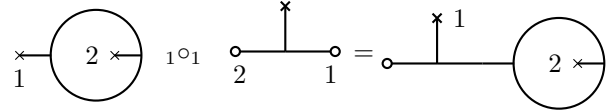
\[ \begin{tikzpicture}[baseline={([yshift=-.4ex]current bounding box.center)},scale=.3]
\draw (2,0) node[cross=2pt,label=above:{}] {};
\draw (6.2,0) node[cross=2pt,label=above:{}] {};
\draw (2,-1) node {$1$};
\draw (5.2,0) node {$2$};
\draw [thick] (6.2,0) to (7.4,0);
\draw [thick] (2,0) to (3.4,0);
\draw [thick] (5.2,0) + (-85:2) arc(-90:270:2);
\end{tikzpicture} \;\;{\;_1\circ_1}\;\;\;  \begin{tikzpicture}[baseline={([yshift=-1.2ex]current bounding
      box.center)},scale=0.3] 
\draw [thick] (0,0) to (0,2);
\draw [thick] (-0.2, 1.8) to (0.2, 2.2);
\draw [thick] (0.2, 1.8) to (-0.2, 2.2);
\draw [thick] (0,0) to (-2,0);
\draw [thick] (0,0) to (2,0);
\draw [thick] (-2.2,0) circle [radius=0.2];
\draw (-2,-1) node {$2$};
\draw [thick] (2.2,0) circle [radius=0.2];
\draw (2,-1) node {$1$};
\end{tikzpicture} = \begin{tikzpicture}[baseline={([yshift=-.4ex]current bounding box.center)},scale=.3]
\draw [thick] (0,0) to (0,2);
\draw [thick] (-0.2, 1.8) to (0.2, 2.2);
\draw [thick] (0.2, 1.8) to (-0.2, 2.2);
\draw [thick] (0,0) to (-2,0);
\draw [thick] (0,0) to (2,0);
\draw [thick] (-2.2,0) circle [radius=0.2];
\draw (6.2,0) node[cross=2pt,label=above:{}] {};
\draw (1,2) node {$1$};
\draw (5.2,0) node {$2$};
\draw [thick] (6.2,0) to (7.4,0);
\draw [thick] (2,0) to (3.4,0);
\draw [thick] (5.2,0) + (-85:2) arc(-90:270:2);
\end{tikzpicture}\]
where the subscript in $\;_1\circ_1$ is to indicate we perform the sewing operation at the white vertex labeled by $1$ in the first graph with the cycle labeled by $1$ in the Mukai graph.

\item The result of sewing both white vertices with the two cycles in the previous example yields a genus one black-and-white graph:
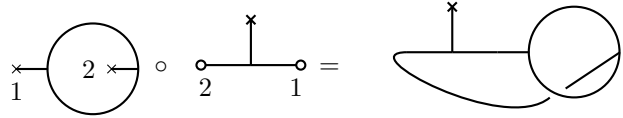
\[ \begin{tikzpicture}[baseline={([yshift=-.4ex]current bounding box.center)},scale=.3]
\draw (2,0) node[cross=2pt,label=above:{}] {};
\draw (6.2,0) node[cross=2pt,label=above:{}] {};
\draw (2,-1) node {$1$};
\draw (5.2,0) node {$2$};
\draw [thick] (6.2,0) to (7.4,0);
\draw [thick] (2,0) to (3.4,0);
\draw [thick] (5.2,0) + (-85:2) arc(-90:270:2);
\end{tikzpicture} \;\;{\circ}\;\;\; \begin{tikzpicture}[baseline={([yshift=-1.2ex]current bounding
      box.center)},scale=0.3] 
\draw [thick] (0,0) to (0,2);
\draw [thick] (-0.2, 1.8) to (0.2, 2.2);
\draw [thick] (0.2, 1.8) to (-0.2, 2.2);
\draw [thick] (0,0) to (-2,0);
\draw [thick] (0,0) to (2,0);
\draw [thick] (-2.2,0) circle [radius=0.2];
\draw (-2,-1) node {$2$};
\draw [thick] (2.2,0) circle [radius=0.2];
\draw (2,-1) node {$1$};
\end{tikzpicture}  = \begin{tikzpicture}[baseline={([yshift=-.4ex]current bounding box.center)},scale=.3]
\draw [thick] (0,0) to (0,2);
\draw [thick] (-0.2, 1.8) to (0.2, 2.2);
\draw [thick] (0.2, 1.8) to (-0.2, 2.2);
\draw [thick] (0,0) to (-2,0);
\draw [thick] (0,0) to (2,0);
\draw [thick] (5,-1.6) to (7.4,0);
\draw [thick] (-2,0) to [out=180, in=220] (4.3, -2);
\draw [thick] (2,0) to (3.4,0);
\draw [thick] (5.2,0) + (-85:2) arc(-90:270:2);
\end{tikzpicture}\]
\item There are two other black-and-white graphs that are important for us. The first one is of type $(g=0,k=1,l=1)$, given by
\begin{equation}\label{graph:circle}
B:= \begin{tikzpicture}[baseline={([yshift=.2ex]current bounding
      box.center)},scale=0.5] 
\draw [thick] (0,0) to (-2,0);
\draw [thick] (-0.2, -0.2) to (0.2, 0.2);
\draw [thick] (0.2, -0.2) to (-0.2, 0.2);
\draw [thick] (-2.2,0) circle [radius=0.2];
\draw (-1,-.5) node {$h$};
\draw (-2.2,-.5) node {$w$};
\draw (-3.3,-.5) node {$l$};
\draw [line width=2pt] (-4.4,0) to (-2.4,0);
\end{tikzpicture} 
\end{equation}

We endow it with the orientation $w\wedge l\wedge h$. This is a closed degree one element in $C^{\sf comb}_\bullet(M_{0,1,1}^\fr) $, geometrically representing the fundamental class of the moduli space of framed annuli with one input and one output (which is homotopy equivalent to $S^1$). Note that the thickened leaf $l$ is the starting half-edge of the white vertex $w$, while the leaf $h$ with a cross tale is the starting half-edge of the cycle. An example of sewing with this graph is the following
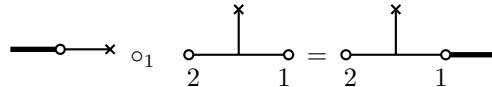
\[\begin{tikzpicture}[baseline={([yshift=-1.2ex]current bounding
      box.center)},scale=0.3] 
\draw [thick] (0,0) to (-2,0);
\draw [thick] (-0.2, -0.2) to (0.2, 0.2);
\draw [thick] (0.2, -0.2) to (-0.2, 0.2);
\draw [thick] (-2.2,0) circle [radius=0.2];
\draw [line width=2pt] (-4.4,0) to (-2.4,0);
\end{tikzpicture} \;\;{\circ_1}\;\;\; \begin{tikzpicture}[baseline={([yshift=-1.2ex]current bounding
      box.center)},scale=0.3] 
\draw [thick] (0,0) to (0,2);
\draw [thick] (-0.2, 1.8) to (0.2, 2.2);
\draw [thick] (0.2, 1.8) to (-0.2, 2.2);
\draw [thick] (0,0) to (-2,0);
\draw [thick] (0,0) to (2,0);
\draw [thick] (-2.2,0) circle [radius=0.2];
\draw (-2,-1) node {$2$};
\draw [thick] (2.2,0) circle [radius=0.2];
\draw (2,-1) node {$1$};
\end{tikzpicture}  = \begin{tikzpicture}[baseline={([yshift=-1.2ex]current bounding
      box.center)},scale=0.3] 
\draw [thick] (0,0) to (0,2);
\draw [thick] (-0.2, 1.8) to (0.2, 2.2);
\draw [thick] (0.2, 1.8) to (-0.2, 2.2);
\draw [thick] (0,0) to (-2,0);
\draw [thick] (0,0) to (2,0);
\draw [thick] (-2.2,0) circle [radius=0.2];
\draw (-2,-1) node {$2$};
\draw [thick] (2.2,0) circle [radius=0.2];
\draw [line width=2pt] (4.4,0) to (2.4,0);
\draw (2,-1) node {$1$};
\end{tikzpicture}\]
The second one is of type $(g=0,k=2, l=0)$, and is called the {\em thickened Mukai graph}, explicitly given by
\begin{equation}\label{graph:thick-mukai}
\mathbb{M}:= \begin{tikzpicture}[baseline={([yshift=-.4ex]current bounding box.center)},scale=0.3]
\draw (2,0) node[cross=2pt,label=above:{}] {};
\draw (4.4,0) node[cross=2pt,label=above:{}] {};
\draw [ thick] (4.3,0) to (3.2,0);
\draw [ thick] (2.1,0) to (3.2,0);
\draw [ thick] (5.2,0) + (-85:2) arc(-90:270:2);
\draw (2,-1) node {$1$};
\draw (5.2,0) node {$2$};
\end{tikzpicture}
\end{equation}
Geometrically, this represents the fundamental class of the moduli space of framed annuli with two inputs and no output. Via the sewing operation, this graph is obtained from sewing the Mukai graph graph with the white vertex of the first special graph, i.e., we have
\[\begin{tikzpicture}[baseline={([yshift=-.4ex]current bounding box.center)},scale=.3]
\draw (2,0) node[cross=2pt,label=above:{}] {};
\draw (6.2,0) node[cross=2pt,label=above:{}] {};
\draw (2,-1) node {$1$};
\draw (5.2,0) node {$2$};
\draw [ thick] (6.2,0) to (7.4,0);
\draw [ thick] (2,0) to (3.4,0);
\draw [ thick] (5.2,0) + (-85:2) arc(-90:270:2);
\end{tikzpicture}  \;\;\prescript{}{2}\circ \;\; \begin{tikzpicture}[baseline={([yshift=-1.2ex]current bounding
      box.center)},scale=0.3] 
\draw [thick] (0,0) to (-2,0);
\draw [thick] (-0.2, -0.2) to (0.2, 0.2);
\draw [thick] (0.2, -0.2) to (-0.2, 0.2);
\draw [thick] (-2.2,0) circle [radius=0.2];
\draw [line width=2pt] (-4.4,0) to (-2.4,0);
\end{tikzpicture}= \begin{tikzpicture}[baseline={([yshift=-.4ex]current bounding box.center)},scale=0.3]
\draw (2,0) node[cross=2pt,label=above:{}] {};
\draw (4.4,0) node[cross=2pt,label=above:{}] {};
\draw [ thick] (4.3,0) to (3.2,0);
\draw [ thick] (2.1,0) to (3.2,0);
\draw [ thick] (5.2,0) + (-85:2) arc(-90:270:2);
\draw (2,-1) node {$1$};
\draw (5.2,0) node {$2$};
\end{tikzpicture}
\]

\end{enumerate}

\subsection{Sen-Zwiebach's DGLA and combinatorial string vertices}
\label{subsec:cei-form}  
There are two main ingredients involved in the construction of CEI: the first being combinatorial Sen-Zwiebach string vertices~\cite{SenZwi} and the second being a $2$-dimensional topological conformal field theory (TCFT) structure on the shifted  Hochschild chain complex. In this subsection, we deal with the notion of combinatorial string vertices.

Let us continue to work with the chain complex $C^{\sf comb}_\bullet(M^\fr_{g,k,l})$ of black-and-white graphs of combinatorial type $(g,k,l)$. Let $\underline{\sgn_k}[-k]$ be the sign representation of the symmetric group $\Sigma_k$, shifted by $[-k]$ in homological degree. For simplicity, we often abbreviate this as $\underline{\sgn}$ when the integer $k$ is clear from the context. Denote by $C^{\sf comb}_\bullet(M^\fr_{g,k,l},\underline{\sgn})$ the $\underline{\sgn}$-twisted chain complex. Observe there are several group actions on $C^{\sf comb}_\bullet(M^\fr_{g,k,l},\underline{\sgn})$:
\begin{enumerate}
\item The symmetric group $\Sigma_k$ acts on it by permuting the labeling of cycles, with a sign twisted by the local system $\underline{\sgn}$.
\item The symmetric group $\Sigma_l$ acts on it by permuting the labeling of white vertices. 
\item We also have $(k+l)$ homological circle actions on it defined by sewing with the circle graph $B$ in Equation~\eqref{graph:circle}
%\[B=\begin{tikzpicture}[baseline={([yshift=-1.2ex]current bounding      box.center)},scale=0.5] 
%\draw [thick] (0,0) to (-2,0);
%\draw [thick] (-0.2, -0.2) to (0.2, 0.2);
%\draw [thick] (0.2, -0.2) to (-0.2, 0.2);
%\draw [thick] (-2.2,0) circle [radius=0.2];
%\draw [line width=2pt] (-4.4,0) to (-2.4,0);
%\end{tikzpicture} 
%\]
at cycles or at white vertices. More precisely, we set
\[
\begin{dcases}
B^c_i(G) := (-1)^{\deg(G)}G \prescript{}{i}\circ B, & \text{for} \quad 1\leq i\leq k;\\
B^w_j(G) := B \circ_j G, & \text{for} \quad 1\leq j\leq l.
\end{dcases}
\]

\end{enumerate}
We may take its homotopy quotient by the $(k+l)$-circle actions, which yields a chain complex
\begin{equation}\label{eq:equiv-chain}
	 \bigg( C^{\sf comb}_\bullet(M^\fr_{g,k,l},\underline{\sgn})[w_1^{-1},\ldots,w_k^{-1},u_1^{-1},\ldots,u_l^{-1}], \quad \partial +
     \sum_{i=1}^{k}w_i B^{c}_i+\sum_{j=1}^{l}w_j B^w_j
     %+w_1\cdot B^c_1+\cdots + w_k\cdot B^c_k + u_1\cdot B^w_1+\cdots +u_l\cdot B^w_l
     \bigg),\end{equation}
where $w_i$'s and $u_j$'s are the corresponding circle parameters, each of homological degree $-2$. The group $\Sigma_k\times \Sigma_l$ still acts on this chain complex. We denote its further quotient complex by $C^{\sf comb}_\bullet(M_{g,k,l}^{\sf fr}, \underline{\sgn})_{\sf hS}$. 

Putting various $g,k,l$'s together we obtain a graded vector space denoted by
\[\widehat{\mathfrak{g}}:= \bigoplus_{\substack{g\geq 0, k\geq 1, l\geq 0\\ 2g-2+k+l>0}} C^{\sf comb}_\bullet(M_{g,k,l}^{\sf fr}, \underline{\sgn})_{\sf hS}[2][[\hbar,\lambda]],\]
where the subscript ${\sf hS}$ denotes the homotopy quotient by the the $(k+l)$ circle actions followed by the quotient of the symmetric group $\Sigma_k\times \Sigma_l$. The formal variables $\hbar$ and $\lambda$ are both of homological degree $-2$. The variable $\hbar$ keeps track of the genus, while $\lambda$ keeps track of the Euler characteristic. In the following, using the gluing maps in Equation~\ref{eq:gluing-maps}, we shall write down a DGLA structure in the combinatorial model $\widehat{\mathfrak{g}}_c^\comb$ explicitly.
In~\cite{CalTu1,CCT}, a DGLA structure was constructed on $\widehat{\mathfrak{g}}$. Its differential is of the form $\eth+\iota+\hbar \Delta$. These operators are defined as follows.
\begin{enumerate}
    \item We denote the equivariant boundary map of the chain complex $C^{\sf comb}_\bullet(M_{g,k,l}^{\sf fr}, \underline{\sgn})_{\sf hS}$ by 
    \begin{equation}\label{eth-operator}
    \eth:=\partial +
    \sum_{i=1}^{k}w_i B^{c}_i+\sum_{j=1}^{l}w_j B^w_j
    \end{equation}
    \item Let $G$ be a black-and-white graph. For each $1\leq j\leq l$ we set:
    \begin{align*}
    	\iota_j(G):= (-1)^{\deg(G)}M \prescript{}{1}\circ_j G,
    \end{align*}
 We then extend $\iota_j$ linearly in the $u_j$-variable, i.e., $u_j\mapsto w_0$, and linearly in all other circle parameters. Here $w_0$ is the circle parameter of the cycle in $M \prescript{}{1}\circ_j G$ that was from the cycle in $M$ that is not used in sewing with $G$. The operator  $\iota$  is defined by 
    $$\iota:=\sum_{j=1}^{l} \iota_j: C^{\sf comb}_\bullet(M_{g,k,l}^{\sf fr},\underline{\sgn})_{\sf hS} \to C^{\sf comb}_\bullet(M_{g,k+1,l-1}^{\sf fr}, \underline{\sgn})_{\sf hS}$$ 

\item 
Let $\pi_{ij}$ be the map by setting the circle parameters $u_i^{-1}=u^{-1}_j=0$. 
By sewing with the thickened Mukai graph $\mathbb{M}$ in Equation~\eqref{graph:thick-mukai}, we define the {\em twisted self-sewing} operator 
\begin{equation}\label{eq:twisted-self-sewing}
\Delta:=\sum_{1\leq i<j\leq l} \mathbb{M} \circ_{\{i,j\}} \pi_{ij}: C^{\sf comb}_\bullet(M_{g,k,l+2}^{\sf fr}, \underline{\sgn})_{\sf hS} \ra C^{\sf comb}_\bullet(M_{g+1,k,l}^{\sf fr}, \underline{\sgn})_{\sf hS}
\end{equation}
\end{enumerate}

Next, we define the Lie bracket in $\hg^\comb$. 
Let $(g',g'',k',k'',l',l'')$ be a tuple that satisfies condition~\eqref{indices-mc-equation}.
For 
%For %two elements 
$$\alpha\in C^{\sf comb}_\bullet(M_{g',k',l'}^{\sf fr}, \mathbb{Q}^\epsilon)_{\sf hS}, \quad \beta \in C^{\sf comb}_\bullet(M_{g'',k'',l''}^{\sf fr}, \mathbb{Q}^\epsilon)_{\sf hS},$$ 
and each $r\geq 1$, 
we define
\begin{equation}\label{eq:def-r-bracket}
	\alpha \underset{[r]}{\circ} \beta :=  \frac{(-1)^{|\beta|(k'-r)}}{r!}
    %\sum_{\substack{I\subset \{1,\ldots,l''\}\\J\subset\{1,\ldots,k'\}\\|I|=|J|=r}} 
    \sum_{\substack{I\subset \{1,\ldots,l''\}\\|I|=r}} 
    \sum_{\substack{J\subset\{1,\ldots,k'\}\\|J|=r}}
    \pi_J(\alpha)\prescript{}{J}\circ B^{\otimes r} \circ_I \;\pi_I(\beta) \in  C^{\sf comb}_\bullet(M_{g,k,l}^{\sf fr}, \mathbb{Q}^\epsilon)_{\sf hS}.
\end{equation}
%let us define a map \begin{align*}\underset{[r]}{\circ}: & C^{\sf comb}_\bullet(M_{g',k',l'}^{\sf fr}, \mathbb{Q}^\epsilon)_{\sf hS} \otimes C^{\sf comb}_\bullet(M_{g'',k'',l''}^{\sf fr}, \mathbb{Q}^\epsilon)_{\sf hS} \ra C^{\sf comb}_\bullet(M_{g,k,l}^{\sf fr}, \mathbb{Q}^\epsilon)_{\sf hS}\end{align*}
Here $B$ is the circle graph~\eqref{graph:circle}. Then, we define the Lie bracket map $\{-,-\}_\hbar$ by
\begin{equation}\label{eq:def-bracket-hbar}
\{\alpha,\beta\}_\hbar := \sum_{r\geq 1} \{\alpha,\beta\}_r \hbar^{r-1}
:=  \sum_{r\geq 1} (-1)^{|\alpha|}\big( \alpha\underset{[r]}{\circ}\beta - (-1)^{|\alpha||\beta|}\beta\underset{[r]}{\circ}\alpha\big) \hbar^{r-1}.
%\begin{split}
%\{\alpha,\beta\}_\hbar &:= \sum_{r\geq 1} \{\alpha,\beta\}_r \hbar^{r-1},\\
%\{\alpha,\beta\}_r &:=  \sum_{r\geq 1} (-1)^{|\alpha|}\big( \alpha\underset{[r]}{\circ}\beta - (-1)^{|\alpha||\beta|}\beta\underset{[r]}{\circ}\alpha\big).
%\end{split}
\end{equation}
The sign $(-1)^{|\alpha|}$ in the front is due to that the map $\{-,-\}_\hbar$ is defined on the shited DGLA $\hg^\comb[1]$. We refer the details and the proof that these operations indeed form a DGLA structure on $\hg^\comb$ to~\cite{CalTu1}.

\begin{Theorem}\label{thm:comb-string-vertex}\cite[Theorem 5.9]{CalTu1}
There exists a unique, up to gauge equivalence, Maurer-Cartan element of $\hg$ of the form \begin{equation}
\label{string-vertex}
\hcV:= \sum_{g,k,l} \widehat{\mathcal{V}}_{g,k,l} \hbar^g \lambda^{2g-2+k+l}
\end{equation}
such that
\[\widehat{\mathcal{V}}_{0,1,2}=\begin{tikzpicture}[baseline={([yshift=-1.2ex]current bounding
      box.center)},scale=0.3] 
\draw [thick] (0,0) to (0,2);
\draw [thick] (-0.2, 1.8) to (0.2, 2.2);
\draw [thick] (0.2, 1.8) to (-0.2, 2.2);
\draw [thick] (0,0) to (-2,0);
\draw (-2,0) node[label=left:{$\frac{1}{2}\;$}] {};
\draw [thick] (0,0) to (2,0);
\draw [thick] (-2.2,0) circle [radius=0.2];
\draw [thick] (2.2,0) circle [radius=0.2];
\end{tikzpicture}\]
The factor $\frac{1}{2}$ is due to the orbifold structure of the moduli space $M_{0,1,2}$. 
\end{Theorem}

%\noindent 
String vertices are defined as a collection of chains $\{\hcV_{g,k,l}\}_{(g,k,l)}$ such that the element in~\eqref{string-vertex} is the unique (up to gauge equivalence) solution in the theorem above. 
When writing down in components, the Maurer-Cartan equation 
\[ (\eth+\iota+\hbar\Delta)\hcV+\frac{1}{2}\{\hcV,\hcV\}_\hbar =0\]
is equivalent to the following system of equations:
\begin{equation}\label{eq:mc}
\eth\widehat{\mathcal{V}}_{g,k,l}+\iota\widehat{\mathcal{V}}_{g,k-1,l+1}+\Delta\widehat{\mathcal{V}}_{g-1,k,l+2}+\frac{1}{2}\sum\{\widehat{\mathcal{V}}_{g',k',l'},\widehat{\mathcal{V}}_{g'',k'',l''}\}_r=0.
 \end{equation}
Here, the summation $\Sigma$ is over all $r\geq 1$ and all $(g', g'', k', k'', l',
l'')$ that satisfy~\eqref{indices-mc-equation}.

\subsection{Topological conformal field theories}\label{subsec:tcft} 
In this subsection, we recall the second ingredient in the construction of CEI. Let $A$ be a finite dimensional, strictly unital $A_\infty$-algebra, over the field $\bbC$. Denote the (shifted) $A_\infty$ structure maps by 
\[ m_k: A[1]^{\otimes k} \to A[1].\]
Recall that strict unitality means that there exists an element $\bone_A\in A$ such that
\begin{align}\label{eq:unitality}
	\begin{dcases}
	m_2(\bone_A,a)  =(-1)^{|a|}m_2(a,\bone_A)=a, & \\
	m_k(\ldots,\bone_A,\ldots)=0, & \mbox{for each $k\geq 3$.}
	\end{dcases}
\end{align}
We use the notations $|-|$ and $|-|'$ to denote the degree in $A$ and its shifted degree in $A[1]$ respectively, i.e., $|-|'=|-|+1 \pmod{2}$. Recall the construction of Hochschild invariants associated with $A$. Denote by $\overline{A}:= A/(\bbC\cdot \bone_A)$ the quotient of $A$ by the one dimensional subspace generated by the unit element. Then we may define the reduced Hochschild chain complex
\begin{equation}\label{eq:hoch-chain}
C_\bullet(A):=\bigoplus_{r \geq 0} A \otimes \overline{A}^{\otimes r},
\end{equation}
with the grading $|a_0\otimes a_1\otimes \ldots\otimes a_r|:= |a_0|+ |a_1|' +  \ldots |a_r|'$. 
We denote by 
$a_0|a_1 \ldots a_r:=a_0\otimes a_1 \otimes\ldots\otimes a_r.$
The Hochschild differential is given by 
\begin{align*} 
b(a_0|a_1 \ldots a_r) =&  \sum_{j=1}^{r-i+1}\sum_{i=1}^{r} (-1)^{\star}a_0| a_1 \ldots m_j(a_i,\ldots, a_{i+j-1}) \ldots a_r\\
 & + \sum_{0\leq i\leq j\leq r} (-1)^@ m_{r-j+i+1}(a_{j+1}, \ldots, a_r, a_0, a_1, \ldots, a_i)|a_{i+1}\ldots a_j,
\end{align*}
where $\star=|a_{0,i-1}|':=|a_0|'+\ldots |a_{i-1}|'$ and $@=|a_{0,j}|'|a_{j+1,r}|'$ is the Koszul sign.

A {\em cyclic structure} on $A$ of dimension $d$ is a non-degenerate pairing $\langle -,-\rangle:A^{\otimes 2}\to \bbC[-d]$ that satisfies 
\begin{equation}\label{eq:cyclic-structure}
\begin{dcases}
\langle a_1, a_2\rangle= (-1)^{|a_1|'|a_2|'}\langle a_2, a_1 \rangle, \\
\langle m_k(a_0, \ldots, a_{k-1}), a_k\rangle=(-1)^{|a_0|'(|a_1|'+\ldots+|a_k|')}\langle m_k(a_1, \ldots, a_{k}), a_0\rangle.
\end{dcases}
\end{equation}
%where $\heartsuit=|a_0|'(|a_1|'+\ldots+|a_k|')$. 

The following is a result of Costello~\cite{Cos1}, Kontsevich-Soibelman~\cite{KonSoi} and Wahl-Westerland~\cite{WahWes}.
\begin{Theorem}\label{thm:tcft}
	 Let $A$ be a finite dimensional, cyclic, and strictly unital $A_\infty$-algebra of dimension $d$, then its shifted Hochschild chain complex $L^A:=C_\bullet(A)[d]$ carries a $2$-dimensional Topological Conformal Field Theory (TCFT) structure, i.e., there are action maps compatible with sewing operations:
	\[ \rho^A_{g,k,l}: C^{\sf comb}_\bullet(M_{g,k,l}^{\sf fr}) \ra {\sf Hom} \big( (L^A)^{\otimes k}, (L^A)^{\otimes l}\big)\]
	with $g\geq 0$, $k\geq 1$, $l\geq 0$, and $2g-2+k+l\geq 0$.
\end{Theorem}

%\vspace{.5cm}
Observe that in a TCFT structure, the number of inputs $k$ should be strictly positive. A TCFT structure carries a rather rich amount of data.  In the following, we consider a few examples. When the context is clear, we shall drop the subscripts $(g,k,l)$ from $\rho^A_{g,k,l}$. 

\begin{example}
We give some examples where $2g-2+k+l=0$. There are two cases: $(g,k,l)=(0,1,1)$ or $(g,k,l)=(0,2,0)$.  We will use the symbol $@$ for the obtained sign, following the Koszul convention for the shifted degrees, by rotating the inputs of the expression from their original order. 

In the first case where $(g,k,l)=(0,1,1)$, let us consider the graph in Equation~\eqref{graph:circle}.  Applying $\rho^A$ to it yields a map denoted by $B: L^A \ra L^A$. This operator yields the well-known Connes' operator~\cite[Section 11]{KonSoi} on the reduced Hochschild chain complex:
\[ B(a_0|a_1 \ldots  a_r)= \sum (-1)^{@} \bone_A|a_i\ldots,a_ra_0a_1\ldots a_{i-1}.\]
Since $\rho^A$ is a chain map, we have $bB+Bb=0$ where $b$ is the Hochschild differential on $L^A$. Furthermore, since $\rho^A$ is also compatible with sewing operations, we also have $B^2=0$.  Such an operator on a chain complex is usually referred to as a circle operator. Let $u$ be a formal variable of homological degree $-2$, one may define
\begin{enumerate}
    \item periodic cyclic chain complex $L^A((u))$ endowed with differential $b+uB$,
    \item negative cyclic chain complex $L^A_+:= L^A[[u]]$, a subcomplex of $L^A((u))$,
    \item cyclic chain complex $L^A_-:= L^A[u^{-1}]$ by the quotient complex $L^A((u))/u\cdot L^A_+$.
\end{enumerate}

In the case where $(g,k,l)=(0,2,0)$, applying $\rho^A$ to the Mukai graph~\eqref{graph:mukai}, we get a chain-level pairing (known as the Mukai pairing) given by:
\begin{equation}\label{eq:mukai-pairing}
\langle-,-\rangle_{\sf Muk}:=\rho^A(M):L^A\otimes L^A \ra \mathbb{C}.\end{equation}
For later usage, we may also extend the Mukai pairingto a pairing between  $L^A_-$ and $L^A_+$ by setting 
\begin{equation}\label{eq:generalized-Mukai-chain-level}
\langle 
x u^{i}, 
y u^{-j}
\rangle_{\Muk}:= 
\delta_i^j(-1)^j\langle x,y\rangle_\Muk.
\end{equation}
As another example, when $(g,k,l)=(0,1,2)$, applying $\rho^A$ to the T-graph~\eqref{graph:T} we obtain a map 
\[ \rho^A(T) : L^A \ra L^A\otimes L^A.\]
This map is the linear dual (using the cyclic pairing on $A$) to the cup product on the Hochschild cochain complex of $A$~\cite[Proposition 6.9]{WahWes}.
\end{example}

We may twist the map $\rho^A_{g,k,l}$ by the local system $\underline{{\sf sgn}}$ to obtain a map
\[ \rho^{A,\tw}_{g,k,l}: C^{\sf comb}_\bullet(M_{g,k,l}^{\sf fr},\underline{{\sf sgn}}) \ra {\sf Hom} \big( (L^A[1])^{\otimes k}, (L^A)^{\otimes l}\big).\]
Explicitly, this map is defined by setting
\begin{align*} 
\rho_{g,k,l}^A \big( G, p_1\wedge\cdots\wedge p_k \big) &:= \rho_{g,k,l}^A(G) \circ s^{\otimes k}
\end{align*}
where $G$ is a black-and-white graph of type $(g,k,l)$, $p_1\wedge\cdots\wedge p_k$ denotes a trivialization of the local system $\underline{{\sf sgn}}$ given by the the input cycles of $G$, and $s: L^A[1] \ra L^A$ is the shift map. Then, we take homotopy quotients on both sides with respect to the $(S_k\times S_l)\ltimes (S^1)^{k+l}$-action to obtain a map still denoted by
\[ \rho_{g,k,l}^{A,{\sf tw}}: C^{\sf comb}_*(M_{g,k,l}^{\sf fr},\underline{\sgn})_\hS \ra {\sf Hom}^\cont \big( \sym^k(L^A_+[1]), \sym^l L^A_-\big).\]
Here, note that taking the $S^1$ quotient of the complex $L^A$ normally yields $L^A_-$. However, since the complex $L^A[1]$ appears in the input place, taking the $S^1$ quotient yields $L^A_+[1]$. And the notation ${\sf Hom}^\cont$ stands for continuous maps in the $u$-adic topology, i.e. linear maps which vanish on elements with large enough $u$-power.
More explicitly, in the case of $k=1$, this action is given by
\begin{equation}\label{eq:action-map}
\rho_{g,1,l}^{A,{\sf tw}}\big( \alpha \cdot u_1^{-a_1}\cdots u_l^{-a_l} w^{-b}\big) (u^{c} x) 
:= %\begin{dcases}
\delta_b^c\ 
p\Big( (-1)^{b} \rho_{g,1,l}^{A,{\sf tw}} (\alpha) (x) \cdot u_1^{-a_1}\cdots u_l^{-a_l}\Big), 
%& \mbox{\; if \;} b=c;\\0, & \mbox{\; otherwise. \;}
%\end{dcases}
\end{equation}
where $p: L_-^{\otimes l} \ra \sym^l L_-$ is the canonical projection map and $\delta_{b}^{c}$ is the Kronecker Delta. For general $k\geq 1$, the map is defined similarly with symmetrization at the inputs.

Putting all these maps together, we obtain a map denoted by
\begin{equation}\label{eq:rho-tw}
 \rho^{A,\tw}: \widehat{\mathfrak{g}} \ra \widehat{\mathfrak{h}}_A:= \bigoplus_{k\geq 1,l\geq 0} {\sf Hom}^\cont \big( \sym^k (L^A_+[1]) , \sym^l (L^A_-) \big)[[\hbar,\lambda]].
 \end{equation}
 The upshot is that by the same construction as in Section~\ref{subsec:cei-form}, we may define a DGLA structure on the right hand side $\widehat{\mathfrak{h}}_A$ so that $\rho^{A,\tw}$ becomes a morphism of DGLA's, see~\cite{CCT} for the details. Indeed, the differential of $\widehat{\mathfrak{h}}_A$ is also of the form $\eth+\iota+\hbar \Delta$, with each operator defined as follows.
\begin{enumerate}
	\item The operator $\eth:=b+uB$ denotes the extension of the equivariant boundary map $b+uB$ on $L^A_-$ or $L^A_+[1]$ to their symmetric powers by Leibniz rule, and on ${\sf Hom} \big( \sym^k (L^A_+[1]) , \sym^l (L^A_-) \big)$ by commutator. 
	\item The operator $\iota: {\sf Hom}^\cont \big( \sym^k (L^A_+[1]) , \sym^l (L^A_-) \big) \ra {\sf Hom}^\cont \big( \sym^{k+1} (L^A_+[1]) , \sym^{l-1} (L^A_-) \big)$ is essentially defined by sewing with the Mukai paring. More precisely, for an element $y\in L^A_+[1]$, we define a construction operator $C_{y}: \sym^l (L^A_-)\to \sym^{l-1}(L^A_-)$ by 
    \[C_{y}(x_1\odot\cdots\odot x_l):= \sum_{i=1}^l(-1)^{|x_i|(|x_1|+\cdots+|x_{i-1}|)}\langle y,x_i\rangle_\Muk\cdot  x_1\odot\cdots\widehat{x_i}\cdots\odot x_l. \]
    Then, for $\alpha \in {\sf Hom}^\cont \big( \sym^k (L^A_+[1]), \sym^l (L^A_-) \big)$, we set
	\begin{equation}
    \label{eq-iota-tcft}
    \iota(\alpha)(y_1\odot\cdots\odot y_{k+1}):= \sum_{j=1}^{k+1}(-1)^{|y_j|(|y_1|+\cdots+|y_{j-1}|+|\alpha|)}C_{y_j}\left(\alpha(y_1\odot\cdots\widehat{y_j}\cdots\odot y_{k+1})\right).
    \end{equation}
	\item The operator $\Delta: {\sf Hom}^\cont \big( \sym^k (L^A_+[1]) , \sym^{l+2}(L^A_-) \big) \ra {\sf Hom}^\cont \big( \sym^{k} (L^A_+[1]) , \sym^{l} (L^A_-) \big)$ is defined by sewing with $\rho^A$ applied to the thickened Mukai graph $\mathbb{M}$ in Equation~\eqref{graph:thick-mukai}, i.e. 
    \begin{equation}\label{eq-delta-tcft}
	\Delta(\alpha):= \sum_{1\leq i<j\leq l+2} \rho^A(\mathbb{M}) \circ_{\{i,j\}}  \pi_{ij}(\alpha),\end{equation}
	where the map $\pi_{ij}$ sets the circle parameters $u_i^{-1}=u^{-1}_j=0$.  Again, composition $ \circ_{\{i,j\}}$ means taking the $i$-th and $j$-th outputs of $\alpha$ as the inputs of $\rho^A(\mathbb{M})$. Algebraically, the operator $\Delta$ is defined by contracting the following symmetric bilinear form on $L^A_-$:
    \[\Delta(x\odot y):= \langle B x_0,y_0\rangle_\Muk, \]
    if $x=x_0+x_{-1}u^{-1}+\cdots$, and $y=y_0+y_{-1}u^{-1}+\cdots$.
	\item Furthermore, the Lie bracket of $\widehat{\mathfrak{h}}_A$ is defined in a way similar to that in Equation~\eqref{eq:def-r-bracket}. Indeed, for each $r\geq 1$, let us define a map $\underset{[r]}{\circ}$ of the form
	\begin{align*}
	 & {\sf Hom}^\cont \big( \sym^{k'} (L^A_+[1]) , \sym^{l'} (L^A_-) \big) \otimes {\sf Hom}^\cont \big( \sym^{k''} (L^A_+[1]) , \sym^{l''} (L^A_-) \big) \to  {\sf Hom}^\cont \big( \sym^k (L^A_+[1]) , \sym^l (L^A_-) \big)
	\end{align*}
	(with $k=k'+k''-r$ and $l=l'+l''-r$) by setting
	\begin{equation}\label{eq:def-r-bracket-algebraic}
		\alpha \underset{[r]}{\circ} \beta :=  \frac{(-1)^{|\beta|(k'-r)}}{r!}
        %\sum_{\substack{I\subset \{1,\ldots,l''\}\\J\subset\{1,\ldots,k'\}\\|I|=|J|=r}} 
        \sum_{\substack{I\subset \{1,\ldots,l''\}\\|I|=r}} 
    \sum_{\substack{J\subset\{1,\ldots,k'\}\\|J|=r}}
    \pi_J(\alpha)\prescript{}{J}\circ B^{\otimes r}\circ_I \;\pi_I(\beta)
	\end{equation}
Then, we define the Lie bracket map $\{-,-\}_\hbar$ by
\begin{equation}\label{eq:def-bracket-hbar-algebraic}
\{\alpha,\beta\}_\hbar := \sum_{r\geq 1} \{\alpha,\beta\}_r \hbar^{r-1}
:=  \sum_{r\geq 1} (-1)^{|\alpha|}\big( \alpha\underset{[r]}{\circ}\beta - (-1)^{|\alpha||\beta|}\beta\underset{[r]}{\circ}\alpha\big) \hbar^{r-1}.
%\begin{split}
%\{\alpha,\beta\}_\hbar &:= \sum_{r\geq 1} \{\alpha,\beta\}_r \hbar^{r-1},\\
%\{\alpha,\beta\}_r &:=  (-1)^{|\alpha|}\big( \alpha\underset{[r]}{\circ}\beta - (-1)^{|\alpha||\beta|}\beta\underset{[r]}{\circ}\alpha\big).
%\end{split}
\end{equation}
\end{enumerate}

Since the DGLA structures of $\hg$ and $\hh_A$ are constructed in a complete parallel manner, it follows that the map $\rho^{A,\tw}$ is indeed a DGLA morphism. The image of the combinatorial string vertices under the action map $\rho^{A,\tw}$ in Equation~\eqref{eq:rho-tw} yields a collection of multi-linear maps
\begin{equation}\label{eq:image-vertex} 
\widehat{\beta}^{A}_{g,k,l}:= \rho^{A,\tw}_{g,k,l}(\widehat{\mathcal{V}}_{g,k,l})\quad\in\quad {\sf Hom}^\cont \big( \sym^k (L^A_+[1]) , \sym^l (L^A_-) \big).\end{equation}
These multilinear maps will serve as building blocks of CEI. They are packaged to a Maurer-Cartan element \begin{equation}\label{hat-beta-A}
\widehat{\beta}^{A}= \sum_{g,k,l} \widehat{\beta}^{A}_{g,k,l} \hbar^g\lambda^{2g-2+k+l}\quad 
\in \quad \widehat{\mathfrak{h}}_A.
\end{equation}

\subsection{Explicit formula of CEI}\label{para:trivialization} 

To summarize the previous discussion, we have obtained
\begin{itemize}
    \item a canonical (up to gauge equivalences) Maurer-Cartan element $\hcV\in \hg$ called {\em combinatorial string vertices}, 
    \item a morphism $\rho^{A,\tw}: \widehat{\mathfrak{g}} \ra \hh_A$ of DGLA's~\eqref{eq:rho-tw}.
\end{itemize}
Now, let us recall the definition of CEI from \cite{CalTu1}.

\subsubsection{Splittings of the non-commutative Hodge filtration.} 
To obtain CEI we also need to choose a splitting of the nc-Hodge filtration of $A$, see~\cite{Cos2,CLT,CalTu2}. Recall from the previous subsection that $L^A=C_\bullet(A)[d]$ is the shifted Hochschild chain complex of $A$, and $L^A_+=L^A[[u]]$ is the shifted negative cyclic chain complex of $A$. The latter is endowed with the equivariant differential $b+uB$. Observe that the natural projection map $L^A_+ \ra L^A$ is a map of complexes. The Mukai pairing induces, by sesquilinear extension, the chain-level higher residue pairing 
$\langle-,-\rangle_{\sf hres}  : L^A_+ \otimes L^A_+ \to \mathbb{C}[[u]]$
with 
\begin{equation}
%	\langle-,-\rangle_{\sf hres} & : L^A_+ \otimes L^A_+ \to \mathbb{C}[[u]],\\
\label{higher-residue-pairing}
\langle u^k x, u^l y\rangle_{\sf hres}:= (-1)^k\langle x, y\rangle_{\sf Muk}\cdot u^{k+l}.
\end{equation}
This defines a chain map (see \cite[Section 5]{She}) that yields the higher residue pairing still denoted by
$$\langle-,-\rangle_\hres : H_\bullet(L^A_+)\otimes H_\bullet(L^A_+) \to \mathbb{C}[[u]].$$

\begin{Definition}
\label{def:nc-splitting}
A splitting of the non-commutative Hodge filtration of $A$ is given by
a map of $\Z$-graded vector spaces
$ s:H_\bullet(L^A)\ra H_\bullet(L^A_+) $
satisfying the following two conditions:
\begin{itemize}
\item[S1.] {\em (Splitting condition.)} The map $s$ splits the canonical projection
  $H_\bullet(L^A_+) \ra H_\bullet(L^A)$.
\item[S2.] {\em (Lagrangian condition.)} For any $x,y\in H_\bullet(L^A)$, we have 
  $\langle s(x),\, s(y)\rangle_{\sf hres} = \langle
  x,\,y\rangle_\Muk.$
\end{itemize}
\end{Definition}

We remark that, under our assumptions on $A$, one can show that such a splitting always exists. This essentially (plus a bit of linear algebra to deal with the Lagrangian condition) follows from Kaledin's resolution~\cite{Kal} of the Kontsevich-Soibelman's Hodge-to-de-rham degeneration conjecture in the $\mathbb{Z}$-graded case.

\subsubsection{Explicit combinatorial formula of CEI}
\label{sec-explicit-cei}
We shall make use of the splitting data $s$ to ``trivialize" the DGLA structure on $\hh_A$.  
Let us denote by $\hh_A^{\sf TRIV}$ the DGLA that has the same underlying graded vector space as $\hh_A$, but is endowed with differential $b+\iota$ and the zero Lie bracket. In other words, the operators $B$, $\Delta$ and $\{-,-\}_\hbar$ defined using the Connes operator are all removed from $\hh_A$. Given a splitting $s$ of the nc-Hodge filtration of $A$, in~\cite{CalTu1} a ``trivialization" $L_\infty$ isomorphism
\begin{equation}\label{eq:Triv-map} 
\cK: \hh_A \ra \hh_A^{\sf TRIV}
\end{equation}
is constructed using an explicit stable graph sum; see Equation~\eqref{eq:trivialization-map-K} below. Thus, we may consider the composition of $L_\infty$ morphisms:
\[ \hg \stackrel{\rho^{A,\tw}}{\longrightarrow} \hh_A \stackrel{\cK}{\longrightarrow} \hh_A^{\sf TRIV}.\]
The explicit formula of CEI obtained in~\cite{CalTu1} is simply given by the push-forward formula %Maurer-Cartan element denoted by
\[\lbA:= \cK_* \rho^{A,\tw}_* \hcV = \cK_*\hbA.\] 
Note that since $\hcV$ is unique up to gauge equivalences, the cohomology class $[\lbA]\in H^*\big(\hh_A^{\sf TRIV}\big) $ depends only on the cyclic $A_\infty$-algebra and the choice of splitting data $s$. Let us write down the element $\lbA$ in components. By construction, it is of the following form: 
\begin{equation}\label{eq:lbA}
%\begin{split}
    \lbA  = \sum_{g,k,l} \lbA_{g,k,l}\hbar^g\lambda^{2g-2+k+l}.
    %\\\lbA_{g,k,l} &:= \lbA_{g,k,l}.
%\end{split}
\end{equation} 
Thus, by fixing a stable pair of integers $(g,n)$, the following element 
\[\sum_{k=1}^n \lbA_{g,k,n-k}\quad \in \quad \bigoplus_{k=1}^n {\sf Hom}^\cont \big( \sym^k (L^A_+[1]) , \sym^{n-k} L^A_- \big)\]
is a $(b+\iota)$-closed element in the chain complex. 
In particular, its first component $\lbA_{g,1,n-1}$
%$$\lbA_{g,1,n-1}\in {\sf Hom}^\cont \big( L^A_+[1] , \sym^{n-1} L^A_- \big)$$
is $b$-closed. 

Let $S_{n-1}$ be the symmetric group of $n-1$ elements.
We extend the Mukai pairing~\eqref{eq:mukai-pairing} to a pairing between  $\sym^{n-1}\big(H_\bullet(L^A)[u^{-1}]\big)$ and $\sym^{n-1}\big(H_\bullet(L^A)[u]\big)$ by setting 
\begin{equation}\label{eq:generalized-Mukai}
\bigg\langle 
%\beta_1u^{-l_1}\cdots\beta_{n-1}u^{-l_{n-1}},
\bigodot\limits_{i=1}^{n-1}\beta_i u^{-l_i}, 
%\alpha_1(-u)^{k_1}\cdots\alpha_{n-1}(-u)^{k_{n-1}}
\bigodot\limits_{i=1}^{n-1}\alpha_i(-u)^{k_i}
\bigg\rangle_{\Muk}:= \sum_{\sigma\in S_{n-1}} \prod_{j=1}^{n-1}
\delta_{l_j}^{k_{\sigma(j)}}\langle \beta_j,\alpha_{\sigma(j)}\rangle_\Muk,
\end{equation}
where $\bigodot$ denotes the symmetric product. The main result proved in~\cite[Theorem 1.3]{CalTu1} is the following

\begin{Theorem}
{\bf (1.)} The push-forward Maurer-Cartan element $\lbA_{g,1,n-1}$ has a combinatorial formula:
    \begin{equation}\label{eq:cei-formula-main}
 \lbA_{g,1,n-1} = \sum_{\GG\in
    \Gamma((g,1,n-1))} \frac{\wt(\GG)}{|\Aut(\GG)|}\rho^{A,\tw}_\GG \in {\sf Hom}^\cont \big( L^A_+[1] , \sym^{n-1} L^A_- \big).
  \end{equation}
{\bf (2.)} Let $\alpha_i\in H_\bullet(L^A)$ and $k_i\geq 0$ for $1\leq i\leq n$, then the CEI may be computed as
\begin{align}
\langle\alpha_1\psi^{k_1},\ldots,\alpha_n\psi^{k_n}\rangle_{g,n}^{A,s}
=&\bigg\langle 
\lbA_{g,1,n-1}
\big(\alpha_n (-u)^{k_n}\big),%\alpha_1(-u)^{k_1}\cdots\alpha_{n-1}(-u)^{k_{n-1}}
\bigodot\limits_{i=1}^{n-1}\alpha_i(-u)^{k_i}\bigg\rangle_{\Muk}
\nonumber\\
=&\sum_{\GG\in
    \Gamma((g,1,n-1))} \frac{\wt(\GG)}{|\Aut(\GG)|}
    \bigg\langle 
\rho^{A,\tw}_\GG
\big(\alpha_n (-u)^{k_n}\big),%\alpha_1(-u)^{k_1}\cdots\alpha_{n-1}(-u)^{k_{n-1}}
\bigodot\limits_{i=1}^{n-1}\alpha_i(-u)^{k_i}\bigg\rangle_{\Muk}.
\label{eq:explicit-evaluation}
\end{align} 
Note that the formal notation $\psi$ (chosen to match with notation used in Gromov-Witten theory) is replaced by $(-u)$ in the negative cyclic chain complex $L^A_+[1]$ on the right hand side.
\end{Theorem}

%In the first part of this theorem, 
Here $\Gamma((g,1,n-1))$ is the set of isomorphism classes of certain {\em partially direct graphs} (Definition~\ref{defi:dir-graph}), %the rational number 
$\wt(\GG)\in \mathbb{Q}$ is a weight function defined in Equation~\eqref{eq:wt-def}, $\Aut(\GG)$ is the automorphism group of a partially directed graph $\GG$, and $\rho^{A,\tw}_\GG\in {\sf Hom}^\cont \big( L^A_+[1] , \sym^{n-1} L^A_- \big)$
is a certain linear map constructed by composing along the graph $\GG$. 
We now explain the details of the notations below.
%The precise meaning of these notions will occupy the remaining part of this section.

\subsubsection{Partial directed graphs}

Now, we proceed to give a more formal account of the discussion above, beginning with the notion of partially directed graphs. We warn the reader that this notion should not to be confused with the notion of black-and-white graphs. In fact the graph sum in Equation~\eqref{eq:cei-formula-main} may be viewed as a sum of partially directed graphs (to be defined below) whose vertices are decorated by string vertices which are linear combinations of black-and-white graphs.

Following~\cite[Section 2]{GetKap}, by a labeled
graph we shall mean a graph $G$ (possibly with leaves) endowed with a
genus labeling function $g: V_G\ra \bbZ_{\geq 0}$ on the set of its
vertices $V_G$. The genus of a labeled graph is defined to be
\[ g(G) = \sum_{v\in V_G} g(v)+{\sf rank\,} H_1(G). \]
We will use the following notation for a labeled graph $G$: $L_G$ denotes the set of leaves of $G$; $n(v)$ denotes the valence of a vertex $v\in V_G$. A labeled graph $G$ is called stable if $2g(v)-2+n(v)>0$ for 
every vertex $v$.

\begin{Definition}
\label{defi:dir-graph}
A partially directed graph of type $(g,k,l)$ is given by a quadruple
$$\GG = \left(G,L_G^{\sf in}\coprod L_G^{\sf out},E^\direct,T\right)$$ 
consisting of the
following data:
\begin{itemize}
\item A labeled graph $G$ of type $(g,k+l)$.
\item A decomposition
$L_G = L_G^{\sf in}\coprod L_G^{\sf out}$
of the set of leaves
$L_G$ such that $|L_G^{\sf in}|=k$ and $|L_G^{\sf out}|=l$. Leaves in
$L_G^{\sf in}$ will be called incoming, while leaves in $L_G^{\sf
  out}$ will be called outgoing.
\item A subset $E^{\direct}\subset E_G$ of edges of $G$ whose
elements are called directed edges, and a direction is chosen on
them. Edges in $E_G-E^{\direct}$ are called undirected.
\item A spanning tree $T\subset E_G$ of the graph $G$.
\end{itemize}

We require the following properties to hold:
\begin{itemize}
\item There is no directed loop.
\item Each vertex has at least one incoming half-edge.
\end{itemize}
  
A partially directed graph is called stable if the underlying
labeled graph is.
\end{Definition}

%\vspace{.5cm}
In the examples preceding the above definition, either we have an obvious spanning tree in the graph (in particular the genus zero case) or edges in the spanning tree is drawn thicker. The reader may verify that indeed, all graphs in the two examples above are partially directed graphs.

%\vspace{.5cm}
Secondly, let us define the weight function $\wt$. Let $\GG=(G,L_G^{\sf in}\coprod L_G^{\sf out},E^\direct,T)$ be a partially directed graph.
Let $e\in T$ be an edge from the spanning tree $T$. We call it
contractible if the graph obtained from $\GG$ by contracting all the
directed edges connecting the two end points of $e$ (if $e$ itself is not directed, we also
contract $e$) is again a valid
partially directed graph. Denote the resulting partially directed
graph by $\GG/e$. Denote by $T^{\sf contr}\subset T$ the set of
contractible edges, and by $E_G^{\sf non-loop}$ the set of non-loop
edges of $G$.

We inductively define $\wt(\GG)$ as follows:
\begin{enumerate}
\item If $|T|=\emptyset$ (equivalently, if $\GG$ only has one vertex)
  then $\wt(\GG)=1$. 
\item 
%In genus zero, the weights are given by $\wt(\mathbb{G})=1$, see~\cite[Section 8.3]{CalTu1}.
In general we set
\begin{equation}\label{eq:wt-def} 
\wt(\GG):= \frac{1}{|E_G^{\sf non-loop}|}\sum_{e\in T^{\sf contr}} \wt(\GG/e).\end{equation}
\end{enumerate}

%In genus zero, the weights are given by $\wt(\mathbb{G})=1$, see~\cite[Section 8.3]{CalTu1}.

In genus zero, the weights of all partially directed graphs are equal to $1$, see~\cite[Section 8.3]{CalTu1}.

\begin{example}
We give some examples of Equation~\eqref{eq:cei-formula-main}.
In the case $g=0$, $n=4$, %all the weights are given by $\wt(\mathbb{G})=1$, 
we have
\begin{equation}\label{eq:lbA-013}
    \lbA_{0,1,3}= \frac{1}{3!} \begin{tikzpicture}[baseline={(current bounding
    		box.center)},scale=0.5] \draw [thick,directed] (3.4,4) to (3.4,2);
    	\node at (3.4,2) {$\bullet$};
    	\draw[thick,directed] (3.4,2) to (1.4,0);
    	\draw[thick,directed] (3.4,2) to (3.4,0);
    	\draw[thick,directed] (3.4,2) to (5.4,0);
    \end{tikzpicture} +\frac{1}{2!} \begin{tikzpicture}[baseline={(current bounding
    	box.center)},scale=0.5] \draw [thick,directed] (3.4,4) to (3.4,2);
    \node at (3.4,2) {$\bullet$};
    \draw[thick,directed] (3.4,2) to (1.4,0);
        \draw[thick,directed] (3.4,2) to (4.4,1);
    \draw[thick,directed] (4.4,1) to (5.4,0);
        \draw[thick,directed](4.4,1)  to (3.4,0);
         \node at (4.4,1) {$\bullet$};
    \end{tikzpicture}
\end{equation}

In the case $g=1$, $n=1$, we have
\begin{equation}\label{eq:lbA-110}
\lbA_{1,1,0}= \begin{tikzpicture}[baseline={(current bounding
box.center)},scale=0.6] \draw [thick,directed] (3.4,4) to (3.4,2);
\node at (3.4,2) {$\bullet$}; \node at (2.4,2) {\tiny $g=1$};\end{tikzpicture}
+\frac{1}{2}\begin{tikzpicture}[baseline={(current bounding
box.center)},scale=0.6] \draw [thick,directed] (3.4,4) to (3.4,2);
\node at (3.4,2) {$\bullet$}; \draw [thick] (3.4,1.5) circle
[radius=.5]; 
\end{tikzpicture}
\end{equation}

%\item 
In the case $g=1$, $n=2$, we express $\lbA_{1,1,1}$ as the following graph sum:
\begin{align*} 
\begin{tikzpicture}[baseline={(current bounding
box.center)},scale=0.6] \draw [thick,directed] (3.4,4) to (3.4,2);
\node at (3.4,2) {$\bullet$}; \node at (2.4,2) {\tiny $g=1$}; \draw
[thick,directed] (3.4,2) to (3.4,0);
\end{tikzpicture}
+\frac{1}{2}\begin{tikzpicture}[baseline={(current bounding
box.center)},scale=0.6] \draw [thick,directed] (3.4,4) to (3.4,2);
\node at (3.4,2) {$\bullet$}; \draw [thick] (3.9,2) circle
[radius=.5]; \draw [thick,directed] (3.4,2) to (3.4, 0);
\end{tikzpicture}
+ \begin{tikzpicture}[baseline={(current bounding
box.center)},scale=0.6] \draw [thick,directed] (3.4,4) to (3.4,2);
\node at (3.4,2) {$\bullet$}; \draw [thick,directed] (3.4,2) to
(2.4,0); \draw [thick,directed] (3.4,2) to (4.4,0); \node at (4.4,0)
{$\bullet$}; \node at (4.4,-.5) {\tiny $g=1$};
\end{tikzpicture}
+
\frac{1}{2} \begin{tikzpicture}[baseline={(current bounding
box.center)},scale=0.5] \draw [thick,directed] (3.4,4) to (3.4,2);
\node at (3.4,2) {$\bullet$}; \node at (3.4,-2) {$\bullet$}; \draw
[ultra thick,directed] (3.4,2) to [out=240, in=120] (3.4,-2); \draw
[thick,directed] (3.4,2) to [out=300, in=60] (3.4,-2); \draw
[thick,directed] (3.4,-2) to (3.4,-4);
\end{tikzpicture}
+\frac{1}{2}\begin{tikzpicture}[baseline={(current
bounding box.center)},scale=0.5] \draw [thick,directed] (3.4,4) to
(3.4,2); \node at (3.4,2) {$\bullet$}; \node at (3.4,-2) {$\bullet$};
\draw [ultra thick,directed] (3.4,2) to [out=240, in=120] (3.4,-2); \draw
[thick] (3.4,2) to [out=300, in=60] (3.4,-2); \draw [thick,directed]
(3.4,-2) to (3.4,-4);
\end{tikzpicture}
+\frac{1}{2}\begin{tikzpicture}[baseline={(current
bounding box.center)},scale=0.5] \draw [thick,directed] (3.4,4) to
(3.4,2); \node at (3.4,2) {$\bullet$}; \node at (3.4,-2) {$\bullet$};
\draw [thick,directed] (3.4,2) to [out=240, in=120] (3.4,-2); \draw
[ultra thick] (3.4,2) to [out=300, in=60] (3.4,-2); \draw [thick,directed]
(3.4,-2) to (3.4,-4);
\end{tikzpicture}
+\frac{1}{2}\begin{tikzpicture}[baseline={(current
bounding box.center)},scale=0.6] \draw [thick,directed] (3.4,4) to
(3.4,2); \node at (3.4,2) {$\bullet$}; \draw [thick,directed] (3.4,2)
to (3.4,0); \draw [thick,directed] (3.4,2) to (5.4,2); \draw [thick]
(5.9,2) circle [radius=.5];\node at (5.4,2) {$\bullet$};
\end{tikzpicture}
\end{align*}

Observe that in this case, only some edges of the graphs are directed, and the thick edges will have different contributions compared with ordinary edges.

%\end{itemize}

\end{example}

\subsubsection{Definition of $\rho^{A,\tw}_\GG$} 

Finally, we explain the map 
$\rho^{A,\tw}_\GG: L^A_+[1]\to \sym^{n-1} L^A_-$
in Equation~\eqref{eq:cei-formula-main}. Given a partially directed graph $\GG=(G,L_G^{\sf in}\coprod L_G^{\sf out},E^\direct,T)$, the map $\rho^{A,\tw}_\GG$ is defined by composing along $G$ with  the various contributions from vertices, edges, and leaves defined as follows. 

Given a splitting map $s: H_\bullet(L^A)\ra H_\bullet(L^A_+)$ in the sense of Definition~\ref{def:nc-splitting}, we may choose a chain-level lift of it, i.e. a map of complexes 
\begin{equation}\label{eq:S-operator}
    S=\id+\sum_{k=1}^{\infty}S_k u^k: L^A \ra L^A_+
\end{equation} 
%of the form $S= \id +S_1 u+ S_2 u^2 +\cdots $ 
with $S_k \in \End(L^A), \; \forall k\geq 1$, such that the induced map in homology is the given splitting $s.$
%$s: H_\bullet(L^A) \ra H_\bullet(L^A_+)$. 
We may consider $S$ as an element in the ring ${\sf End}(L^A)[[u]]$ consisting of operator valued power series in the $u$-variable. 
By construction, the power series $S$ begins with $\id$, hence it is invertible and we denote its inverse by
\begin{equation}\label{eq:R-operator}
R:= S^{-1}=\id+\sum_{k=1}^{\infty} R_k u^k. %\id+R_1u+R_2u^2+\cdots.
\end{equation} 

Recall that 
$$L^A_-:= L^A[u^{-1}]\cong L^A((u))/u\cdot L^A_+.$$ 
The construction in~\eqref{eq:S-operator} and~\eqref{eq:R-operator} induces the operators $S: L^A_+[1]\to L^A_+[1]$ and $R:L^A_-\to L^A_-$. Then, the contributions at vertices and leaves are given by:
\begin{itemize}
\item At each vertex $v$ in a partially directed graph, we assign the multi-linear map $\hbA_{g(v),k(v),l(v)}$ from Equation~\eqref{eq:image-vertex}. This gives the contribution ${\sf Cont}(v)$ at a vertex.
\item At incoming (outgoing) leaves we assign the operator $S$ ($R$ respectively) defined above. This defines the contribution at leaves of a partially directed graph.
\end{itemize}

The edge contributions are more involved as there are different types of edges in a partially direct graph. Indeed, using maps $R$ and $S$, we may define an operator $H: L^A_-\otimes L^A_- \ra \mathbb{C}$ by 
\begin{equation}\label{eq:homotopy-H}
H(x\cdot u^{-i}, y\cdot u^{-j}) =  (-1)^j \sum_{l=0}^j 
\left\langle\, S_l
  R_{i+j+1-l} \,x,\, y\,\right\rangle_\Muk.
  \end{equation}
Denote by $H^\sym:\sym^2 L^A_- \ra \mathbb{C}$ its symmetrization.
Let us also define $F: L^A_- \ra L^A_+[1]$ by
\begin{equation}\label{eq:homotopy-F}
 F(x\cdot u^{-i}) := -\sum_{j=0}^\infty u^{j}\sum_{l=0}^j S_lR_{i+j+1-l}\, x.
 \end{equation}
Lastly, we choose any $\delta: L^A_-\otimes L^A_- \ra \mathbb{C}$ is such that 
\begin{equation}\label{eq:homotopy-delta}
[b+uB, \delta] = H- H^\sym,
\end{equation}
i.e., it bounds the failure of the operator $H$ being symmetric. This $\delta$ always exists by the Lagrangian property of the splitting map $s$ (see Definition~\ref{def:nc-splitting}). Then the edge contributions in the graph sum~\eqref{eq:cei-formula-main} are given by
\begin{enumerate}
\item for directed edges in the spanning tree $T$ we assign the
homotopy operator $F: L^A_- \ra L^A_+[1]$; 
\item for other directed edges, we assign the operator $\Theta: L^A_- \ra L^A_+[1]$ given by the circle action, i.e., 
\begin{equation}\label{eq:Theta}
	\Theta(x\cdot u^{-i}):= \delta_0^i Bx.
%    \begin{cases}	Bx, \;\;\mbox{if $k=0$}\\0,\;\;\mbox{otherwise;}\end{cases}
\end{equation}
The map $\Theta$ is often called twisted sewing, see~\cite[Section 4.2]{CalTu1}.
\item for undirected edges in the spanning tree $T$ we assign the
homotopy operator $\delta: L^A_-\otimes L^A_- \ra \mathbb{C}$; 
\item for other undirected edges, we assign the homotopy operator $H^{\sym}: \sym^2 L^A_- \ra \mathbb{C}$.
\end{enumerate} 

The following result can be directly checked from definitions.
\begin{Lemma}
\cite[Proposition 7.5, Proposition 8.1]{CalTu1}
There are commutator relations
\begin{align}
\label{eq:delta-appear}
[b+uB,H]&=\Delta.\\
\label{eq:theta-appear}
[b+uB,F]&=-(b+uB)F+F(b+uB)=\Theta.
\end{align}
Here, the sign $(-1)$ is due to the shift of $L_{+}^{A}[1].$
For any $x\cdot u^{-i}, y\cdot u^{-j}\in L^A_-,$ we also have
\begin{equation}\label{eq:iota-F-H}
    \iota(x\cdot u^{-i})\big(F(y\cdot u^{-j})\big)=H(y\cdot u^{-j}, x\cdot u^{-i}).
\end{equation}
\end{Lemma}

The map $\rho^{A,\tw}_\GG$ in Equation~\eqref{eq:cei-formula-main} is then obtained by the composition along the graph $\GG$ (which is always valid since $\GG$ contains no directed loops) with the contributions from vertices, edges, and leaves as described above. We give an example here.
For the first graph $\GG$ in Equation~\eqref{eq:lbA-013}, 
the map $\rho^{A,\tw}_\GG$ is given by
\[
x\cdot u^k\mapsto R^{\odot 3}\left(\hbA_{0,1,3}(S(x\cdot u^k))\right).
\]
For the second graph $\GG$ in Equation~\eqref{eq:lbA-013},
%the map $\rho^{A,\tw}_\GG: L^A_+[1] \to \sym^3 L^A_-$ 
the map $\rho^{A,\tw}_\GG$ is depicted as

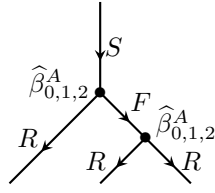
\[\begin{tikzpicture}[baseline={(current bounding
   	box.center)},scale=0.6] \draw [thick,directed] (3.4,4) to (3.4,2);
    \node at (3.4,2) {$\bullet$};
             \node at (3.7,3) {$S$};
             \node at (2.5,2.2) {$\hbA_{0,1,2}$};
    \draw[thick,directed] (3.4,2) to (1.4,0);
                 \node at (1.8,1) {$R$};
        \draw[thick,directed] (3.4,2) to (4.4,1);
                     \node at (4.3,1.8) {$F$};
                     \node at (5.3,1.3) {$\hbA_{0,1,2}$};
    \draw[thick,directed] (4.4,1) to (5.4,0);
        \draw[thick,directed](4.4,1)  to (3.4,0);
         \node at (4.4,1) {$\bullet$};
                          \node at (3.3,.5) {$R$};
                 \node at (5.4,.5) {$R$};
    \end{tikzpicture}\]
    
More explicitly, let $\hbA_{0,1,2}(S(x\cdot u^k))=Y\odot Z$. Then we have
\[\rho^{A,\tw}_\GG(x\cdot u^k)= R(Y)\odot R^{\odot 2}\left(\hbA_{0,1,2}(F(Z))\right)\in \sym^3 L^A_-.
\]

\subsubsection{The construction of $\cK$.}
Using partially directed graphs and the weight function, we may also write down explicitly the $L_\infty$ morphism $\cK: \hh_A \ra \hh_A^{\sf TRIV}$ in~\eqref{eq:trivialization-map-K}. Indeed, for each $m\geq 1$ we need to construct a map
\[\cK_m: \sym^m (\hh_A[1]) \ra
  \hh_A^{\sf TRIV}[1].\]
It will be defined as a sum over partially directed graphs:
\begin{equation}\label{eq:trivialization-map-K} \cK_m (\gamma_1\cdot \hbar^{g_1},\ldots,\gamma_m\cdot
\hbar^{g_m})= \sum_{(\GG,\tau)} \frac{\wt(\GG)}{\Aut(\GG,\tau)}
\cK_{(\GG,\tau)} (\gamma_1\cdot
\hbar^{g_1},\ldots,\gamma_m\cdot\hbar^{g_m}).\end{equation}
The summation above is over all marked partially directed graphs with
$m$ vertices for various $(g,k,l)$.  Here, a marking means a bijection
$\tau:\{1,\ldots,m\}\ra V_\GG$.  Let $(\GG,\tau)\in 
\widetilde{\Gamma(g,k,l)}_m$ be a marked partially directed
graph with $m$ vertices. We proceed to define the map $\cK_{(\GG,\tau)}$.
This map is only non-zero when the combinatorial types match at each
vertex, i.e., for each $1\leq i \leq m$, we have $k_i= k_{\tau(i)}$,
$l_i = l_{\tau(i)}$, and $g_i=g_{\tau(i)}$. When this is the case, the
result
$ \cK_{(\GG,\tau)}(\gamma_1\cdot
\hbar^{g_1},\ldots,\gamma_m\cdot\hbar^{g_m})$ is defined as follows:
\begin{itemize}
\item The tensor $\widetilde{\gamma_i}$ is assigned to the vertex $\tau(i)$. Here,
for a map $\gamma_i\in {\sf Hom} \big( \sym^{k_i} (L^A_+[1]) , \sym^{l_i}L^A_- \big)$, its de-symmetrization $\widetilde{\gamma_i}$
%\[ \widetilde{\gamma_i} \in \Hom \left((L^A_+[1])^{\otimes k_i}, (L^A_-)^{\otimes l_i}\right) \]
is the composition
\[ (L^A_+[1])^{\otimes k_i} \ra \sym^{k_i}(L^A_+[1])
\stackrel{\gamma}{\lra} \sym^{l_i}(L^A_-) \ra (L^A_-)^{\otimes l_i}. \]
\item The assignments at various types of edges and leaves are the same as in the previous paragraph.
\item The result $\cK_{(\GG,\tau)}(\gamma_1\cdot
  \hbar^{g_1},\ldots,\gamma_m\cdot\hbar^{g_m})$ is obtained by composition along $\GG$. 
\end{itemize}
We refer the details of the signs involved to the original paper~\cite{CalTu1}.
See also Section~\ref{sec-K-construction}.

\section{The dilaton equation}\label{sec:dilaton}
This section is devoted to prove the dilaton equation in the context of CEI. The idea is to show that string vertices can be chosen to be compatible with certain forgetful maps. Throughout the section, we continue to work with notation introduced in Section~\ref{subsec:cei-form}.

\subsection{The combinatorial forgetful maps}\label{subsec:f1}  
%The map $\cF$ is defined in the following way.
Let $\Gamma$ be an oriented black-and-white graph, and 
%\[ \alpha = \Gamma w_1^{-a_1}\cdots w_k^{-a_k} u_1^{-b_1}\cdots u_{l+1}^{-b_{l+1}}\] in $C^{\sf comb}_\bullet(M_{g,k,l+1}^{\sf fr}, \underline{\sgn})_{\sf hS}$, 
\begin{equation}\label{alpha-chain} 
\alpha = \Gamma \cdot \prod_{i=1}^{k}w_i^{-a_i}\prod_{j=1}^{l+1}u_j^{-b_j} \in C^{\sf comb}_\bullet(M_{g,k,l+1}^{\sf fr}, \underline{\sgn})_{\sf hS}
\end{equation}
be an equivariant chain in~\eqref{eq:equiv-chain}. Let $v_j$ be the $j$-th white vertex in $\Gamma$ such that ${\sf val}(v_j)=1$. Denote by $e$ the unique edge attached to $v_{j}$ and $v$ the other vertex of $e$. 
We have the following two special cases.

\begin{enumerate}
\item If $v$ is a black vertex and ${\sf val}(V)=3$, we define a new black-and-white graph $\Gamma/e$ as illustrated in the following picture. 
Note that in $\Gamma/e$, the black vertex $v$ is no longer present. 

\begin{equation}
\label{graph-remove-edge}
\begin{tikzpicture}[baseline={([yshift=-0.5ex]current bounding
		box.center)},scale=0.4] 
	\draw [thick] (-2.2,2) to (-2.2,0);
	\draw [thick] (-2.2, -2) to (-2.2, 0);
	\draw (-2.2,0) node[label=left:{$v$}] {};
	\draw (1,0) node[label=below:{$v_{j}$}] {};
	\draw (-3,0) node[label=left:{$\Gamma=$}] {};
    \node at (-.6,.8) {$e$};
	\draw (9,0) node[label=left:{$\Gamma/e=$}] {};
	\draw [thick] (-2.2,0) to (1,0);
	\draw [thick] (1.2,0) circle [radius=0.2];
	\draw [thick] [|->] (4,0) to (5.5,0);
	\draw [thick] (9.8,2) to (9.8,0);
	\draw [thick] (9.8, -2) to (9.8, 0);
\end{tikzpicture}
\end{equation}

%\medskip

\item 
If $v$ is a white vertex and the starting half-edge of $v$ is contained in $e$, we define a black-and-white graph $\Gamma\backslash v_{j}$ by removing the half-edge at $v_{j}$, while the remaining half-edge becomes the starting leaf of the white vertex $v$. For example, we have 

\begin{equation}
\label{graph-remove-white-vertex}
\begin{tikzpicture}[baseline={([yshift=-0.5ex]current bounding
		box.center)},scale=0.4] 
	\draw [thick] (-3,0) circle [radius=0.2];
	\draw (-3,0) node[label=below:{$v$}] {};
	\draw [thick] (-.3,0) circle [radius=0.2];
	\draw (-.3,0) node[label=below:{$v_{j}$}] {};
	\draw (-6,0) node[label=left:{$\Gamma=$}] {};
    \node at (-1,.8) {$e$};
	\draw [thick] (-6,-0) to (-3.2,0);
	\draw [thick] (-4,2) to (-3.2,0);
	\draw [thick] (-4,-2) to (-3.2,0);
	\draw [thick] (-1,2) to (-2.8,0);
	\draw [thick] (-1,-2) to (-2.8,0);
	\draw [thick] (-2.8,0) to (-.5,0);
	\draw [line width=2.4pt] (-1.5,0) to (-2.8,0);
	\draw [thick] [|->] (1.5,0) to (3,0);
	\draw (8,0) node[label=left:{$\Gamma\backslash v_j=$}] {};
	\draw [thick] (11,0) circle [radius=0.2];
	\draw (11,0) node[label=below:{$v$}] {};
	\draw [thick] (8,-0) to (10.8,0);
	\draw [thick] (10,2) to (10.8,0);
	\draw [thick] (10,-2) to (10.8,0);
	\draw [thick] (13,2) to (11.2,0);
	\draw [thick] (13,-2) to (11.2,0);
	\draw [line width=2.4pt] (11.2,0) to (12.7,0);
\end{tikzpicture}
\end{equation}
\end{enumerate}

Using the constructions above, for each $1\leq j\leq l+1$ we define 
\begin{equation}\label{eq:def-forgetful1}
\cF_j(\alpha):= 
\begin{cases}
\Gamma/e 
\cdot \delta_0^{b_j}
\prod\limits_{i=1}^{k} w_i^{-a_i}\prod\limits_{i=1}^{l+1} u_i^{-b_i}, &\;\;
\mbox{in cases of $\eqref{graph-remove-edge}$,}\\
%\mbox{if $b_j=0$, ${\sf val}(v_j)=1,$ $v\in V_b$, and ${\sf val}(v)=3$},\\
\Gamma\backslash v_j 
\cdot \delta_0^{b_j} \prod\limits_{i=1}^{k} w_i^{-a_i}\prod\limits_{i=1}^{l+1} u_i^{-b_i}, &\;\;
\mbox{in cases of $\eqref{graph-remove-white-vertex}$,}\\
%\mbox{if $b_j=0, {\sf val}(v_j)=1, v\in V_w,$ and the starting half-edge of $v$ is in $e$},\\
0, &\;\;\mbox{otherwise.}
\end{cases}
\end{equation}

The sum of all $\cF_j$'s will be denoted by
\begin{equation}
\label{map-forget-white-vertex}
\cF^\sym:=\sum_{j=1}^{l+1} \cF_j: C^{\sf comb}_\bullet(M_{g,k,l+1}^{\sf fr}, \underline{\sgn})_{\sf hS} \ra C^{\sf comb}_\bullet(M_{g,k,l}^{\sf fr}, \underline{\sgn})_{\sf hS}.
\end{equation}

\subsubsection{Algebraic meaning of the combinatorial forgetful map.} Let $A$ be a finite dimensional strictly unital cyclic $A_\infty$ algebra of dimension $d$. Using the strict unit $\bone_A$, we define a linear functional on $L_-^A=C_\bullet(A)[d][u^{-1}]$ by setting
\begin{equation}
\label{eq:cy-form} 
\omega_A (a_0|a_1|\cdots|a_n u^{-k}) := \delta_n^0\delta^k_0 \langle \bone_A, a_0\rangle.
%\begin{cases}
%\langle \bone_A, a_0\rangle \;\;\mbox{if}\;\; n=k=0,\\0, \;\; \mbox{otherwise}.
%\end{cases}
\end{equation}
Observe that $\omega_A: L_-^A \ra \mathbb{C}$ is a chain map with respect to the differential $\eth=b+uB$. This linear functional is usually referred to as a {\em compact Calabi-Yau structure} on $A$. Contraction with this linear functional  yields a map denoted by $C_{\omega_A}: \sym^{l+1} L^A_- \to \sym^{l} L^A_-$ for each $l\geq 0$. By post-composition, it further induces a map still denoted by 
\[ C_{\omega_A}: {\sf Hom}^c\big( \sym^k (L^A_+[1]), \sym^{l+1} L^A_-\big) \to {\sf Hom}^c\big( \sym^k (L^A_+[1]), \sym^{l} L^A_-\big). \]

\begin{Proposition}
\label{prop:cd}
Then the following diagram is commutative:
\[\begin{CD}
C^{\sf comb}_\bullet(M_{g,k,l+1}^{\sf fr}, \underline{\sgn})_{\sf hS}  @>{\rho^{A,\tw}}>> {\sf Hom}^c\big( \sym^k (L^A_+[1]), \sym^{l+1} L^A_-\big)\\
@V \cF^\sym VV        @VV C_{\omega_A} V\\
C^{\sf comb}_\bullet(M_{g,k,l}^{\sf fr}, \underline{\sgn})_{\sf hS}  @>{\rho^{A,\tw}}>> {\sf Hom}^c\big( \sym^k (L^A_+[1]), \sym^l L^A_-\big)
\end{CD}\]
\end{Proposition}
\begin{proof}
This follows from the strict unital property of $\bone_A$. It suffices to observe that the definition~\eqref{eq:def-forgetful1} of $\cF_j$, under the action map $\rho^{A,\tw}$, corresponds to applying the linear function $\omega_A$ at the $j$-th output. Indeed, the rules in defining $\cF_j$ at black vertices correspond to the unitality condition~\eqref{eq:unitality}; while the rules at white vertices are due to the fact that we are using the reduced Hochschild chain complex (which quotients out the unit except at the starting position of a Hochschild chain).
\end{proof}

\subsection{A differential graded Lie module and semi-direct product DGLA} 

Due to the symmetric group action, we can not deduce a direct compatibility between combinatorial string vertices and the forgetful maps in Equation~\eqref{eq:def-forgetful1}. This is circumvented by using a differential graded Lie module $\widehat{\mathfrak{n}}$ over the Sen-Zwiebach's DGLA $\hg$ and considering its associated semi-direct product DGLA $\hg\ltimes \widehat{\mathfrak{n}}$.

Recall from Section~\ref{subsec:cei-form} the chain complex
$C^{\sf comb}_\bullet(M_{g,k,l}^{\sf fr}, \underline{\sgn})_{\sf hS}$
is obtained from the combinatorial chain complex $C^\comb_\bullet(M^\fr_{g,k,l})$ quotiented by both the circle and symmetric group actions. In order to ``know" where to apply the forgetful map, we introduce a slight modification of the moduli space $M^\fr_{g,k,l+1}$ which will be denoted by $M^{\fr}_{g,k,l+P}$. The modification is that there is a distinguished white vertex labeled by $P$. Hence this space has an action by the group 
$(S^1)^{k+l+1} \ltimes (\Sigma_k\times \Sigma_l)$ (not $\Sigma_{l+1}$). Denote its homotopy quotient chain complex by
$C^{\sf comb}_\bullet(M_{g,k,l+P}^{\sf fr}, \underline{\sgn})_{\sf hS}.$ We have a natural labeling map
\[ \mathfrak{l}: C^{\sf comb}_\bullet(M_{g,k,l+1}^{\sf fr}, \underline{\sgn})_{\sf hS}\to C^{\sf comb}_\bullet(M_{g,k,l+P}^{\sf fr}, \underline{\sgn})_{\sf hS},\]
defined by $\mathfrak{l}(\alpha):=\sum_{j=1}^{l+1}\mathfrak{l}_j(\alpha)$ where the chain $\mathfrak{l}_j(\alpha)$ is simply $\alpha$ with its $j-th$ white vertex labeled by $P$. Putting these chain complexes together, let us define
\begin{equation} \label{fraknhat}
\hn:= \bigoplus_{\substack{g\geq 0, k\geq 1, l\geq 0\\ 2g-2+k+l\geq 0}} C^{\sf comb}_\bullet(M_{g,k,l+P}^{\sf fr}, \underline{\sgn})_{\sf hS}[2][[\hbar,\lambda]].
\end{equation}
Observe that $\widehat{\mathfrak{n}}$ 
%in \eqref{frakmhat} 
is a differential graded Lie module over $\widehat{\mathfrak{g}}$ with the action map $\widehat{\mathfrak{g}}\otimes \widehat{\mathfrak{n}} \ra \widehat{\mathfrak{n}}$ given by twisted sewing between inputs and outputs that are not labeled by $P$. We shall slight abuse the notation $\{-,-\}_\hbar$ for both the Lie bracket in $\hg$ as well as the Lie module structure map. Using this dg Lie module structure, we may form a semi-direct product DGLA denoted by $\widehat{\mathfrak{g}} \ltimes \widehat{\mathfrak{n}}$. 

\begin{Lemma}\label{lem:semi-unique-gn}
The DGLA $\widehat{\mathfrak{g}} \ltimes \widehat{\mathfrak{n}}$ has a Maurer-Cartan element (unique up to gauge equivalences) %of the following form
\[ \widehat{\mathcal{V}} + \widehat{\mathcal{U}}=\sum_{g,k\geq 1,l\geq0} \widehat{\mathcal{V}}_{g,k,l} \hbar^g \lambda^{2g-2+k+l} + \sum_{g,k\geq 1,l\geq 1} \widehat{\mathcal{U}}_{g,k,l} \hbar^g \lambda^{2g-2+k+l}\]
with $\hcV$ a combinatorial string vertex, %the component 
$\widehat{\mathcal{U}}_{g,k,l}\in C^{\sf comb}_\bullet(M_{g,k,(l-1)+P}^{\sf fr}, \underline{\sgn})_{\sf hS}$, and $\widehat{\mathcal{U}}_{0,1,2}=\mathfrak{l}\big(\widehat{\mathcal{V}}_{0,1,2}\big)$.
\end{Lemma}
\begin{proof}
The existence can be achieved by setting
$\widehat{\mathcal{U}}_{g,k,l}=\mathfrak{l}\big(\widehat{\mathcal{V}}_{g,k,l}\big).$
The uniqueness can be deduced in the same way as that of the string vertex $\widehat{\mathcal{V}}$, using the vanishing of relevant homology groups. Note that adding a labeling by $P$ at a white vertex does not change the fact that the relevant homology groups vanish. For more details, see~\cite{Cos2} and~\cite[Section 5]{CalTu1}.
\end{proof}

\subsection{Maurer-Cartan elements compatible with forgetful maps} Forgetting the white vertex labeled by $P$ using Equation~\eqref{eq:def-forgetful1} yields a map $\cF_P: C^{\sf comb}_\bullet(M_{g,k,l+P}^{\sf fr}, \underline{\sgn})_{\sf hS} \ra C^{\sf comb}_\bullet(M_{g,k,l}^{\sf fr}, \underline{\sgn})_{\sf hS}$. Using this map we introduce the following map
\[\cF_P^{u}\;:=\cF_P \circ M_{u_P}: C^{\sf comb}_\bullet(M_{g,k,l+P}^{\sf fr}, \underline{\sgn})_{\sf hS} \ra C^{\sf comb}_\bullet(M_{g,k,l}^{\sf fr}, \underline{\sgn})_{\sf hS}.\]
Here the operator $M_{u_P}$ is the multiplication by $u_P$ operator where $u_P$ is the circle parameter at the distinguished white vertex $P$. Explicitly, it acts on an equivariant chain in  $C^{\sf comb}_\bullet(M_{g,k,l+P}^{\sf fr}, \underline{\sgn})_{\sf hS}$ by
\begin{equation}
\label{multiply-u_j}
M_{u_P} (\Gamma \cdot \prod_{i=1}^{k}w_i^{-a_i} u_P^{-b} \prod_{j=1}^{l}u_j^{-b_j}):= 
\begin{cases}
	\Gamma \cdot \prod_{i=1}^{k}w_i^{-a_i} u_P^{1-b} \prod_{j=1}^{l}u_j^{-b_j}, & \mbox{if $b\geq 1$;}\\
	0, & \mbox{if $b=0$.}
\end{cases}
\end{equation}

We give an example of these maps defined above. The explicit formula (see~\cite[Section 5.4]{CalTu1}) of $\hcV_{0,1,3}$ is given by
\[\hcV_{0,1,3} =\frac{1}{2}\;
  \begin{tikzpicture}[baseline={([yshift=-1.2ex]current bounding
      box.center)},scale=0.35] 
\draw (0,1.5) node[cross=2pt,label=above:{}] {};
\draw [thick] (0,0.2) to (0,1.5);
\draw [thick] (-0.2,0) to (-2,0);
\draw [thick] (0.2,0) to (2,0);
\draw [thick] (-2.2,0) circle [radius=0.2];
\draw [thick] (2.2,0) circle [radius=0.2];
\draw [thick] (0,0) circle [radius=0.2];
\draw [ultra thick] (.2,0) to (1,0);
\end{tikzpicture}+\frac{1}{2}\;
  \begin{tikzpicture}[baseline={([yshift=-1.2ex]current bounding
      box.center)},scale=0.35] 
\draw (1,1) node[cross=2pt,label=above:{}] {};
\draw [thick] (1,0) to (1,1);
\draw [thick] (-0.2,0) to (-2,0);
\draw [thick] (0.2,0) to (2,0);
\draw [thick] (-2.2,0) circle [radius=0.2];
\draw [thick] (2.2,0) circle [radius=0.2];
\draw [thick] (0,0) circle [radius=0.2];
\draw [ultra thick] (0,-.2) to (0,-1);\end{tikzpicture}
+\frac{1}{2}\;
  \begin{tikzpicture}[baseline={([yshift=-.4ex]current bounding
      box.center)},scale=0.35] 
\draw (1,1) node[cross=2pt,label=above:{}] {};
\draw [thick] (1,0) to (1,1);
\draw [thick] (0,0) to (-2,0);
\draw [thick] (0,0) to (2,0);
\draw [thick] (-2.2,0) circle [radius=0.2];
\draw [thick] (2.2,0) circle [radius=0.2];
\draw [thick] (0,-1) circle [radius=0.2];
\draw [thick] (0,0) to (0,-.8);
\node at (1.2,-1) {$\scriptstyle{u^{-1}}$};
\end{tikzpicture}+\frac{1}{6}\;
  \begin{tikzpicture}[baseline={([yshift=-2ex]current bounding
      box.center)},scale=0.35] 
    \draw (-1,1) node[cross=2pt,label=above:{}] {}; \draw [thick]
    (-1,0) to (-1,1); \draw [thick] (0,0) to (-2,0); \draw [thick]
    (0,0) to (2,0); \draw [thick] (-2.2,0) circle [radius=0.2]; \draw
    [thick] (2.2,0) circle [radius=0.2]; \draw [thick] (1,1) circle
    [radius=0.2]; \draw [thick] (1,0) to (1,.8); \node at (-1,1.5)
    {$\scriptstyle{w^{-1}}$};\end{tikzpicture}\]
We may apply the labeling map to $\hcV_{0,1,3}$ to obtain a chain $\mathfrak{l}(\hcV_{0,1,3})\in C^{\sf comb}_\bullet(M_{0,1,2+P}^{\sf fr}, \underline{\sgn})_{\sf hS}$. Let us compute $\cF_P^u\big( \mathfrak{l}(\hcV_{0,1,3})\big) \in C^{\sf comb}_\bullet(M_{0,1,2}^{\sf fr}, \underline{\sgn})_{\sf hS}$. By~\eqref{multiply-u_j}, the nontrivial contribution of computing $\cF^u\big( \mathfrak{l}(\hcV_{0,1,3})\big)$ comes only from the third graph. Computing using Equation~\eqref{eq:def-forgetful1} verifies the following identity:
\begin{equation}
\label{string-vertex-initial-recursion}
\cF_P^u\big( \mathfrak{l}(\hcV_{0,1,3})\big)= \hcV_{0,1,2}.
\end{equation}

\begin{Lemma}\label{lem:commutator-Fu}
	The map $\cF_P^u\;$ satisfies the following commutator identities
    \begin{equation}
        \label{commutator-FP}
        [\eth, \cF_P^u]=[\Delta, \cF_P^u]=[\iota, \cF_P^u]=0.
    \end{equation}
%	\begin{align*}
%		\eth \cF_P^u\; & = \cF_P^u\; \eth,\\
%		\Delta \cF_P^u\; & = \cF_P^u\; \Delta,\\
%		\iota \cF_P^u\; &= \cF_P^u\;\iota.
%	\end{align*}
	Furthermore, for $\alpha\in \hg$ and $\beta\in \hn$, we have
    \[ \cF_P^u\{\alpha,\beta\}_\hbar = \{\alpha,\cF_P^u\beta\}_\hbar. \]
    \end{Lemma}

\begin{proof}
Recall that the equivariant boundary map $\eth$ is given in~\eqref{eth-operator}.
Multiplication by each $u$ clearly commutes with $\eth$. Thus, to check $\cF_P^u\;$ commutes with $\eth$, it suffices to show that 
$$[\eth, \cF_P]= \eth \cF_P-\cF_P\eth=0.$$
Let us first verify $[\partial, \cF_P] = 0$. Indeed, the only possibly non-vanishing terms in this commutator are two graphs from computing $\partial \cF_P(\beta)$ (with $\beta\in \hn$) when the white vertex $P$ and its unique adjacent vertex $v$ are both white vertices. These two graphs are given by
	
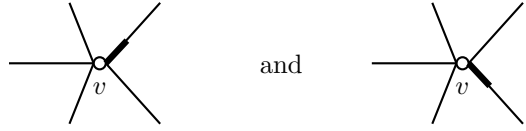
\[\begin{tikzpicture}[baseline={([yshift=-0.5ex]current bounding
			box.center)},scale=0.4] 
		\draw [thick] (-3,0) circle [radius=0.2];
		\draw (-3,0) node[label=below:{$v$}] {};
		\draw [thick] (-6,-0) to (-3.2,0);
		\draw [thick] (-4,2) to (-3.2,0);
		\draw [thick] (-4,-2) to (-3.2,0);
		\draw [thick] (-1,2) to (-2.8,0);
		\draw [thick] (-1,-2) to (-2.8,0);
		\draw [line width=2.4pt] (-2.8,0) to (-2.1,0.75);
		\node at (3,0) {\mbox{and}};
		\draw [thick] (9,0) circle [radius=0.2];
		\draw (9,0) node[label=below:{$v$}] {};
		\draw [thick] (6,-0) to (8.8,0);
		\draw [thick] (8,2) to (8.8,0);
		\draw [thick] (8,-2) to (8.8,0);
		\draw [thick] (11,2) to (9.2,0);
		\draw [thick] (11,-2) to (9.2,0);
		\draw [line width=2.4pt] (9.2,0) to (9.9,-0.75);
	\end{tikzpicture}\]

However, these two graphs also appear from applying $\cF_P$ to the following two terms in $\partial \beta$.
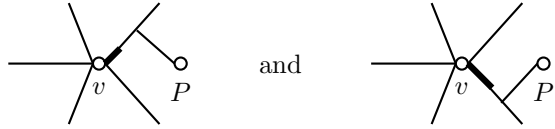
\[\begin{tikzpicture}[baseline={([yshift=-0.5ex]current bounding
		box.center)},scale=0.4] 
	\draw [thick] (-3,0) circle [radius=0.2];
	\draw (-3,0) node[label=below:{$v$}] {};
	\draw [thick] (-.3,0) circle [radius=0.2];
	\draw (-.3,0) node[label=below:{$P$}] {};
	\draw [thick] (-6,-0) to (-3.2,0);
	\draw [thick] (-4,2) to (-3.2,0);
	\draw [thick] (-4,-2) to (-3.2,0);
	\draw [thick] (-1,2) to (-2.8,0);
	\draw [thick] (-1,-2) to (-2.8,0);
	\draw [thick] (-1.75,1.1) to (-.5,0);
		\draw [line width=2.4pt] (-2.8,0) to (-2.3,0.5);
	\node at (3,0) {$\mbox{and}$};
	\draw [thick] (9,0) circle [radius=0.2];
	\draw (9,0) node[label=below:{$v$}] {};
		\draw [thick] (11.7,0) circle [radius=0.2];
			\draw (11.7,0) node[label=below:{$P$}] {};
	\draw [thick] (6,-0) to (8.8,0);
	\draw [thick] (8,2) to (8.8,0);
	\draw [thick] (8,-2) to (8.8,0);
	\draw [thick] (11,2) to (9.2,0);
	\draw [thick] (11,-2) to (9.2,0);
	\draw [thick]  (10.3,-1.3) to (11.5,0);
	\draw [line width=2.4pt] (9.2,0) to (10,-0.8);
\end{tikzpicture}\]
Other commutator identities  of $\cF_P^u\;$ with $\Delta$, $\iota$ and $\{-,-\}_\hbar$ all follow from the fact that the distinguished vertex $P$ is not used in these operators.
\end{proof}

\begin{Proposition}
\label{prop:forgetful}
There exists a Maurer-Cartan element $\hcV+\widehat{\mathcal{U}}$ as in Lemma~\ref{lem:semi-unique-gn} such that %the following equation holds: 
\begin{equation}\label{eq:dilaton-mc}
\widehat{\mathcal{V}}_{g,k,l} =\frac{1}{2g-2+k+l} 
\cF_P^u\;\widehat{\mathcal{U}}_{g,k,l+1}.
\end{equation} 
\end{Proposition}

\begin{proof}
We construct the desired element $\widehat{\mathcal{U}}$
%$=\sum\widehat{\mathcal{U}}_{g,k,l} \hbar^g\lambda^{2g-2+k+l}$ 
as follows. By definition, we have $\widehat{\mathcal{U}}_{0,1,2} = \mathfrak{l} \hcV_{0,1,2}$. Using Equation~\eqref{string-vertex-initial-recursion}, we may set $\widehat{\mathcal{U}}_{0,1,3}=\mathfrak{l}(\hcV_{0,1,3})$ to make Equation~\eqref{eq:dilaton-mc} hold for the initial case $(g,k,l)=(0,1,2)$. 
Now, we inductively construct the desired element $\widehat{\mathcal{U}}$ as follows. Let $(g,k,l+1)$ be a stable triple with $k\geq 1$ and $l\geq 0$. Let $-\chi=2g-2+k+l>0$ be the negative of Euler characteristic. Assume $\widehat{\mathcal{U}}_{g',k',l'+1}$ has been constructed for $(g',k',l'+1)$ satisfying one of the following conditions:
\begin{itemize}
    \item $2g'-2+k'+l'<-\chi$,
    \item $2g'-2+k'+l'=-\chi$, and $g'<g$,
    \item $2g'-2+k'+l'=-\chi$, $g'=g$, and $k'<k$.
\end{itemize}
We proceed to construct the desired $\widehat{\mathcal{U}}_{g,k,l+1}$, we first choose any $\widehat{\mathcal{U}}'_{g,k,l+1}$ that satisfies the Maurer-Cartan equation. Using the commutators in Lemma~\ref{lem:commutator-Fu} and the induction, we obtain
\begin{align*}
\eth(\cF_P^u\;\widehat{\mathcal{U}}'_{g,k,l+1})
=&  \cF_P^u\; \eth \widehat{\mathcal{U}}'_{g,k,l+1}\\
=& \cF_P^u\;\big( -\Delta \widehat{\mathcal{U}}_{g-1,k,l+3} - \iota \widehat{\mathcal{U}}_{g,k-1,l+2} - \sum \{ \hcV_{g_1,k_1,l_1},\widehat{\mathcal{U}}_{g_2,k_2,l_2+1}\}_r \big)\\
=& -\Delta (\cF_P^u\; \widehat{\mathcal{U}}_{g-1,k,l+3}) - \iota ( \cF_P^u\;\widehat{\mathcal{U}}_{g,k-1,l+2}) -\sum\cF_P^u \{\hcV_{g_1,k_1,l_1},\widehat{\mathcal{U}}_{g_2,k_2,l_2+1}\}_r\\
=&-(2g-2+k+l)\left(\Delta (\hcV_{g-1,k,l+2})+\iota(\hcV_{g,k-1,l+1})\right) -\sum \{\hcV_{g_1,k_1,l_1},\cF_P^u\widehat{\mathcal{U}}_{g_2,k_2,l_2+1}\}_r
\end{align*}
According to~\eqref{indices-mc-equation}, the summation of $(g_1,k_1,l_1)$ and $(g_2,k_2,l_2)$ in the second equality satisfies
\begin{equation}\label{indices-induction}
%\begin{dcases}
g_1+g_2 =g-r+1,\quad
k_1+k_2= k+r,\quad
l_1+l_2= l+r.
%\end{dcases}
\end{equation}
Thus, we use induction to deduce
\begin{align*}
&\sum\{ \hcV_{g_1,k_1,l_1},\cF_P^u\widehat{\mathcal{U}}_{g_2,k_2,l_2+1}\}_r\\
=&\sum (2g_2-2+k_2+l_2)\{ \hcV_{g_1,k_1,l_1},\hcV_{g_2,k_2,l_2}\}_r\\
=&\frac{1}{2}\big(\sum (2g_1-2+k_1+l_1)\{ \hcV_{g_1,k_1,l_1},\hcV_{g_2,k_2,l_2}\}_r+ \sum (2g_2-2+k_2+l_2)\{ \hcV_{g_1,k_1,l_1},\hcV_{g_2,k_2,l_2}\}_r\big)\\
=&(2g-2+k+l)\frac{1}{2}\sum\{ \hcV_{g_1,k_1,l_1},\hcV_{g_2,k_2,l_2}\}_r
\end{align*}
In the second last equation, we used that  
\[(2g_1-2+k_1+l_1)+(2g_2-2+k_2+l_2)=2(g-r+1)-4+(k+r)+(l+r)=2g-2+k+l.\]
This implies
\[\eth(\cF_P^u\;\widehat{\mathcal{U}}'_{g,k,l+1})
=-(2g-2+k+l)\left(\Delta (\hcV_{g-1,k,l+2})+\iota(\hcV_{g,k-1,l+1})+\frac{1}{2}\sum\{ \hcV_{g_1,k_1,l_1},\hcV_{g_2,k_2,l_2}\}_r\right).\]
Thus, $\frac{1}{2g-2+k+l} \cF_P^u\;\widehat{\mathcal{U}}'_{g,k,l+1}$ also satisfies the Maurer-Cartan equation~\eqref{eq:mc}. 
By the uniqueness of string vertices, there exists $\gamma$ such that
\[ \frac{1}{2g-2+k+l} \cF_P^u\;\widehat{\mathcal{U}}'_{g,k,l+1} = \hcV_{g,k,l} +\eth \gamma.\]
Since the map $\cF_P^u$ is a surjective chain map, we choose a lift $\widetilde{\gamma}$ such that $\cF_P^u\;(\widetilde{\gamma})=\gamma$. Finally, we set %the desired 
\[\widehat{\mathcal{U}}_{g,k,l+1}:= \widehat{\mathcal{U}}'_{g,k,l+1} - (2g-2+k+l)\eth\widetilde{\gamma}.\]
One easily verifies that we have
\[ \frac{1}{2g-2+k+l} \cF_P^u\;\widehat{\mathcal{U}}_{g,k,l+1} = \frac{1}{2g-2+k+l} \cF_P^u\;\widehat{\mathcal{U}}'_{g,k,l+1} -\eth{\gamma}= \hcV_{g,k,l}.\]
This completes the proof.
\end{proof}

\subsection{The dilaton equation}

In this section, we prove the CEI analogue of the dilaton equation. We begin with a discussion of noncommutative Calabi-Yau structures, which enables us to identify the ``unit" element in CEI. 

%\subsubsection{Calabi-Yau structures versus cyclic structures.} 
Let $A$ be an $A_\infty$ algebra, and let $\Omega\in HC_\bullet^-(A)$ be a Calabi-Yau structure on $A$. As before, we always assume $A$ is smooth and proper. Using homological transferring of $A_\infty$ structures, without loss of generality, we may assume that $A$ is finite dimensional. In the construction of CEI, the first step is applying a formal Darboux theorem (following Kontsevich-Soibelman~\cite{KonSoi,Cho}, and also~\cite{AmoTu2} for its unital version) to obtain a cyclic $A_\infty$ structure on $A$. As shown in~\cite{AmoTu2}, the CEI (defined using $A$) does not depend on the choice of this cyclic model.

Now, the cyclic structure is related to the original Calabi-Yau structure as follows. Let ${\sf Hom}^c\big( L_-^A, \mathbb{C} \big)$ be the space of $u$-adic continuous linear functionals on $L_-^A$. There is a natural map
\[D: L_+^A \ra {\sf Hom}^\cont\big( L_-^A, \mathbb{C} \big)\]
defined using the Mukai pairing by setting
\[ D(x\cdot u^i) (y\cdot u^{-j})=\langle y\cdot u^{-j}, x\cdot u^i\rangle_\Muk.\]
Shklyarov~\cite{Shk} proves that in homology, this map induces an isomorphism which we still denote by $D: HC_\bullet^-(A)[d]\to {\sf Hom}^\cont\big( HC_\bullet(A)[d], \mathbb{C} \big)$. Then the Calabi-Yau structure is related to the cyclic structure by the following equation:
\begin{equation}
\label{sm-CY-structure}
D\Omega = \omega_A,
\end{equation}
where the linear functional $\omega_A$ was defined in Equation~\eqref{eq:cy-form}. For this reason, the element $\Omega$  plays the role of ``unit" in categorical enumerative invariants. 
With this in mind, we now state the following 
\begin{Theorem}
\label{thm:dilaton}
Assume that the splitting map $s$ is unital in the sense that $\Omega\in {\sf Im}(s)$. Then the CEI of $(A,\Omega, s)$ satisfies the dilaton equation, i.e., for insertions $\alpha_1,\ldots,\alpha_n \in HH_\bullet(A)[d]$, we have
\begin{equation}\label{eq:dilaton}
\langle\Omega\psi, \alpha_1\psi^{k_1},\ldots,\alpha_n\psi^{k_n}\rangle_{g,n+1}^{A,\Omega,s} = (2g-2+n)\cdot \langle \alpha_1\psi^{k_1},\ldots,\alpha_n\psi^{k_n}\rangle_{g,n}^{A,\Omega,s}.\end{equation}
Here $[\Omega]\in H_\bullet(L^A)$ is image of $\Omega$ under the natural projection map $H_\bullet(L_+^A)\to H_\bullet(L^A)$.
%with insertions $\alpha_1,\ldots,\alpha_n \in HH_\bullet(A)[d]$, and $k_1,\ldots,k_n\geq 0$. 
\end{Theorem}

%\subsubsection{The strategy of the proof}

%\begin{proof} 

In order to prove the dilaton equation, we need the following construction.
Denote by $\mathbb{C}[\epsilon]$ the ring of dual numbers, with $\epsilon^2=0$. We may extend the DGLA map $\rho^{A,\tw}:\hg \ra \hh_A$ in~\eqref{eq:rho-tw} to a map 
$\widetilde{\rho^{A,\tw}}: \widehat{\mathfrak{g}} \ltimes \widehat{\mathfrak{n}} \ra \widehat{\mathfrak{h}}_A [\epsilon],$
by setting 
\[\widetilde{\rho^{A,\tw}}\big((\alpha,\beta)\big):= \rho^{A,\tw}(\alpha)+\epsilon\lambda^{-1} \rho^{A,\tw}(\cF_P^u\beta).\]
Post-composing $\widetilde{\rho^{A,\tw}}$ with the trivialization $L_\infty$ morphism $\cK$ (extended $\epsilon$-linearly) in Equation~\eqref{eq:Triv-map}, we obtain an $L_\infty$ morphism
\[ \widehat{\mathfrak{g}} \ltimes \widehat{\mathfrak{n}} \stackrel{\widetilde{\rho^{A,\tw}}}{\longrightarrow} \hh_A[\epsilon] \stackrel{\cK}{\longrightarrow} \hh_A^{\sf TRIV}[\epsilon].
\]
Now, in the semi-direct product DGLA $\hg\ltimes \hn$, we have two Maurer-Cartan elements $(\hcV,\mathfrak{l}\hcV)$ and $(\hcV,\widehat{\mathcal{U}})$ (from Proposition~\ref{prop:forgetful}). Let us describe how each side of Equation~\eqref{eq:dilaton} relates to these two elements.

\iffalse
In order to prove Equation~\eqref{eq:dilaton}, we shall prove the following two identities:
\begin{enumerate}
    \item (Proved in Proposition~\ref{prop:lhs-dilaton}) \[\langle \Omega\psi, \alpha_1\psi^{k_1},\ldots,\alpha_n\psi^{k_n} \rangle_{g, n+1}^{A,\Omega,s} = \left\langle \big( \cK_* \widetilde{\rho^{A,\tw}}_*(\hcV,\mathfrak{l}\hcV)\big)^\epsilon_{g,1,n-1}(\alpha_n(-u)^{k_n}),
    \bigodot\limits_{i=1}^{n-1}\alpha_i(-u)^{k_i}
    \right\rangle_{\Muk},\]
    \item (Proved in Proposition~\ref{prop:rhs-dilaton}) \[(2g-2+n)\cdot \langle \alpha_1\psi^{k_1},\ldots,\alpha_n\psi^{k_n}\rangle_{g,n}^{A,\Omega,s} = \left\langle \big( \cK_* \widetilde{\rho^{A,\tw}}_*(\hcV,\widehat{\mathcal{U}})\big)^\epsilon_{g,1,n-1}(\alpha_n(-u)^{k_n}),
    \bigodot\limits_{i=1}^{n-1}\alpha_i(-u)^{k_i}
    \right\rangle_{\Muk}. \]
\end{enumerate}
Then Equation~\eqref{eq:dilaton} follows from the uniqueness in Lemma~\ref{lem:semi-unique-gn} proving the two Maurer-Cartan elements $(\hcV,\mathfrak{l}\hcV)$ and $(\hcV,\widehat{\mathcal{U}})$ are gauge equivalent.
%\end{proof}

We proceed to prove the two propositions used in the proof above. 
\fi

Analogously to Equation~\eqref{eq:cy-form}, we define another linear functional $\omega_A^{[1]}$ on $L_-^A=C_\bullet(A)[d][u^{-1}]$ by 
\begin{equation}
\label{eq:cy-form-mu} 
\omega_A^{[1]} (a_0|a_1|\cdots|a_n u^{-k}) := \delta_n^0\delta^k_1 \langle \bone_A, a_0\rangle.
\end{equation}
Observe that $\omega_A^{[1]}=\omega_A\circ M_u$ where $M_u: L^A_-\to L^A_-$ is the multiplication by $u$ map. This implies that $\omega^{[1]}_A: L_-^A \ra \mathbb{C}$ is also a chain map with respect to differential $\eth=b+uB$. It also implies, using Equation~\eqref{sm-CY-structure}, that $\omega^{[1]}_A$ is represented by $\Omega(-u)$, i.e.,
\begin{equation}\label{eq:cy-form-mu-dual} \omega_A^{[1]}(\alpha)= \langle \alpha, \Omega(-u)\rangle_\Muk.\end{equation}
Contraction by the linear functional yields a map denoted by $C_{\omega^{[1]}_A}: \sym^{l+1} L^A_- \to \sym^{l} L^A_-$ for each $l\geq 0$. By post-composition, it further induces a map still denoted by 
\[ C_{\omega^{[1]}_A}: {\sf Hom}^c\big( \sym^k (L^A_+[1]), \sym^{l+1} L^A_-\big) \to {\sf Hom}^c\big( \sym^k (L^A_+[1]), \sym^{l} L^A_-\big). \]

\begin{Lemma}\label{lem:Fu-cd}
The following diagram is commutative:
\[\begin{CD}
C^{\sf comb}_\bullet(M_{g,k,l+1}^{\sf fr}, \underline{\sgn})_{\sf hS}  @>{\rho^{A,\tw}}>> {\sf Hom}^c\big( \sym^k (L^A_+[1]), \sym^{l+1} L^A_-\big)\\
@V \cF_P^u\circ \mathfrak{l} VV        @VV C_{\omega^{[1]}_A} V\\
C^{\sf comb}_\bullet(M_{g,k,l}^{\sf fr}, \underline{\sgn})_{\sf hS}  @>{\rho^{A,\tw}}>> {\sf Hom}^c\big( \sym^k (L^A_+[1]), \sym^l L^A_-\big)
\end{CD}\]
\end{Lemma}

\begin{proof}
By definition, we have $\cF_P^u\circ \mathfrak{l}= \sum_{j=1}^{l+1} \cF_j\circ M_{u_j}$. As in the proof of Proposition~\ref{prop:cd}, the map $\cF_j$ corresponds to applying the linear functional $\omega_A$ at the $j$-th output. Now, it suffices to observe that $\omega_A^{[1]}=\omega_A\circ M_u$ where $M_u: L^A_-\to L^A_-$ is the multiplication by $u$ map.
\end{proof}

\begin{Proposition}\label{prop:lhs-dilaton}
    The left hand side of Equation~\eqref{eq:dilaton} satisfies the following identity:
    \[\langle [\Omega]\psi, \alpha_1\psi^{k_1},\ldots,\alpha_n\psi^{k_n} \rangle_{g, n+1}^{A,\Omega,s} = \bigg\langle \big( \cK_* \widetilde{\rho^{A,\tw}}_*(\hcV,\mathfrak{l}\hcV)\big)^\epsilon_{g,1,n-1}(\alpha_n(-u)^{k_n}),
    \bigodot\limits_{i=1}^{n-1}\alpha_i(-u)^{k_i}
    \bigg\rangle_{\Muk}. \]
\end{Proposition}

\begin{proof}
Using the previous Lemma~\ref{lem:Fu-cd}, we may compute the Maurer-Cartan element
\[ \widetilde{\rho^{A,\tw}}\big((\hcV,\mathfrak{l}\hcV)\big)= \rho^{A,\tw}(\hcV)+\rho^{A,\tw}\big((\cF_P^u\circ\mathfrak{l})\hcV\big)\epsilon= \hbA+C_{\omega^{[1]}_A}\hbA \epsilon.\]
Then, using Equation~\eqref{eq:trivialization-map-K} we obtain
\begin{align*}
    \Big(\cK_* \widetilde{\rho^{A,\tw}}_*(\hcV,\mathfrak{l}\hcV)\Big)_{g,1,n-1} & = \sum_{m\geq 1} \frac{1}{m!}\sum_{\GG\in \Gamma((g,1,n-1))_m}\sum_{\tau}\frac{\wt(\GG)}{|\Aut(\GG,\tau)|}\cK_{(\GG,\sigma)}\big((\hbA+C_{\omega^{[1]}_A}\hbA \epsilon)^{\otimes m}\big)
\end{align*}
where recall $\tau$ is a bijection $\tau:\{1,\ldots,m\}\ra V_\GG$. Since $\epsilon^2=0$, the $\epsilon$-component of the above sum is %given by
\[ \Big(\cK_* \widetilde{\rho^{A,\tw}}_*(\hcV,\mathfrak{l}\hcV)\Big)_{g,1,n-1}^\epsilon  = \sum_{\GG\in \Gamma((g,1,n-1))} \sum_{v\in V_\GG} \frac{\wt(\GG)}{|\Aut(\GG,v)|}\rho^{A,\tw}_{(\GG,v)}.\]
Here, the construction of $\rho^{A,\tw}_{(\GG,v)}: L^A_+[1] \to \sym^{n-1}L^A_-$ is the same as the construction of $\rho^{A,\tw}_\GG$ in Equation~\eqref{eq:cei-formula-main} except at the distinguished vertex $v$ we assign the operator $C_{\omega^{[1]}_A}\hbA $ from the coefficient of $\epsilon$. Consider the graph $\GG'$ obtained from $\GG$ by adding an outgoing leaf denoted by $l_P$ at the distinguished vertex $v$. The map $\rho^{A,\tw}_{(\GG,v)}$ is then obtained from $\rho^{A,\tw}_{\GG'}$ by applying the linear functional $\omega_A^{[1]}$ at the added outgoing leaf $l_P$. Using this observation, we may rewrite the above summation as 
\[ \Big(\cK_* \widetilde{\rho^{A,\tw}}_*(\hcV,\mathfrak{l}\hcV)\Big)_{g,1,n-1}^\epsilon = \sum_{\GG'\in \Gamma((g,1,n))} \frac{\wt(\GG')}{|\Aut(\GG')|}C_{\omega^{[1]}_A}\left(\rho^{A,\tw}_{\GG'}\right).\]
Here we have used the fact that $\wt(\GG)=\wt(\GG')$ since by definition (see Equation~\eqref{eq:wt-def}) the weight of a partially direct graph is independent of leaves. 

Thus, we have
\begin{align*}
 &  \bigg\langle \big( \cK_* \widetilde{\rho^{A,\tw}}_*(\hcV,\mathfrak{l}\hcV)\big)^\epsilon_{g,1,n-1}(\alpha_n(-u)^{k_n}),
    \bigodot\limits_{i=1}^{n-1}\alpha_i(-u)^{k_i}
    \bigg\rangle_{\Muk} \\
    = & \sum_{\GG'\in \Gamma((g,1,n))} \frac{\wt(\GG')}{|\Aut(\GG')|}\bigg\langle C_{\omega^{[1]}_A}\rho^{A,\tw}_{\GG'}(\alpha_n(-u)^{k_n}),
    \bigodot\limits_{i=1}^{n-1}\alpha_i(-u)^{k_i}
    \bigg\rangle_{\Muk}\\
    = &  \sum_{\GG'\in \Gamma((g,1,n))} \frac{\wt(\GG')}{|\Aut(\GG')|}\bigg\langle\rho^{A,\tw}_{\GG'}(\alpha_n(-u)^{k_n}),
    [\Omega](-u)\bigodot\limits_{i=1}^{n-1}\alpha_i(-u)^{k_i}
    \bigg\rangle_{\Muk}\;\;\;\mbox{(by Equation~\eqref{eq:cy-form-mu-dual})}
    %\\= &  \langle \Omega\psi, \alpha_1\psi^{k_1},\ldots,\alpha_n\psi^{k_n} \rangle_{g, n+1}^{A,\Omega,s} \;\;\; \mbox{(by Equation~\eqref{eq:cei-formula-main})}. 
\end{align*}
By Equation~\eqref{eq:explicit-evaluation}, this is equal to the left hand side of the dilaton equation~\eqref{eq:dilaton}.
\end{proof}
 
\begin{Proposition}\label{prop:rhs-dilaton}
    The right hand side of Equation~\eqref{eq:dilaton} satisfies the following identity:
    \[(2g-2+n)\cdot \langle \alpha_1\psi^{k_1},\ldots,\alpha_n\psi^{k_n}\rangle_{g,n}^{A,\Omega,s} = \bigg\langle \big( \cK_* \widetilde{\rho^{A,\tw}}_*(\hcV,\widehat{\mathcal{U}})\big)^\epsilon_{g,1,n-1}(\alpha_n(-u)^{k_n}),
    \bigodot\limits_{i=1}^{n-1}\alpha_i(-u)^{k_i}
    \bigg\rangle_{\Muk}. \]
\end{Proposition}

\begin{proof}
Using Proposition~\ref{prop:forgetful}, we may compute the Maurer-Cartan element
\[ \widetilde{\rho^{A,\tw}}_{*}\big((\hcV,\widehat{\mathcal{U}})\big)= \rho^{A,\tw}(\hcV)+\rho^{A,\tw}\big(\cF_P^u\;\widehat{\mathcal{U}}\big)\epsilon= \hbA+\lambda\partial_\lambda(\hbA) \epsilon.\]
Recall the formula of $\hbA$ in~\eqref{hat-beta-A}, 
%$= \sum_{g,k,l} \widehat{\beta}^{A}_{g,k,l} \hbar^g\lambda^{2g-2+k+l}$, hence 
we have
$$\lambda\partial_\lambda\hbA=\sum_{g,k,l} (2g-2+k+l)\widehat{\beta}^{A}_{g,k,l} \hbar^g\lambda^{2g-2+k+l}.$$
Then, using Equation~\eqref{eq:trivialization-map-K} we obtain
\begin{align*}
    \Big(\cK_* \widetilde{\rho^{A,\tw}}_*(\hcV,\widehat{\mathcal{U}})\Big)_{g,1,n-1} & = \sum_{m\geq 1} \frac{1}{m!}\sum_{\GG\in \Gamma((g,1,n-1))_m}\sum_{\tau}\frac{\wt(\GG)}{|\Aut(\GG,\tau)|}\cK_{(\GG,\sigma)}\big((\hbA+\lambda\partial_\lambda\,\hbA \epsilon)^{\otimes m}\big)
\end{align*}
Since $\epsilon^2=0$, the $\epsilon$-component of the above sum is given by
\[ \Big(\cK_* \widetilde{\rho^{A,\tw}}_*(\hcV,\widehat{\mathcal{U}})\Big)_{g,1,n-1}^\epsilon  = \sum_{\GG\in \Gamma((g,1,n-1))} \sum_{v\in V_\GG} \frac{\wt(\GG)}{|\Aut(\GG,v)|}\rho^{A,\tw}_{(\GG,v)}.\]
Here, the construction of $\rho^{A,\tw}_{(\GG,v)}: L^A_+[1] \to \sym^{n-1}L^A_-$ is the same as the construction of $\rho^{A,\tw}_\GG$ in Equation~\eqref{eq:cei-formula-main} except at the distinguished vertex $v$ we assign the operator $\lambda\partial_\lambda\,\hbA$ from the coefficient of $\epsilon$. Since the operator $\lambda\partial_\lambda$ only scales the vertex contribution at $v$ by $\big(2g(v)-2+k(v)+l(v)\big)$, and 
$$\sum_{v\in \GG} \big(2g(v)-2+k(v)+l(v)\big) = 2g-2+n.$$ The above summation is equal to
\[ \Big(\cK_* \widetilde{\rho^{A,\tw}}_*(\hcV,\widehat{\mathcal{U}})\Big)_{g,1,n-1}^\epsilon =\big(2g-2+n\big) \sum_{\GG\in \Gamma((g,1,n-1))} \frac{\wt(\GG)}{|\Aut(\GG)|}\rho^{A,\tw}_{\GG}.\]
Using this identity, we may compute
\begin{align*}
 &   \bigg\langle \big( \cK_* \widetilde{\rho^{A,\tw}}_*(\hcV,\widehat{\mathcal{U}})\big)^\epsilon_{g,1,n-1}(\alpha_n(-u)^{k_n}),
    \bigodot\limits_{i=1}^{n-1}\alpha_i(-u)^{k_i}
    \bigg\rangle_{\Muk} \\
    = & \big(2g-2+n\big)  \sum_{\GG\in \Gamma((g,1,n-1))} \frac{\wt(\GG)}{|\Aut(\GG)|}\bigg\langle\rho^{A,\tw}_{\GG}(\alpha_n(-u)^{k_n}),
    \bigodot\limits_{i=1}^{n-1}\alpha_i(-u)^{k_i}
    \bigg\rangle_{\Muk}.
    %\\= & \big(2g-2+n\big)  \langle \alpha_1\psi^{k_1},\ldots,\alpha_n\psi^{k_n} \rangle_{g, n+1}^{A,\Omega,s} \;\;\; \mbox{(by Equation~\eqref{eq:cei-formula-main})}. 
\end{align*}
By Equation~\eqref{eq:explicit-evaluation}, this is equal to the right hand side of the dilaton equation~\eqref{eq:dilaton}.
\end{proof}

Now, Theorem~\ref{thm:dilaton} follows from Proposition~\ref{prop:lhs-dilaton}, Proposition~\ref{prop:rhs-dilaton}, and the uniqueness in Lemma~\ref{lem:semi-unique-gn} proving the two Maurer-Cartan elements $(\hcV,\mathfrak{l}\hcV)$ and $(\hcV,\widehat{\mathcal{U}})$ are gauge equivalent.

%\newpage 

\section{A partial recursion of string vertices}\label{sec:recursion}

In Section~\ref{sec:dilaton}, we used the compatibility of string vertices with the forgetful map defined in~\eqref{eq:def-forgetful1} to prove the dilaton equation.  For the string  and the divisor equations, it turns out we need to consider a different type of forgetful map. Namely, in this section we consider another forgetful map which only forgets the framing data at a white vertex. Then we prove a more involved compatibility between string vertices and this second type forgetful map. This compatibility (Theorem~\ref{thm:mc-recursion}) will be used later in Sections~\ref{sec:string} and~\ref{sec:divisor} to prove the string and the divisor equations.

\subsection{The forget-framing map} 
\label{sec-forget-framing}
Recall from Section~\ref{subsec:cei-form} the notation
$C^{\sf comb}_\bullet(M_{g,k,l}^{\sf fr}, \underline{\sgn})_{\sf hS}$
for the combinatorial chain model of $M^\fr_{g,k,l}$ quotiented by both the circle and symmetric group actions. In this section, we introduce a slight modification of the moduli space $M^\fr_{g,k,l+1}$ which will be denoted by $M^{\fr}_{g,k,l+Q}$. The modification is that there is a distinguished white vertex labeled by $Q$ that is {\em not} framed. Hence this space has an action by the group 
$(S^1)^{k+l} \ltimes (\Sigma_k\times \Sigma_l).$
Denote its homotopy quotient chain complex by
$C^{\sf comb}_\bullet(M_{g,k,l+Q}^{\sf fr}, \underline{\sgn})_{\sf hS}.$

\subsubsection{The definition}
We shall define a chain map 
\[ f: C^{\sf comb}_\bullet(M_{g,k,l+1}^{\sf fr}, \underline{\sgn})_{\sf hS} \ra C^{\sf comb}_\bullet(M_{g,k,l+Q}^{\sf fr}, \underline{\sgn})_{\sf hS}\]
which ``forgets the framing" at a white vertex and label it by $Q$. Again, let $\alpha \in C^{\sf comb}_\bullet(M_{g,k,l+1}^{\sf fr}, \underline{\sgn})_{\sf hS}$
be an equivariant chain as in~\eqref{eq:equiv-chain}, with $\Gamma$ an oriented black-and-white graph, we have
\begin{equation}
\label{def-forget-framing}
f(\alpha)= \sum_{j=1}^{l+1} f_j(\alpha).
\end{equation}
For each $j$, if $b_j=0$ and the starting half-edge at the $j$-th white vertex of $\Gamma$ is not a leaf, we define $f_j(\alpha)$ by forgetting the starting half-edge labeling at the $j$-th white vertex and label this white vertex $Q$. In all other cases, we set $f_j(\alpha)=0$.

The definition of $f_j$ is illustrated in the following figures. 
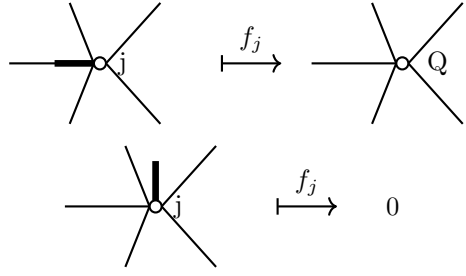
\[
\begin{tikzpicture}[baseline={([yshift=-0.5ex]current bounding
      box.center)},scale=0.4] 
       \draw [thick] (-3,0) circle [radius=0.2];
       \draw (-3,0) node[label=right:{j}] {};
\draw [thick] (-6,-0) to (-3.2,0);
\draw [line width=2.5pt] (-4.5,0) to (-3.2,0);
\draw [thick] (-4,2) to (-3.2,0);
\draw [thick] (-4,-2) to (-3.2,0);
\draw [thick] (-1,2) to (-2.8,0);
\draw [thick] (-1,-2) to (-2.8,0);
\draw [thick] [|->] (1,0) to (3,0);
\node at (2,.8) {$f_j$};
 \draw [thick] (7,0) circle [radius=0.2];
\draw [thick] (4,-0) to (6.8,0);
\draw [thick] (6,2) to (6.8,0);
\draw [thick] (6,-2) to (6.8,0);
\draw [thick] (9,2) to (7.2,0);
\draw [thick] (9,-2) to (7.2,0);
\draw (7.2,0) node[label=right:{Q}] {};
\end{tikzpicture}\]
\[  \begin{tikzpicture}[baseline={([yshift=-0.5ex]current bounding
      box.center)},scale=0.4] 
       \draw [thick] (-3,0) circle [radius=0.2];
       \draw (-3,0) node[label=right:{j}] {};
\draw [thick] (-6,-0) to (-3.2,0);
\draw [line width=2.5pt] (-3,1.5) to (-3,0.2);
\draw [thick] (-4,2) to (-3.2,0);
\draw [thick] (-4,-2) to (-3.2,0);
\draw [thick] (-1,2) to (-2.8,0);
\draw [thick] (-1,-2) to (-2.8,0);
\draw [thick] [|->] (1,0) to (3,0);
\node at (2,.8) {$f_j$};
\draw (4,0) node[label=right:{$0$}] {};
\end{tikzpicture}\]

\subsubsection{Examples}

Using the definition of the forget-framing morphism $f$, 
we consider the following graphs
\begin{align}
	%\begin{split}
	{\sf Get}_0^\sym  := &\frac{1}{2} \left(\begin{tikzpicture}[baseline={([yshift=-0.5ex]current bounding
			box.center)},scale=0.3] 
		\draw [thick] (-4,0) to (-2.4,0);
		\draw [thick] (-4.2, -.2) to (-3.8, .2);
		\draw [thick] (-4.2, .2) to (-3.8, -.2);
		\draw [thick] (0,0) to (-2,0);
		\draw [line width=2.5pt] (-2.2,0.2) to (-2.2,1);
		\draw (1,0) node[label=right:{$Q$}] {};
		\draw [thick] (0,0) to (1,0);
		\draw [thick] (-2.2,0) circle [radius=0.2];
		\draw [thick] (1.2,0) circle [radius=0.2];
	\end{tikzpicture} +\begin{tikzpicture}[baseline={([yshift=-0.5ex]current bounding
			box.center)},scale=0.3] 
		\draw [thick] (-4,0) to (-2.4,0);
		\draw [thick] (-4.2, -.2) to (-3.8, .2);
		\draw [thick] (-4.2, .2) to (-3.8, -.2);
		\draw [thick] (0,0) to (-2,0);
		\draw [line width=2.5pt] (-2.2,-0.2) to (-2.2,-1);
		\draw (1,0) node[label=right:{$Q$}] {};
		\draw [thick] (0,0) to (1,0);
		\draw [thick] (-2.2,0) circle [radius=0.2];
		\draw [thick] (1.2,0) circle [radius=0.2];
	\end{tikzpicture} \right)
    = f\left(\frac{1}{2} \begin{tikzpicture}[baseline={([yshift=-1ex]current bounding
  		box.center)},scale=0.3] 
  	\draw [thick] (-4,0) to (-2.4,0);
  	\draw [thick] (-4.2, -.2) to (-3.8, .2);
  	\draw [thick] (-4.2, .2) to (-3.8, -.2);
  	\draw [thick] (0,0) to (-2,0);
  	\draw [line width=2.5pt] (-2.2,0.2) to (-2.2,1);
  	\draw [thick] (0,0) to (1,0);
  	\draw [thick] (-2.2,0) circle [radius=0.2];
  	\draw [thick] (1.2,0) circle [radius=0.2];
  \end{tikzpicture}
  +{\frac{1}{2}}\begin{tikzpicture}[baseline={([yshift=-0.5ex]current bounding
			box.center)},scale=0.3] 
		\draw [thick] (-4,0) to (-2.4,0);
		\draw [thick] (-4.2, -.2) to (-3.8, .2);
		\draw [thick] (-4.2, .2) to (-3.8, -.2);
		\draw [thick] (0,0) to (-2,0);
		\draw [line width=2.5pt] (-2.2,-0.2) to (-2.2,-1);
		%\draw (1,0) node[label=right:{$Q$}] {};
		\draw [thick] (0,0) to (1,0);
		\draw [thick] (-2.2,0) circle [radius=0.2];
		\draw [thick] (1.2,0) circle [radius=0.2];
	\end{tikzpicture}
  \right)\label{graph:getzler1}
  \\
  {\sf Get}_1^\sym := &\frac{1}{2} 
	\left( \begin{tikzpicture}[baseline={([yshift=-1.2ex]current bounding
			box.center)},scale=0.3] 
		\draw [thick] (0,0) to (0,2);
		\draw [thick] (-0.2, 1.8) to (0.2, 2.2);
		\draw [thick] (0.2, 1.8) to (-0.2, 2.2);
		\draw [thick] (0,0) to (-2,0);
		\draw (-2,0) node[label=left:{$Q$}] {};
		\draw [thick] (0,0) to (2,0);
		\draw [thick] (-2.2,0) circle [radius=0.2];
		\draw [thick] (2.2,0) circle [radius=0.2];
	\end{tikzpicture}+ \begin{tikzpicture}[baseline={([yshift=-1.2ex]current bounding
			box.center)},scale=0.3] 
		\draw [thick] (0,0) to (0,2);
		\draw [thick] (-0.2, 1.8) to (0.2, 2.2);
		\draw [thick] (0.2, 1.8) to (-0.2, 2.2);
		\draw [thick] (0,0) to (-2,0);
		\draw (5,0) node[label=left:{$Q$}] {};
		\draw [thick] (0,0) to (2,0);
		\draw [thick] (-2.2,0) circle [radius=0.2];
		\draw [thick] (2.2,0) circle [radius=0.2];
	\end{tikzpicture}\right)
    =  f\left(\frac{1}{2}\begin{tikzpicture}[baseline={([yshift=-1ex]current bounding
  		box.center)},scale=0.3] 
  	\draw [thick] (0,0) to (0,2);
  	\draw [thick] (-0.2, 1.8) to (0.2, 2.2);
  	\draw [thick] (0.2, 1.8) to (-0.2, 2.2);
  	\draw [thick] (0,0) to (-2,0);
  	\draw [thick] (0,0) to (2,0);
  	\draw [thick] (-2.2,0) circle [radius=0.2];
  	\draw [thick] (2.2,0) circle [radius=0.2];
  \end{tikzpicture}\right)=f(\hcV_{0,1,2})
  \label{graph:getzler2}
%\end{split}
\end{align} 

\iffalse
\begin{align*}
  {\sf Get}_0^\sym &= f\left(\frac{1}{2} \begin{tikzpicture}[baseline={([yshift=-1ex]current bounding
  		box.center)},scale=0.3] 
  	\draw [thick] (-4,0) to (-2.4,0);
  	\draw [thick] (-4.2, -.2) to (-3.8, .2);
  	\draw [thick] (-4.2, .2) to (-3.8, -.2);
  	\draw [thick] (0,0) to (-2,0);
  	\draw [line width=2.5pt] (-2.2,0.2) to (-2.2,1);
  	\draw [thick] (0,0) to (1,0);
  	\draw [thick] (-2.2,0) circle [radius=0.2];
  	\draw [thick] (1.2,0) circle [radius=0.2];
  \end{tikzpicture}\right),\\
  & \\
  {\sf Get}_1^\sym &=  f\left(\frac{1}{2}\begin{tikzpicture}[baseline={([yshift=-1ex]current bounding
  		box.center)},scale=0.3] 
  	\draw [thick] (0,0) to (0,2);
  	\draw [thick] (-0.2, 1.8) to (0.2, 2.2);
  	\draw [thick] (0.2, 1.8) to (-0.2, 2.2);
  	\draw [thick] (0,0) to (-2,0);
  	\draw [thick] (0,0) to (2,0);
  	\draw [thick] (-2.2,0) circle [radius=0.2];
  	\draw [thick] (2.2,0) circle [radius=0.2];
  \end{tikzpicture}\right)=f(\hcV_{0,1,2}).
\end{align*}
\fi

We refer to these graphs as {\em Getzler graphs} due to their appearance in Getzler's explicit formula of a connection operator on periodic cyclic homology~\cite{Get}. For our purposes, we need a symmetric version of this connection operator which leads to the graphs above. This is explained in more detail in Section~\ref{app:getzler}.   

%For later convenience, the graph ${\sf Get}_1^\sym$ will also be denoted by
%\begin{equation}
%\label{graph-W012}
%\widehat{\mathcal{W}}_{0,1,2}:=f(\hcV_{0,1,2})= {\sf Get}_1^\sym.
%\end{equation}
%so as to state Theorem~\ref{thm:mc-recursion} in a more compact way. 

Recall that $\partial$ is the boundary operator defined in Section~\ref{sec:cei} and $B$ is the graph defined in~\eqref{graph:circle}. 
The following properties of Getzler graphs can be verified directly.
\begin{Lemma}
Let ${\sf Get}_1^\sym$ and ${\sf Get}_0^\sym$ both be endowed with the canonical orientation. Then we have
\begin{align}
\partial \big({\sf Get}_1^{\sf sym}\big)  &= 0, \label{getzler1-boundary}\\
\partial \big({\sf Get}_0^{\sf sym}\big)  &= T_1+T_2+T_3, \label{getzler0-boundary}
\end{align}
where
\begin{align*}
T_1&:= \begin{tikzpicture}[baseline={([yshift=-0.5ex]current bounding
      box.center)},scale=0.4] 
\draw [thick] (-4,0) to (-2.4,0);
\draw [thick] (-4.2, -.2) to (-3.8, .2);
\draw [thick] (-4.2, .2) to (-3.8, -.2);
\draw [thick] (0,0) to (-2,0);
\draw [line width=2.5pt] (-3.2,0) to (-2.4,0);
\draw (1,0) node[label=right:{Q}] {};
\draw [thick] (0,0) to (1,0);
\draw [thick] (-2.2,0) circle [radius=0.2];
\draw [thick] (1.2,0) circle [radius=0.2];
\end{tikzpicture}\\
T_2&:= \begin{tikzpicture}[baseline={([yshift=-0.5ex]current bounding
      box.center)},scale=0.4] 
\draw [thick] (-4,0) to (-2.4,0);
\draw [thick] (-4.2, -.2) to (-3.8, .2);
\draw [thick] (-4.2, .2) to (-3.8, -.2);
\draw [thick] (0,0) to (-2,0);
\draw [line width=2.5pt] (-2,0) to (-1,0);
\draw (1,0) node[label=right:{Q}] {};
\draw [thick] (0,0) to (1,0);
\draw [thick] (-2.2,0) circle [radius=0.2];
\draw [thick] (1.2,0) circle [radius=0.2];
\end{tikzpicture}\\
T_3&:=\frac{1}{2} \begin{tikzpicture}[baseline={([yshift=-1.5ex]current bounding
      box.center)},scale=0.4] 
\draw [thick] (-4,0) to (-2,0);
\draw [thick] (-4.2, -.2) to (-3.8, .2);
\draw [thick] (-4.2, .2) to (-3.8, -.2);
\draw [thick] (0,0) to (-2,0);
\draw [thick] (-2.2,0) to (-2.2,.8);
\draw [line width=2.5pt] (-2.2,1.2) to (-2.2,2);
\draw (1,0) node[label=right:{Q}] {};
\draw [thick] (0,0) to (1,0);
\draw [thick] (-2.2,1) circle [radius=0.2];
\draw [thick] (1.2,0) circle [radius=0.2];
\end{tikzpicture} +\frac{1}{2}\begin{tikzpicture}[baseline={([yshift=0.5ex]current bounding
      box.center)},scale=0.4] 
\draw [thick] (-4,0) to (-2,0);
\draw [thick] (-4.2, -.2) to (-3.8, .2);
\draw [thick] (-4.2, .2) to (-3.8, -.2);
\draw [thick] (0,0) to (-2,0);
\draw [thick] (-2.2,0) to (-2.2,-.8);
\draw [line width=2.5pt] (-2.2,-1.2) to (-2.2,-2);
\draw (1,0) node[label=right:{Q}] {};
\draw [thick] (0,0) to (1,0);
\draw [thick] (-2.2,-1) circle [radius=0.2];
\draw [thick] (1.2,0) circle [radius=0.2];
\end{tikzpicture}
\end{align*}
The graphs $T_1$, $T_2$, and $T_3$ are endowed with an induced orientation being in the boundary of ${\sf Get}_0^\sym$.
\end{Lemma}
%It is useful to compute the boundary of the Getzler graph ${\sf Get}_0^{\sf sym}$ defined in Equation~\eqref{graph:getzler}. 

\begin{Lemma}\label{lem:boundary-orientation}
Let ${\sf Get}_1^\sym$ and ${\sf Get}_0^\sym$ be both endowed with the canonical orientation. Then the following identities hold.
	\begin{align}
	B \circ {\sf Get}_1^\sym &= T_3. \label{twist-getzler1}\\
	{\sf Get}_1^\sym \circ B  &= -T_2 + T_4. \label{getzler1-twist}
	\end{align}
	where \begin{align*}
	T_4 &:=  \begin{tikzpicture}[baseline={([yshift=0.5ex]current bounding
			box.center)},scale=0.4] 
		\draw [thick] (-4,0) to (-2.4,0);
		\draw [thick] (-4.2, -.2) to (-3.8, .2);
		\draw [thick] (-4.2, .2) to (-3.8, -.2);
		\draw [thick] (0,0) to (-2,0);
		\draw (-2.5,.2) node[label=below:{Q}] {};
		\draw [thick] (0,0) to (1,0);
		\draw [thick] (-2.2,0) circle [radius=0.2];
		\draw [thick] (1.2,0) circle [radius=0.2];
	\end{tikzpicture}, 
\end{align*}
and the orientation of $T_4$ is determined by \eqref{getzler1-twist}.	
\end{Lemma}

\iffalse
We also define another graph:
\begin{align*}
	T_4 &:=  \begin{tikzpicture}[baseline={([yshift=0.5ex]current bounding
			box.center)},scale=0.4] 
		\draw [thick] (-4,0) to (-2.4,0);
		\draw [thick] (-4.2, -.2) to (-3.8, .2);
		\draw [thick] (-4.2, .2) to (-3.8, -.2);
		\draw [thick] (0,0) to (-2,0);
		\draw (-2.5,.2) node[label=below:{Q}] {};
		\draw [thick] (0,0) to (1,0);
		\draw [thick] (-2.2,0) circle [radius=0.2];
		\draw [thick] (1.2,0) circle [radius=0.2];
	\end{tikzpicture}.
\end{align*}
The orientation of $T_4$ is determined by part $(2.)$ of the following Lemma.
\begin{itemize}
		\item[(1.)] $B \circ {\sf Get}_1^\sym = T_3$.
		\item[(2.)] ${\sf Get}_1^\sym \circ B  = -T_2 + T_4$.
\end{itemize}
\fi

We now consider another example $f(\hcV_{0,2,1})$.

\begin{Lemma}
We may take the combinatorial string vertex $\widehat{\mathcal{V}}_{0,2,1}$ to be the following linear combination of black-and-white graphs:
\begin{equation}
\label{eq:021}
  \frac{1}{4} \begin{tikzpicture}[baseline={([yshift=-.4ex]current bounding box.center)},scale=0.3]
\draw (2,0) node[cross=2pt,label=above:{}] {};
\draw (6.2,0) node[cross=2pt,label=above:{}] {};
\draw [thick] (6.2,0) to (7.2,0);
\draw [thick] (2.1,0) to (3.2,0);
\draw [thick] (5.2,0) + (-85:2) arc(-85:265:2);
\draw [line width=2.4pt] (5.2,0) + (235:2) arc(235:265:2);
\draw [thick] (5.2,-2) circle (.2);
\end{tikzpicture}+\frac{1}{4} \begin{tikzpicture}[baseline={([yshift=-.4ex]current bounding box.center)},scale=0.3]
\draw (2,0) node[cross=2pt,label=above:{}] {};
\draw (6.2,0) node[cross=2pt,label=above:{}] {};
\draw [thick] (6.2,0) to (7.2,0);
\draw [thick] (2.1,0) to (3.2,0);
\draw [thick] (5.2,0) + (85:2) arc(85:-265:2);
\draw [line width=2.4pt] (5.2,0) + (95:2) arc(95:125:2);
\draw [thick] (5.2,2) circle (.2);
\end{tikzpicture}
+\frac{1}{2}\begin{tikzpicture}[baseline={([yshift=-.4ex]current bounding box.center)},scale=0.3]
\draw (2,0) node[cross=2pt,label=above:{}] {};
\draw (6.2,0) node[cross=2pt,label=above:{}] {};
\draw [thick] (6.2,0) to (7.2,0);
\draw [thick] (2.1,0) to (3.2,0);
\draw [thick] (5.2,0) circle [radius=2];
\draw (1.8,-1.6) circle (.2);
\draw [thick] (3.2,0) to (1.9,-1.4);
\end{tikzpicture}+\frac{1}{2}\begin{tikzpicture}[baseline={([yshift=-.4ex]current bounding box.center)},scale=0.3]
\draw (2,0) node[cross=2pt,label=above:{}] {};
\draw (6.2,0) node[cross=2pt,label=above:{}] {};
\draw [thick] (6.2,0) to (7.2,0);
\draw [thick] (2.1,0) to (3.2,0);
\draw [thick] (5.2,0) circle [radius=2];
\draw (1.8,1.6) circle (.2);
\draw [thick] (3.2,0) to (1.9,1.4);
\end{tikzpicture}
\end{equation}
We remark this particular choice of $\widehat{\mathcal{V}}_{0,2,1}$  is different (but homologous by uniqueness) from that in~\cite{CCT}.
\end{Lemma}

\begin{proof}
The desired string vertex is defined by the equation
\[ \eth \widehat{\mathcal{V}}_{0,2,1} = - \iota\widehat{\mathcal{V}}_{0,1,2}= \begin{tikzpicture}[baseline={([yshift=-1.2ex]current bounding
      box.center)},scale=0.3] 
\draw [thick] (0,0) to (0,2);
\draw [thick] (-0.2, 1.8) to (0.2, 2.2);
\draw [thick] (0.2, 1.8) to (-0.2, 2.2);
\draw [thick] (0,0) to (-2,0);
\draw [thick] (-4.7, -.2) to (-4.3, .2);
\draw [thick] (-4.7, .2) to (-4.3, -.2);
\draw [thick] (-6,0) to (-4.5,0);
\draw (-6,0) node[label=left:{$-\frac{1}{2}\;$}] {};
\draw [thick] (0,0) to (2,0);
\draw [thick] (-4,0) circle [radius=2];
\draw [thick] (2.2,0) circle [radius=0.2];
\end{tikzpicture}  
\begin{tikzpicture}[baseline={([yshift=-1.2ex]current bounding
      box.center)},scale=0.3] 
\draw [thick] (0,0) to (0,2);
\draw [thick] (-0.2, 1.8) to (0.2, 2.2);
\draw [thick] (0.2, 1.8) to (-0.2, 2.2);
\draw [thick] (0,0) to (-2,0);
\draw (-2,0) node[label=left:{$-\frac{1}{2}\;$}] {};
\draw [thick] (0,0) to (2,0);
\draw [thick] (-2.2,0) circle [radius=0.2];
\draw [thick] (4,0) circle [radius=2];
\draw [thick] (4.7, -.2) to (4.3, .2);
\draw [thick] (4.7, .2) to (4.3, -.2);
\draw [thick] (6,0) to (4.5,0);
\end{tikzpicture}\]

In this equation, the graphs are endowed with their canonical orientation since all vertices have odd valence. Let us consider the first graph in the above equation. We may choose appropriate orientation so that we have
\[ \eth\left( \frac{1}{2}\begin{tikzpicture}[baseline={([yshift=-.4ex]current bounding box.center)},scale=0.3]
\draw (2,0) node[cross=2pt,label=above:{}] {};
\draw (6.2,0) node[cross=2pt,label=above:{}] {};
\draw [thick] (6.2,0) to (7.2,0);
\draw [thick] (2.1,0) to (3.2,0);
\draw [thick] (5.2,0) circle [radius=2];
\draw (1.8,1.6) circle (.2);
\draw [thick] (3.2,0) to (1.9,1.4);
\end{tikzpicture} \right) = 
-\frac{1}{2} \ 
\begin{tikzpicture}[baseline={([yshift=-1.2ex]current bounding
      box.center)},scale=0.3] 
\draw [thick] (0,0) to (0,2);
\draw [thick] (-0.2, 1.8) to (0.2, 2.2);
\draw [thick] (0.2, 1.8) to (-0.2, 2.2);
\draw [thick] (0,0) to (-2,0);
\draw [thick] (-4.7, -.2) to (-4.3, .2);
\draw [thick] (-4.7, .2) to (-4.3, -.2);
\draw [thick] (-6,0) to (-4.5,0);
%\draw (-6,0) node[label=left:{$-\frac{1}{2}\;$}] {};
\draw [thick] (0,0) to (2,0);
\draw [thick] (-4,0) circle [radius=2];
\draw [thick] (2.2,0) circle [radius=0.2];
\end{tikzpicture} 
+ 
\frac{1}{2}\ 
\begin{tikzpicture}[baseline={([yshift=-.4ex]current bounding box.center)},scale=0.3]
\draw (2,0) node[cross=2pt,label=above:{}] {};
\draw (6.2,0) node[cross=2pt,label=above:{}] {};
\draw [thick] (6.2,0) to (7.2,0);
\draw [thick] (2.1,0) to (3.2,0);
\draw [thick] (5.2,0) circle [radius=2];
\draw (1.8,1.6) circle (.2);
\draw [thick] (3.4,.8) to (1.9,1.4);
\end{tikzpicture}\]
Then to cancel out the second term, we may choose an orientation of the  graph $\begin{tikzpicture}[baseline={([yshift=-.4ex]current bounding box.center)},scale=0.3]
\draw (2,0) node[cross=2pt,label=above:{}] {};
\draw (6.2,0) node[cross=2pt,label=above:{}] {};
\draw [thick] (6.2,0) to (7.2,0);
\draw [thick] (2.1,0) to (3.2,0);
\draw [thick] (5.2,0) + (85:2) arc(85:-265:2);
\draw [line width=2.4pt] (5.2,0) + (95:2) arc(95:125:2);
\draw [thick] (5.2,2) circle (.2);
\end{tikzpicture}$ so that 
\[ \eth \left( \frac{1}{4} \begin{tikzpicture}[baseline={([yshift=-.4ex]current bounding box.center)},scale=0.3]
\draw (2,0) node[cross=2pt,label=above:{}] {};
\draw (6.2,0) node[cross=2pt,label=above:{}] {};
\draw [thick] (6.2,0) to (7.2,0);
\draw [thick] (2.1,0) to (3.2,0);
\draw [thick] (5.2,0) + (85:2) arc(85:-265:2);
\draw [line width=2.4pt] (5.2,0) + (95:2) arc(95:125:2);
\draw [thick] (5.2,2) circle (.2);
\end{tikzpicture} \right) = - \frac{1}{2}\begin{tikzpicture}[baseline={([yshift=-.4ex]current bounding box.center)},scale=0.3]
\draw (2,0) node[cross=2pt,label=above:{}] {};
\draw (6.2,0) node[cross=2pt,label=above:{}] {};
\draw [thick] (6.2,0) to (7.2,0);
\draw [thick] (2.1,0) to (3.2,0);
\draw [thick] (5.2,0) circle [radius=2];
\draw (1.8,1.6) circle (.2);
\draw [thick] (3.4,.8) to (1.9,1.4);
\end{tikzpicture}\]
Here, it is essential that we have twisted the inputs by the sign representation ${\sf sgn}_2$. The discussion of the other two terms is analogous.
\end{proof}

Applying the framing forgetful map $f$ to the string vertex $\hcV_{0,2,1}$ yields 
 \begin{align}
 \label{eq:W-021}
 \begin{split}
  %\widehat{\mathcal{W}}_{0,2,1}:=
  f\left(\hcV_{0,2,1}\right)
   =  \frac{1}{4} \begin{tikzpicture}[baseline={([yshift=.5ex]current bounding box.center)},scale=0.3]
\draw (2,0) node[cross=2pt,label=above:{}] {};
\draw (6.2,0) node[cross=2pt,label=above:{}] {};
\draw [thick] (6.2,0) to (7.2,0);
\draw [thick] (2.1,0) to (3.2,0);
\draw [thick] (5.2,0) + (-85:2) arc(-85:265:2);
\draw [thick] (5.2,-2) circle (.2);
\draw (5.2,-2) node[label=above:{Q}] {};
\end{tikzpicture}+\frac{1}{4} \begin{tikzpicture}[baseline={([yshift=-.5ex]current bounding box.center)},scale=0.3]
\draw (2,0) node[cross=2pt,label=above:{}] {};
\draw (6.2,0) node[cross=2pt,label=above:{}] {};
\draw [thick] (6.2,0) to (7.2,0);
\draw [thick] (2.1,0) to (3.2,0);
\draw [thick] (5.2,0) + (85:2) arc(85:-265:2);
\draw [thick] (5.2,2) circle (.2);
\draw (5.2,2.2) node[label=below:{Q}] {};
\end{tikzpicture}
+\frac{1}{2}\begin{tikzpicture}[baseline={([yshift=.4ex]current bounding box.center)},scale=0.3]
\draw (2,0) node[cross=2pt,label=above:{}] {};
\draw (6.2,0) node[cross=2pt,label=above:{}] {};
\draw [thick] (6.2,0) to (7.2,0);
\draw [thick] (2.1,0) to (3.2,0);
\draw [thick] (5.2,0) circle [radius=2];
\draw (1.8,-1.6) circle (.2);
\draw [thick] (3.2,0) to (1.9,-1.4);
\draw (1.6,-1.6) node[label=right:{Q}] {};
\end{tikzpicture}+\frac{1}{2}\begin{tikzpicture}[baseline={([yshift=-.4ex]current bounding box.center)},scale=0.3]
\draw (2,0) node[cross=2pt,label=above:{}] {};
\draw (6.2,0) node[cross=2pt,label=above:{}] {};
\draw [thick] (6.2,0) to (7.2,0);
\draw [thick] (2.1,0) to (3.2,0);
\draw [thick] (5.2,0) circle [radius=2];
\draw (1.8,1.6) circle (.2);
\draw [thick] (3.2,0) to (1.9,1.4);
\draw (1.6,1.6) node[label=right:{Q}] {};
\end{tikzpicture}
\end{split}
\end{align}

\subsection{Two combinatorial constructions and their commutators}
We proceed to construct two degree-two maps:
\begin{align}
	\mathcal{T}&: C^{\sf comb}_\bullet(M_{g,k,l}^{\sf fr}, \underline{\sgn})_{\sf hS} \ra C^{\sf comb}_\bullet(M_{g,k,l+Q}^{\sf fr}, \underline{\sgn})_{\sf hS}. \label{outgoing-T}\\
	\mathcal{S}&: C^{\sf comb}_\bullet(M_{g,k-1,l+1}^{\sf fr}, \underline{\sgn})_{\sf hS} \ra C^{\sf comb}_\bullet(M_{g,k,l+Q}^{\sf fr}, \underline{\sgn})_{\sf hS}. \label{incoming-S}
\end{align}

Recall $\alpha$ in~\eqref{alpha-chain} 
%Throughtout the following, we shall let 
%\[ \alpha = \Gamma w_1^{-a_1}\cdots w_k^{-a_k} u_1^{-b_1}\cdots u_{l}^{-b_{l}}\in C^{\sf comb}_\bullet(M_{g,k,l}^{\sf fr}, \underline{\sgn})_{\sf hS}\]
is an equivariant chain in~\eqref{eq:equiv-chain}.
We define 
\begin{align}
\label{components-T}
\begin{dcases}
	H^{\sf in} (\alpha):= \sum_{i=1}^k \alpha \;_i \circ {\sf Get}_0^\sym, & 	H^{\sf out}(\alpha):=  \sum_{i=1}^l {\sf Get}_0^\sym \circ_i \alpha,\\
	G^{\sf in} (\alpha):=  \sum_{j=1}^k w_j^{-1} \alpha  \;_j \circ {\sf Get}_1^\sym, &
	G^{\sf out}(\alpha):= \sum_{j=1}^l u_j^{-1} {\sf Get}_1^\sym \circ_j \alpha,
\end{dcases}
\end{align}
where the superscript ``in" (respectively ``out") is to indicate that the operator involves incoming cycles (respectively outgoing white vertices).

%They shall provide us with an explicit formula to write down another Maurer-Cartan element of $\widehat{\mathfrak{g}} \ltimes \widehat{\mathfrak{m}}$ that is gauge equivalent to $\hcV+ f(\hcV)$.

%\subsubsection{Adding an outgoing vertex.}\label{subsec:T} 

%We shall proceed to define a map  \[\mathcal{T}: C^{\sf comb}_\bullet(M_{g,k,l}^{\sf fr}, \underline{\sgn})_{\sf hS} \ra C^{\sf comb}_\bullet(M_{g,k,l+Q}^{\sf fr}, \underline{\sgn})_{\sf hS}\]
% So we have $\mathcal{T}^{\sf in}= H^{\sf in}+G^{\sf in}$, and $\mathcal{T}^{\sf out}= H^{\sf out}+ G^{\sf out}$. 
%  Its definition is a bit involved and is split into three operators:
% \[ \mathcal{T}= \mathcal{T}^{\sf in} + \mathcal{T}^{\sf v} + \mathcal{T}^{\sf out},\]

The map $\mathcal{T}$ in \eqref{outgoing-T} adds a distinguished outgoing white vertex that is not framed. 
It is defined by
\begin{equation}
\label{construction-T}
\mathcal{T}=H^{\sf in} -H^{\sf out}+G^{\sf in}+G^{\sf out}+ \mathcal{T}^{\sf v},
\end{equation}
where the operator $\mathcal{T}^{\sf v}$ when applied to a black-and-white graph $\Gamma$, is defined by summing over all possible ways of choosing a black vertex in $\Gamma$ and just turn it into a white vertex and label it by $Q$. This is illustrated in the following picture. 

\[\begin{tikzpicture}[baseline={([yshift=0ex]current bounding box.center)},scale=0.4] 
	\draw [thick] (-6,-0) to (-3,0);
	\draw [thick] (-4,2) to (-3,0);
	\draw [thick] (-4,-2) to (-3,0);
	\draw [thick] (-1,2) to (-3,0);
	\draw [thick] (-1,-2) to (-3,0);
	\draw [thick] [|->] (1,0) to (3,0);
	\node at (2,.8) {$\mathcal{T}^{\sf v}$};
	\draw [thick] (7,0) circle [radius=0.2];
	\draw [thick] (4,-0) to (6.8,0);
	\draw [thick] (6,2) to (6.8,0);
	\draw [thick] (6,-2) to (6.8,0);
	\draw [thick] (9,2) to (7.2,0);
	\draw [thick] (9,-2) to (7.2,0);
	\draw (7.2,0) node[label=right:{Q}] {};
\end{tikzpicture}\]

\begin{remark}
\begin{enumerate}
\item
The operator $\mathcal{T}^{\sf in}:=H^{\sf in}+G^{\sf in}$ sews the unique white vertex that is not labeled in the Getzler graphs with all the cycles in a given black-and-white graph. In the case of ${\sf Get}_1^\sym$, we also multiply by the circle parameter $w_j^{-1}$ after sewing. 

\item
The operator $\mathcal{T}^{\sf out}:=H^{\sf out}+G^{\sf out}$ sews the unique cycle in the Getzler graphs with all the white vertices in a given black-and-white graph.
\end{enumerate}
\iffalse
The operator $\mathcal{T}^{\sf in}$ is defined by
\begin{equation}\label{eq:T-in} \mathcal{T}^{\sf in}(\Gamma):= \sum_{i=1}^k \Gamma \;_i \circ {\sf Get}_0^\sym + \sum_{j=1}^k w_j^{-1} \Gamma  \;_j \circ {\sf Get}_1^\sym\end{equation}
when acting on a black-and-white graph $\Gamma$ of type $(g,k,l)$ and extended to $C^{\sf comb}_\bullet(M_{g,k,l}^{\sf fr}, \underline{\sgn})_{\sf hS}$ linearly in the circle parameters $w^{-1}$'s and $u^{-1}$'s. In other words, the operator $\mathcal{T}^{\sf in}$ sews the unique white vertex that is not labeled in the Getzler graphs with all the cycles in a given black-and-white graph. In the case of ${\sf Get}_1^\sym$, we also multiply by the circle parameter $w_j^{-1}$ after sewing. 

The operator $\mathcal{T}^{\sf out}$ is defined in a similar way by
\begin{equation}\label{eq:T_out}
	 \mathcal{T}^{\sf out} (\Gamma) := \sum_{i=1}^l {\sf Get}_0^\sym \circ_i \Gamma + \sum_{j=1}^l u_j^{-1} {\sf Get}_1^\sym \circ_j \Gamma. \end{equation}
That is, the operator $\mathcal{T}^{\sf out}$ sews the unique cycle in the Getzler graphs with all the white vertices in a given black-and-white graph.
\fi
\end{remark}

%\subsubsection{Adding an incoming cycle.}\label{subsec:S} 

%Next, we define the another map  \[\mathcal{S} : C^{\sf comb}_\bullet(M_{g,k-1,l+1}^{\sf fr}, \underline{\sgn})_{\sf hS} \ra C^{\sf comb}_\bullet(M_{g,k,l+Q}^{\sf fr}, \underline{\sgn})_{\sf hS}.\]

We have seen that $f(\hcV_{0,1,2})={\sf Get}_1^\sym$ plays an important role in the construction of $\mathcal{T}$. 
Now the string vertex $\hcV_{0,2,1}$ (or rather $f(\hcV_{0,2,1})$) will play a similar role in the construction of $\mathcal{S}$. We define
\begin{equation}\label{construction-S} \mathcal{S}(\alpha):= \sum_{j=1}^{l+1} u_j^{-1}  %\widehat{\mathcal{W}}_{0,2,1}
f(\hcV_{0,2,1})\circ_j \alpha,
\end{equation}
where $\alpha$ is as in Equation~\eqref{alpha-chain}. In other words, the map $\mathcal{S}$ is defined by sewing at the inputs of $%\widehat{\mathcal{W}}_{0,2,1}
f(\hcV_{0,2,1})$ with all the white vertices of a black-and-white graph, followed with a multiplication by the circle parameter at the sewed white vertex. Again, we extend $\mathcal{S}$ linearly in all the circle parameter.  

\begin{remark}
\begin{enumerate}
\item Although the construction of $\mathcal{T}$ in the previous subsection is intuitively clear from the point of view of Getzler graphs, the construction of $\mathcal{S}$ is more of a homotopy nature. Part of the reason for the existence of $\mathcal{S}$ is that the operator $\mathcal{T}$ does not commute with $\iota$. 

\item 
We note that the use of the circle parameter $u^{-1}_j$ is slightly confusing since the $j$-th white vertex is no longer in $%\widehat{\mathcal{W}}_{0,2,1} 
f(\hcV_{0,2,1})\circ_j \alpha$. However, it should be clear what we mean by this: $u_j$ is now representing the circle parameter associated with the unique cycle of $%\widehat{\mathcal{W}}_{0,2,1}
f(\hcV_{0,2,1})
$ that is not sewed with $\alpha$. We also observe that in homological chain degree, this is a degree $3$ operation since $u_j^{-1}$ is of degree $2$ while the string vertex $\mathcal{V}_{0,2,1}$ is of degree $1$. Thus, after taking into account the shifting in the local system $\underline{{\sf sgn}}$, the map $\mathcal{S}$ is of homological degree $2$, as desired.
\end{enumerate}
\end{remark}

%Thus, before we define the operator $\mathcal{S}$, we need to have an explicit formula of this string vertex.
%With these preparations, we may define $\mathcal{S} : C^{\sf comb}_\bullet(M_{g,k-1,l+1}^{\sf fr}, \underline{\sgn})_{\sf hS} \ra C^{\sf comb}_\bullet(M_{g,k,l+Q}^{\sf fr}, \underline{\sgn})_{\sf hS}$ by setting 

\subsubsection{Commutators.}\label{para:commutators}   
Now we study the commutators between the new operators $\mathcal{T}, \mathcal{S}$ and the operators $\eth, \iota, \Delta.$ We begin with operators that commute.
\begin{Lemma}
The following two identities hold.
\begin{equation}
\label{commutator-iota-S}
[\iota,\mathcal{S}]=[\Delta,\mathcal{S}]=0.
\end{equation}
\end{Lemma}
\begin{proof}
    These two identities both follow from the fact that all three operators are applied to outputs and yield no new output. Hence they commute with each other.
\end{proof}    
%There are a few more commutator identities which vanishes by simple reasons. We gather them in the following lemma.
The following relations are obtained for a similar reason.
\begin{Lemma}
We have
\begin{align*}
    &[uB^w,H^{\sf in}]=[uB^w,w^{-1}G^{\sf in}]=[wB^c,H^{\sf in}]=0;\\
    &[wB^c, H^{\sf out}]=[wB^c, G^{\sf out}]=[uB^w,H^{\sf out}]=0;\\
    &[\partial, w^{-1}G^{\sf in}]=[\partial, u^{-1}G^{\sf out}]=0;\\
    &[wB^c, H^{\vee}]=[uB^w,H^{\vee}]=0.
\end{align*}
\end{Lemma}
%\begin{proof} The third line follows from the identity in~\eqref{getzler1-boundary}.  \end{proof}

Now let us describe the non-vanishing commutators.
%Now we compute the commutator $[\eth,\mathcal{T}]$. 
\begin{Lemma}\label{commutator-lemma-1}
We have
\begin{equation}
\label{commutator-eth-T}[\eth,\mathcal{T}] (\alpha) + \{ \alpha, %\widehat{\mathcal{W}}_{0,1,2} 
f(\hcV_{0,1,2})
\} =0.
\end{equation}
%Note that the Lie bracket is taken in the semi-direct product DGLA $\widehat{\mathfrak{g}}\ltimes \widehat{\mathfrak{m}}$.
\end{Lemma}
\begin{proof}
We divide the proof into four parts. 

\noindent {\bf (A.)} Since $\mathcal{T}=  \mathcal{T}^{\sf in} + \mathcal{T}^{\sf v} + \mathcal{T}^{\sf out}$, let us start with the commutator $[\eth,\mathcal{T}^{\sf in}]$. We recall that the equivariant boundary map is $\eth=\partial+uB+wB$. Since $\mathcal{T}^{\sf in}$ only involves cycles, it follows that $[uB,\mathcal{T}^{\sf in}]=0$. It remains to compute $[\partial,\mathcal{T}^{\sf in}]$ and $[wB,\mathcal{T}^{\sf in}]$. From the definition of $\mathcal{T}^{\sf in}$ in Equation~\eqref{construction-T}, we have 
\[ [\partial,\mathcal{T}^{\sf in}] (\alpha) = (-1)^{\deg(\alpha)}\sum_{i=1}^k \alpha\;_i\circ \partial {\sf Get}_0^\sym   = (-1)^{\deg(\alpha)}\sum_{i=1}^k \alpha\;_i\circ (T_1+T_2+T_3).\] 
The commutator $[wB,\mathcal{T}^{\sf in}](\alpha)$ consists of two parts:
\begin{align*}
	(A1)&:= (-1)^{\deg(\alpha)} \sum_{i=1}^k (\alpha \;_i\circ {\sf Get}_1^\sym) \;_i \circ B=(-1)^{\deg(\alpha)}\sum_{i=1}^k \alpha \;_i\circ ({\sf Get}_1^\sym \circ B);\\
	(A2)&:=-(-1)^{\deg(\alpha)}\sum_{i=1}^k \delta_{a_i\neq 1}\;(\alpha \;_i\circ B) \;_i\circ {\sf Get}_1^\sym =-(-1)^{\deg(\alpha)}\sum_{i=1}^k \delta_{a_i\neq 0}\; \alpha \;_i\circ (B \circ {\sf Get}_1^\sym). 
\end{align*}
In the second equation, the ``delta" function $\delta_{a_i\neq 0}$ is equal to $1$ if $a_i\neq 0$, otherwise it is zero. And $B$ is the circle graph in Equation~\eqref{graph:circle}. Using Equation~\eqref{twist-getzler1}, we observe that the term $(A2)$ almost cancel with $(-1)^{\deg(\alpha)}\sum_{i=1}^k \alpha\;_i\circ T_3$ in $[\partial,\mathcal{T}^{\sf in}]$, with the exception when the power $a_i=0$, i.e.,
\[ (A2) + (-1)^{\deg(\alpha)}\sum_{i=1}^k \alpha\;_i\circ T_3 =(-1)^{\deg(\alpha)} \sum_{i=1}^k \delta_{a_i=0} \; \alpha \;_i\circ (B \circ {\sf Get}_1^\sym). \]
Observe that the right-hand side also appears in the bracket $\{\alpha, %\widehat{\mathcal{W}}_{0,1,2} 
f(\hcV_{0,1,2})\}=\{\alpha, {\sf Get}_1^\sym\}$ with an opposite sign. Indeed, by Equation~\eqref{eq:def-r-bracket}, the term $\sum_{i=1}^k \delta_{a_i=0} \; \alpha \;_i\circ (B \circ {\sf Get}_1^\sym)$ appears with a sign $(-1)^{|\alpha|}(-1)^{|{\sf Get}_1^\sym|(k_\alpha-1)}=-(-1)^{|\alpha|+k_\alpha}=-(-1)^{\deg(\alpha)}$, as desired.

On the other hand, using Equation~\eqref{getzler1-twist}, we have 
\[ (A1)+ (-1)^{\deg(\alpha)}\sum_{i=1}^k \alpha\;_i\circ T_2 = (-1)^{\deg(\alpha)}\sum_{i=1}^k \alpha\;_i\circ T_4.\]
Putting these equations together yields
\[ [\eth,\mathcal{T}^{\sf in}] (\alpha)= (-1)^{\deg(\alpha)}\sum_{i=1}^k \alpha\;_i\circ (T_1+T_4) +  (-1)^{\deg(\alpha)}\sum_{i=1}^k \delta_{a_i=0} \; \alpha \;_i\circ (B \circ %\widehat{\mathcal{W}}_{0,1,2}
f(\hcV_{0,1,2})).\]

\noindent {\bf (B.)} A similar analysis in the case of $[\eth, \mathcal{T}^{\sf out}]$ yields the equation
\[ [\eth,\mathcal{T}^{\sf out}](\alpha) = \sum_{j=1}^l -(T_1+T_4) \circ_j \alpha + \sum_{j=1}^l \delta_{b_j=0}\; (%\widehat{\mathcal{W}}_{0,1,2}
f(\hcV_{0,1,2}))\circ B)\circ_j \alpha.\]
The negative sign in $\sum_{j=1}^l -(T_1+T_4) \circ_j \alpha$ is due to the negative sign in $H^{\sf out}$ in the Equation~\eqref{construction-T}.

\noindent {\bf (C.)} In this part, let us consider the commutator $[\eth, \mathcal{T}^{\sf v}]$.  We split it into two types of terms depending on it acting at a white vertex or a black vertex. At a black vertex, the commutator $[\eth, \mathcal{T}^{\sf v}]$ is by splitting out a white vertex of valence $1$ or $2$:
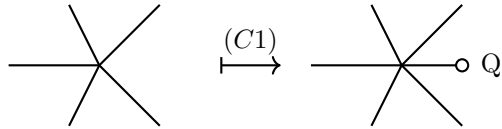
\[  \begin{tikzpicture}[baseline={([yshift=-0.5ex]current bounding
		box.center)},scale=0.4] 
	\draw [thick] (-6,-0) to (-3,0);
	\draw [thick] (-4,2) to (-3,0);
	\draw [thick] (-4,-2) to (-3,0);
	\draw [thick] (-1,2) to (-3,0);
	\draw [thick] (-1,-2) to (-3,0);
	\draw [thick] [|->] (1,0) to (3,0);
	\node at (2,.8) {$(C1)$};
	\draw [thick] (9,0) circle [radius=0.2];
	\draw [thick] (4,-0) to (7,0);
	\draw [thick] (6,2) to (7,0);
	\draw [thick] (6,-2) to (7,0);
	\draw [thick] (9,2) to (7,0);
	\draw [thick] (9,-2) to (7,0);
	\draw [thick] (8.8,0) to (7,0);
	\draw (9,0) node[label=right:{Q}] {};
\end{tikzpicture}\]
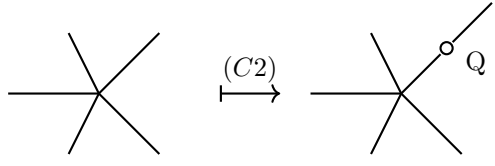
\[  \begin{tikzpicture}[baseline={([yshift=-0.5ex]current bounding
		box.center)},scale=0.4] 
	\draw [thick] (-6,-0) to (-3,0);
	\draw [thick] (-4,2) to (-3,0);
	\draw [thick] (-4,-2) to (-3,0);
	\draw [thick] (-1,2) to (-3,0);
	\draw [thick] (-1,-2) to (-3,0);
	\draw [thick] [|->] (1,0) to (3,0);
	\node at (2,.8) {$(C2)$};
	\draw [thick] (8.5,1.5) circle [radius=0.2];
	\draw [thick] (4,-0) to (7,0);
	\draw [thick] (6,2) to (7,0);
	\draw [thick] (6,-2) to (7,0);
	\draw [thick] (10,3) to (8.7,1.7);
	\draw [thick] (9,-2) to (7,0);
	\draw [thick] (8.3,1.3) to (7,0);
	\draw (8.5,1) node[label=right:{Q}] {};
\end{tikzpicture}\]
At a white vertex, we have $[\eth, \mathcal{T}^{\sf v}]$ acts by
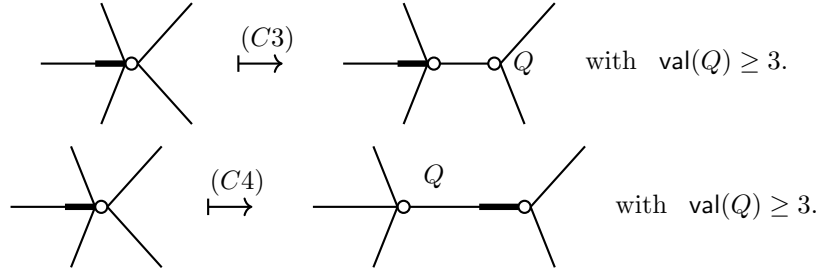
\[  \begin{tikzpicture}[baseline={([yshift=-0.5ex]current bounding
		box.center)},scale=0.4] 
	\draw [thick] (-3,0) circle [radius=0.2];
	\draw [thick] (-6,-0) to (-3.2,0);
	\draw [thick] (-4,2) to (-3.2,0);
	\draw [thick] (-4,-2) to (-3.2,0);
	\draw [thick] (-1,2) to (-2.8,0);
	\draw [thick] (-1,-2) to (-2.8,0);
	\draw [line width=2.4pt] (-4.2,0) to (-3.2,0);
	\draw [thick] [|->] (.5,0) to (2,0);
	\node at (1.5,.8) {$(C3)$};
	\draw [thick] (7,0) circle [radius=0.2];
	\draw [thick] (4,0) to (6.8,0);
	\draw [thick] (6,2) to (6.8,0);
	\draw [thick] (6,-2) to (6.8,0);
	\draw [line width=2.4pt] (5.8,0) to (6.8,0);
	\draw [thick] (7.2,0) to (8.8,0);
	\draw [thick] (9,0) circle [radius=0.2];
	\draw [thick] (9.2,0) to (11,2);
	\draw [thick] (9.2,0) to (10,-2);
	\draw (9,0) node[label=right:{$Q$}] {};
\end{tikzpicture} \mbox{\;\;\;\;with\;\;\;} {\sf val}(Q)\geq 3.\]
\[  \begin{tikzpicture}[baseline={([yshift=-0.5ex]current bounding
		box.center)},scale=0.4] 
	\draw [thick] (-3,0) circle [radius=0.2];
	\draw [thick] (-6,-0) to (-3.2,0);
	\draw [thick] (-4,2) to (-3.2,0);
	\draw [thick] (-4,-2) to (-3.2,0);
	\draw [thick] (-1,2) to (-2.8,0);
	\draw [thick] (-1,-2) to (-2.8,0);
	\draw [line width=2.4pt] (-4.2,0) to (-3.2,0);
	\draw [thick] [|->] (.5,0) to (2,0);
	\node at (1.5,.8) {$(C4)$};
	\draw [thick] (7,0) circle [radius=0.2];
	\draw [thick] (4,0) to (6.8,0);
	\draw [thick] (6,2) to (6.8,0);
	\draw [thick] (6,-2) to (6.8,0);
	\draw [line width=2.4pt] (9.5,0) to (10.8,0);
	\draw [thick] (7.2,0) to (10,0);
	\draw [thick] (11,0) circle [radius=0.2];
	\draw [thick] (11.2,0) to (13,2);
	\draw [thick] (11.2,0) to (12,-2);
	\draw (8,0) node[label=above:{$Q$}] {};
\end{tikzpicture} \mbox{\;\;\;\;with\;\;\;} {\sf val}(Q)\geq 3.\]

\noindent{\bf (D.)} Putting the commutators above together yields
\begin{align*}
	& [\eth,\mathcal{T}] (\alpha) + \{ \alpha, %\widehat{\mathcal{W}}_{0,1,2} 
    f(\hcV_{0,1,2})
    \} \\
	= & (-1)^{\deg(\alpha)}\sum_{i=1}^k \alpha\;_i\circ (T_1+T_4) - \sum_{j=1}^l (T_1+T_4) \circ_j \alpha +(C1)+(C2)+(C3)+(C4)\\
	= & 0.
\end{align*}
The cancellation in the last equality is mixed. We only sketch some of the details. First, we observe that both the sum $ -\sum_{j=1}^l T_1 \circ_j \alpha+ (C3)$ and $ -\sum_{j=1}^l T_4 \circ_j \alpha+ (C4)$ are almost zero except the terms with ${\sf val}(Q)=1,2$ are missing in $(C3)$ and $(C4)$. Denote these two missing terms by $(D1)$ and $(D2)$ (both are part of $ -\sum_{j=1}^l T_1 \circ_j \alpha$) which are depicted as in the following pictures.
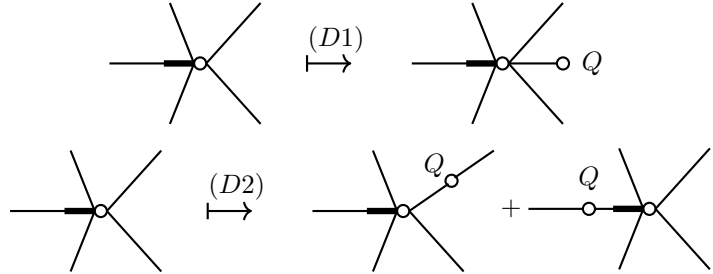
\[  \begin{tikzpicture}[baseline={([yshift=-0.5ex]current bounding
		box.center)},scale=0.4] 
	\draw [thick] (-3,0) circle [radius=0.2];
	\draw [thick] (-6,-0) to (-3.2,0);
	\draw [thick] (-4,2) to (-3.2,0);
	\draw [thick] (-4,-2) to (-3.2,0);
	\draw [thick] (-1,2) to (-2.8,0);
	\draw [thick] (-1,-2) to (-2.8,0);
	\draw [line width=2.4pt] (-4.2,0) to (-3.2,0);
	\draw [thick] [|->] (.5,0) to (2,0);
	\node at (1.5,.8) {$(D1)$};
	\draw [thick] (7,0) circle [radius=0.2];
	\draw [thick] (4,0) to (6.8,0);
	\draw [thick] (6,2) to (6.8,0);
	\draw [thick] (6,-2) to (6.8,0);
	\draw [line width=2.4pt] (5.8,0) to (6.8,0);
	\draw [thick] (7.2,0) to (8.8,0);
	\draw [thick] (9,0) circle [radius=0.2];
	\draw [thick] (7.2,0) to (9,2);
	\draw [thick] (7.2,0) to (9,-2);
	\draw (9,0) node[label=right:{$Q$}] {};
\end{tikzpicture}\]
\[  \begin{tikzpicture}[baseline={([yshift=-0.5ex]current bounding
		box.center)},scale=0.4] 
	\draw [thick] (-3,0) circle [radius=0.2];
	\draw [thick] (-6,-0) to (-3.2,0);
	\draw [thick] (-4,2) to (-3.2,0);
	\draw [thick] (-4,-2) to (-3.2,0);
	\draw [thick] (-1,2) to (-2.8,0);
	\draw [thick] (-1,-2) to (-2.8,0);
	\draw [line width=2.4pt] (-4.2,0) to (-3.2,0);
	\draw [thick] [|->] (.5,0) to (2,0);
	\node at (1.5,.8) {$(D2)$};
	\draw [thick] (7,0) circle [radius=0.2];
	\draw [thick] (4,0) to (6.8,0);
	\draw [thick] (6,2) to (6.8,0);
	\draw [thick] (6,-2) to (6.8,0);
	\draw [line width=2.4pt] (5.8,0) to (6.8,0);
	\draw [thick] (8.6,1) circle [radius=0.2];
	\draw [thick] (7.2,0) to (8.5,.9);
	\draw [thick] (8.7,1.1) to (10,2);
	\draw [thick] (7.2,0) to (9,-2);
	\draw (9,1.5) node[label=left:{$Q$}] {};
\end{tikzpicture}+ \begin{tikzpicture}[baseline={([yshift=-0.5ex]current bounding
	box.center)},scale=0.4] 
\draw [thick] (-3,0) circle [radius=0.2];
\draw [thick] (-5,0) circle [radius=0.2];
	\draw (-5,0) node[label=above:{$Q$}] {};
\draw [thick] (-7,-0) to (-5.2,0);
\draw [thick] (-4,2) to (-3.2,0);
\draw [thick] (-4,-2) to (-3.2,0);
\draw [thick] (-1,2) to (-2.8,0);
\draw [thick] (-1,-2) to (-2.8,0);
\draw [line width=2.4pt] (-4.2,0) to (-3.2,0);
\draw [thick] (-4.8,0) to (-3.2,0);
\end{tikzpicture}\]
Then, we observe that $(-1)^{\deg(\alpha)}\sum_{i=1}^k \alpha\;_i\circ (T_1)$  cancels with $(C1)+(D1)$, since both are splitting out a uni-valent white vertex in all possible ways at all vertices of $\Gamma$. In the remaining terms,  we observe that $(-1)^{\deg(\alpha)}\sum_{i=1}^k \alpha\;_i\circ (T_4)$ cancels with $(C2)+(D2)$. 
\end{proof}

\begin{Lemma}\label{commutator-lemma-3}
We have
\begin{equation}
\label{commutator-eth-S}
[\iota, \mathcal{T}] (\alpha)+[\eth, \mathcal{S}](\alpha) +\{ \alpha,  %\widehat{\mathcal{W}}_{0,2,1}
f(\hcV_{0,2,1})
\}_1=0.
\end{equation}
\end{Lemma}
\begin{proof}
By the defining property of $%\widehat{\mathcal{W}}_{0,2,1}
f(\hcV_{0,2,1})
$ we have
\begin{equation}\label{eq:defining-w021} 
\partial %\widehat{\mathcal{W}}_{0,2,1} 
f(\hcV_{0,2,1})= - \iota %\widehat{\mathcal{W}}_{0,1,2}
f(\hcV_{0,1,2}).\end{equation}
This equation implies that
\[ [\iota, G^{\sf in} + G^{\sf out}]+ [\partial,\mathcal{S}]=0,\]
since we have ${\sf Get}_1^\sym
%= \widehat{\mathcal{W}}_{0,1,2} 
= f(\hcV_{0,1,2})$. The remaining terms, $[\iota, H^{\sf in}+H^{\sf out}+\mathcal{T}^{\sf v}]$ and $[uB+wB,\mathcal{S}]$ almost cancel each other, except in the case when the circle parameter has power that is equal to zero, which is precisely twisted sewing with 
%$\widehat{\mathcal{W}}_{0,2,1}$
$f(\hcV_{0,2,1})$.
\end{proof}

\begin{Lemma}\label{commutator-lemma-2}
Let $\{-,-\}_2$ be the second twisted sewing operation involving two cycles and two white outgoing vertices.
We have 
\begin{equation}
\label{commutator-T-Delta}
[\mathcal{T}, \Delta] (\alpha) =\{ \alpha, %\widehat{\mathcal{W}}_{0,2,1}
f(\hcV_{0,2,1})
\}_2 
\end{equation}
%Note that the Lie bracket is taken in the semi-direct product DGLA $\widehat{\mathfrak{g}}\ltimes \widehat{\mathfrak{m}}$.
\end{Lemma}
\begin{proof}
The twisted self-sewing operator $\Delta$ defined in Equation~\eqref{eq:twisted-self-sewing} is by sewing with the thickened Mukai graph~\eqref{graph:thick-mukai}. Thus, form this we may observe that the commutator $[\mathcal{T},\Delta](\alpha)$ is from the composition $\mathcal{T}^{\sf v} \Delta (\alpha)$ when the operator $\mathcal{T}^{\sf v}$ is applied to the black vertex in the thickened Mukai graph. In other words,  we have 
\[ [\mathcal{T},\Delta](\alpha) = \sum_{1\leq i<j \leq l}  \mathcal{T}^{\sf v}(\mathbb{M})\circ_{\{i,j\}} \alpha,\]
where $ \mathcal{T}^{\sf v}(\mathbb{M})$ is given by the following graph.
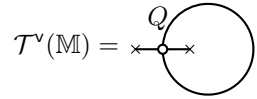
\[\mathcal{T}^{\sf v}(\mathbb{M})=\begin{tikzpicture}[baseline={([yshift=-.4ex]current bounding box.center)},scale=0.3]
\draw (2,0) node[cross=2pt,label=above:{}] {};
\draw (4.4,0) node[cross=2pt,label=above:{}] {};
\draw [thick] (3.4,0) to (4.4,0);
\draw [thick] (2,0) to (3,0);
\draw [thick] (5.2,0) + (-175:2) arc(-175:175:2);
\draw [thick] (3.2,0) circle (.2);
\draw (3,0) node[label=above:{$Q$}] {};
\end{tikzpicture}\]
To finish the proof, note that in the twisted sewing operation $\{-,-\}_2$, we need to apply the two circle operators corresponding to the two input cycles before sewing with $\alpha$. That is, we compute $B_1B_2 (
%\widehat{\mathcal{W}}_{0,2,1}
f(\hcV_{0,2,1}))
$, with $
%\widehat{\mathcal{W}}_{0,2,1}
f(\hcV_{0,2,1})$ as in Equation~\eqref{eq:W-021}. Computing $B_1B_2 (
%\widehat{\mathcal{W}}_{0,2,1}
f(\hcV_{0,2,1}))$ yields precisely $\mathcal{T}^{\sf v}(\mathbb{M})$ in the picture above.
\end{proof}

\begin{Lemma}\label{commutator-lemma-4}
For each $r\geq 1$, both $\mathcal{T}$ and $\mathcal{S}$ are derivations of $\{-,-\}_r$ defined in Equation~\eqref{eq:def-r-bracket}.
\end{Lemma}
\begin{proof}
    Let us first consider the operator $\mathcal{T}=\mathcal{T}^{\sf in}+\mathcal{T}^{\sf out}+\mathcal{T}^{\sf v}$. Recall the $r$-th bracket performs twisted sewing along $r$ cycles and white vertices only when both are with the circle power equal to zero.  Let us continue to use notation in the proof of Lemma~\ref{commutator-lemma-3} above. We see that both $H^{\sf in}$ and $H^{\sf out}$ are derivations of $\{-,-\}_r$ because the Getzler graph ${\sf Get}_0^\sym$ is in the kernel of twisting (i.e., pre-composing or post-composing with the circle graph $B$~\eqref{graph:circle}). On the other hand, both $G^{\sf in}$ and $G^{\sf out}$ are also derivations because they are both endowed with non-zero circle powers. And the operator $\mathcal{T}^{\sf v}$ is a derivation as by definition of $\{-,-\}_r$, the distinguished white vertex is never sewed with any cycles. Finally, the other operation $\mathcal{S}$ is a derivation again because it is endowed with non-zero circle power.
\end{proof}

\subsection{A partial recursion property of combinatorial string vertices}

Putting the moduli spaces in Section~\ref{sec-forget-framing} together yields
\begin{equation} \label{frakmhat}
\widehat{\mathfrak{m}}:= \bigoplus_{\substack{g\geq 0, k\geq 1, l\geq 0\\ 2g-2+k+l\geq 0}} C^{\sf comb}_\bullet(M_{g,k,l+Q}^{\sf fr}, \underline{\sgn})_{\sf hS}[2][[\hbar,\lambda]].
\end{equation}
Observe that $\widehat{\mathfrak{m}}$ 
%in \eqref{frakmhat} 
is a differential graded Lie module over $\widehat{\mathfrak{g}}$ with the action map $\widehat{\mathfrak{g}}\otimes \widehat{\mathfrak{m}} \ra \widehat{\mathfrak{m}}$ given by $\{-,-\}_\hbar$ with the twisted sewing performed between inputs and outputs except at $Q$. Here we have slightly abused the notation $\{-,-\}_\hbar$ since $Q$ is not involved in any of the twisted sewing operations. Using the dg-Lie module structure, we may form a semi-direct product DGLA denoted by $\widehat{\mathfrak{g}} \ltimes \widehat{\mathfrak{m}}$.

\begin{Lemma}\label{lem:semi-unique}
The DGLA $\widehat{\mathfrak{g}} \ltimes \widehat{\mathfrak{m}}$ has a Maurer-Cartan elememt, unique up to gauge equivalences, of the following form
\[ \widehat{\mathcal{V}} + \widehat{\mathcal{W}}=\sum_{g,k\geq 1,l\geq0} \widehat{\mathcal{V}}_{g,k,l} \hbar^g \lambda^{2g-2+k+l} + \sum_{g,k\geq 1,l\geq 1} \widehat{\mathcal{W}}_{g,k,l} \hbar^g \lambda^{2g-2+k+l}\]
such that $\widehat{\mathcal{V}}_{0,1,2}$ is as in Theorem~\ref{thm:comb-string-vertex} and $\widehat{\mathcal{W}}_{0,1,2}=f\big(\widehat{\mathcal{V}}_{0,1,2}\big)$.
\end{Lemma}
\begin{proof}
The existence can be achieved by setting
$\widehat{\mathcal{W}}_{g,k,l}=f\big(\widehat{\mathcal{V}}_{g,k,l}\big).$
The uniqueness can be deduced in the same way as that of the string vertex $\widehat{\mathcal{V}}$, using the vanishing of relevant homology groups. Note that forgetting the framing does not change the homology groups of moduli spaces after taking homotopy quotient since the circle group acts freely on framings. For more details, see~\cite{Cos2} and~\cite[Section 5]{CalTu1}.
\end{proof}

%In the proof above, we see that $\widehat{\mathcal{V}} + f\big(\widehat{\mathcal{V}}\big)$ is a Maurer-Cartan element of the semi-direct product DGLA $\widehat{\mathfrak{g}} \ltimes \widehat{\mathfrak{m}}$. 

%The key to prove the string equation and the divisor equation for CEI replies on 
%With the constructions of $\mathcal{T}$ and $\mathcal{S}$ ready, 

The following result %on gauge equivalence of Maurer-Cartan elements 
is crucial in proving the string equation and the divisor equation for CEI.

%Now we can state the main result of the section.

\begin{Theorem}
\label{thm:mc-recursion}
	Let $\hcV$ be a combinatorial string vertex, i.e., a Maurer-Cartan element as in Theorem~\ref{thm:comb-string-vertex}. 
Then the following two Maurer-Cartan elements of the semi-direct product DGLA $\widehat{\mathfrak{g}} \ltimes \widehat{\mathfrak{m}}$
\begin{align}
    \mc_1 &:=\big( \widehat{\mathcal{V}},f(\widehat{\mathcal{V}})\big),\\
    \mc_2 &:=\Big( \hcV, \sum_{g,k,l\geq 1} \big( \mathcal{T}\widehat{\mathcal{V}}_{g,k,l-1}+ \mathcal{S} \widehat{\mathcal{V}}_{g,k-1,l}\big) \hbar^g \lambda^{2g-2+k+l}+f(\hcV_{0,1,2})+f(\hcV_{0,2,1})\Big)
\end{align}
are gauge equivalent.
We will refer to the two terms $f(\hcV_{0,1,2})$ and $f(\hcV_{0,2,1})$ as  ``exceptional terms".
\iffalse
\begin{equation}
\label{second-mc-element}
\mc_2:= \Big( \hcV, \sum_{g,k,l\geq 1} \big( \mathcal{T}\widehat{\mathcal{V}}_{g,k,l-1}+ \mathcal{S} \widehat{\mathcal{V}}_{g,k-1,l}\big) \hbar^g \lambda^{2g-2+k+l}+f(\hcV_{0,1,2})+f(\hcV_{0,2,1})\Big)\in \hg\ltimes \widehat{\mathfrak{m}}. 
\end{equation}
is a Maurer-Cartan element that is gauge equivalent to the Maurer-Cartan element 
$\mc_1:=\big( \widehat{\mathcal{V}},f(\widehat{\mathcal{V}})\big).$

Let us define $\widehat{\mathcal{W}}_{g,k,l}\in C^{\sf comb}_\bullet(M_{g,k,(l-1)+Q}^{\sf fr}, \underline{\sgn})_{\sf hS}$ by 
	\begin{align*}
		\widehat{\mathcal{W}}_{g,k,l} &:= \mathcal{T}\widehat{\mathcal{V}}_{g,k,l-1}+ \mathcal{S} \widehat{\mathcal{V}}_{g,k-1,l},\\
		\widehat{\mathcal{W}}_{0,1,2} &:= f(\hcV_{0,1,2}),\\
		\widehat{\mathcal{W}}_{0,2,1} &:= f(\hcV_{0,2,1})
	\end{align*}
	Then the element 
	\[ \mc_2:=\sum_{g,k\geq 1,l\geq0} \widehat{\mathcal{V}}_{g,k,l} \hbar^g \lambda^{2g-2+k+l} + \sum_{g,k\geq 1,l\geq 1} \widehat{\mathcal{W}}_{g,k,l} \hbar^g \lambda^{2g-2+k+l}\]
	is a Maurer-Cartan element of the semi-direct product DGLA $\widehat{\mathfrak{g}} \ltimes \widehat{\mathfrak{m}}$. Furthermore, it is gauge equivalent to the Maurer Cartan element $\mc_1:=\widehat{\mathcal{V}} + f\big(\widehat{\mathcal{V}}\big)$.
    \fi
\end{Theorem}
\begin{proof}
%We proceed to prove Theorem~\ref{thm:mc-recursion}. 
The idea of the proof is similar as the proof of Proposition~\ref{prop:forgetful}, but this time by making use of the uniqueness of solutions of the Maurer-Cartan equation in Lemma~\ref{lem:semi-unique}. In order to do this, we need the commutators of $\mathcal{T}$ and $\mathcal{S}$ with various operators $\eth$, $\Delta$, $\iota$ and $\{-,-\}_r (r\geq 1)$ in $\hg\ltimes \widehat{\mathfrak{m}}$.

We apply the operator $\mathcal{T}$ to the Maurer-Cartan equation~\eqref{eq:mc} and use the commutator relation~\eqref{commutator-T-Delta} to obtain 
\begin{align*}
\mathcal{T} \eth\widehat{\mathcal{V}}_{g,k,l} +\mathcal{T}\iota\widehat{\mathcal{V}}_{g,k-1,l+1}
&= -\mathcal{T}\Delta\widehat{\mathcal{V}}_{g-1,k,l+2}-\mathcal{T}\left(\frac{1}{2}\sum \{\widehat{\mathcal{V}}_{g_1,k_1,l_1},\widehat{\mathcal{V}}_{g_2,k_2,l_2}\}_r\right),\nonumber\\
&= -\Delta\mathcal{T}\widehat{\mathcal{V}}_{g-1,k,l+2}-\{ \widehat{\mathcal{V}}_{g-1,k,l+2}, %\widehat{\mathcal{W}}_{0,2,1}
f(\hcV_{0,2,1})
\}_2 -\mathcal{T}\left(\frac{1}{2}\sum \{\widehat{\mathcal{V}}_{g_1,k_1,l_1},\widehat{\mathcal{V}}_{g_2,k_2,l_2}\}_r\right).
%\label{T-eth-iota}
\end{align*}
Similarly, applying the operator $\mathcal{S}$ to~\eqref{eq:mc} and using~\eqref{commutator-iota-S}, we obtain 
\begin{align*}
%\label{S-eth}
\mathcal{S} \eth \widehat{\mathcal{V}}_{g,k-1,l+1} 
&=
-\mathcal{S}\iota\widehat{\mathcal{V}}_{g,k-2,l+2}-\mathcal{S}\Delta\widehat{\mathcal{V}}_{g-1,k-1,l+3}-\mathcal{S}\left(\frac{1}{2}\sum \{\widehat{\mathcal{V}}_{g_1,k'_1,l'_1},\widehat{\mathcal{V}}_{g_2,k'_2,l'_2}\}_r\right)\\
&=-\iota\mathcal{S}\widehat{\mathcal{V}}_{g,k-2,l+2}-\Delta\mathcal{S}\widehat{\mathcal{V}}_{g-1,k-1,l+3}-\mathcal{S}\left(\frac{1}{2}\sum \{\widehat{\mathcal{V}}_{g_1,k'_1,l'_1},\widehat{\mathcal{V}}_{g_2,k'_2,l'_2}\}_r\right).
\end{align*}
Finally, we use the relations~\eqref{commutator-eth-T}, ~\eqref{commutator-eth-S}, and then the two equations above together to obtain 
\begin{align*}
&\eth\big(\mathcal{T}\widehat{\mathcal{V}}_{g,k,l}+\mathcal{S}\widehat{\mathcal{V}}_{g,k-1,l+1}\big)\\
=&\mathcal{T}\eth\widehat{\mathcal{V}}_{g,k,l}-\{\widehat{\mathcal{V}}_{g,k,l},
%\widehat{\mathcal{W}}_{0,1,2}
f(\hcV_{0,1,2})
\}+\left([\mathcal{T},\iota]+\mathcal{S}\eth\right)\widehat{\mathcal{V}}_{g,k-1,l+1}-\{\widehat{\mathcal{V}}_{g,k-1,l+1},
%\widehat{\mathcal{W}}_{0,2,1}
f(\hcV_{0,2,1})
\}_1\\
=&-\Delta\mathcal{T}\widehat{\mathcal{V}}_{g-1,k,l+2}-\{\widehat{\mathcal{V}}_{g-1,k,l+2}, %\widehat{\mathcal{W}}_{0,2,1}
f(\hcV_{0,2,1})
\}_2 -\mathcal{T}\left(\frac{1}{2}\sum \{\widehat{\mathcal{V}}_{g_1,k_1,l_1},\widehat{\mathcal{V}}_{g_2,k_2,l_2}\}_r\right)\\
&-\{\widehat{\mathcal{V}}_{g,k,l},
%\widehat{\mathcal{W}}_{0,1,2}
f(\hcV_{0,1,2})
\}- \iota\mathcal{T}\widehat{\mathcal{V}}_{g,k-1,l+1}-\{\widehat{\mathcal{V}}_{g,k-1,l+1},%\widehat{\mathcal{W}}_{0,2,1}
f(\hcV_{0,2,1})
\}_1\\
&-\iota\mathcal{S}\widehat{\mathcal{V}}_{g,k-2,l+2}-\Delta\mathcal{S}\widehat{\mathcal{V}}_{g-1,k-1,l+3}-\mathcal{S}\left(\frac{1}{2}\sum \{\widehat{\mathcal{V}}_{g_1,k'_1,l'_1},\widehat{\mathcal{V}}_{g_2,k'_2,l'_2}\}_r\right).
%=&-\{\widehat{\mathcal{V}}_{g,k,l},\widehat{\mathcal{W}}_{0,1,2}\}-\{\widehat{\mathcal{V}}_{g,k-1,l+1},\widehat{\mathcal{W}}_{0,2,1}\} -\{\widehat{\mathcal{V}}_{g-1,k,l+2},\widehat{\mathcal{W}}_{0,2,1}\}_2 \\
%&- \iota\big(  \mathcal{T}\widehat{\mathcal{V}}_{g,k-1,l+1}+\iota\mathcal{S}\widehat{\mathcal{V}}_{g,k-2,l+2}\big)- \Delta \big( \mathcal{T}\widehat{\mathcal{V}}_{g-1,k,l+2}+\mathcal{S}\widehat{\mathcal{V}}_{g-1,k-1,l+3}\big)\\
%& -\sum \{\widehat{\mathcal{V}}_{g_1,k_1,l_1},\mathcal{T}\widehat{\mathcal{V}}_{g_2,k_2,l_2}+\mathcal{S}\widehat{\mathcal{V}}_{g_2,k_2-1,l_2+1}\}_r
\end{align*}
This is exactly the Maurer-Cartan equation in the semi-direct product DGLA. Now the theorem follows from Lemma~\ref{lem:semi-unique}.
\end{proof}

\section{The string equation}\label{sec:string}

In this section, we prove the following theorem known as the string equation in the context of CEI. Throughout the section, we shall continue to use the notation from Section~\ref{sec:cei} and Section~\ref{sec:recursion}.

\begin{Theorem}
\label{thm:string}
Under the same conditions as in Theorem~\ref{thm:dilaton}, the CEI of $(A,\Omega,s)$ satisfies the string equation, i.e., for insertions $\alpha_1,\ldots,\alpha_n \in HH_\bullet(A)[d]$, we have 
\begin{align}\label{eq:string-equation}
 \langle [\Omega],\alpha_1\psi^{k_1},\ldots,\alpha_n\psi^{k_n}\rangle_{g, n+1}^{A,\Omega,s}= %\sum_{1\leq j\leq n, k_j\geq 1} 
 \sum_{j=1}^{n}\langle \alpha_1\psi^{k_1},\ldots,\alpha_j\psi^{k_j-1},\ldots,\alpha_n\psi^{k_n}\rangle_{g, n}^{A,\Omega,s}.
  \end{align}
%with insertions $\alpha_1,\ldots,\alpha_n \in HH_\bullet(A)[d]$, and $k_1,\ldots,k_n\geq 0$. 
\end{Theorem}

We remark that the term when $k_j=0$ on the right hand side vanishes automatically by definition.

\subsection{Strategy of the proof}
The proof of this theorem is much more involved than Theorem~\ref{thm:dilaton}. We need to use the partial recursion property proved in Theorem~\ref{thm:mc-recursion}. Let us briefly explain the idea of the proof before heading into the details.  Denote by $\mathbb{C}[\epsilon]$ the ring of dual numbers, with $\epsilon^2=0$. We extend the DGLA map $\rho^{A,\tw}:\hg \ra \hh_A$ in~\eqref{eq:rho-tw} to
$\rho^{A,\tw,\omega_A} : \widehat{\mathfrak{g}} \ltimes \widehat{\mathfrak{m}} \ra \widehat{\mathfrak{h}}_A [\epsilon]$ by setting  %such that for $\alpha\in \hg$ and $\beta \in \widehat{\mathfrak{m}}$, 
\begin{equation}
    \rho^{A,\tw,\omega_A} \big((\alpha,\beta)\big):= \rho^{A,\tw}(\alpha)+\lambda^{-1}C_{\omega_A}^Q\big(\rho^{A,\tw}(\beta)\big)\epsilon, 
\end{equation}
for $\alpha\in \hg$ and $\beta \in \widehat{\mathfrak{m}}$,
where the operator $C_{\omega_A}^Q$ is defined by applying the functional $\omega_A$ in~\eqref{eq:cy-form} to the Hochschild chain at the distinguished white vertex $Q$ in $\beta\in \widehat{\mathfrak{m}}$.

Post-composing with the trivialization $L_\infty$ morphism $\cK$ (extended $\epsilon$-linearly) in Equation~\eqref{eq:Triv-map}, we obtain an $L_\infty$ morphism
\[ \widehat{\mathfrak{g}} \ltimes \widehat{\mathfrak{m}} \stackrel{\rho^{A,\tw,\omega_A}}{\longrightarrow} \hh_A[\epsilon] \stackrel{\cK}{\longrightarrow} \hh_A^{\sf TRIV}[\epsilon].
\]
Recall that $\mc_1$ and $\mc_2$ in Theorem~\ref{thm:mc-recursion} are two gauge equivalent Maurer-Cartan elements of $\widehat{\mathfrak{g}} \ltimes \widehat{\mathfrak{m}}$.
By Theorem~\ref{thm:mc-recursion}, the push-forward Maurer-Cartan elements $\cK_*\rho^{A,\tw,\omega_A}_*\mc_1$ and $\cK_*\rho^{A,\tw,\omega_A}_*\mc_2$ are gauge equivalent. Since the DGLA $\hh_A^{\sf TRIV}[\epsilon]$ has a trivial Lie bracket, this implies
\begin{equation}\label{eq:push-forward-equal}
[\cK_*\rho^{A,\tw,\omega_A}_*\mc_1]=[\cK_*\rho^{A,\tw,\omega_A}_*\mc_2]\in H_\bullet(\hh_A^{\sf TRIV}[\epsilon], b+\iota)
\end{equation}
%as homology classes in $\hh_A^{\sf TRIV}[\epsilon]$. 

Note that since the differential in $\hh_A^{\sf TRIV}[\epsilon]$ is just $b+\iota$, 
this implies that 
\[[\big(\cK_*\rho^{A,\tw,\omega_A}_*\mc_1\big)_{g,1,n-1}^\epsilon]=[\big(\cK_*\rho^{A,\tw,\omega_A}_*\mc_2\big)_{g,1,n-1}^\epsilon], 
\]
for each stable $(g,n)$ with $n\geq 1$, 
where $Z^\epsilon$ is the $\epsilon$-component of an element $Z\in  \hh_A^{\sf TRIV}[\epsilon]$.

The strategy for proving Theorem~\ref{thm:string} consists of two main steps:
\begin{enumerate}
    \item Express the left hand side of Equation~\eqref{eq:string-equation} using $[\big(\cK_*\rho^{A,\tw,\omega_A}_*\mc_1\big)_{g,1,n-1}^\epsilon]$.
    \item Express the right hand side of Equation~\eqref{eq:string-equation} using $[\big(\cK_*\rho^{A,\tw,\omega_A}_*\mc_2\big)_{g,1,n-1}^\epsilon]$.
\end{enumerate}
Theorem~\ref{thm:string} would then follow from the equality~\eqref{eq:push-forward-equal} above.

\begin{Lemma}\label{lem:string-lhs-mc}
The left hand side of the string equation~\eqref{eq:string-equation} is given by
\[
\bigg\langle [\big(\cK_*\rho^{A,\tw,\omega_A}_*\mc_1\big)_{g,1,n-1}^\epsilon]
\big(\alpha_n (-u)^{k_n}\big),\bigodot\limits_{i=1}^{n-1}\alpha_i(-u)^{k_i}    \bigg\rangle_{\Muk}
\]
\end{Lemma}

\begin{proof}
We may compute the Maurer-Cartan element
\[ \rho^{A,\tw,\omega_A}_*\mc_1= \rho^{A,\tw}(\hcV)+C_{\omega_A}^Q\big(\rho^{A,\tw}(f\hcV)\big)\epsilon= \hbA+C_{\omega_A}(\hbA) \epsilon.\]
Then, using Equation~\eqref{eq:trivialization-map-K} we obtain
\begin{align*}
    \Big(\cK_* \rho^{A,\tw,\omega_A}_*\mc_1\Big)_{g,1,n-1} & = \sum_{m\geq 1} \frac{1}{m!}\sum_{\GG\in \Gamma((g,1,n-1))_m}\sum_{\tau}\frac{\wt(\GG)}{|\Aut(\GG,\tau)|}\cK_{(\GG,\sigma)}\big((\hbA+C_{\omega_A}(\hbA) \epsilon)^{\otimes m}\big)
\end{align*}
where recall $\tau$ is a bijection $\tau:\{1,\ldots,m\}\ra V_\GG$. Since $\epsilon^2=0$, the $\epsilon$-component of the above sum is given by
\[ \Big(\cK_* \rho^{A,\tw,\omega_A}_*\mc_1\Big)_{g,1,n-1}^\epsilon  = \sum_{\GG\in \Gamma((g,1,n-1))} \sum_{v\in V_\GG} \frac{\wt(\GG)}{|\Aut(\GG,v)|}\rho^{A,\tw}_{(\GG,v)}.\]
Here in the equation above, the map $\rho^{A,\tw}_{(\GG,v)}: L^A_+[1] \to \sym^{n-1}L^A_-$ is the same as the construction of $\rho^{A,\tw}_\GG$ in Equation~\eqref{eq:cei-formula-main} except at the distinguished vertex $v$ we assign the operator $C_{\omega_A}(\hbA) $ from the coefficient of $\epsilon$. Consider the graph $\GG'$ obtained from $\GG$ by adding an outgoing leaf denoted by $l_P$ at the distinguished vertex $v$. The map $\rho^{A,\tw}_{(\GG,v)}$ is then obtained from $\rho^{A,\tw}_{\GG'}$ by applying the linear functional $\omega_A$ at the added outgoing leaf $l_P$. Using this observation, we may rewrite the above summation as 
\[ \left(\cK_* \rho^{A,\tw,\omega_A}_*\mc_1\right)_{g,1,n-1}^\epsilon = C_{\omega_A}\bigg(\sum_{\GG'\in \Gamma((g,1,n))} \frac{\wt(\GG')}{|\Aut(\GG')|}\rho^{A,\tw}_{\GG'}\bigg).\]
Here we have used the fact that $\wt(\GG)=\wt(\GG')$ since by definition (see Equation~\eqref{eq:wt-def}) the weight of a partially direct graph is independent of leaves. 
Thus, we have
\begin{align*}
 &  \bigg\langle \Big(\cK_* \rho^{A,\tw,\omega_A}_*\mc_1\Big)_{g,1,n-1}^\epsilon(\alpha_n(-u)^{k_n}),
    \bigodot\limits_{i=1}^{n-1}\alpha_i(-u)^{k_i}
    \bigg\rangle_{\Muk} \\
    = & \bigg\langle C_{\omega_A}\bigg(\sum_{\GG'\in \Gamma((g,1,n))} \frac{\wt(\GG')}{|\Aut(\GG')|}\rho^{A,\tw}_{\GG'}(\alpha_n(-u)^{k_n})\bigg),
    \bigodot\limits_{i=1}^{n-1}\alpha_i(-u)^{k_i}
    \bigg\rangle_{\Muk}\\
    = & \bigg\langle \sum_{\GG'\in \Gamma((g,1,n))} \frac{\wt(\GG')}{|\Aut(\GG')|}\rho^{A,\tw}_{\GG'}(\alpha_n(-u)^{k_n}),
    [\Omega]\odot\bigodot\limits_{i=1}^{n-1}\alpha_i(-u)^{k_i}
    \bigg\rangle_{\Muk}\;\;\;\mbox{(by Equation~\eqref{sm-CY-structure})}
%\\    = &  \langle \Omega, \alpha_1\psi^{k_1},\ldots,\alpha_n\psi^{k_n} \rangle_{g, n+1}^{A,\Omega,s} \;\;\; \mbox{(by Equation~\eqref{eq:cei-formula-main})}. 
\end{align*}
By Equation~\eqref{eq:explicit-evaluation}, this is equal to the left hand side of the string equation~\eqref{eq:string-equation}.
\end{proof}

Let us denote by $M_{u^{-1}}$ the multiplication by $u^{-1}$, either defined on $L^A_+[1]$ or on $L^A_-$. On $L^A_+[1]$, we have $M_{u^{-1}}(x u^{i=0})=0$.
This operator extends by derivation to %symmetric powers 
$\sym^k(L_+^A[1])$ and $\sym^l(L_-^A)$. 
It also extends to an operator on $\Hom^c\big( \sym^k(L_+^A[1]),\sym^l(L_-^A)\big)$ by commutator $[M_{u^{-1}},-]$, i.e., we set
\begin{equation}\label{commutator-mu} [M_{u^{-1}},\Phi]:= M_{u^{-1}}\Phi - \Phi M_{u^{-1}}.
\end{equation}

\begin{Lemma}
The right-hand side of the string equation~\eqref{eq:string-equation} is given by
$$
\bigg\langle [M_{u^{-1}}, \lbA_{g,1,n-1}]
(\alpha_n (-u)^{k_n}),
%\alpha_1(-u)^{k_1}\cdots\alpha_{n-1}(-u)^{k_{n-1}}
\bigodot\limits_{i=1}^{n-1}\alpha_i(-u)^{k_i}
\bigg\rangle_\Muk
$$
\end{Lemma}
\begin{proof}
By the formula of CEI%~\eqref{eq:cei-formula-main} and
~\eqref{eq:explicit-evaluation}, we have
\begin{align*}
   & \sum_{j=1}^{n} \langle \alpha_1\psi^{k_1},\ldots,\alpha_j\psi^{k_j-1},\ldots,\alpha_n\psi^{k_n}\rangle_{g,n}^{A,\Omega,s} \\
  % = & \sum_{j=1}^{n-1} \langle \alpha_1 \psi^{k_1},\ldots,\alpha_j\psi^{k_j-1},\ldots,\alpha_n\psi^{k_n}\rangle_{0,n}^{A,\Omega,s}+ \langle \alpha_1 \psi^{k_1},\ldots,\alpha_n\psi^{k_n-1}\rangle_{0,n}^{A,\Omega,s}\\
   = & \sum_{j=1}^{n-1} \bigg\langle %(\mathcal{K}_*\rho_*^{A,\tw}\hcV)_{0,1,n-1}
   \lbA_{g,1,n-1}
   (\alpha_n (-u)^{k_n}),
   %\alpha_1(-u)^{k_1}\cdots\alpha_j(-u)^{k_j-1}\cdots\alpha_{n-1}(-u)^{k_{n-1}}
   \bigodot\limits_{i=1}^{n-1}\alpha_i(-u)^{k_i-\delta_i^j}
   \bigg\rangle_\Muk 
   %\hspace{3cm} 
   +\bigg\langle %(\mathcal{K}_*\rho_*^{A,\tw}\hcV)_{0,1,n-1}
   \lbA_{g,1,n-1}
   (\alpha_n (-u)^{k_n-1}),
   %\alpha_1(-u)^{k_1}\cdots\alpha_{n-1}(-u)^{k_{n-1}}
   \bigodot\limits_{i=1}^{n-1}\alpha_i(-u)^{k_i}
   \bigg\rangle_\Muk\\
   = & \bigg\langle M_{u^{-1}} %(\mathcal{K}_*\rho_*^{A,\tw}\hcV)_{0,1,n-1}
   \lbA_{g,1,n-1}
   (\alpha_n (-u)^{k_n}), 
   %\alpha_1 (-u)^{k_1}\cdots\alpha_{n-1}(-u)^{k_{n-1}}
   \bigodot\limits_{i=1}^{n-1}\alpha_i(-u)^{k_i}
   \bigg\rangle_\Muk 
   %\hspace{3cm} 
      %\;\;
      -\bigg\langle %(\mathcal{K}_*\rho_*^{A,\tw}\hcV)_{0,1,n-1}
      \lbA_{g,1,n-1}
      M_{u^{-1}}(\alpha_n (-u)^{k_n}),
      %\alpha_1(-u)^{k_1}\cdots\alpha_{n-1}(-u)^{k_{n-1}}
      \bigodot\limits_{i=1}^{n-1}\alpha_i(-u)^{k_i}
      \bigg\rangle_\Muk
      %\\=& \left\langle [M_{u^{-1}}, \lbA_{g,1,n-1}](\alpha_n (-u)^{k_n}),
%\alpha_1(-u)^{k_1}\cdots\alpha_{n-1}(-u)^{k_{n-1}}\bigodot\limits_{i=1}^{n-1}\alpha_i(-u)^{k_i}\right\rangle_\Muk
%\\=&\sum_{\GG\in\Gamma((g=0,1,n-1))} \frac{1}{|\Aut(\GG)|} 
%\left\langle[M_{u^{-1}}, \rho^{A,\tw}_\GG](\alpha_n (-u)^{k_n}),\bigodot\limits_{i=1}^{n-1}\alpha_i(-u)^{k_i}\right\rangle_\Muk
\end{align*}
%Now the formula follows from Equation~\eqref{eq:cei-formula-main}. 
%the term $(\mathcal{K}_*\rho_*^{A,\tw}\hcV)_{0,1,n-1}$ is given by the following graph sum
% \begin{equation}\label{eq:string-rhs-genus-zero} \sum_{\GG\in\Gamma((g=0,1,n-1))} \frac{1}{|\Aut(\GG)|}  \rho^{A,\tw}_\GG.  \end{equation}
%see Section~\ref{para:trivialization}.
Now, the result follows from the definition in~\eqref{commutator-mu}.
\end{proof}

\iffalse
\begin{align*}
&\sum_{j=1}^{n} \langle \alpha_1\psi^{k_1},\ldots,\alpha_j\psi^{k_j-1},\ldots,\alpha_n\psi^{k_n}\rangle_{0,n}^{A,\Omega,s}\\
=&
\sum_{\GG\in\Gamma((g=0,1,n-1))} \frac{1}{|\Aut(\GG)|} 
\bigg\langle
[M_{u^{-1}}, \rho^{A,\tw}_\GG](\alpha_n (-u)^{k_n}),
%\alpha_1(-u)^{k_1}\cdots\alpha_{n-1}(-u)^{k_{n-1}}
\bigodot\limits_{i=1}^{n-1}\alpha_i(-u)^{k_i}
\bigg\rangle_\Muk
\end{align*}
%\[[M_{u^{-1}},(\mathcal{K}_*\rho_*^{A,\tw}\hcV)_{0,1,n-1}] = \sum_{\GG\in\Gamma((g=0,1,n-1))} \frac{1}{|\Aut(\GG)|} [M_{u^{-1}}, \rho^{A,\tw}_\GG], \]
Note that we have used the fact that all weights are equal to $1$ for genus zero graphs.
\fi

In order to prove the string equation~\eqref{eq:string-equation}, it suffices to prove
\begin{equation}
\label{string-main-identity}
[M_{u^{-1}},\lbA_{g,1,n-1}]=[(\cK_*\rho^{A,\tw,\omega_A}_*\mc_2)_{g,1,n-1}^\epsilon]
\in H_\bullet\left( {\sf Hom}^\cont \big( L^A_+[1] , \sym^{n-1} L^A_- \big), b\right).
\end{equation}

\subsection{Calculations of push-forward}
%The identification of the right hand side of Equation~\eqref{eq:string-equation} with the appropriate component of the second push-forward $\cK_*\rho^{A,\tw,\omega_A}_*\mc_2$ is much more involved. 

The proof of equation~\eqref{string-main-identity} is much more involved. 
Let us first compute the components of $\rho^{A,\tw,\omega_A}_*\mc_2$ more explicitly. 

Let $\Gamma$ be an oriented black-and-white graph and $\alpha_Q$ be an equivariant chain 
\begin{equation*}
\alpha_Q = \Gamma \cdot \prod_{i=1}^{k}w_i^{-a_i}\prod_{i=1}^{l}u_i^{-b_i} \in C^{\sf comb}_\bullet(M_{g,k,l+Q}^{\sf fr}, \underline{\sgn})_{\sf hS}.
\end{equation*}
The map
$\cF': C^{\sf comb}_\bullet(M_{g,k,l+Q}^{\sf fr}, \underline{\sgn})_{\sf hS} \to C^{\sf comb}_\bullet(M_{g,k,l}^{\sf fr}, \underline{\sgn})_{\sf hS}$ is defined by
\begin{equation}\label{eq:def-forgetful-prime}
\cF'(\alpha_Q)= 
\begin{cases}
(\Gamma/e)\cdot \prod\limits_{i=1}^{k} w_i^{-a_i}\prod\limits_{i=1}^{l} u_i^{-b_i}, &\mbox{if ${\sf val}(Q)=1,$ $v\in V_b$, and ${\sf val}(v)=3$},\\
(\Gamma\backslash Q) \cdot\prod\limits_{i=1}^{k} w_i^{-a_i}\prod\limits_{i=1}^{l} u_i^{-b_i}, &\mbox{if ${\sf val}(Q)=1, v\in V_w,$ and the starting half-edge of $v$ is along $e$},\\
0, &\mbox{otherwise.}
\end{cases}
\end{equation} 
The construction of graphs $\Gamma/e$ and $\Gamma\backslash Q$ here is the same as that used in Equation~\eqref{eq:def-forgetful1}. They are illustrated in the following pictures, respectively. 
\[\begin{tikzpicture}[baseline={([yshift=-0.5ex]current bounding
		box.center)},scale=0.4] 
	\draw [thick] (-6.2,2) to (-6.2,0);
	\draw [thick] (-6.2, -2) to (-6.2, 0);
	\draw (-5.8,0) node[label=left:{$v$}] {};
	\draw (-7,0) node[label=left:{$\Gamma=$}] {};
    \node at (-4.6,.8) {$e$};
	\draw (5,0) node[label=left:{$\Gamma/e=$}] {};
	\draw (-3,0) node[label=below:{$Q$}] {};
	\draw [thick] (-6.2,0) to (-3,0);
	\draw [thick] (-2.8,0) circle [radius=0.2];
	\draw [thick] [|->] (-0.5,0) to (1,0);
%	\node at (4.8,.8) {$\cF'$};
	\draw [thick] (5,2) to (5,0);
	\draw [thick] (5, -2) to (5, 0);
\end{tikzpicture}\]
\[\begin{tikzpicture}[baseline={([yshift=-0.5ex]current bounding
		box.center)},scale=0.4] 
	\draw (-6,0) node[label=left:{$\Gamma=$}] {};
	\draw [thick] (-3,0) circle [radius=0.2];
	\draw (-3,0) node[label=below:{$v$}] {};
	\draw [thick] (-.3,0) circle [radius=0.2];
    \node at (-1.3,.6) {$e$};
	\draw (-.3,0) node[label=below:{$Q$}] {};
	\draw [thick] (-6,-0) to (-3.2,0);
	\draw [thick] (-4,2) to (-3.2,0);
	\draw [thick] (-4,-2) to (-3.2,0);
	\draw [thick] (-1,2) to (-2.8,0);
	\draw [thick] (-1,-2) to (-2.8,0);
	\draw [thick] (-2.8,0) to (-.5,0);
	\draw [line width=2.4pt] (-1.5,0) to (-2.8,0);
	\draw [thick] [|->] (1.5,0) to (3,0);
%	\node at (3,.8) {$\cF'$};
	\draw (7,0) node[label=left:{$\Gamma\backslash Q=$}] {};
	\draw [thick] (10,0) circle [radius=0.2];
	\draw (10,0) node[label=below:{$v$}] {};
	\draw [thick] (7,-0) to (9.8,0);
	\draw [thick] (9,2) to (9.8,0);
	\draw [thick] (9,-2) to (9.8,0);
	\draw [thick] (12,2) to (10.2,0);
	\draw [thick] (12,-2) to (10.2,0);
	\draw [line width=2.4pt] (10.2,0) to (11.7,0);
\end{tikzpicture}\]

\begin{Lemma}\label{lem:F_Q-cd}
The following diagram is commutative:
\[\begin{CD}
C^{\sf comb}_\bullet(M_{g,k,l+Q}^{\sf fr}, \underline{\sgn})_{\sf hS}  @>{\rho^{A,\tw}}>> {\sf Hom}^c\big( \sym^k (L^A_+[1]), L^A_-\otimes \sym^{l} L^A_-\big)\\
@V \cF' VV        @VV C_{\omega_A}^Q V\\
C^{\sf comb}_\bullet(M_{g,k,l}^{\sf fr}, \underline{\sgn})_{\sf hS}  @>{\rho^{A,\tw}}>> {\sf Hom}^c\big( \sym^k (L^A_+[1]), \sym^l L^A_-\big)
\end{CD}\]
Here in the top line, the output element at the distinguished white vertex $Q$ lies in the first component in the tensor product $L^A_-\otimes \sym^{l} L^A_-$, and the contraction map $C_{\omega_A}^Q$ is applied to this component.
\end{Lemma}

%For simplicity, let us denote the maps in Lemma~\ref{lem:string-evaluation-exceptional} and Lemma~\ref{lem:string-evaluation} by
%\begin{align}\label{eq:cont-string-v_0}   
%\begin{split}
%\left\{ \begin{aligned} 
%\gamma_{g,k,l} & := \Big( \rho^{A,\tw,\omega_A}_*\mc_2\Big)_{g,k,l}^\epsilon= [M_{u^{-1}},\hbA_{g,k,l}],  \mbox{\;\;\;\;for stable $(g,k,l)$.}\\
%    \gamma_{0,1,1} & := \Big( \rho^{A,\tw,\omega_A}_*\mc_2\Big)_{0,1,1}^\epsilon.
%\end{aligned} \right.
%   \end{split}
%\end{align}

For simplicity, let us denote by
\begin{align}\label{eq:cont-string-v_0}   
  \gamma_{g,k,l}  := \Big( \rho^{A,\tw,\omega_A}_*\mc_2\Big)_{g,k,l}^\epsilon.
\end{align}

\begin{Lemma}
\label{lem:string-evaluation}
%For the exceptional terms, we have
We have
\begin{equation}
    \label{exceptional-constraction}
\begin{dcases}
%\Big( \rho^{A,\tw,\omega_A}_*\mc_2\Big)_{0,2,0}^\epsilon
\gamma_{0,2,0}=0; &\\
%\Big( \rho^{A,\tw,\omega_A}_*\mc_2\Big)_{0,1,1}^\epsilon
\gamma_{0,1,1}(x \cdot u^i)=\delta_0^i \cdot x, & \forall \quad x\cdot u^i \in L^A_+[1];\\
\gamma_{g,k,l}=[M_{u^{-1}},\hbA_{g,k,l}], & \text{for } (g,k,l) \text{ stable}. 
\end{dcases}
\end{equation}
\end{Lemma}
\begin{proof}
For $\gamma_{0,2,0}$, we notice that 
the first two graphs of $f\widehat{\mathcal{V}}_{0,2,1}$ in~\eqref{eq:W-021} have ${\rm val}(Q)=2$ and the last two graphs have ${\rm val}(v)=4$. By~\eqref{eq:def-forgetful-prime}, we have $\cF'(f\widehat{\mathcal{V}}_{0,2,1})=0$. Using Lemma~\ref{lem:F_Q-cd}, we have
$$\gamma_{0,2,0}:=\Big( \rho^{A,\tw,\omega_A}_*\mc_2\Big)_{0,2,0}^\epsilon
=C^Q_{\omega_A}\big(\rho^{A,\tw}(f\widehat{\mathcal{V}}_{0,2,1})\big)= \rho^{A,\tw}\big(\cF'(f\widehat{\mathcal{V}}_{0,2,1})\big)=0.$$

For the second exceptional term, we first compute
\begin{equation}
\label{eq:forgetting-getzler-1}
\cF'({\sf Get}_1^\sym)= \cF'\left( \frac{1}{2} 
 \begin{tikzpicture}[baseline={([yshift=-1.2ex]current bounding
		box.center)},scale=0.3] 
	\draw [thick] (0,0) to (0,2);
	\draw [thick] (-0.2, 1.8) to (0.2, 2.2);
	\draw [thick] (0.2, 1.8) to (-0.2, 2.2);
	\draw [thick] (0,0) to (-2,0);
	\draw (-2,0) node[label=left:{$Q$}] {};
	\draw [thick] (0,0) to (2,0);
	\draw [thick] (-2.2,0) circle [radius=0.2];
	\draw [thick] (2.2,0) circle [radius=0.2];
\end{tikzpicture}+\frac{1}{2} \begin{tikzpicture}[baseline={([yshift=-1.2ex]current bounding
		box.center)},scale=0.3] 
	\draw [thick] (0,0) to (0,2);
	\draw [thick] (-0.2, 1.8) to (0.2, 2.2);
	\draw [thick] (0.2, 1.8) to (-0.2, 2.2);
	\draw [thick] (0,0) to (-2,0);
	\draw (5,0) node[label=left:{$Q$}] {};
	\draw [thick] (0,0) to (2,0);
	\draw [thick] (-2.2,0) circle [radius=0.2];
	\draw [thick] (2.2,0) circle [radius=0.2];
\end{tikzpicture}\right)= \begin{tikzpicture}[baseline={([yshift=-1.2ex]current bounding
	box.center)},scale=0.3] 
\draw [thick] (0,-1) to (0,2);
\draw [thick] (-0.2, 1.8) to (0.2, 2.2);
\draw [thick] (0.2, 1.8) to (-0.2, 2.2);
\draw [thick] (0,-1.2) circle [radius=0.2];
\end{tikzpicture}
\end{equation}

Using Lemma~\ref{lem:F_Q-cd} we obtain
\[
\gamma_{0,1,1}:=\Big( \rho^{A,\tw,\omega_A}_*\mc_2\Big)_{0,1,1}^\epsilon = C_{\omega_A}^Q\big(\rho^{A,\tw}({\sf Get}_1^\sym)\big)=\rho^{A,\tw}(\cF'{\sf Get}_1^\sym)=\rho^{A,\tw}\left(\; \begin{tikzpicture}[baseline={([yshift=-1.2ex]current bounding
		box.center)},scale=0.3] 
	\draw [thick] (0,-1) to (0,2);
	\draw [thick] (-0.2, 1.8) to (0.2, 2.2);
	\draw [thick] (0.2, 1.8) to (-0.2, 2.2);
	\draw [thick] (0,-1.2) circle [radius=0.2];
\end{tikzpicture}\;\right). \]
This is exactly the map that sends $x\cdot u^{i}\in L^A_+[1]$ to $x$ if $i=0$ and to $0$ otherwise.
%to $\delta^{i}_{0}\cdot x\in L^A_-$. 

%\end{proof}

%\begin{Lemma}\label{lem:string-evaluation}
%For $(g,k,l)$ stable, i.e., $2g-2+k+l>0$, we have 
%$$\Big( \rho^{A,\tw,\omega_A}_*\mc_2\Big)_{g,k,l}^\epsilon =  [M_{u^{-1}},\hbA_{g,k,l}].$$
\iffalse
\[ \rho^{A,\tw,\omega_A}\big( \mathcal{T} \widehat{\mathcal{V}}_{g,k,l} + \mathcal{S} \widehat{\mathcal{V}}_{g,k-1,l+1} \big)=  [M_{u^{-1}},\hbA_{g,k,l}].\]
for stable $(g,k,l)$'s. And for the exceptional terms, we have
\begin{itemize}
  \item The map $\rho^{A,\tw,\omega_A}\big( f\widehat{\mathcal{V}}_{0,1,2} \big): L^A_+[1] \ra L^A_-$ is given by
  \[x \cdot u^i \mapsto \begin{cases}
      x, \;\;\;\mbox{if $i=0$,}\\
      0, \;\;\;\mbox{otherwise.}
  \end{cases}\]   
\item $ \rho^{A,\tw,\omega_A}\big( f\widehat{\mathcal{V}}_{0,2,1} \big)=0$. 
\end{itemize} 
\fi

%\end{Lemma}
%\begin{proof}
%Using $\cF'$ constructed above, we may compute as
%\begin{align*}
%\rho^{A,\tw,\omega_A}\big( \mathcal{T} \widehat{\mathcal{V}}_{g,k,l} + \mathcal{S} \widehat{\mathcal{V}}_{g,k-1,l+1} \big) & = \rho^{A,\tw}\Big( \cF' \big( \mathcal{T} \widehat{\mathcal{V}}_{g,k,l} + \mathcal{S} \widehat{\mathcal{V}}_{g,k-1,l+1} \big) \Big)
%\end{align*}

Now we consider the stable terms. 
Recall the construction of $\mathcal{T}=H^{\sf in} +H^{\sf out}+G^{\sf in}+G^{\sf out}+ \mathcal{T}^{\sf v}$ from Equation~\eqref{construction-T}. Using the explicit formulas of $\cF'$, we have $\cF'({\sf Get}_0^\sym)=0$. This implies 
\[ \cF' H^{\sf in}= \cF' H^{\sf out} = 0.\]
We also have $\cF' H^{\sf v}=0$ since a black vertex that is turned white by $H^{\sf v}$ would be at least trivalent. 

Similarly, from the construction of $\mathcal{S}$ in~\eqref{construction-S} and the equation $\cF'(f\widehat{\mathcal{V}}_{0,2,1})=0$, we also have
\[\cF'\mathcal{S} =0.\]
Now, let us compute $\cF'G^{\sf in}$ and $\cF'G^{\sf out}$.
The formula~\eqref{eq:forgetting-getzler-1} plays the role of the identity graph when computing compositions of black-and-white graphs, hence, we have
\begin{align*}
\cF'G^{\sf in}\widehat{\mathcal{V}}_{g,k,l} 
& = \cF'\left(\sum_{j=1}^k w_j^{-1} \widehat{\mathcal{V}}_{g,k,l} \prescript{}{j}\circ {\sf Get}_1^\sym \right) 
    = \sum_{j=1}^k w_j^{-1} 
    \left(
    \widehat{\mathcal{V}}_{g,k,l} \prescript{}{j}\circ \;\; 
    \begin{tikzpicture}[baseline={([yshift=-1.2ex]current bounding box.center)},scale=0.3] 
		\draw [thick] (0,-1) to (0,2);
		\draw [thick] (-0.2, 1.8) to (0.2, 2.2);
		\draw [thick] (0.2, 1.8) to (-0.2, 2.2);
		\draw [thick] (0,-1.2) circle [radius=0.2];
	\end{tikzpicture}
    \;
    \right) 
    =  \left(\sum_{j=1}^k w_j^{-1} \right)  \widehat{\mathcal{V}}_{g,k,l},\\
	\cF'G^{\sf out}\widehat{\mathcal{V}}_{g,k,l} & = \cF'\left( \sum_{i=1}^l u_i^{-1} {\sf Get}_1^\sym \circ_i \widehat{\mathcal{V}}_{g,k,l}\right) = \sum_{i=1}^l u_i^{-1} 
    \left(\; \begin{tikzpicture}[baseline={([yshift=-1.2ex]current bounding box.center)},scale=0.3] 
		\draw [thick] (0,-1) to (0,2);
		\draw [thick] (-0.2, 1.8) to (0.2, 2.2);
		\draw [thick] (0.2, 1.8) to (-0.2, 2.2);
		\draw [thick] (0,-1.2) circle [radius=0.2];
	\end{tikzpicture}
    \;\circ_i \widehat{\mathcal{V}}_{g,k,l}\right) 
    = \left(\sum_{i=1}^l u_i^{-1}\right)  \widehat{\mathcal{V}}_{g,k,l}.
\end{align*}
%Putting these vanishing results together, we obtain
%\[ \rho^{A,\tw,\omega_A}\big( \mathcal{T} \widehat{\mathcal{V}}_{g,k,l} + \mathcal{S} \widehat{\mathcal{V}}_{g,k-1,l+1} \big)  = \rho^{A,\tw}\Big( \cF' \big(G^{\sf in} + G^{\sf out}\big) \widehat{\mathcal{V}}_{g,k,l}\Big).\]
Putting all these formulas together, we have
\begin{align*}
 \rho^{A,\tw,\omega_A}\big( \mathcal{T} \widehat{\mathcal{V}}_{g,k,l} + \mathcal{S} \widehat{\mathcal{V}}_{g,k-1,l+1} \big) & = \rho^{A,\tw}\Big( \cF' \big( \mathcal{T} \widehat{\mathcal{V}}_{g,k,l} + \mathcal{S} \widehat{\mathcal{V}}_{g,k-1,l+1} \big) \Big)\\
 &  = \rho^{A,\tw}\Big( \cF' \big(G^{\sf in} + G^{\sf out}\big) \widehat{\mathcal{V}}_{g,k,l}\Big)\\
&=\rho^{A,\tw}\Big((u_1^{-1}+\cdots+u_l^{-1}+w_1^{-1}+\cdots+w_k^{-1})\cdot \widehat{\mathcal{V}}_{g,k,l}\Big)\\
 &= [M_{u^{-1}}, \hbA_{g,k,l}].
\end{align*}
Note that in the last step we used the fact that, by construction of $\rho^{A,\tw}$ (see Equation~\eqref{eq:action-map}), there is a sign depending on the power of the circle parameter $w$'s. This makes the final calculation a commutator with $M_{u^{-1}}$, as desired.
\end{proof}

\subsection{Some commutators relations}

We list some commutator relations here. 
\begin{Lemma}
We have
\begin{equation}
\label{eq:gamma-commutator}
[b+uB, \gamma_{0,1,1}]=0.
\end{equation}
\end{Lemma}
\begin{proof}
Using Equation~\eqref{exceptional-constraction}, we can verify directly 
\begin{align*}
[b+uB,\gamma_{0,1,1}]\big(\sum_{i\geq 0} x_i\cdot u^{i}\big)
%&= (b+uB)(x_0)-(-1)^{|\gamma_{0,1,1}|\cdot |b+uB|}\gamma_{0,1,1} \Big(-(b+uB)\big(\sum_{i\geq 0} x_i\cdot u^{i}\big)\Big)\\
&= (b+uB)(x_0)-(-1)\gamma_{0,1,1} 
\Big((-1)(b+uB)\big(\sum_{i\geq 0} x_i\cdot u^{i}\big)\Big)\\
&=b(x_0)-b(x_0)\\
&=0.    
\end{align*}
The term $uB(x_0)$ in $(b+uB)(x_0)$ vanishes because here $b+uB\in \Hom(L_-^A, L_-^A).$ In the first equality, we have the first $(-1)$ because both $\gamma_{0,1,1}$ and $b+uB$ are odd, and we have the second $(-1)$ in front of $b+uB$ due to the shift in $L_{+}^A[1]$.
\end{proof}

Next, we consider the commutators of $M_{u^{-1}}$. 
\begin{Lemma}
\label{lemma-commutator-Mu}
We have 
\begin{align}
\label{eq:iota-u}[M_{u^{-1}}, \iota]&=0.
\\
\label{eq:Delta-Mu}
    [M_{u^{-1}}, \Delta]&=0.
\\
\label{eq:pair-string}
 [M_{u^{-1}}, b+uB] &= \{\gamma_{0,1,1},-\}_\hbar.
\end{align}
\end{Lemma}
\begin{proof}
These equations can be checked directly from the definitions. In particular, the second equation holds because the operator $\Delta: L_-^A\otimes L_-^A\to \bbC$ vanishes on entries with strictly negative $u$-power.
\end{proof}

From the formula~\eqref{eq:cei-formula-main} of $\lbA_{g=0,1,n-1}$, we need to compute the commutators between $M_{u^{-1}}$ and the operators $S$, $R$ and $F$ defined in Equations~\eqref{eq:S-operator}, ~\eqref{eq:R-operator}~\eqref{eq:homotopy-F}.

%The notation $S$, $R$ and $F$ were defined as in Equations~\eqref{eq:S-operator}~\eqref{eq:R-operator}~\eqref{eq:homotopy-F}. Thus, in order to compute the commutator, we are led to compute the commutators between $M_{u^{-1}}$ and the operators $S$, $R$ and $F$. 

\begin{Lemma}\label{lem:string-leaf-cont}
We have the following identities.
%\begin{enumerate}
  %  \item As maps from $L_+^A[1] \to L_+^A[1]$, we have
    \begin{align}
        \label{commutator-S}
        [M_{u^{-1}}, S] &= F\circ \gamma_{0,1,1}\circ S\quad \in\quad{\sf End}(L_+^A[1]).\\
 %   \item As maps from $L_-^A \to L_-^A$, we have
        \label{commutator-R}
        [M_{u^{-1}}, R] &= R\circ \gamma_{0,1,1}\circ F
        \quad \in \quad {\sf End}(L_-^A).\\
  %  \item As maps from $L_-^A\to L_+^A[1]$, we have
        \label{commutator-F}
        [M_{u^{-1}},F] &= F\circ \gamma_{0,1,1}\circ F\quad \in \quad {\sf Hom}(L_-^A, L_+^A[1]).\\
[M_{u^{-1}},H]&=%H(M_{u^{-1}}(-),-)+H(-,M_{u^{-1}}(-))&=-H(\gamma_{0,1,1}F(-),-).
H(\gamma_{0,1,1}\circ F(-),-) \quad \in \quad \Hom(L_-^A\otimes L_-^A, \mathbb{C}).
\label{lem:edge-cancel-H}    
    \end{align}
%    \end{enumerate}
\end{Lemma}
\begin{proof}
The proof is a direct computation using the formula of $\gamma_{0,1,1}$ in~\eqref{exceptional-constraction} and   explicit formulas (Equations~\eqref{eq:S-operator}~\eqref{eq:R-operator}~\eqref{eq:homotopy-F}~\eqref{eq:homotopy-H}) of $S$, $R$, $F$ and $H$. 

Let us prove equation~\eqref{commutator-S}. 
For any element $x\cdot u^i\in L_+^A[1]$, if $i\geq 1$, we have
$$[M_{u^{-1}},S](x\cdot u^i)
%=-F\circ \gamma_{0,1,1}\circ S(\alpha u^k)
=M_{u^{-1}}(S(x\cdot u^i))-S(x\cdot u^{i-1})
=S(x)u^{i-1}-S(x)u^{i-1}
=0.$$ 
If $i=0$, we have
\begin{align*}
    [M_{u^{-1}},S](x)  =M_{u^{-1}}(S(x))-0
    =M_{u^{-1}} \left( x + \sum_{k=1}^{\infty} S_k(x)u^k\right)
    =\sum_{j=0}^{\infty} S_{j+1}(x) u^j.
\end{align*}
Using the formula of $\gamma_{0,1,1}$ in~\eqref{exceptional-constraction}, the formula of $F$, and the relation $R=S^{-1}$, we have
\begin{align*}
    F\circ \gamma_{0,1,1}\circ S(x\cdot u^i)
    =
   \delta_0^i F(x) = - \delta_0^i\sum_{j=0}^{\infty} \sum_{l=0}^{j} S_l R_{j+1-l}(x)u^{j}  = \delta_0^i\sum_{j=0}^{\infty}S_{j+1}(x)u^{j}.
\end{align*} 
Thus, we obtain equation~\eqref{commutator-S}.
The other two equations can be proved similarly.
%The equations in $(2)$ and $(3)$ are also such direct verifications.
\end{proof}

These formulas may be graphically represented as follows:
%at an incoming leaf as
%Similarly, the formula in~\eqref{commutator-R} may be depicted at an outgoing leaf as
%And the formula in~\eqref{commutator-F} is depicted at an edge as
\[
\text{Equation}~\eqref{commutator-S}: \quad 
\begin{tikzpicture}
[baseline={([yshift=.3ex]current bounding
      box.center)},scale=0.5] 
\draw[->-,thick] (0,4) to [] (0,0);
\draw[thick] (-.5,3.5) -- (.5,3.5);
\draw (1.9,3.5) node {$(-M_{u^{-1}})$};
\draw (.6,2) node {$S$};
\filldraw[black] (0,0) circle (4pt) node[anchor=north]{};
\end{tikzpicture}
=\begin{tikzpicture}
[baseline={([yshift=.3ex]current bounding
      box.center)},scale=0.5] 
\draw[->-,thick] (0,4) to [] (0,0);
\draw[thick] (-.5,.5) -- (.5,.5);
\draw (1.9,.5) node {$(-M_{u^{-1}})$};
\draw (.6,2) node {$S$};
\filldraw[black] (0,0) circle (4pt) node[anchor=north]{};
\end{tikzpicture} +\;\;\;\; 
\begin{tikzpicture}
[baseline={([yshift=.3ex]current bounding
      box.center)},scale=0.5] 
\draw[->-,thick] (0,4) to [] (0,2);
\draw[->-,thick] (0,2) to [] (0,0);
\draw (.6,3.5) node {$S$};
\draw (.6,1) node {$F$};
\filldraw[black] (0,0) circle (4pt) node[anchor=north]{};
\filldraw[black] (0,2) circle (4pt) node[anchor=west]{$\gamma_{0,1,1}$};
\end{tikzpicture}\]
\[
\text{Equation}~\eqref{commutator-R}: \quad 
\begin{tikzpicture}
[baseline={([yshift=.3ex]current bounding
     box.center)},scale=0.5] 
\draw[->-,thick] (0,4) to [] (0,0);
\draw[thick] (-.5,3.5) -- (.5,3.5);
\draw (1.9,3.5) node {$(-M_{u^{-1}})$};
\draw (.6,2) node {$R$};
\filldraw[black] (0,4) circle (4pt) node[anchor=north]{};
\end{tikzpicture}
=\begin{tikzpicture}[baseline={([yshift=.3ex]current bounding
      box.center)},scale=0.5] 
\draw[->-,thick] (0,4) to [] (0,0);
\draw[thick] (-.5,.5) -- (.5,.5);
\draw (1.9,.5) node {$(-M_{u^{-1}})$};
\draw (.6,2) node {$R$};
\filldraw[black] (0,4) circle (4pt) node[anchor=north]{};
\end{tikzpicture} +\;\;\;\; \begin{tikzpicture}[baseline={([yshift=.3ex]current bounding
      box.center)},scale=0.5] 
\draw[->-,thick] (0,4) to [] (0,2);
\draw[->-,thick] (0,2) to [] (0,0);
\draw (.6,3.2) node {$F$};
\draw (.6,1) node {$R$};
\filldraw[black] (0,4) circle (4pt) node[anchor=north]{};
\filldraw[black] (0,2) circle (4pt) node[anchor=west]{$\gamma_{0,1,1}$};
\end{tikzpicture}\]
%And the formula in~\eqref{commutator-F} is depicted at an edge as
%\medskip
\[
\text{Equation}~\eqref{commutator-F}: \quad 
\begin{tikzpicture}[baseline={([yshift=.3ex]current bounding
      box.center)},scale=0.5] 
      \filldraw[black] (0,4) circle (4pt) node[anchor=south]{};
\draw[->-,thick] (0,4) to [] (0,0);
\draw[thick] (-.5,3.5) -- (.5,3.5);
\draw (1.9,3.5) node {$(-M_{u^{-1}})$};
\draw (.6,2) node {$F$};
\filldraw[black] (0,0) circle (4pt) node[anchor=north]{};
\end{tikzpicture}=\begin{tikzpicture}[baseline={([yshift=.3ex]current bounding
      box.center)},scale=0.5] 
                \filldraw[black] (0,4) circle (4pt) node[anchor=south]{};
\draw[->-,thick] (0,4) to [] (0,0);
\draw[thick] (-.5,.5) -- (.5,.5);
\draw (1.9,.5) node {$(-M_{u^{-1}})$};
\draw (.6,2) node {$F$};
\filldraw[black] (0,0) circle (4pt) node[anchor=north]{};
\end{tikzpicture} +\;\;\;\; \begin{tikzpicture}[baseline={([yshift=.3ex]current bounding
      box.center)},scale=0.5] 
            \filldraw[black] (0,4) circle (4pt) node[anchor=south]{};
\draw[->-,thick] (0,4) to [] (0,2);
\draw[->-,thick] (0,2) to [] (0,0);
\draw (.6,3.5) node {$F$};
\draw (.6,1) node {$F$};
\filldraw[black] (0,0) circle (4pt) node[anchor=north]{};
\filldraw[black] (0,2) circle (4pt) node[anchor=west]{$\gamma_{0,1,1}$};
\end{tikzpicture}.
\]

%\begin{example}
We give an example of the graphical calculation when $n=4$. Consider the first graph of $\lbA_{0,1,3}$ in Equation~\eqref{eq:lbA-013}. We have 
\iffalse 
    &\frac{1}{3!} 
    \begin{tikzpicture}[baseline={(current bounding
    		box.center)},scale=0.5] \draw [thick,directed] (3.4,4) to (3.4,2);
    	\node at (3.4,2) {$\bullet$};
    	\draw[thick,directed] (3.4,2) to (1.4,0);
    	\draw[thick,directed] (3.4,2) to (3.4,0);
    	\draw[thick,directed] (3.4,2) to (5.4,0);
    \end{tikzpicture} \circ (-M_{u^{-1}})\\
    =
    \fi
\begin{align*}
    &  \frac{1}{3!} 
    \begin{tikzpicture}[baseline={(current bounding
    		box.center)},scale=0.5] \draw [thick,directed] (3.4,4) to (3.4,2);
    	\node at (3.4,2) {$\bullet$};
        \node at (5.5,3.5) {$(-M_{u^{-1}})$};
        \draw[thick] (3,3.5) -- (3.8,3.5);
    	\draw[thick,directed] (3.4,2) to (1.4,0);
    	\draw[thick,directed] (3.4,2) to (3.4,0);
    	\draw[thick,directed] (3.4,2) to (5.4,0);
    \end{tikzpicture} \\
     =&   \frac{1}{3!} 
    \begin{tikzpicture}[baseline={(current bounding
    		box.center)},scale=0.5] \draw [thick,directed] (3.4,4) to (3.4,2);
    	\node at (3.4,2) {$\bullet$};
        \node at (5.5,2.5) {$(-M_{u^{-1}})$};
        \draw[thick] (3,2.5) -- (3.8,2.5);
    	\draw[thick,directed] (3.4,2) to (1.4,0);
    	\draw[thick,directed] (3.4,2) to (3.4,0);
    	\draw[thick,directed] (3.4,2) to (5.4,0);
    \end{tikzpicture} +\frac{1}{3!} 
    \begin{tikzpicture}[baseline={(current bounding
    		box.center)},scale=0.5] 
            \draw [thick,directed] (3.4,4) to (3.4,3);
            \draw [thick,directed] (3.4,3) to (3.4,2);
    	\node at (3.4,3) {$\bullet$};
        \node at (4.5,3) {$\gamma_{0,1,1}$};
    	\node at (3.4,2) {$\bullet$};
    	\draw[thick,directed] (3.4,2) to (1.4,0);
    	\draw[thick,directed] (3.4,2) to (3.4,0);
    	\draw[thick,directed] (3.4,2) to (5.4,0);
    \end{tikzpicture}\\
    = & \frac{1}{2!} 
    \begin{tikzpicture}[baseline={(current bounding
    		box.center)},scale=0.5] \draw [thick,directed] (3.4,4) to (3.4,2);
    	\node at (3.4,2) {$\bullet$};
        \node at (1.2,1.8) {$(-M_{u^{-1}})$};
        \draw[thick] (2.7,1.8) -- (3.2,1.3);
    	\draw[thick,directed] (3.4,2) to (1.4,0);
    	\draw[thick,directed] (3.4,2) to (3.4,0);
    	\draw[thick,directed] (3.4,2) to (5.4,0);
    \end{tikzpicture}  +  \frac{1}{3!} 
    \begin{tikzpicture}[baseline={(current bounding
    		box.center)},scale=0.5] \draw [thick,directed] (3.4,4) to (3.4,2);
    	\node at (3.4,2) {$\bullet$};
        \node at (4.5,2) {$\gamma_{0,1,3}$};
    	\draw[thick,directed] (3.4,2) to (1.4,0);
    	\draw[thick,directed] (3.4,2) to (3.4,0);
    	\draw[thick,directed] (3.4,2) to (5.4,0);
    \end{tikzpicture} +\frac{1}{3!} 
    \begin{tikzpicture}[baseline={(current bounding
    		box.center)},scale=0.5] 
            \draw [thick,directed] (3.4,4) to (3.4,3);
            \draw [thick,directed] (3.4,3) to (3.4,2);
    	\node at (3.4,3) {$\bullet$};
        \node at (4.5,3) {$\gamma_{0,1,1}$};
    	\node at (3.4,2) {$\bullet$};
    	\draw[thick,directed] (3.4,2) to (1.4,0);
    	\draw[thick,directed] (3.4,2) to (3.4,0);
    	\draw[thick,directed] (3.4,2) to (5.4,0);
    \end{tikzpicture}.\\
    =& (-M_{u^{-1}})\circ  \frac{1}{3!} 
    \begin{tikzpicture}[baseline={(current bounding
    		box.center)},scale=0.5] \draw [thick,directed] (3.4,4) to (3.4,2);
    	\node at (3.4,2) {$\bullet$};
    	\draw[thick,directed] (3.4,2) to (1.4,0);
    	\draw[thick,directed] (3.4,2) to (3.4,0);
    	\draw[thick,directed] (3.4,2) to (5.4,0);
    \end{tikzpicture} + \frac{1}{2!} 
    \begin{tikzpicture}[baseline={(current bounding
    		box.center)},scale=0.5] \draw [thick,directed] (3.4,4) to (3.4,2);
    	\node at (3.4,2) {$\bullet$};
    	\draw[thick,directed] (3.4,2) to (2.4,1);
        \draw[thick,directed] (2.4,1) to (1.4,0);
        \node at (2.4,1) {$\bullet$};
        \node at (1.5,1.2) {$\gamma_{0,1,1}$};
    	\draw[thick,directed] (3.4,2) to (3.4,0);
    	\draw[thick,directed] (3.4,2) to (5.4,0);
    \end{tikzpicture}  +  \frac{1}{3!} 
    \begin{tikzpicture}[baseline={(current bounding
    		box.center)},scale=0.5] \draw [thick,directed] (3.4,4) to (3.4,2);
    	\node at (3.4,2) {$\bullet$};
        \node at (4.5,2) {$\gamma_{0,1,3}$};
    	\draw[thick,directed] (3.4,2) to (1.4,0);
    	\draw[thick,directed] (3.4,2) to (3.4,0);
    	\draw[thick,directed] (3.4,2) to (5.4,0);
    \end{tikzpicture} +\frac{1}{3!} 
    \begin{tikzpicture}[baseline={(current bounding
    		box.center)},scale=0.5] 
            \draw [thick,directed] (3.4,4) to (3.4,3);
            \draw [thick,directed] (3.4,3) to (3.4,2);
    	\node at (3.4,3) {$\bullet$};
        \node at (4.5,3) {$\gamma_{0,1,1}$};
    	\node at (3.4,2) {$\bullet$};
    	\draw[thick,directed] (3.4,2) to (1.4,0);
    	\draw[thick,directed] (3.4,2) to (3.4,0);
    	\draw[thick,directed] (3.4,2) to (5.4,0);
    \end{tikzpicture}.
\end{align*}

Thus the commutator $[M_{u^{-1}},-]$ is given by 
\begin{equation}\label{eq:example-string-013-i}
\left[\ M_{u^{-1}},\frac{1}{3!} 
    \begin{tikzpicture}[baseline={(current bounding
    		box.center)},scale=0.5] \draw [thick,directed] (3.4,4) to (3.4,2);
    	\node at (3.4,2) {$\bullet$};
    	\draw[thick,directed] (3.4,2) to (1.4,0);
    	\draw[thick,directed] (3.4,2) to (3.4,0);
    	\draw[thick,directed] (3.4,2) to (5.4,0);
    \end{tikzpicture}\ \right] =   \frac{1}{2!} 
    \begin{tikzpicture}[baseline={(current bounding
    		box.center)},scale=0.5] \draw [thick,directed] (3.4,4) to (3.4,2);
    	\node at (3.4,2) {$\bullet$};
    	\draw[thick,directed] (3.4,2) to (2.4,1);
        \draw[thick,directed] (2.4,1) to (1.4,0);
        \node at (2.4,1) {$\bullet$};
        \node at (1.5,1.2) {$\gamma_{0,1,1}$};
    	\draw[thick,directed] (3.4,2) to (3.4,0);
    	\draw[thick,directed] (3.4,2) to (5.4,0);
    \end{tikzpicture}  + \frac{1}{3!} 
    \begin{tikzpicture}[baseline={(current bounding
    		box.center)},scale=0.5] \draw [thick,directed] (3.4,4) to (3.4,2);
    	\node at (3.4,2) {$\bullet$};
        \node at (4.5,2) {$\gamma_{0,1,3}$};
    	\draw[thick,directed] (3.4,2) to (1.4,0);
    	\draw[thick,directed] (3.4,2) to (3.4,0);
    	\draw[thick,directed] (3.4,2) to (5.4,0);
    \end{tikzpicture} +\frac{1}{3!} 
    \begin{tikzpicture}[baseline={(current bounding
    		box.center)},scale=0.5] 
            \draw [thick,directed] (3.4,4) to (3.4,3);
            \draw [thick,directed] (3.4,3) to (3.4,2);
    	\node at (3.4,3) {$\bullet$};
        \node at (4.5,3) {$\gamma_{0,1,1}$};
    	\node at (3.4,2) {$\bullet$};
    	\draw[thick,directed] (3.4,2) to (1.4,0);
    	\draw[thick,directed] (3.4,2) to (3.4,0);
    	\draw[thick,directed] (3.4,2) to (5.4,0);
    \end{tikzpicture}.\end{equation}

We observe that the sum of graphs on the right hand side of Equation~\eqref{eq:example-string-013-i} is precisely the contribution of $(\cK_*\rho^{A,\tw,\omega_A}_*\mc_2)_{0,1,3}^\epsilon$ from the first term in~\eqref{eq:lbA-013}.    

\subsection{Towards proving the string equation}
\subsubsection{A proof in genus zero.}
\label{para:string-genus-zero} 
%Now we give a proof of the string equation~\eqref{eq:string-equation} in genus zero.
%According to the previous discussions, we need to prove the following equation
%\[ [M_{u^{-1}},\overline{\iota}(F^{A,s}_{0,n})]=(\cK_*\rho^{A,\tw,\omega_A}_*\mc_2)_{0,1,n-1}^\epsilon.\] 
%This is equivalent to \[ -\lbA_{g,1,n-1}M_{u^{-1}}= - M_{u^{-1}} \lbA_{g,1,n-1} + (\cK_*\rho^{A,\tw,\omega_A}_*\mc_2)_{0,1,n-1}^\epsilon.\]
%The operators associated with vertices, leaves, and edges are as labeled in the graph. 
%Now we prove equation~\eqref{string-main-identity} when $g=0$.
We first prove Equation~\eqref{string-main-identity} when $g=0$ and obtain the following.
\begin{Proposition}\label{prop:string-genus-zero}
The string equation~\eqref{eq:string-equation} holds when $g=0$.
\end{Proposition}
\begin{proof}
We fix a partially directed graph $\GG\in \Gamma((g=0,1,n-1))$, and consider the term $- \rho^{A,\tw}_\GG M_{u^{-1}}$ on the left hand side. We may think of the operator $-M_{u^{-1}}$ as it flows along $\GG$ from top to bottom. As it passes the (unique) incoming leaf, using the identity $(1)$ in Lemma~\ref{lem:string-leaf-cont} we see that it yields a binary vertex at the incoming leaf labeled by the operator $\gamma_{0,1,1}$.

Continuing to pass the operator $-M_{u^{-1}}$ at a vertex, we obtain the commutator $[M_{u^{-1}}, \hbA_{g(v),k(v),l(v)}]$, which is equal to $\gamma_{g(v),k(v),l(v)}$ by definition. Then, we pass the operator $-M_{u^{-1}}$ along an internal edge. In this case, using the identity $(3)$ in Lemma~\ref{lem:string-leaf-cont}, we obtain a binary vertex at the edge again labeled by the operator $\gamma_{0,1,1}$. Continuing in this way, we keep passing the operator $-M_{u^{-1}}$ down until finally it is at one of the outgoing leaves of the tree $\GG$. Then, using the identity $(2)$ in Lemma~\ref{lem:string-leaf-cont} we obtain a binary vertex at the outgoing leaf labeled by the operator $\gamma_{0,1,1}$. In the end, this process yields the term $-M_{u^{-1}}\lbA_{g,1,n-1}$, while the intermediate commutator terms involving an operator from the $\gamma$'s give precisely the desired component $(\cK_*\rho^{A,\tw,\omega_A}_*\mc_2)_{0,1,n-1}^\epsilon$.
\end{proof}

\subsubsection{An example in genus one} 
In general, unlike the identity~\eqref{string-main-identity} as $b$-homology classes, the elements
$[M_{u^{-1}},\lbA_{g,1,n-1}]$ and 
%(\mathcal{K}_*\rho_*^{A,\tw}\hcV)_{g,1,n-1}
$(\cK_*\rho^{A,\tw,\omega_A}_*\mc_2)_{g,1,n-1}^\epsilon$
would not be equal in chain level. 
It would only hold as homology classes.

For example, when $(g,n)=(1,1)$, using the graph sum formula in Equation~\eqref{eq:lbA-110}, 
we have the commutator %$[M_{u^{-1}},\lbA_{1,1,0}]$ is given by
\begin{align*} 
[M_{u^{-1}},\lbA_{1,1,0}]
=\left(\begin{tikzpicture}[baseline={(current bounding
box.center)},scale=0.5] \draw [thick,directed] (3.4,4) to (3.4,2);
\node at (3.4,2) {$\bullet$}; \node at (3.4,1.4) {\small $g=1$}; 
\end{tikzpicture}
+\frac{1}{2}\;\;\begin{tikzpicture}[baseline={(current
bounding box.center)},scale=0.5] \draw [thick,directed] (3.4,4) to
(3.4,2); \node at (3.4,2) {$\bullet$}; \draw [thick]
(3.4,1) circle [radius=1];
\end{tikzpicture}\right) \circ (-M_{u^{-1}})
\end{align*}
Meanwhile, the component $(\cK_*\rho^{A,\tw,\omega_A}_*\mc_2)_{1,1,0}^\epsilon$ is given by
\begin{align*} 
\left( \begin{tikzpicture}[baseline={(current bounding
box.center)},scale=0.5] \draw [thick,directed] (3.4,4) to (3.4,2);
\node at (3.4,2) {$\bullet$}; \node at (4.3,2) {\small $\gamma_{1,1,0}$}; 
\node at (3.4,1.2) {\small $g=1$};
\end{tikzpicture}+\begin{tikzpicture}[baseline={(current
bounding box.center)},scale=0.5] \draw [thick,directed] (3.4,4) to
(3.4,2); \node at (3.4,2) {$\bullet$}; \node at (4.5,2) {\small $\gamma_{0,1,1}$};\draw [thick,directed] (3.4,2) to (3.4,0); \node at (3.4,0) {$\bullet$}; \node at (3.4,-.4) {\small $g=1$};
\end{tikzpicture}\right)+ \left( \frac{1}{2}\;\;\begin{tikzpicture}[baseline={(current
bounding box.center)},scale=0.5] \draw [thick,directed] (3.4,4) to
(3.4,2); \node at (3.4,2) {$\bullet$}; \node at (4.5,2.5) {\small $\gamma_{0,1,2}$};\draw [thick]
(3.4,1) circle [radius=1];
\end{tikzpicture} +\frac{1}{2}\begin{tikzpicture}[baseline={(current
bounding box.center)},scale=0.5] \draw [thick,directed] (3.4,4) to
(3.4,2); \node at (3.4,2) {$\bullet$}; \node at (4.5,2) {\small $\gamma_{0,1,1}$};\draw [thick,directed] (3.4,2) to (3.4,0); \node at (3.4,0) {$\bullet$};\draw [thick]
(3.4,-1) circle [radius=1];
\end{tikzpicture}\right)
+\frac{1}{2}
\begin{tikzpicture}[baseline={(current
bounding box.center)},scale=0.5] 
\draw [thick,directed] (3.4,4) to (3.4,2); 
\node at (3.4,2) {$\bullet$}; 
\node at (3.4,-2) {$\bullet$};
\draw [ultra thick,directed] (3.4,2) to [out=240, in=120] (3.4,-2); 
\draw [thick] (3.4,2) to [out=300, in=60] (3.4,-2); 
\node at (3.4,-2.5) {\small $\gamma_{0,1,1}$};
\end{tikzpicture}
+\frac{1}{2}
\begin{tikzpicture}[baseline={(current
bounding box.center)},scale=0.5] \draw [thick,directed] (3.4,4) to
(3.4,2); \node at (3.4,2) {$\bullet$}; \node at (3.4,-2) {$\bullet$};
\draw [thick,directed] (3.4,2) to [out=240, in=120] (3.4,-2); \draw
[ultra thick] (3.4,2) to [out=300, in=60] (3.4,-2); \node at (3.4,-2.5) {\small $\gamma_{0,1,1}$};
\end{tikzpicture}.
\end{align*}

Using Lemma~\ref{lem:string-evaluation} and Lemma~\ref{lem:string-leaf-cont}, we obtain the following
\begin{align*} 
(\cK_*\rho^{A,\tw,\omega_A}_*\mc_2)_{1,1,0}^\epsilon
-
[M_{u^{-1}},\lbA_{1,1,0}]= \frac{1}{2}\begin{tikzpicture}[baseline={(current
bounding box.center)},scale=0.5] \draw [thick,directed] (3.4,4) to
(3.4,2); \node at (3.4,2) {$\bullet$}; \node at (3.4,-2) {$\bullet$};
\draw [ultra thick,directed] (3.4,2) to [out=240, in=120] (3.4,-2); \draw
[thick] (3.4,2) to [out=300, in=60] (3.4,-2); \node at (3.4,-2.5) {\small $\gamma_{0,1,1}$};
\end{tikzpicture}+\frac{1}{2}\begin{tikzpicture}[baseline={(current
bounding box.center)},scale=0.5] \draw [thick,directed] (3.4,4) to
(3.4,2); \node at (3.4,2) {$\bullet$}; \node at (3.4,-2) {$\bullet$};
\draw [thick,directed] (3.4,2) to [out=240, in=120] (3.4,-2); \draw
[ultra thick] (3.4,2) to [out=300, in=60] (3.4,-2); \node at (3.4,-2.5) {\small $\gamma_{0,1,1}$};
\end{tikzpicture} -\frac{1}{2}\begin{tikzpicture}[baseline={(current
bounding box.center)},scale=0.5] \draw [thick,directed] (3.4,4) to
(3.4,2); \node at (3.4,2) {$\bullet$}; \draw [thick] (3,2.4) to (3.2, 1.6);\draw [thick]
(3.4,1) circle [radius=1];
\end{tikzpicture} - \frac{1}{2}\begin{tikzpicture}[baseline={(current
bounding box.center)},scale=0.5] \draw [thick,directed] (3.4,4) to
(3.4,2); \node at (3.4,2) {$\bullet$}; \draw [thick] (3.8,2.4) to (3.6, 1.6);\draw [thick]
(3.4,1) circle [radius=1];
\end{tikzpicture}.
\end{align*}

As in the last section, the short line segment is labeled by the operator $-M_{u^{-1}}$. This equality shows that, unlike the genus zero case, we do not get an exact cancellation.
%Equation~\eqref{string-main-identity} holds as homology classes.

However, the sum of the four terms above is equal to zero after taking the $b$-homology. 
Indeed, using Equation~\eqref{lem:edge-cancel-H}, one can show that the difference above is bounded by the following element
\[ \frac{1}{2}\begin{tikzpicture}[baseline={(current
bounding box.center)},scale=0.5] \draw [thick,directed] (3.4,4) to
(3.4,2); \node at (3.4,2) {$\bullet$}; \node at (3.4,-2) {$\bullet$};
\draw [ultra thick,directed] (3.4,2) to [out=240, in=120] (3.4,-2); \draw
[ultra thick] (3.4,2) to [out=300, in=60] (3.4,-2); \node at (3.4,-2.5) {\small $\gamma_{0,1,1}$};
\end{tikzpicture}.\]

\subsubsection{A proof in higher genus}
A complete proof of the string equation~\eqref{eq:string-equation} (for an arbitrary genus) will be given in Section~\ref{app:proof-string}. Instead of working with the explicit formula of $\mathcal{K}$, we will use a different strategy to work with its corresponding pseudo-isotopy families of DGLA's. This is not surprising since, already in~\cite{CalTu1}, the map $\mathcal{K}$ was constructed using pseudo-isotopies. Since the proof of the divisor equation~\eqref{eq:divisor} uses the same technique, we gather both proofs in Section~\ref{app:proof-string}.

\section{The divisor equation}\label{sec:divisor}

In this section, we prove the divisor equation in the context of CEI associated with a family of cyclic $A_\infty$-algebras, assuming strict unitality, smoothness, and properness. 

\subsection{The symmetric Getzler connection}
\label{app:getzler}
In order to discuss the variational properties of CEI, it is necessary to use Getzler connection~\cite{Get} on the periodic cyclic homology, often called Getzler-Gauss-Manin connection. Let us briefly recall this construction. Let $A$ be a strictly unital $A_\infty$-algebra over $\bbC$. Analogous to the definition of Hochschild chain complex in~\eqref{eq:hoch-chain}, the reduced Hochschild cochain complex is defined by
\[ C^\bullet(A):= \prod_{n\geq 0} {\sf Hom} (\overline{A}[1]^{\otimes n} , A).\]
 The algebraic structure on the Hochschild pair $( C^\bullet(A), C_\bullet(A) )$ is extremely rich. Indeed, for a Hochschild cochain $\varphi$, Getzler~\cite{Get} defines operators 
\begin{align*}
b\{\varphi\} &: C_\bullet(A) \ra C_\bullet(A),\\
B\{\varphi\} &: C_\bullet(A) \ra C_\bullet(A).
\end{align*}
They are explicitly defined as follows. As in Subsection~\ref{subsec:tcft}, we use $@$ to denote the Koszul sign.
\begin{enumerate}
    \item We set\footnote{This operator $B\{\varphi\}$ is denoted by $B^{1,1}(\varphi,-)$ in~\cite[Section 3.6]{She}.}
    \begin{align}\label{eq:B11}
	B\{\varphi\}(a_0|a_1\ldots a_n):= & \sum (-1)^{\star_2} \bone_A | a_{j+1}\ldots \varphi(a_{i+1}\ldots)\ldots a_n a_0\ldots a_j
\end{align}
where $\star_2=|\varphi|'(|a_{j+1}'|+\ldots+|a_i|')+@$. 
\item 
Let $m:=\prod_{l\geq 1} m_l \in C^\bullet(A)$ be the $A_\infty$-structure map. We set
 \begin{align}\label{eq:b11}
   b\{\varphi\}(a_0|a_1\ldots a_n):= & \sum (-1)^{\star_3} m(a_{j+1} \ldots \varphi(a_{k+1}\ldots)\ldots a_0 \ldots a_i)|a_{i+1}\ldots a_j 
   \end{align}
	where $\star_3= |\varphi|'(|a_{j+1}|'+\ldots+|a_k|')+ @$.
    The action of $C^\bullet(A)$ on $C_\bullet(A)$ is compatible with the differential and induces the {\em cap product map} in cohomology denoted by
    $$[\varphi]\cap[\alpha]:= [ b\{\varphi\}(\alpha)]$$ 
    for $[\varphi]\in HH^\bullet(A)$ and $[\alpha]\in HH_\bullet(A)$, see~\cite[Theorem 1.9]{Get} or~\cite[Example 3.14]{She}.
\end{enumerate}
We remark that $b\{\}=b$ and $B\{\}=B$ in the case where no Hochschild cochain is present.
%Using these two types of operators, 
We define an operator %$\iota\{\varphi\}: C_\bullet(A)((u))\to C_\bullet(A)((u))$ 
on the periodic cyclic chain complex by 
\begin{equation}~\label{eq:iota-hochschild-calculus}
\iota\{\varphi\}:=b\{\varphi\}+uB\{\varphi\}: C_\bullet(A)((u))\to C_\bullet(A)((u)).
\end{equation}

Now, let us set $\cA:= A\otimes_\bbC \bbC[[t]]$. Consider an $A_\infty$-structure on $\cA$, linear over $\bbC[[t]]$. Denote its $A_\infty$-structure by 
\[\{m_k: \cA[1]^{\otimes k} \ra \cA[1]\}_{k\geq 0}.\]
We still assume $\cA$ to be strictly unital with a unit $\bone_\cA=\bone_A\otimes 1$. The derivative $\frac{dm}{dt}$ is a Hochschild cochain (in fact, a cocycle) in $C^\bullet(\cA)$. Its cohomology class is the so-called {\em Kodaira-Spencer class}, given by
\begin{equation}\label{eq-KS}
\KS(\partial_t):= \left[\frac{dm}{dt}\right] \in HH^*(\cA).
\end{equation}
Getzler~\cite[Proposition 3.1]{Get} defines a connection on the periodic cyclic homology of $\cA$. This connection is often called the {\em Getzler-Gauss-Manin connection} and is explicitly given by
\begin{equation}\label{GMM-connection}
\nabla_{\partial_t}^{\rm GGM} := \partial_t - u^{-1}\iota\Big\{\frac{dm}{dt}\Big\}= \partial_t - B\Big\{\frac{dm}{dt}\Big\} - u^{-1} b\Big\{\frac{dm}{dt}\Big\}.
\end{equation}
\begin{remark}
Note that there is a sign difference from this formula with Getzler's original definition. %This is also remarked by  Sheridan in
See also~\cite[Definition 3.29]{She}. 
\end{remark}
In this paper, we shall follow Sheridan's sign conventions in~\cite{She}. The following result is a consequence of the $2D$ TCFT structure in Theorem~\ref{thm:tcft}.
\begin{Corollary}
\label{cor-getzler-connection}
We have the following  graphical presentations for Getzler's operators: 
%\begin{enumerate}
%\item
%the operator $B\{\frac{dm}{dt}\}$ is represented by
\begin{align}
B\Big\{\frac{dm}{dt}\Big\}&={\sf Get}_0:=  \begin{tikzpicture}[baseline={([yshift=-0.5ex]current bounding
      box.center)},scale=0.4] 
\draw [thick] (-4,0) to (-2.4,0);
\draw [thick] (-4.2, -.2) to (-3.8, .2);
\draw [thick] (-4.2, .2) to (-3.8, -.2);
\draw [thick] (0,0) to (-2,0);
\draw [line width=2.5pt] (-2.2,0.2) to (-2.2,1);
\draw (1,0) node[label=right:{$\frac{dm}{dt}$}] {};
\draw [thick] (0,0) to (1,0);
\draw [thick] (-2.2,0) circle [radius=0.2];
\draw [thick] (1.2,0) circle [radius=0.2];
\end{tikzpicture}
\label{graphic-B}
\\
%\item 
%and the operator $-b\{\frac{dm}{dt}\}$ is represented by
-b\Big\{\frac{dm}{dt}\Big\}&={\sf Get}_1:=
\begin{tikzpicture}[baseline={([yshift=-1.2ex]current bounding
      box.center)},scale=0.4] 
\draw [thick] (0,0) to (0,2);
\draw [thick] (-0.2, 1.8) to (0.2, 2.2);
\draw [thick] (0.2, 1.8) to (-0.2, 2.2);
\draw [thick] (0,0) to (-2,0);
\draw (-2,0) node[label=left:{$\frac{dm}{dt}$}] {};
\draw [thick] (0,0) to (2,0);
\draw [thick] (-2.2,0) circle [radius=0.2];
\draw [thick] (2.2,0) circle [radius=0.2];
\end{tikzpicture}
\label{graphic-b}
\end{align}
%\end{enumerate}
\end{Corollary}
In both graphs, the cyclic orderings of the vertices are taken clockwise. 
Getzler proves that the connection operator commutes with the cyclic differential $b+uB$, 
and thus descends to the periodic cyclic homology $HP_*(\cA)$. This can also be proved using the above graphical presentation by computing
\[(\partial +u\cdot B)\left({\sf Get}_0+ u^{-1} {\sf Get}_1\right)= \begin{tikzpicture}[baseline={([yshift=-0.5ex]current bounding
      box.center)},scale=0.4] 
\draw [thick] (-4,0) to (-2.4,0);
\draw [thick] (-4.2, -.2) to (-3.8, .2);
\draw [thick] (-4.2, .2) to (-3.8, -.2);
\draw [thick] (0,0) to (-2,0);
\draw [line width=2.5pt] (-3.2,0) to (-2.4,0);
\draw (1,0) node[label=right:{$\frac{dm}{dt}$}] {};
\draw [thick] (0,0) to (1,0);
\draw [thick] (-2.2,0) circle [radius=0.2];
\draw [thick] (1.2,0) circle [radius=0.2];
\end{tikzpicture}
+ \begin{tikzpicture}[baseline={([yshift=-2ex]current bounding
      box.center)},scale=0.4] 
\draw [thick] (-4,0) to (-2.4,0);
\draw [thick] (-4.2, -.2) to (-3.8, .2);
\draw [thick] (-4.2, .2) to (-3.8, -.2);
\draw [thick] (0,0) to (-2,0);
\draw [line width=2.5pt] (-2,0) to (-1,0);
\draw (-2,0) node[label=above:{$\frac{dm}{dt}$}] {};
\draw [thick] (0,0) to (1,0);
\draw [thick] (-2.2,0) circle [radius=0.2];
\draw [thick] (1.2,0) circle [radius=0.2];
\end{tikzpicture}\]
The right hand side is precisely the Lie action $[\frac{d}{dt}, b]$, which proves that 
$$[ b+uB,  \nabla_{\partial_t}]=0.$$ 
For our purpose, we shall need to use a symmetric version of Getzler's connection formula. %Graphically, it is defined by
Following the graphical presentation in Corollary~\ref{cor-getzler-connection}, we define
\begin{equation}
\label{eq:sym-getzler}
\nabla_{\partial_t}^{\sf Get} := \partial_t - \frac{1}{2} \left(
  \begin{tikzpicture}[baseline={([yshift=-0.5ex]current bounding
      box.center)},scale=0.3] 
\draw [thick] (-4,0) to (-2.4,0);
\draw [thick] (-4.2, -.2) to (-3.8, .2);
\draw [thick] (-4.2, .2) to (-3.8, -.2);
\draw [thick] (0,0) to (-2,0);
\draw [line width=2.5pt] (-2.2,0.2) to (-2.2,1);
\draw (1,0) node[label=right:{$\frac{dm}{dt}$}] {};
\draw [thick] (0,0) to (1,0);
\draw [thick] (-2.2,0) circle [radius=0.2];
\draw [thick] (1.2,0) circle [radius=0.2];
\end{tikzpicture} 
+\begin{tikzpicture}[baseline={([yshift=-0.5ex]current bounding
      box.center)},scale=0.3] 
\draw [thick] (-4,0) to (-2.4,0);
\draw [thick] (-4.2, -.2) to (-3.8, .2);
\draw [thick] (-4.2, .2) to (-3.8, -.2);
\draw [thick] (0,0) to (-2,0);
\draw [line width=2.5pt] (-2.2,-0.2) to (-2.2,-1);
\draw (1,0) node[label=right:{$\frac{dm}{dt}$}] {};
\draw [thick] (0,0) to (1,0);
\draw [thick] (-2.2,0) circle [radius=0.2];
\draw [thick] (1.2,0) circle [radius=0.2];
\end{tikzpicture} \right)
- \frac{u^{-1}}{2} 
\bigg(\begin{tikzpicture}[baseline={([yshift=1.2ex]current bounding
      box.center)},scale=0.3] 
\draw [thick] (0,0) to (0,2);
\draw [thick] (-0.2, 1.8) to (0.2, 2.2);
\draw [thick] (0.2, 1.8) to (-0.2, 2.2);
\draw [thick] (0,0) to (-2,0);
\draw (-2.2,0) node[label=left:{$\frac{dm}{dt}$}] {};
\draw [thick] (0,0) to (2,0);
\draw [thick] (-2.2,0) circle [radius=0.2];
\draw [thick] (2.2,0) circle [radius=0.2];
\end{tikzpicture}+ \begin{tikzpicture}[baseline={([yshift=1.2ex]current bounding
      box.center)},scale=0.3] 
\draw [thick] (0,0) to (0,2);
\draw [thick] (-0.2, 1.8) to (0.2, 2.2);
\draw [thick] (0.2, 1.8) to (-0.2, 2.2);
\draw [thick] (0,0) to (-2,0);
\draw (5,0) node[label=left:{$\frac{dm}{dt}$}] {};
\draw [thick] (0,0) to (2,0);
\draw [thick] (-2.2,0) circle [radius=0.2];
\draw [thick] (2.2,0) circle [radius=0.2];
\end{tikzpicture}\bigg)
\end{equation}
The notation $\nabla_{\partial_t}^{\sf Get}$ is chosen to distinguish it from the original one $\nabla_{\partial_t}^{\sf GMM}$. 
One can check the following identity of the difference between the two connections:
\[\nabla^{\sf GMM}_{\partial_t} - \nabla_{\partial_t}^{\sf Get} = (\partial +u B) \left( \frac{u^{-1}}{2}\cdot \begin{tikzpicture}[baseline={([yshift=-0.5ex]current bounding
      box.center)},scale=0.3] 
\draw [thick] (-4,0) to (-2.4,0);
\draw [thick] (-4.2, -.2) to (-3.8, .2);
\draw [thick] (-4.2, .2) to (-3.8, -.2);
\draw [thick] (0,0) to (-2,0);
\draw (1,0) node[label=right:{$\frac{dm}{dt}$}] {};
\draw [thick] (0,0) to (1,0);
\draw [thick] (-2.2,0) circle [radius=0.2];
\draw [thick] (1.2,0) circle [radius=0.2];
\end{tikzpicture} \right)\]
As a consequence, the two connections are equal in homology.
\begin{Lemma}\label{lem:symmetric-getzler}
The Getzler connection $\nabla^{\sf GMM}_{\partial_t}$ and its symmetric version $\nabla_{\partial_t}^{\sf Get}$ are homotopic operators. Hence, they induce the same connection on the periodic cyclic homology $HP_*(\cA)$.
\end{Lemma}

\begin{remark}
%Algebraically, 
One can verify that the symmetric Getzler connection $\nabla_{\partial_t}^{\sf Get}$ is the average of the original Getzler's connection $\nabla^{\sf GMM}_{\partial_t}$ with its conjugate connection $r^{-1}\circ \nabla_{\partial_t}\circ r$ where $r: C_\bullet(\cA) \ra C_\bullet(\cA^{\sf op})$ is the canonical isomorphism 
$$r(a_0|a_1\ldots a_s) := (-1)^\dagger a_0|a_s\ldots a_1, \quad  \text{with} \quad \dagger = \sum_{1\leq i< j\leq s} (|a_i|+1)(|a_j|+1).$$
\end{remark}

%Throughout the rest of the section we shall use $\nabla^{\sf Get}$ to denote the symmetric Getzler connection as defined in Equation~\eqref{eq:sym-getzler}. 

\subsection{The divisor equation}\label{subsec:divisor}
Continuing with the notation of the previous subsection, now let us consider a family of cyclic $A_\infty$-algebra structures on $\cA=A\otimes_\bbC\bbC[[t]]$. We shall assume that the cyclic structure is independent of $t$, i.e., the cyclic structure is obtained from an inner product $\langle-,-\rangle$ on $A$. This condition may be easily achieved in the geometric context in which we are interested (as in Section~\ref{subsec:cei-family}). This assumption implies that
\begin{equation}\label{eq:flat-cyclic-metric}
\frac{d}{dt} \langle m_k(a_1,\ldots,a_k),a_0\rangle = \left\langle \frac{dm_k}{dt}(a_1,\ldots,a_k),a_0\right\rangle.
\end{equation}
In order to write down the divisor equation in our context, we need to introduce some notation. Let $s$ be a splitting of the nc-Hodge filtration of $\cA$, linear over the base ring $\bbC[[t]]$. 
This induces an isomorphism (by $u$-linear extension) denoted by
\[\widetilde{s}: HH_\bullet(\cA)((u)) \ra HP_*(\cA).\]
In addition, it induces a connection $\nabla^{s}$ on $HH_\bullet(\cA)((u))$ using the following commutative diagram. 
\[\begin{CD}
    HH_\bullet(\cA)((u)) @>\nabla^{s}>> HH_\bullet(\cA)((u))\\
        @V\widetilde{s}VV   @V\widetilde{s}VV\\
            HP_*(\cA) @>\nabla^{\sf Get}>> HP_*(\cA) 
\end{CD}\]
By the construction in~\eqref{eq:sym-getzler} and the definition of the Kodaira-Spencer class in~\eqref{eq-KS}, we have
\[ \nabla^{\sf Get}_{\partial_t}(\alpha) 
=-\KS(\partial_t)\cap \alpha \cdot u^{-1} + \mbox{regular part in $u$-variable.}\]
As a result, the connection  $\nabla^{s}$ is also of this form and therefore it does not restrict to the subspace $HH_\bullet(\cA)[[u]]$. However, using the canonical decomposition $$HH_\bullet(\cA)((u))\cong u^{-1} HH_\bullet(\cA)[u^{-1}]\oplus HH_\bullet(\cA)[[u]],$$ we may define a connection on $HH_\bullet(\cA)[[u]]$ by setting
\begin{equation}\label{eq:conn-truncated}
\nabla^{s,+}_{\partial_t}(\alpha\cdot u^k):= \pi^+\big(\nabla_{\partial_t}^{s}(\alpha\cdot u^k)\big) = \nabla_{\partial_t}^{s}(\alpha\cdot u^k)+ \delta_k^0 \KS(\partial_t)\cap \alpha \cdot u^{-1},
\end{equation}
where $\pi^+: HH_\bullet(\cA)((u))\to HH_\bullet(\cA)[[u]]$ is the canonical projection map. Since in our explicit formula of CEI~\eqref{eq:explicit-evaluation}, the variable $\psi$ corresponds to $-u$, thus we define a connection on $HH_\bullet(\cA)((\psi))$ by the following commutative diagram.
\[\begin{CD}
    HH_\bullet(\cA)((u)) @>\nabla^{s}>> HH_\bullet(\cA)((u))\\
        @V\cong VV   @V\cong VV\\
           HH_\bullet(\cA)((\psi)) @>\nabla^{\sf s}>> HH_\bullet(\cA)((\psi))
\end{CD}\]
The vertical isomorphisms are both given by $u\mapsto -\psi$. Similarly, we also obtain a connection on $HH_\bullet(\cA)[[\psi]]$ by setting
\begin{equation}
\label{connection-HH}
    \nabla^{s,+}_{\partial_t}(\alpha\cdot \psi^k):= \nabla_{\partial_t}^{s}(\alpha\cdot \psi^k)- \delta_k^0 \KS(\partial_t)\cap \alpha \cdot \psi^{-1}.
\end{equation}
In the divisor equation we need to assume the following admissibility condition: the connection $\nabla^{s}_{\partial_{t}}$ operator preserves the opposite subspace, i.e.
\begin{equation}\label{eq:divisor-admissibility}
  \nabla^s_{\partial_t}\big(HH_\bullet(\cA)[u^{-1}]\big)
  \subset HH_\bullet(\cA)[u^{-1}].
\end{equation}
We call a splitting satisfying~\eqref{eq:divisor-admissibility} admissible in the direction $\partial_t$. Equivalently, the connection matrix of $ \nabla^s_{\partial_t}$ has only powers in $u^0$ and $u^{-1}$. The condition holds for the Blanc--To\"en splitting in a holomorphic base direction by the Hodge decomposition of the Gauss--Manin connection. With these preparations, we may state the following

\begin{Theorem}\label{thm:divisor}
For an admissible splitting in the direction $\partial_t$, we have the following equation for $n\geq1$, $2g-2+n>0$, and $k_i\geq0$
\begin{align}\label{eq:divisor}
\begin{split}
& 
\langle \nabla^{s,+}_{\partial_t}(-[\Omega]\psi),\alpha_1\psi^{k_1},\ldots,\alpha_n\psi^{k_n}\rangle_{g, n+1}^{\cA,\Omega,s}
\\=
& 
\partial_t\langle \alpha_1\psi^{k_1},\ldots,\alpha_n\psi^{k_n}\rangle_{g, n}^{\cA,\Omega,s}- \sum_{j=1}^{n} \langle \alpha_1\psi^{k_1},\ldots,\nabla^{s,+}_{\partial_t}(\alpha_j\psi^{k_j}),\ldots,\alpha_n\psi^{k_n}\rangle_{g, n}^{\cA,\Omega,s}
\end{split}
\end{align}
\end{Theorem}

\begin{remark}
In the case where $k_1=\cdots=k_n=0$, the above equation simplifies to
\[ \langle \nabla^{s,+}_{\partial_t}(-[\Omega]\psi),\alpha_1,\ldots,\alpha_n\rangle_{g, n+1}^{\cA,\Omega,s}=\partial_t\langle \alpha_1,\ldots,\alpha_n\rangle_{g, n}^{\cA,\Omega,s}- \sum_{j=1}^{n} \langle \alpha_1,\ldots,\nabla^{s,+}_{\partial_t}(\alpha_j),\ldots,\alpha_n\rangle_{g, n}^{\cA,\Omega,s}.\]
Furthermore, if $\Omega$ and all $\alpha_j$'s are flat with respect to $\nabla^{s,+}_{\partial_t}$ in~\eqref{connection-HH}, that is, $\nabla^{s,+}_{\partial_t}([\Omega])=\nabla^{s,+}_{\partial_t}(\alpha_j)=0,$
then the equation above reduces to 
\[ \langle -\KS(\partial_t)\cap[\Omega], \alpha_1,\ldots,\alpha_n\rangle_{g, n+1}^{\cA,\Omega,s}=\partial_t\langle \alpha_1,\ldots,\alpha_n\rangle_{g,n}^{\cA,\Omega,s}.\]
This equation resembles the divisor equation in Gromov-Witten theory. In this way, Equation~\eqref{eq:divisor} may be considered the divisor property in a possibly non-flat basis.
\end{remark}

\subsubsection{The strategy of the proof}
The proof of Theorem~\ref{thm:divisor} is parallel to that of Theorem~\ref{thm:string}. Indeed, since $\cA$ is a family of cyclic $A_\infty$-structures, we define a linear functional $\zeta: L^\cA_-\to \bbC[[t]]$ %using $\frac{dm}{dt}$ 
by 
\begin{equation}\label{eq-KS-connection} 
\zeta(a_0|a_1\ldots a_k u^{-n}):= \delta_0^n(-1)^{|a_0|'\sum_{r=1}^k |a_r|'} \Big\langle \frac{dm_k}{dt}(a_1,\ldots,a_k),a_0\Big\rangle. 
\end{equation}
Contraction with this linear functional yields a map $C_\zeta: \sym^{l+1} L^\cA_- \to \sym^l L^\cA_-$.  By post-composition, it further induces a map still denoted by 
\[ C_{\zeta}: {\sf Hom}^c\big( \sym^k (L^\cA_+[1]), \sym^{l+1} L^\cA_-\big) \to {\sf Hom}^c\big( \sym^k (L^\cA_+[1]), \sym^{l} L^\cA_-\big). \]
Denote by $\mathbb{C}[\epsilon]$ the ring of dual numbers, i.e., $\epsilon^2=0$. 
Recall that the semi-direct DGLA $\widehat{\mathfrak{g}} \ltimes \widehat{\mathfrak{m}}$ is as in Theorem~\ref{thm:mc-recursion}.
Now, we may extend the DGLA map $\rho^{\cA,\tw}:\hg \ra \hh_\cA$ in~\eqref{eq:rho-tw} to a DGLA map $\rho^{\cA,\tw,\zeta} : \widehat{\mathfrak{g}} \ltimes \widehat{\mathfrak{m}} \ra \widehat{\mathfrak{h}}_\cA [\epsilon]$ by setting %such that for $\alpha\in \hg$ and $\beta \in \widehat{\mathfrak{m}}$,
\begin{equation}
\label{zeta-map}
    \rho^{\cA,\tw,\zeta} \big((\alpha,\beta)\big):= \rho^{\cA,\tw}(\alpha)+\lambda^{-1}C_{\zeta}^Q\big(\rho^{\cA,\tw}(\beta)\big)\epsilon
\end{equation}
for $\alpha\in \hg$ and $\beta \in \widehat{\mathfrak{m}}$,
where the operator $C_{\zeta}^Q$ is defined by applying the functional $\zeta$ to the Hochschild chain at the distinguished white vertex $Q$ in $\beta\in \widehat{\mathfrak{m}}$.  Post-composing with the trivialization $L_\infty$ morphism $\cK$ in Equation~\eqref{eq:Triv-map}, we obtain an $L_\infty$ morphism
\[ \cK\circ \rho^{\cA,\tw,\zeta}:\widehat{\mathfrak{g}}\ltimes \widehat{\mathfrak{m}} \ra \hh_\cA^{\sf TRIV}[\epsilon].\]
Then, the proof of Theorem~\ref{thm:divisor} consists of two main steps:
\begin{enumerate}
    \item Express the left hand side of Equation~\eqref{eq:divisor} using $[\big(\cK_*\rho^{\cA,\tw,\zeta}_*\mc_1\big)_{g,1,n-1}^\epsilon]$.
    \item Express the right hand side of Equation~\eqref{eq:divisor} using $[\big(\cK_*\rho^{\cA,\tw,\zeta}_*\mc_2\big)_{g,1,n-1}^\epsilon]$.
\end{enumerate}
Here $\mc_1$ and $\mc_2$ are as in Theorem~\ref{thm:mc-recursion}. Theorem~\ref{thm:divisor} would then follow from Theorem~\ref{thm:mc-recursion} proving that the two push-forward Maurer-Cartan elements are gauge equivalent. 

\subsubsection{The left hand side}
Let us begin with the left hand side identification. Consider the following diagram:
\begin{equation}
   \begin{tikzcd}
   H_\bullet(L^\cA)[u^{-1}] \arrow[d,"s"]\arrow{dr}[name=U]{\langle -,\nabla^{s,+}_{\partial_t} ([\Omega] u)\rangle_\Muk} &[5em] \\
   H_\bullet(L^\cA_-)  \arrow[r,"\zeta"]& \bbC [[t]].
\end{tikzcd}
\end{equation} 
\iffalse
We need to introduce the so-called residue pairing
\[ HP_\bullet(\cA)\otimes HP_\bullet(\cA) \to \bbC[[t]],\]
defined using the higher residue pairing in~\eqref{higher-residue-pairing} as follows.  First, we extend the higher residue pairing to the full periodic cyclic homology by inverting the $u$-variable. Then we set the residue pairing to be the coefficient of $u^0$ in the higher residue pairing, explicitly given by
\[ \langle x u^k , y u^l\rangle_{\res}:= (-1)^k\delta^{k+l}_0\langle x, y\rangle_{\sf Muk}. \]
\fi

\begin{Lemma}\label{lem:functional-zeta}
%In cohomology, the map $\zeta$ is represented by residue pairing with $\nabla^{\sf Get}_{\partial_t}(\Omega\cdot u)$, i.e., 
The diagram above is commutative, i.e. for any $\alpha\in  H_\bullet(L^\cA)[u^{-1}]$, we have
\begin{equation}
\label{zeta-connection}
\zeta\big(s(\alpha)\big)=\left\langle \alpha, \nabla_{\partial_t}^{s,+}([\Omega] u) \right\rangle_\Muk.
\end{equation}
\end{Lemma}
\begin{proof}
According to~\cite[Corollary 5.39]{She}, the higher residue pairing in~\eqref{higher-residue-pairing} is covariantly constant with respect to the Getzler connection. From this, we deduce that
\begin{align*}
\left\langle s(\alpha), u\nabla^{\sf Get}_{\partial_t}\Omega \right\rangle_\hres & =- u\partial_t\langle s(\alpha), \Omega\rangle_\hres +\left\langle u\nabla^{\sf Get}_{\partial_t} s(\alpha),\Omega\right\rangle_\hres.
\end{align*}
By Definition~\eqref{def:nc-splitting}, the splitting map $s$ intertwines the higher residue pairing with the Mukai pairing, hence the equation above implies that
\begin{align*}
\left\langle \alpha, \nabla_{\partial_t}^{s,+}([\Omega] u) \right\rangle_\Muk & = \mbox{the coefficient of $u^0$ in\;\;} \left\langle s(\alpha), u\nabla^{\sf Get}_{\partial_t}\Omega \right\rangle_\hres\\
 & =-\partial_t\langle u\alpha, [\Omega]\rangle_\Muk +\left\langle \nabla_{\partial_t}^{s,+} \alpha,[\Omega]\right\rangle_\Muk\\
 &= \langle -u\partial_t\alpha, [\Omega]\rangle_\Muk + \Big\langle u \partial_t \alpha -uB\Big\{\frac{dm}{dt}\Big\}(\alpha)- b\Big\{\frac{dm}{dt}\Big\}(\alpha), [\Omega]\Big\rangle_\Muk 
 \\
 &= -\omega_\cA\Big(uB\Big\{\frac{dm}{dt}\Big\}(\alpha)+ b\Big\{\frac{dm}{dt}\Big\}(\alpha)\Big).
\end{align*}
 For the third equality above, we first use Equation~\eqref{sm-CY-structure} to obtain 
$\langle u\alpha, [\Omega]\rangle_\Muk =\omega_\cA(u\alpha).$
Then, by Equation~\eqref{eq:cy-form}, $\omega_\cA(u\alpha)$ vanishes unless $\alpha=a_0 u^{-1}$ for some $a_0$. If so, since the cyclic pairing $\langle-,-\rangle:\cA\otimes\cA \to \bbC[[t]]$ is obtained from the cyclic pairing on $A$ by scalar extension, we have
\[\partial_t\langle u\alpha, [\Omega]\rangle_\Muk
=\partial_t\big(\omega_\cA(u\alpha)\big) =
\partial_t\langle \bone_A,a_0\rangle=\langle \bone_A,\partial_t a_0\rangle=\omega_\cA\big(\partial_t(u\alpha)\big) = \langle u\partial_t\alpha, [\Omega]\rangle_\Muk. \]
%\[ \partial_t\langle u\alpha, \Omega\rangle_\res = \partial_t\big( \omega_\cA(u\alpha)\big) = \omega_\cA\big(\partial_t(u\alpha)\big) = \langle u\partial_t\alpha, \Omega\rangle_\res. \]
%Here the second equality is due to Equation~\eqref{eq:cy-form} defining the linear functional $\omega_\cA$. 

By Formula~\eqref{eq:B11}, $B\big\{\frac{dm}{dt}\big\}(\alpha)$ is of the form $\bone_A|\cdots$.
Applying $\omega_\cA$, we have 
$\omega_\cA\left(uB\{\frac{dm}{dt}\}(\alpha)\right)=0$.

Finally, for $\alpha=a_0|a_1\ldots a_k u^{-n}$, by the formula of $b\{\frac{dm}{dt}\}$ in Equation~\eqref{eq:b11}, 
Equation~\eqref{eq:cy-form},
Equation~\eqref{eq:unitality}, and Equation~\eqref{eq:cyclic-structure}, we have %may compute as
\begin{align*}
%\left\langle\alpha, u\nabla^{\sf Get}_{\partial_t}\Omega \right\rangle_\res  = 
-\omega_\cA\Big(b\big\{\frac{dm}{dt}\big\}(\alpha)\Big)
    & = -\delta_0^{n}(-1)^{|a_0|'\sum_{r=1}^k |a_r|'} \Big\langle \bone_A, m_2\Big(\frac{dm_k}{dt}(a_1,\ldots,a_k),a_0\Big)\Big\rangle\\
    & = \delta_0^{n}(-1)^{|a_0|'\sum_{r=1}^k |a_r|'} \Big\langle \frac{dm_k}{dt}(a_1,\ldots,a_k),a_0\Big\rangle.
    %\\&= \zeta(\alpha).
\end{align*}
This is $\zeta(\alpha)$ defined in Equation~\eqref{eq-KS-connection} and the proof is completed.
%Here, the last equality uses the cyclic structure (see Equation~\eqref{eq:cyclic-structure}) on $\cA$.
\end{proof}

%Using the lemma above, we obtain the following
%may compute the left hand side of the divisor equation using Equation~\eqref{eq:cei-formula-main} and 

\begin{Lemma}\label{lem:left-hand-side-genus-zero}
The left side of the divisor equation~\eqref{eq:divisor} is given by 
%using Equation~\eqref{eq:cei-formula-main}
%The following identity holds:
\begin{align*}
         %\langle \nabla^{s,+}_{\partial_t}(\Omega\psi),\alpha_1\psi^{k_1},\ldots,\alpha_n\psi^{k_n}\rangle_{g, n}^{\cA,\Omega,s}=  
        \bigg\langle [\big(\cK_*\rho^{\cA,\tw,\zeta}_*\mc_1\big)_{g,1,n-1}^\epsilon](\alpha_n(-u)^{k_n}),
        %\alpha_1(-u)^{k_1}\cdots\alpha_{n-1}(-u)^{k_{n-1}}
        \bigodot\limits_{i=1}^{n-1}\alpha_i(-u)^{k_i}
        \bigg\rangle_\Muk.
    \end{align*}
\end{Lemma}

\begin{proof}
By Equation~\eqref{eq:cei-formula-main} and Equation~\eqref{zeta-connection}, we have
\begin{align*}
& \langle \nabla^{s,+}_{\partial_t}(-\Omega\psi),\alpha_1\psi^{k_1},\ldots,\alpha_n\psi^{k_n}\rangle_{g, n}^{\cA,\Omega,s}\\
%= &  \left\langle [\lbA_{g,1,n}]\big(\alpha_n (-u)^{k_n}\big),\nabla^{s,+}_{\partial_t}(-\Omega u)
%\cdot\alpha_1(-u)^{k_1}\cdots\alpha_{n-1}(-u)^{k_{n-1}}
%\bigodot\limits_{i=1}^{n-1}\alpha_i(-u)^{k_i}
%\right\rangle_{\Muk}\\
= &\sum_{\GG\in
    \Gamma((g,1,n))} \frac{\wt(\GG)}{|\Aut(\GG)|}  
    \bigg\langle \rho^{A,\tw}_\GG(\alpha_n(-u)^{k_n}),\nabla^{s,+}_{\partial_t}([\Omega] u)
    %\cdot\alpha_1(-u)^{k_1}\cdots\alpha_{n-1}(-u)^{k_{n-1}} 
    \odot\bigodot\limits_{i=1}^{n-1}\alpha_i(-u)^{k_i}
    \bigg\rangle_{\Muk}\\
   = &\sum_{\GG\in
    \Gamma((g,1,n))} \frac{\wt(\GG)}{|\Aut(\GG)|} \bigg\langle C_\zeta \left(\rho^{A,\tw}_\GG(\alpha_n(-u)^{k_n})\right),
    %\alpha_1(-u)^{k_1}\cdots\alpha_{n-1}(-u)^{k_{n-1}} 
    \bigodot\limits_{i=1}^{n-1}\alpha_i(-u)^{k_i}
    \bigg\rangle_{\Muk}
    %\\= & \left\langle [\big(\cK_*\rho^{\cA,\tw,\zeta}_*\mc_1\big)_{g,1,n-1}^\epsilon](\alpha_n(-u)^{k_n}),    \bigodot\limits_{i=1}^{n-1}\alpha_i(-u)^{k_i}\right\rangle_\Muk.
\end{align*}
From the definition of $\cK_*\rho^{\cA,\tw,\zeta}_*\mc_1$ in~\eqref{zeta-map}, the rest of the proof is similar to that of Lemma~\ref{lem:string-lhs-mc}.
\end{proof}

\subsubsection{The right hand side} 
%Our next goal is to identify the right hand side of Equation~\eqref{eq:divisor} with $\big(\cK_*\rho^{\cA,\tw,\zeta}_*\mc_2\big)_{g,1,n-1}^\epsilon$ in the push-forward Maurer-Cartan element. 
%As in the case of string equation, this identification is much more involved. 

%Now we consider the right hand side of the divisor equation~\eqref{eq:divisor}.
We define
\begin{equation}
\label{nabla-commutator}
[\nabla^{s,\pm}_{\partial_t},
\lbA_{g,1,n-1}
]:=
\nabla^{s,-}_{\partial_t}
\lbA_{g,1,n-1}-
\lbA_{g,1,n-1}\nabla^{s,+}_{\partial_t}
.\end{equation}
\begin{Lemma}
The right hand side of the divisor equation~\eqref{eq:divisor} is given by
$$
\bigg\langle 
[\nabla^{s,\pm}_{\partial_t},
\lbA_{g,1,n-1}
](\alpha_n(-u)^{k_n}),
    %\alpha_1(-u)^{k_1}\cdots\alpha_{n-1}(-u)^{k_{n-1}}
    \bigodot\limits_{i=1}^{n-1}\alpha_i(-u)^{k_i}\bigg\rangle_\Muk.
    $$
\end{Lemma}
\begin{proof}
Using Equation~\eqref{eq:explicit-evaluation} and the compatibility between the connection and the pairing~\cite[Corollary 5.39]{She}, we have
\begin{align*}
     &\partial_t
     \langle \alpha_1\psi^{k_1},\ldots,\alpha_n\psi^{k_n}\rangle_{g, n}^{\cA,\Omega,s} \\
     =& \partial_t \bigg\langle \lbA_{g,1,n-1}(\alpha_n(-u)^{k_n}), %\alpha_1(-u)^{k_1}\cdots\alpha_{n-1}(-u)^{k_{n-1}}
    \bigodot\limits_{i=1}^{n-1}\alpha_i(-u)^{k_i}
    \bigg\rangle_\Muk\\
     =& \bigg\langle \nabla^{s,-}_{\partial_t}\Big(\lbA_{g,1,n-1}(\alpha_n(-u)^{k_n})\Big), %\alpha_1(-u)^{k_1}\cdots\alpha_{n-1}(-u)^{k_{n-1}}
    \bigodot\limits_{i=1}^{n-1}\alpha_i(-u)^{k_i}
    \bigg\rangle_\Muk+\bigg\langle \lbA_{g,1,n-1}(\alpha_n(-u)^{k_n}), %\alpha_1(-u)^{k_1}\cdots\alpha_{n-1}(-u)^{k_{n-1}}
\nabla^{s,+}_{\partial_t}\bigg(\bigodot\limits_{i=1}^{n-1}\alpha_i(-u)^{k_i}\bigg)
    \bigg\rangle_\Muk\\
    %\\&+\sum_{j=1}^{n-1} \left\langle \lbA_{g,1,n-1}(\alpha_nu^{k_n}),\alpha_1(-u)^{k_1}\cdots\nabla^{s,+}_{\partial_t}(\alpha_j(-u)^{k_j})\cdots\alpha_{n-1}(-u)^{k_{n-1}}\right\rangle_\Muk\\
     =& \bigg\langle \nabla^{s,-}_{\partial_t}\Big(\lbA_{g,1,n-1}(\alpha_n(-u)^{k_n})\Big), %\alpha_1(-u)^{k_1}\cdots\alpha_{n-1}(-u)^{k_{n-1}}
    \bigodot\limits_{i=1}^{n-1}\alpha_i(-u)^{k_i}
    \bigg\rangle_\Muk 
     + \sum_{j=1}^{n-1} \langle \alpha_1\psi^{k_1},\ldots,\nabla^{s,+}_{\partial_t}(\alpha_j\psi^{k_j}),\ldots,\alpha_n\psi^{k_n}\rangle_{g}^{\cA,\Omega,s}.
\end{align*}
Thus, the right hand side of Equation~\eqref{eq:divisor} is given by
\begin{align*}
    &\bigg\langle \nabla^{s,-}_{\partial_t}\Big(\lbA_{g,1,n-1}(\alpha_n(-u)^{k_n})\Big), %\alpha_1(-u)^{k_1}\cdots\alpha_{n-1}(-u)^{k_{n-1}}
    \bigodot\limits_{i=1}^{n-1}\alpha_i(-u)^{k_i}
    \bigg\rangle_\Muk 
    %\hspace{2cm} 
    - \langle \alpha_1\psi^{k_1},\ldots,\alpha_{n-1}\psi^{k_{n-1}},\nabla^{s,+}_{\partial_t}(\alpha_n\psi^{k_n})\rangle_g^{\cA,\Omega,s}\\
    = & \bigg\langle \nabla^{s,-}_{\partial_t}\Big(\lbA_{g,1,n-1}(\alpha_n(-u)^{k_n})\Big),
    %\alpha_1(-u)^{k_1}\cdots\alpha_{n-1}(-u)^{k_{n-1}}
    \bigodot\limits_{i=1}^{n-1}\alpha_i(-u)^{k_i}\bigg\rangle_\Muk 
    %\hspace{2cm} 
    - \bigg\langle \big(\lbA_{g,1,n-1}
    \nabla^{s,+}_{\partial_t}\big(\alpha_n(-u)^{k_n}\big), %\alpha_1(-u)^{k_1}\cdots\alpha_{n-1}(-u)^{k_{n-1}}
    \bigodot\limits_{i=1}^{n-1}\alpha_i(-u)^{k_i}\bigg\rangle_\Muk
    %\\=&\left\langle [\nabla^{s,\pm}_{\partial_t},\lbA_{g,1,n-1}](\alpha_n(-u)^{k_n}),\bigodot\limits_{i=1}^{n-1}\alpha_i(-u)^{k_i}\right\rangle_\Muk
\end{align*}
Now, the result follows from the definition~\eqref{nabla-commutator}.
\end{proof}

Thus, in order to prove the divisor equation~\eqref{eq:divisor}, 
it remains to prove that
\begin{equation}\label{eq:identify-rhs} [\nabla^{s,\pm}_{\partial_t},
%\big(\cK_*\rho^{\cA,\tw}_*\hcV\big)_{g,1,n-1}
\lbA_{g,1,n-1}] 
= [\big(\cK_*\rho^{\cA,\tw,\zeta}_*\mc_2\big)_{g,1,n-1}^\epsilon]\in H_\bullet\left( {\sf Hom}^\cont \big( L^A_+[1] , \sym^{n-1} L^A_- \big), b\right).
\end{equation} 
%with both sides considered as homology classes.

\subsection{Calculations of push-forward}
%We need to compute  $\rho^{\cA,\tw,\zeta}_*\mc_2$ explicitly. 
Let us denote by $\nabla_{\partial_t}^{\sf Get}: L^\cA((u)) \ra L^\cA((u))$ the chain-level symmetric Getzler connection that acts on the periodic cyclic chain complex of $A$. Its formula is in Equation~\eqref{eq:sym-getzler}. By construction, it restricts to $L^\cA_-$ which we denote by
\[\nabla^{{\sf Get},-}_{\partial_t}: L^\cA_- \ra L^\cA_-.\] 
Although it does not restrict to $L^\cA_+$, we can define an operator $\nabla^{{\sf Get},+}_{\partial_t}: L^\cA_+ \ra L^\cA_+$ by setting
\[\nabla^{{\sf Get},+}_{\partial_t}(\alpha\cdot u^k):= 
    \pi^+\big(\nabla^{{\sf Get}}_{\partial_t}(\alpha\cdot u^k)\big),\]
where $\pi^+: L^\cA((u))\to L^\cA[[u]]$ is the canonical projection map.
Let us denote by 
\begin{equation}
\label{eta-definition}
\eta_{g,k,l}:=\big(\rho^{\cA,\tw,\zeta}_*\mc_2\big)_{g,k,l}^\epsilon \in {\sf Hom}^\cont \big( \sym^k (L^\cA_+[1]) , \sym^l (L^\cA_{-}) \big).
\end{equation}

\iffalse
The component $\big(\rho^{\cA,\tw,\zeta}_*\mc_2\big)_{g,k,l}^\epsilon$ is given by the following list.
\begin{enumerate}
\item For %the exceptional component
$(g,k,l)=(0,1,1)$, we denote this map by $\eta_{0,1,1}: L^\cA_+[1] \ra L^\cA_-$. It is explicitly given by
\item For %the exceptional component 
$(g,k,l)=(0,2,0)$, we denote this map by $\eta_{0,2,0}\in {\sf Hom}^\cont \big( \sym^2 (L^\cA_+[1]) , \mathbb{C}[[t]] \big)$ explicitly given by
\end{enumerate}
\fi

\begin{Lemma}\label{lem:divisor-evaluation}
For stable terms, we have
\begin{align}
    \rho^{\cA,\tw,\zeta}\big( \mathcal{T} \widehat{\mathcal{V}}_{g,k,l} \big)  &= [\nabla_{\partial_t}^{\sf Get,\pm}, \hbcA_{g,k,l}] := \nabla_{\partial_t}^{{\sf Get},-}\hbcA_{g,k,l} - \hbcA_{g,k,l}\nabla_{\partial_t}^{{\sf Get},+},\\
    \rho^{\cA,\tw,\zeta}\big( \mathcal{S} \widehat{\mathcal{V}}_{g,k-1,l+1} \big) &= \sum_{j=1}^{l+1} u_j^{-1}\cdot \eta_{0,2,0}\circ_j \hbcA_{g,k-1,l+1}.
\end{align}
For the exceptional terms, we have 
\begin{align}
%\begin{dcases}
  \eta_{0,1,1} &= \frac{1}{2} 
	\bigg( \begin{tikzpicture}[baseline={([yshift=-1.2ex]current bounding
			box.center)},scale=0.3] 
		\draw [thick] (0,0) to (0,2);
		\draw [thick] (-0.2, 1.8) to (0.2, 2.2);
		\draw [thick] (0.2, 1.8) to (-0.2, 2.2);
		\draw [thick] (0,0) to (-2,0);
		\draw (-2,0) node[label=left:{$\frac{dm}{dt}$}] {};
		\draw [thick] (0,0) to (2,0);
		\draw [thick] (-2.2,0) circle [radius=0.2];
		\draw [thick] (2.2,0) circle [radius=0.2];
	\end{tikzpicture}+ \begin{tikzpicture}[baseline={([yshift=-1.2ex]current bounding
			box.center)},scale=0.3] 
		\draw [thick] (0,0) to (0,2);
		\draw [thick] (-0.2, 1.8) to (0.2, 2.2);
		\draw [thick] (0.2, 1.8) to (-0.2, 2.2);
		\draw [thick] (0,0) to (-2,0);
		\draw (5,0) node[label=left:{$\frac{dm}{dt}$}] {};
		\draw [thick] (0,0) to (2,0);
		\draw [thick] (-2.2,0) circle [radius=0.2];
		\draw [thick] (2.2,0) circle [radius=0.2];
	\end{tikzpicture}\bigg),\\
\eta_{0,2,0}&= \rho^{\cA,\tw}
\left(\frac{1}{4} \begin{tikzpicture}[baseline={([yshift=.5ex]current bounding box.center)},scale=0.3]
\draw (2,0) node[cross=2pt,label=above:{}] {};
\draw (6.2,0) node[cross=2pt,label=above:{}] {};
\draw [thick] (6.2,0) to (7.2,0);
\draw [thick] (2.1,0) to (3.2,0);
\draw [thick] (5.2,0) + (-85:2) arc(-85:265:2);
\draw [thick] (5.2,-2) circle (.2);
\draw (5.2,-2) node[label=above:{${\scriptstyle \frac{dm}{dt}}$}] {};
\end{tikzpicture}
+\frac{1}{4} \begin{tikzpicture}[baseline={([yshift=-.5ex]current bounding box.center)},scale=0.3]
\draw (2,0) node[cross=2pt,label=above:{}] {};
\draw (6.2,0) node[cross=2pt,label=above:{}] {};
\draw [thick] (6.2,0) to (7.2,0);
\draw [thick] (2.1,0) to (3.2,0);
\draw [thick] (5.2,0) + (85:2) arc(85:-265:2);
\draw [thick] (5.2,2) circle (.2);
\draw (5.2,2.2) node[label=below:{${\scriptstyle \frac{dm}{dt}}$}] {};
\end{tikzpicture}
+\frac{1}{2}\begin{tikzpicture}[baseline={([yshift=.4ex]current bounding box.center)},scale=0.3]
\draw (2,0) node[cross=2pt,label=above:{}] {};
\draw (6.2,0) node[cross=2pt,label=above:{}] {};
\draw [thick] (6.2,0) to (7.2,0);
\draw [thick] (2.1,0) to (3.2,0);
\draw [thick] (5.2,0) circle [radius=2];
\draw (1.8,-1.6) circle (.2);
\draw [thick] (3.2,0) to (1.9,-1.4);
\draw (1.3,-2.2) node[label=right:{${\scriptstyle \frac{dm}{dt}}$}] {};
\end{tikzpicture}
+\frac{1}{2}\begin{tikzpicture}[baseline={([yshift=-.4ex]current bounding box.center)},scale=0.3]
\draw (2,0) node[cross=2pt,label=above:{}] {};
\draw (6.2,0) node[cross=2pt,label=above:{}] {};
\draw [thick] (6.2,0) to (7.2,0);
\draw [thick] (2.1,0) to (3.2,0);
\draw [thick] (5.2,0) circle [radius=2];
\draw (1.8,1.6) circle (.2);
\draw [thick] (3.2,0) to (1.9,1.4);
\draw (1.3,2.2) node[label=right:{${\scriptstyle \frac{dm}{dt}}$}] {};
\end{tikzpicture}\right).
%\end{dcases}
\end{align}
\end{Lemma}
\begin{proof}
For the stable component, recall by construction, the operator $\mathcal{T}=H^{\sf in} -H^{\sf out}+G^{\sf in}+G^{\sf out}+ \mathcal{T}^{\sf v}$ in Equation~\eqref{construction-T} consists of three parts: inputs, outputs and at black vertices. The parts at the inputs and outputs are exactly the operator from Getzler's connection formula~\eqref{eq:sym-getzler}, using Corollary~\ref{cor-getzler-connection}. The commutator $[\partial_t,-]$ corresponds to the part of $\mathcal{T}$ that switches a black vertex to a white vertex, since, by construction, this operation replaces $m(t)$ by $\frac{dm}{dt}$. Here we have also used the flatness of cyclic pairing, i.e., Equation~\eqref{eq:flat-cyclic-metric}.  The formula of $\rho^{\cA,\tw,\zeta}\big( \mathcal{S} \widehat{\mathcal{V}}_{g,k-1,l+1} \big)$ follows directly from its definition of $\mathcal{S}$~\eqref{incoming-S} and Equation~\eqref{eq:W-021}.

The exceptional terms follow from the explicit formulas~\eqref{graph:getzler1}, \eqref{graph:getzler2}, and~\eqref{eq:W-021} respectively. 
\end{proof}

\begin{remark}
The operator $\eta_{0,1,1}$ is exactly the pole part of the symmetric Getzler connection. The other exceptional term $\eta_{0,2,0}$ %\in {\sf Hom} \big( \sym^2 (L^\cA_+[1]) , \mathbb{C}[[t]] \big)$
also has an algebraic interpretation using% its defining identity
~\eqref{eq:defining-w021}: it bounds the failure of the operator $\eta_{0,1,1}$ being a self-adjoint operator with respect to the chain-level Mukai pairing. This follows from the Maurer-Cartan equation $\eth (f\hcV_{0,2,1})+\iota(f\hcV_{0,1,2})=0$, which implies that 
$$[b+uB,\eta_{0,2,0}] + \iota \eta_{0,1,1} =0.$$
This homotopy formula should be compared with a similar one in~\cite[Lemma 5.38]{She}.
\end{remark}

\subsection{Genus zero divisor equation}
\label{subsec:divisor-genus-zero}
Now, we give a proof of the divisor equation in genus zero. By the above discussions, it remains to prove that Equation~\eqref{eq:identify-rhs} holds when $g=0$. In this case, our strategy is to use the explicit formula of CEI~\eqref{eq:cei-formula-main}. The proof will be in complete parallel to that of Proposition~\ref{prop:string-genus-zero}.  We need to introduce some notation for the chain-level operators corresponding to $\nabla^{s}_{\partial_t}$ and $\nabla^{s,+}_{\partial_t}$. Indeed, let 
$S: L^\cA \to L^\cA[[u]]$
be a chain-level lift of the given splitting map $s:H_\bullet(L^\cA) \to H_\bullet(L^\cA[[u]])$. Denote by
\[ \widetilde{S}: \big(L^\cA((u)),b\big) \ra \big(L^\cA((u)),b+uB\big)\]
the induced isomorphism obtained by extending $S$ by $u$-linearity to the periodic chains. Its inverse map is denoted by $\widetilde{R}$. Then we obtain a pull-back connection denoted by
\[\nabla^{S}_{\partial_t}: \big(L^\cA((u)),b\big) \ra \big(L^\cA((u)),b\big)\]
defined by $\nabla^{S}_{\partial_t}:=   \widetilde{R}\circ \nabla_{\partial_t}^{\sf Get} \circ \widetilde{S}$. We may also define its projection onto the negative and positive $u$-parts which are denoted by
\begin{align*}
    \nabla^{S,+}_{\partial_t}& : L^\cA[[u]] \ra L^\cA[[u]],\\
     \nabla^{S,-}_{\partial_t}& : L^\cA[u^{-1}] \ra L^\cA[u^{-1}].
\end{align*}
The connection $\nabla^{S,+}_{\partial_t}$ is the chain-level lift of the connection $\nabla^{s,+}_{\partial_t}$ defined in Equation~\eqref{eq:conn-truncated}.

\begin{Lemma}\label{lem:chain-level-admissible-lift}
Let $s$ be admissible in the direction $\partial_t$ in the sense of Equation~\eqref{eq:divisor-admissibility}. Then there exists a chain-level lift $S$ of $s$ for which
\[
 \nabla^S_{\partial_t}\bigl(L^\cA[u^{-1}]\bigr)\subset L^\cA[u^{-1}].
\]
Equivalently, the connection matrix of $\nabla^S_{\partial_t}$ has only powers $u^{-1}$ and $u^0$. We call such an $S$ a chain-level admissible lift of $s$ in the direction $\partial_t$.
\end{Lemma}

\begin{proof}
Start with any chain-level lift $S$ as above. Since $\widetilde S$ and its inverse have only nonnegative powers of $u$ and the Getzler connection has at most a simple pole in $u$, the connection $\nabla^S_{\partial_t}$ also has at most a simple pole. Its $u^0$ part is a connection commuting with $b$, since $[\nabla^S_{\partial_t},b]=0$. Use its formal parallel frame to make $b$ constant and write $\nabla^S_{\partial_t}=\partial_t+A(t,u)$ for some $t$-linear operator on $L^{\cA}((u))$. By~\eqref{eq:divisor-admissibility}, the positive powers of $u$ in $A(t,u)$ induce zero on homology.

Let $L$ be the fiber of $L^\cA$ at $t=0$. We proceed to inductively construct an automorphism of the form $G=\id+\sum_{n\geq1}t^nG_n(u)$ with $G_n(u)\in u\operatorname{End}(L)[[u]]$, which after its gauge transformation would remove the positive powers of $u$ in $A(t,u)$ . Assume such a $G_{<n}=1+\sum_{1\leq j<n}t^jG_j(u)$ has been constructed to removed the positive powers of $A(t,u)$ modulo $t^{n-1}$. Set
\[
 G_n=-\frac{1}{n}\Big(\mbox{Coefficient of $t^{{n-1}}$ in  $\left(G_{<n}^{-1} \nabla^S_{\partial_t} G_{<n} - \partial_t\right)_{>0}$}\Big),
\]
where the subscript $(-)_{>0}$ denotes the strictly positive powers of $u$ in $(-)$. This coefficient commutes with $b$ and induces zero on homology. Adding $t^nG_n$ to $G_{<n}$ changes the coefficient at order $t^{n-1}$ by $nG_n$, from differentiating $t^nG_n$. The new coefficient therefore vanishes modulo $t^{n}$. Composing $G$ with $\widetilde S$ then yields the desired chain-level admissible lift of $s$ .
\end{proof}

With these preparations, we proceed to prove the divisor equation in genus zero. The proof is almost identical to that of Proposition~\ref{prop:string-genus-zero}, with the following role switches:
\begin{align*} 
& M_{u^{-1}}  \leftrightarrow \nabla^{S,\pm}_{\partial_t} \mbox{\;\;\; at leaves (with} \pm \mbox{depends on in/out)}, \\
&M_{u^{-1}}  \leftrightarrow \nabla^{{\sf Get},\pm}  \mbox{\;\;\; at edges,}\\
&\gamma_{g,k,l}  \leftrightarrow \eta_{g,k,l} \mbox{\;\;\; at a vertex}.
\end{align*}

% we can compute the commutators.
\begin{Lemma}\label{lem:divisor-leaf-cont}
Let $S$ be a chain-level admissible lift of $s$. Then we have the following commutators.
\begin{equation}
\label{commutators-divisor-eq}
    \begin{dcases}
      \nabla_{\partial_t}^{\sf Get,+}\circ S - S\circ \nabla_{\partial_t}^{S,+}=F\circ \eta_{0,1,1} \circ S, & \text{at an incoming leaf;}\\
      \nabla_{\partial_t}^{S,-} \circ R - R\circ \nabla^{\sf Get,-}_{\partial_t}=R\circ \eta_{0,1,1} \circ F, & \text{at an outgoing leaf};\\
       \nabla_{\partial_t}^{{\sf Get},+}\circ F-F\circ \nabla_{\partial_t}^{\sf Get,-} = F\circ \eta_{0,1,1} \circ F, & \text{at an internal directed edge}.
    \end{dcases}
\end{equation}
%\begin{itemize}
%\item[(a.)] At an incoming leaf, we have
%\[ \nabla_{\partial_t}^{\sf Get,+}\circ S - S\circ \nabla_{\partial_t}^{S,+}=F\circ \eta_{0,1,1} \circ S.\]
%\item[(b.)] At an outgoing leaf, we have
%\[ \nabla_{\partial_t}^{S,-} \circ R - R\circ \nabla^{\sf Get,-}_{\partial_t}=R\circ \eta_{0,1,1} \circ F. \]
%\item[(c.)] At an internal directed edge, we have
%\[ \nabla_{\partial_t}^{{\sf Get},+}\circ F-F\circ \nabla_{\partial_t}^{\sf Get,-} = F\circ \eta_{0,1,1} \circ F. \]
%\end{itemize}
\end{Lemma}
\begin{proof}
The proof is a straight-forward computation using the definitions of $S$, $R$, and $F$ in Equations~\eqref{eq:S-operator}~\eqref{eq:R-operator}~\eqref{eq:homotopy-F}, similar to Lemma~\ref{lem:string-leaf-cont}.

Let us prove the first identity here, and the rest two can be proved in the same way. First,  observe that for $\alpha u^k$ with $k\geq 1$, we have
\[\big(\nabla_{\partial_t}^{\sf Get,+}\circ S - S\circ \nabla_{\partial_t}^{S,+}\big)(\alpha u^k)
=\nabla_{\partial_t}^{\sf Get}\big( S (\alpha u^k)\big)- S\big(\nabla_{\partial_t}^{S}(\alpha u^k)\big)
=(S\circ\nabla_{\partial_t}^{S}-\nabla_{\partial_t}^{{\sf Get}}\circ S)(\alpha u^k)
=0.\]
Meanwhile, by Lemma~\ref{lem:divisor-evaluation}, $\eta_{0,1,1}$ vanishes on $S(\alpha u^k)$ for $k\geq 1$, we also have
\[\big(F\circ \eta_{0,1,1} \circ S\big)(\alpha u^k) = 0.\]
Now, let us focus on the case $k=0$. Indeed, we have
\begin{align*}
S\big(\nabla_{\partial_t}^{S,+} (\alpha)\big) & = S\big(\nabla_{\partial_t}^{S} (\alpha)\big) - S\big(u^{-1}\eta_{0,1,1}(\alpha)\big),\\
\nabla_{\partial_t}^{{\sf Get},+} S(\alpha) &=  \nabla_{\partial_t}^{{\sf Get}} S(\alpha)- u^{-1}\eta_{0,1,1}(\alpha).
\end{align*}
Taking the difference and noting that $S\nabla_{\partial_t}^{S}=\nabla_{\partial_t}^{{\sf Get}} S$, we obtain 
\begin{align*}
\big( \nabla_{\partial_t}^{\sf Get,+}\circ S - S\circ \nabla_{\partial_t}^{S,+}\big)(\alpha)  
&=-u^{-1}\eta_{0,1,1}(\alpha)+S\big(u^{-1}\eta_{0,1,1}(\alpha)\big)\\
&= \sum_{k\geq 0} S_{k+1}(\eta_{0,1,1}(\alpha)) u^k\\
&= -\sum_{k\geq 0} \bigg(\sum_{l=0}^k S_l R_{k+1-l}\bigg)\eta_{0,1,1}(\alpha) u^k\\
&= F\big(\eta_{0,1,1}(\alpha)\big) \\
&=\big(F\circ \eta_{0,1,1} \circ S\big)(\alpha).
\end{align*}    
So, the first identity holds.
\end{proof}

\begin{Proposition}
The divisor equation~\eqref{eq:divisor} holds when $g=0$.
\end{Proposition}

\begin{proof}
According to the previous discussions, it suffices to prove the chain level equality of Equation~\eqref{eq:identify-rhs} at genus zero. That is,  %we need to prove the following equation
\[ [\nabla^{S,\pm}_{\partial_t},%(\mathcal{K}_*\rho_*^{A,\tw}\hcV)_{0,1,n-1}
\lbA_{0,1,n-1}]= (\cK_*\rho^{A,\tw,\zeta}_*\mc_2)_{0,1,n-1}^\epsilon.\] 

The Lemma~\ref{lem:divisor-leaf-cont} puts us in the same situation as in the proof of Proposition~\ref{prop:string-genus-zero}. Indeed, let us fix a partially directed graph $\GG\in \Gamma((g=0,1,n-1))_{m}$. Consider its commutator with the connection $\nabla_{\partial_t}^{S,\pm}$. We think of the operator $\nabla_{\partial_t}^{S,+}$ as flowing along $\GG$ from top to bottom. As it passes the input leaf contribution $S$, using the first identity in~\eqref{commutators-divisor-eq} we see that it produces a binary vertex at the input leaf labeled by the operator $\eta_{0,1,1}$. After passing the input leaf, the operator becomes $\nabla_{\partial_t}^{\sf Get,+}$. As we pass the operator $\nabla_{\partial_t}^{\sf Get,+}$ at a vertex $v$, we obtain the commutator $\eta_{0,1,l(v)}$. Then, the operator becomes $\nabla_{\partial_t}^{\sf Get,-}$ at an output of $v$. In this case, using the third identity in~\eqref{commutators-divisor-eq}, we obtain a binary vertex at the edge, again labeled by the operator $\eta_{0,1,1}$. Continuing in this way, we keep passing the connection operator down until it finally is at one of the outgoing leaves of the tree $\GG$. Then, using the second identity in~\eqref{commutators-divisor-eq}, we obtain a binary vertex at the outgoing leaf labeled by the operator $\eta_{0,1,1}$, and the connection operator becomes $\nabla_{\partial_t}^{S,-}$. 

On the right hand side we read off the coefficient in the component $(\cK_*\rho^{A,\tw,\zeta}_*\mc_2)^\epsilon_{0,1,n-1}$. By the construction of $\mathcal{K}$ in Equation~\eqref{eq:trivialization-map-K} and that $\epsilon^2=0$, the right hand side is given by a sum over partially directed graphs such that exactly one of the vertex is labeled by the contribution using $\eta_{0,1,1}$. This is precisely the same as what we previously obtained when computing the commutator $[\nabla^{S,\pm}_{\partial_t},%(\mathcal{K}_*\rho_*^{A,\tw}\hcV)_{0,1,n-1}
\lbA_{g,1,n-1}
]$ using Lemma~\ref{lem:divisor-leaf-cont}.
\end{proof}

The proof of Theorem~\ref{thm:divisor} in the case of arbitrary genus is included in Section~\ref{app:proof-string}.

\section{The holomorphic anomaly equations}\label{sec:hae}

In this section we prove the holomorphic anomaly equations in Theorem~\ref{thm:main}. We shall continue to work with notations in Section~\ref{para:intro-1.1} and Section~\ref{sec-intro-1.2}. The main idea of the proof is to compare the complex conjugate splitting with a holomorphic splitting, and then use the compatibility of CEI with the Givental group action.  We refer to~\cite{Giv,Pan} for basics of the Givental group action.

\subsection{Family version CEI}
\label{subsec:cei-family} 
The construction of CEI in~\cite{CalTu1} is done over a field (of characteristic zero). A crucial observation used in this paper is that the same construction also works for a cyclic $A_\infty$-algebra $\cA$ over a base commutative ring $T$ (again of characteristic zero) with the following additional assumptions on $\cA$:
\begin{itemize}
\item[(1.)] The $A_\infty$-algebra $\cA$ is a projective $T$-module of finite rank.
\item[(2.)] The Hochschild homology (relative to $T$) $HH_\bullet(\cA/T)$ is a projective $T$-module of finite rank.
\end{itemize}
More generally, we may also work with a non-affine base space $M$ and assuming that both $\cA$ and its Hochschild homology $HH_\bullet(\cA)$ (relative to $\cO_M$) are locally free of finite rank over $\cO_M$. The well-definedness of CEI over a base ring, under the above assumptions, is particularly clear from the explicit formula of CEI~\eqref{eq:cei-formula-main}. Indeed, by projectivity of $\cA$ and $HH_\bullet(\cA)$, we still have the existence of a chain-level splittings $S$ and $R=S^{-1}$ used in that formula.

\begin{Lemma}
Let $p:\mathfrak{X}\ra M$ be a family of smooth projective Calabi-Yau $3$-folds over a smooth base $M={\sf Spec\;} T$. Then there exists a cyclic $A_\infty$-algebra $\cA$ linear over $T$ satisfying $(1.)$, $(2.)$ above. Furthermore, for any point $b\in M$ the specialization $\cA_b$ is $A_\infty$-Morita equivalent to $D_{dg}^b({\sf Coh}(\mathfrak{X}_b))$.
\end{Lemma}

\begin{proof}
Since $p$ is a projective family, it is uniformaly polarized, i.e., $\mathfrak{X}\hookrightarrow \mathbb{P}_T^N$ for some fixed $N$. Denote by $\mathbb{E}=\cO_\mathfrak{X}\oplus\cdots\oplus \cO_\mathfrak{X}(N)$ the vector bundle on $\mathfrak{X}$ obtained by the restriction from $P_T^N$ of the Beilinson's generator. We may consider a \u Cech resolution of $\sEnd_{\mathfrak{X}/M}(\mathbb{E})$ denoted by $C^*\big(\cU,\sEnd_{\mathfrak{X}/M}(\mathbb{E})\big)$. Denote its cohomology by $\cA$. We claim that $\cA$ satisfies all the desired properties. The locally freeness of $\cA$ is a direct compution of coherent cohomology groups $\Ext^*(\cO_\mathfrak{X}(i),\cO_\mathfrak{X}(j))$. The locally freeness of its Hochschild homology follows from Keller~\cite{Kel} and the Comparison result by Hochschild-Konstant-Rosenberg~\cite{HKR}. Finally, the existence of a cyclic structure is by Kontsevich-Soibelman~\cite{KonSoi,Cho}, and~\cite{AmoTu2} for the unital version.
\end{proof}

\begin{remark}
In the case of the quintic family and mirror quintic family, explicit formulas of the cyclic $A_\infty$-algebra $\cA$ can be derived~\cite{Tu,Tu2} using Kontsevich's deformation quantization formula~\cite{Kon2}.
\end{remark}

\subsection{Comparison of Hodge structures} Our next goal is to understand the nc splitting data in Definition~\ref{def:nc-splitting} in more geometric terms. This requires a comparison of nc-Hodge structures with the usual geometric Hodge structures. 
Let $p: \mathfrak{X} \ra M$ be a smooth and projective family of Calabi-Yau $3$-folds, with a smooth base space $M$. Following Griffiths~\cite{Gri}, we obtain the following data:
\begin{itemize}
\item The Hodge bundle $\mathbb{H}:= R^3p_*\underline{\mathbb{C}} \otimes \cO_M$.
\item The Hodge filtration $\mathbb{L}= F^3\mathbb{H}\subset F^2\mathbb{H}\subset F^1\mathbb{H} \subset F^0\mathbb{H}=\mathbb{H}$.
\item The Gauss-Manin connection $\nabla^{\sf GM}$.
\item A rational lattice $R^3p_*\underline{\mathbb{Q}}$ in $\mathbb{H}$.
\item The intersection pairing $Q(-,-)$ on $\mathbb{H}$.
\end{itemize}
These data satisfy certain compatibility conditions, such as Griffiths transversality, flatness of $Q(-,-)$ and so on, axiomized as variation of Hodge structures~\cite{Gri}.

On the other hand, there exists an entirely parallel construction in the world of non-commutative geometry, following the works~\cite{Bar,KonSoi,KKP,Shk,Get,Bla} and~\cite{Sai1,Sai2}. We obtain a similar set of data:
\begin{itemize}
\item The nc-Hodge bundle $\mathbb{H}^{nc}:= HP_{\sf odd}(\mathfrak{X}/M)$.
\item The nc-Hodge filtration $F^\bullet\mathbb{H}^{nc}=HC^-_\bullet(\mathfrak{X}/M)$.
\item The Getzler connection $\nabla^{\sf Get}$.
\item The Blanc-Toen lattice given by the image of the topological Chern character %map 
${\sf ch}^{\sf Top}\otimes \mathbb{Q}\subset \mathbb{H}^{nc}$. %HP_{\sf odd}(\mathfrak{X}/M)
\item The residue pairing $Q^{nc}$ on $\mathbb{H}^{nc}.$% $HP_{\sf odd}(\mathfrak{X}/M)$.
\end{itemize}

The following comparison result is proved in~\cite{Tu3}.
\begin{Theorem}\label{conj:comparison}
There exists an isomorphism $\tilde{J}: \mathbb{H}^{nc} \ra \mathbb{H}$ that intertwines all the above structures.
\end{Theorem}

Using the comparison theorem above, we may understand the non-commutative splitting data in Definition~\ref{def:nc-splitting} in terms of more classical geometric data. 
\begin{Lemma}\label{lem:splitting-linear-alg}
Fix a base point $b\in M$. Then a splitting of the nc-Hodge filtration of $\cA_b$ in the sense of Definition~\ref{def:nc-splitting} is equivalent to the following data:
    \begin{enumerate}
        \item a choice of Lagrangian complement of $F^2HP_{\sf odd}(\cA_b)\subset HP_{\sf odd}(\cA_b)$,
        \item a choice of linear complement of $F^3HP_{\sf odd}(\cA_b) \subset  F^2HP_{\sf odd}(\cA_b)$. 
    \end{enumerate}
\end{Lemma}
\begin{proof}
In the following, we shall write $HH_j=HH_j(\cA_b)$ and $HC^-_j=HC^-_j(\cA_b)$, for simplicity. By HKR isomorphism and $\pi_1(\mathfrak{X}_b)=0$ imply that the even part of Hochschild homology is concentrated in degree zero. The projection $HC^-_0\ra HH_0$ is then an isomorphism by Hodge-to-de Rham degeneration, so the even part splits canonically. 

For the odd part we use the Rees correspondence of~\cite[Lemma 2.7]{GPS}. By periodicity, realize $V=HP_{\sf odd}(\cA_b)$ as $HP_{-3}(\cA_b)$. In our homological grading, $u$ has degree $-2$, and the Hodge filtration, indexed as in Theorem~\ref{conj:comparison}, is
\[
F^p V = u^p HC^-_{2p-3}\subset V,\qquad
\operatorname{Gr}_F^p V
= u^p HC^-_{2p-3}/u^{p+1}HC^-_{2p-1}
\cong HH_{2p-3}.
\]
The last isomorphism is induced by dividing by $u^p$ and reducing modulo $u$. In particular, $F^0V=V$ and $F^4V=0$. A graded section $s$ of the Hochschild projection therefore gives a decomposition of this filtration with pieces $V^p=u^p s(HH_{2p-3})$, for $0\leq p\leq 3$. Conversely, given a decomposition $F^pV=V^p\oplus F^{p+1}V$, one recovers $s$ by lifting each class of $HH_{2p-3}$ to $V^p$ via the decomposition (followed by multiplying the lift by $u^{-p}$).

Under the same correspondence the pairing $Q^{nc}$ on $V$ is, up to a fixed nonzero scalar, the coefficient of $u^3$ in the higher residue pairing of two representatives of degree $-3$. Indeed, to pair opposite degrees as in~\cite[Lemma 2.7]{GPS}, one multiplies the second representative by $u^{-3}$. For $x\in HH_{2p-3}$ and $y\in HH_{2q-3}$ we have
\[
[u^3]\langle u^p s(x),u^q s(y)\rangle_{\sf hres}
= (-1)^p [u^{3-p-q}]\langle s(x),s(y)\rangle_{\sf hres},
\]
where $[u^{k}](F(u))$ denotes the coefficient of $u^{k}$ in $F(u)$. From this identity we see that the Lagrangian condition on $s$ is equivalent to $Q^{nc}(V^p,V^q)=0$ for $p+q\ne 3$. The polarized Rees correspondence also gives $Q^{nc}(F^p V,F^q V)=0$ for $p+q>3$, and the perfect pairings on complementary graded pieces give $\dim F^2V=\frac12\dim V$ and $\dim F^1V=\dim V-\dim F^3V$. Thus $F^2V$ is Lagrangian and $(F^3V)^\perp=F^1V$, where $\perp$ is taken with respect to $Q^{nc}$. Hence a splitting $s$ gives the complements $V^0\oplus V^1$ and $V^2$ in $(1)$ and $(2)$, respectively.

Conversely, let $U$ be the Lagrangian complement in $(1)$ and let $V^2$ be the complement in $(2)$. Set
\[
V^3=F^3V,\qquad V^1=U\cap (V^3)^\perp,\qquad V^0=U\cap (V^2)^\perp.
\]
The pairing between the complementary Lagrangians $U$ and $F^2V=V^2\oplus V^3$ is perfect. Thus $U=V^0\oplus V^1$, and $(V^3)^\perp=F^1V$ gives $F^1V=V^1\oplus F^2V$. The resulting decomposition of $F^\bullet V$ pairs only complementary indices, by the defining annihilator conditions and the isotropy of $U$ and $F^2V$. It is the unique such decomposition with the prescribed complements, since within $U$ the two lower pieces must be the annihilators of $V^2$ and $V^3$, respectively. The inverse Rees construction therefore gives the required splitting $s$.
\end{proof}

The Lemma~\ref{lem:splitting-linear-alg} shows that the splittings of the nc-Hodge filtration of $\cA_b$ can be described in terms of nc-Hodge filtration on the periodic cyclic homology of $\cA_b$. Hence, by Theorem~\ref{conj:comparison}, %it makes sense to make the following definition: 
we may define a splitting of the classical Hodge filtration of $\mathfrak{X}_b$ given by the following data:
\begin{enumerate}
        \item a choice of Lagrangian complement of $F^2H^3(\mathfrak{X}_b)\subset H^3(\mathfrak{X}_b)$,
        \item a choice of a linear complement of $F^3H^3(\mathfrak{X}_b) \subset  F^2H^3(\mathfrak{X}_b)$. 
\end{enumerate}
The following lemma then easily follows from Lemma~\ref{lem:splitting-linear-alg}. 

\begin{Lemma}\label{lem:splitting-comparison}
There exists a canonical bijection between the set of splittings of nc-Hodge filtration of $\cA_b$ and the set of splittings of the classical Hodge filtration of $\mathfrak{X}_b$.
\end{Lemma}

To this end, observe that using the rational structure $H^3(\mathfrak{X}_b,\mathbb{Q})$, we have a canonically defined splitting of the classical Hodge filtration by setting the Lagrangian complement by $\overline{F^2H^3(\mathfrak{X}_b)}$, the complex conjugate of $F^2H^3(\mathfrak{X}_b)$, and the linear complement by $F^2H^3(\mathfrak{X}_b)\cap \overline{F^1H^3(\mathfrak{X}_b)}$. Of course, this is simply the classical $(p,q)$-decomposition
\[ H^3(\mathfrak{X}_b) \cong H^{3,0}(\mathfrak{X}_b)\oplus H^{2,1}(\mathfrak{X}_b)\oplus H^{1,2}(\mathfrak{X}_b)\oplus H^{0,3}(\mathfrak{X}_b).\]
The corresponding nc-splitting under the bijection of Lemma~\ref{lem:splitting-comparison} will be called the {\em Blanc-Toen} splitting of the nc-Hodge filtration, and will be denoted by $s^{\sf BT}$.

\subsection{Givental's quantization formula}
\subsubsection{The Givental group} 
Consider a splitting of the nc-Hodge filtration $$s: HH_\bullet(\cA_b)[3] \ra HC_\bullet^-(\cA_b)[3].$$ 
%be a splitting of the nc-Hodge filtration. 
In our setup, the even part of Hochschild splits canonically. It suffices to restrict $s$ to the odd part.
%\[ s: \bigoplus_{i=0}^3 HH_{3-2i}(\cA_b)[3] \ra HC^-_*(\cA_b)[3].\]
To give a splitting map $s$ is equivalent to give a $\mathbb{C}[[u]]$-linear isomorphism (obtained by extending $s$ linearly in $u$)
\[ \widetilde{s}: \bigoplus_{i=0}^3 HH_{3-2i}(\cA_b)[3][[u]] \ra HC^-_{\bullet}(\cA_b)[3].\]
Let $s'$ be another splitting of the nc-Hodge filtration. Then we obtain an automorphism
\[ \widetilde{s}^{-1} \widetilde{s'}: \bigoplus_{i=0}^3 HH_{3-2i}(\cA_b)[3][[u]] \ra \bigoplus_{i=0}^3 HH_{3-2i}(\cA_b)[3][[u]].\]
By Definition~\ref{def:nc-splitting}, a splitting preserves the $\Z$-grading, which implies that the automorphism above restricts to the subspace of polynomials in $u$ (instead of power series). Furthermore, the Lagrangian condition in Definition~\ref{def:nc-splitting} implies certain compatibility of the automorphism with respect to the Mukai pairing $\langle-,-\rangle_\Muk$. This discussion leads us to introduce the so-called homogeneous Givental group~\cite{Giv} associated with the vector space $HH_{\sf odd}(\cA_b)[3]$ endowed with the symmetric pairing $\langle-,-\rangle_{\Muk}$:
\begin{equation}\label{eq:giv-group} {\sf Giv}:= \{\mathcal{R}(u)\in \End\big( HH_{\sf odd}(\cA_b)[3]\big)[u] \mid \mathcal{R}^*(-u)\mathcal{R}(u)=\id.\}\end{equation}
Here, the adjoint $L^*$ of an operator $L\in \End\big( HH_{\sf odd}(\cA_b)[3]\big)$ is computed using the Mukai pairing. Our previous discussion shows that the set of splittings of the nc-Hodge filtration of $\cA_b$ is a torsor over the Givental group ${\sf Giv}$. In fact, from the previous discussion, this torsor has a canonical base point given by the splitting $s^{\sf BT}$.

%In the end of this section, 
Now, we give a more explicit description of the elements in ${\sf Giv}$. Indeed, we may write an element $\mathcal{R}(u)\in {\sf Giv}$ in the decomposition $\bigoplus_{i=0}^3 HH_{3-2i}(\cA_b)[3]$  denoting its components by 
\iffalse
it takes the form
\[ \mathcal{R}(u) = \begin{bmatrix}
    \id & \mathcal{R}_{1,0}\cdot u & \mathcal{R}_{2,0}\cdot u^2 & \mathcal{R}_{3,0}\cdot u^3 \\
    0  & \id & \mathcal{R}_{2,1} \cdot u & \mathcal{R}_{3,1}\cdot u^2\\
    0 & 0 & \id & \mathcal{R}_{3,2}\cdot u\\
    0 & 0 & 0 & \id
\end{bmatrix}\]
\fi
 $$\mathcal{R}_{i,j}\cdot u^{i-j}:HH_{3-2i}(\cA_b)[3]\to HH_{3-2j}(\cA_b)[3]\quad \text{for} \quad i> j.$$ 
Then equation $\mathcal{R}^*(-u)\mathcal{R}(u)=\id$ implies that $$\mathcal{R}_{3,2}=\mathcal{R}_{1,0}^*, \quad \mathcal{R}_{3,1}=(\mathcal{R}_{2,1}^*\mathcal{R}_{1,0}^*- \mathcal{R}_{2,0}^*), \quad \mathcal{R}_{2,1}^*=\mathcal{R}_{2,1}.$$ 
Thus, we may write $\mathcal{R}(u)$ as 
\begin{align}\label{eq:R(u)} 
\mathcal{R}(u) =& \begin{bmatrix}
    \id & \mathcal{R}_{1,0}\cdot u & \mathcal{R}_{2,0}\cdot u^2 & \mathcal{R}_{3,0}\cdot u^3 \\
    0  & \id & \mathcal{R}_{2,1} \cdot u & (\mathcal{R}_{2,1}^*\mathcal{R}_{1,0}^*- \mathcal{R}_{2,0}^*)\cdot u^2\\
    0 & 0 & \id & \mathcal{R}_{1,0}^*\cdot u\\
    0 & 0 & 0 & \id
\end{bmatrix}\\
=&
\begin{bmatrix}
    \id & 0 & \mathcal{R}_{2,0}\cdot u^2 & (\mathcal{R}_{3,0}^*-\mathcal{R}_{1,0}\mathcal{R}_{2,0}^*)\cdot u^3 \\
    0  & \id & \mathcal{R}_{2,1}\cdot u & - \mathcal{R}_{2,0}^*\cdot u^2\\
    0 & 0 & \id & 0\\
    0 & 0 & 0 & \id
\end{bmatrix} \cdot \begin{bmatrix}
    \id & \mathcal{R}_{1,0}\cdot u & 0 & 0 \\
    0  & \id & 0 & 0\\
    0 & 0 & \id & \mathcal{R}_{1,0}^*\cdot u\\
    0 & 0 & 0 & \id
\end{bmatrix}\nonumber
\end{align}
\iffalse
Such a matrix may be factored as
\[ \mathcal{R}(u) = \begin{bmatrix}
    \id & 0 & \mathcal{R}_{2,0}\cdot u^2 & (\mathcal{R}_{3,0}^*-\mathcal{R}_{1,0}\mathcal{R}_{2,0}^*)\cdot u^3 \\
    0  & \id & \mathcal{R}_{2,1}\cdot u & - \mathcal{R}_{2,0}^*\cdot u^2\\
    0 & 0 & \id & 0\\
    0 & 0 & 0 & \id
\end{bmatrix} \cdot \begin{bmatrix}
    \id & \mathcal{R}_{1,0}\cdot u & 0 & 0 \\
    0  & \id & 0 & 0\\
    0 & 0 & \id & \mathcal{R}_{1,0}^*\cdot u\\
    0 & 0 & 0 & \id
\end{bmatrix} \]
\fi
Observe that this decomposition is compatible with the linear algebra data $(1)$ and $(2)$ in Lemma~\ref{lem:splitting-linear-alg} in the sense that the left matrix factor changes the Lagrangian complement, while the right matrix factor changes the linear complement. 
\iffalse
\begin{remark}
%{\bf A remark on notation.} 
In the following subsections, we shall implicitly identify an  element $\mathcal{R}(u)$ in Equation~\eqref{eq:R(u)} in the Givental's group with the corresponding
upper triangular matrix 
\[ \mathcal{R}:= \begin{bmatrix}
    \id & \mathcal{R}_{1,0} & \mathcal{R}_{2,0} & \mathcal{R}_{3,0} \\
    0  & \id & \mathcal{R}_{2,1}  & (\mathcal{R}_{2,1}^*\mathcal{R}_{1,0}^*- \mathcal{R}_{2,0}^*)\\
    0 & 0 & \id & \mathcal{R}_{1,0}^*\\
    0 & 0 & 0 & \id
\end{bmatrix}.\]
\end{remark}
\fi

\subsubsection{Givental group action on CEI} 
Consider a collection of functionals $$F^{X,\Omega,s}:= \left\{ F_{g,n}^{X,\Omega,s}: \sym^n \big(HH_\bullet(X)[3][[u]]\big) \to \mathbb{C} \mid (g,n) \text{ is stable}, n\geq 1\right\}$$ where each functional 
%For each stable pair $(g,n)$, the CEI defines a functional %denoted by
$F_{g,n}^{X,\Omega,s}$ is defined by the CEI. 
%Denote by $F^{X,\Omega,s}:= \big(F_{g,n}^{X,\Omega,s}\big)_{2g-2+n>0}$ these functionals. 
Let $\mathcal{R}(u)\in \Giv$ be an element in the Givental group of $HH_\bullet(X)[3]$. Following~\cite{Giv} and~\cite[Section 1.2]{Pan}, we define a new collection of functionals %denoted by 
\[\left\{ \big(\mathcal{R} F^{X,\Omega,s}\big)_{g,n}:= \sum_{G\in \Gamma((g,n))} \frac{1}{|\Aut(G)|}\prod_v {\sf Cont}(v)\prod_e {\sf Cont}(e) \prod_l {\sf Cont}(l)\right\},\]
where
\begin{enumerate}
    \item the vertex contribution is ${\sf Cont}(v):= F^{X,\Omega,s}_{g(v),n(v)}$ with $g(v)$ and $n(v)$ denote the genus and the number of half-edges and leaves at the vertex;
    \item the leg contribution is by the operator ${\sf Cont}(l):=\mathcal{R}(u)$;
    \item the edge contribution is defined as follows. Let $\{e_i\}$ be a basis vector of $HH_\bullet(X)[3]$. In this basis, the Mukai pairing is given by $M_{ij}:=\langle e_i,e_j\rangle_\Muk$. Denote by $M^{ij}$ its inverse pairing in the dual basis $\{e_i^\vee\}$. 
    \begin{equation}\label{eq:giv-propagater} {\sf Cont}(e):= \sum_{i,j} \frac{M^{ij}- \sum_{k,l} \mathcal{R}_k^i(u')M^{kl}\mathcal{R}_l^j(u'')}{u'+u''} e_i\otimes e_j \in HH_\bullet(X)[3][[u']]\otimes HH_\bullet(X)[3][[u'']].\end{equation}
\end{enumerate}

As mentioned in the introduction, the compatibility of CEI with Givental's action is proved in~\cite{CalTu1} with strictly postive number of insertions, i.e., for stable pair $(g,n)$ such that $n\geq 1$, we have
\begin{equation}\label{eq:R-action-gn}
    F_{g,n}^{X,\Omega,s\mathcal{R}^{-1}}= \big(\mathcal{R} F^{X,\Omega,s}\big)_{g,n}.
\end{equation}

Following Costello~\cite{Cos2}, we force the dilaton equation~\eqref{eq:dilaton} for the case of $n=0$ to define CEIs with no insertions $\langle \emptyset \rangle_{g,0}^{X,\Omega,s}$ for $g \geq 2$ by 
\begin{equation}\label{eq:cei-no-insertion}
    F_g^{X,\Omega,s}:= \langle \emptyset \rangle_{g,0}^{X,\Omega,s}
    :=\frac{1}{2g-2} \langle [\Omega]\psi \rangle_{g,1}^{X,\Omega,s}, \quad \forall g\geq2.
\end{equation}
We may extend the compatibility of CEIs (including those with no insertions) with Givental's action.
\begin{Lemma}
Let $s'= s \mathcal{R}$ for some element $R\in {\sf Giv}$ in the Givental group~\eqref{eq:giv-group}. We have %Then for each $g\geq 2$, we have
\[ F_g^{X,\Omega,s}= \big(\mathcal{R} F^{X,\Omega,s'}\big)_g=\sum_{G\in \Gamma((g,0))} \prod_{v\in V_G} F_{g(v),n(v)}^{X,\Omega,s'} \prod_{e\in E_G} {\sf Cont}(e), \quad \forall g\geq2.\]
Here, $\Gamma((g,0))$ are stable genus $g$ graphs with no leaves and ${\sf Cont}(e)$ stands for Givental's propagator.
\end{Lemma}

\begin{proof}
Since we know the compatibility with positive number of insertions, we have
\begin{align*}
F_g^{X,\Omega,s} = \frac{1}{2g-2} \langle [\Omega]\psi \rangle_{g,1}^{X,\Omega,s}
 = \frac{1}{2g-2} \sum_{G'\in \Gamma((g,1))} \prod_{v\in V_{G'}} F_{g(v),n(v)}^{X,\Omega,s'} \prod_{e\in E_{G'}} {\sf Cont}(e) \prod_l {\sf Cont}(l)
\end{align*}
Since the unique leg of $G'$ has insertion $\Omega\psi$, its attached vertex cannot be of type $(0,3)$, which implies that $G'$ with its unique leg removed is a stable graph in $\Gamma((g,0))$. In contrast, fix a stable graph $G\in \Gamma((g,0))$, we may add a leg to any of its vertices to produce a stable graph in $\Gamma((g,1))$. By this observation and the Dilaton equation~\eqref{eq:dilaton}, we may simply the above summation to
\begin{align*}
&\frac{1}{2g-2} \sum_{G\in \Gamma((g,0))} \Big(\sum_{v\in V_G} (2g(v)-2+n(v))\Big)\prod_{v\in V_{G}} F_{g(v),n(v)}^{X,\Omega,s'} \prod_{e\in E_{G}} {\sf Cont}(e)\\
& = \frac{1}{2g-2} \cdot (2g-2) \sum_{G\in \Gamma((g,0))} \prod_{v\in V_{G}} F_{g(v),n(v)}^{X,\Omega,s'} \prod_{e\in E_{G}} {\sf Cont}_{R}(e)\\
& = \sum_{G\in \Gamma((g,0))} \prod_{v\in V_{G}} F_{g(v),n(v)}^{X,\Omega,s'} \prod_{e\in E_{G}} {\sf Cont}(e).
\end{align*}
This completes the proof.
%We are done.
\end{proof}

Analogously to the ancestor potential in Gromov-Witten theory, we define a generating function 
\[\mathfrak{D}^{X,\Omega,s}:= \exp\left(\sum_{g,n} F^{X,\Omega,s}_{g,n}  \cdot \hbar^{g-1} \right) \in \widehat{\sym} \big(HH_\bullet(X)[3][[u]]\big)^\vee((\hbar)).\]
Following~\cite{Giv}, let $\widehat{\mathcal{R}}$ be the quantization operator of $\mathcal{R}$.
%$$\mathcal{R}\mapsto \widehat{\mathcal{R}}\in \End\left( \widehat{\sym} \big(HH_\bullet(X)[3][[u]]\big)^\vee((\hbar))\right),$$
Equation~\eqref{eq:R-action-gn} may be succinctly written as
\begin{equation}\label{eq:R-action-potential} \mathfrak{D}^{X,\Omega,s\mathcal{R}^{-1}}= \widehat{\mathcal{R}}\mathfrak{D}^{X,\Omega,s}
\in \widehat{\sym} \big(HH_\bullet(X)[3][[u]]\big)^\vee((\hbar)).\end{equation}
Note that we shall not need the precise formula for the association $\mathcal{R}\mapsto \widehat{\mathcal{R}}$. For us, it is only a placeholder to package the data from Equation~\eqref{eq:R-action-gn}.

\subsection{The holomorphic anomaly equation}
The idea to prove Theorem~\ref{thm:main} is by comparing CEI in the Blanc-Toen splitting with any holomorphic splitting. We shall continue to work within the setup of Section~\ref{subsec:cei-family} and notation therein. By the comparison result in Theorem~\ref{conj:comparison}, we shall also implicitly identify the nc-Hodge structure with the classical Hodge structure. 

Let $p:\mathfrak{X}\ra M$ be a family of smooth projective Calabi-Yau $3$-folds over a smooth base $M$. We have the canonical splitting $s^{\sf BT}$ which is a $C^\infty$ splitting over $M$. Let us choose any holomorphic splitting $s^{\sf hol}$ over $M$, which exists after possible making $M$ smaller. Since any two splittings are related by an element in the homogeneous Givental group, or equivalently, an upper-triangular symplectic matrix, see~\eqref{eq:giv-group}, we have 
\begin{equation}\label{eq:gauge-trans} 
s^{\sf hol}=s^{\sf BT} \mathfrak{A}
\end{equation}
for some upper-triangular symplectic matrix $\mathfrak{A}$. Since $s^{\sf BT}$ is only a smooth splitting over the base, the matrix $\mathfrak{A}$ is also a smooth section over $M$ with values in the Givental group.  Let $\mathfrak{b}^{\sf hol}$ denote a holomorphic basis in the image of $s^{\sf hol}$. This determines a $C^\infty$-basis $\mathfrak{b}^{\sf BT}$ by the identity
$\mathfrak{b}^{\sf BT} \mathfrak{A} = \mathfrak{b}^{\sf hol}.$ 
Note that in this way of writing down the transition matrix, we view the basis as horizontal vectors. 

\begin{Lemma}
The matrix $\mathfrak{A}$ satisfies the identity   
\begin{equation}\label{C-bar-operator}
\db \mathfrak{A}\cdot  \mathfrak{A}^{-1}= -\overline{C} \quad \in \quad \Omega_M^{0,1}\left(\bigoplus_{p=0}^2\Hom(\mathbb{H}^{p,3-p}, \mathbb{H}^{p+1,2-p})\right)
\end{equation}
where $\overline{C}$ is the component of the Gauss-Manin connection as defined in Equation~\eqref{eq:gm-components}.
\end{Lemma}

\begin{proof}
Observe that the complex structure on the bundle $\mathbb{H}$ is defined using the $(0,1)$-part of the Gauss-Manin connection $\nabla^{0,1}$. Since the basis $\mathfrak{b}^{\sf hol}$ is by definition a holomorphic basis of $\mathbb{H}$, in this basis the connection $\nabla^{0,1}$ acts by zero. Now, if we change basis according to $\mathfrak{b}^{\sf BT}  = \mathfrak{b}^{\sf hol}\mathfrak{A}^{-1} $, the same connection acts by
$-\db \mathfrak{A}\cdot  \mathfrak{A}^{-1}.$
On the other hand, we know that the connection $\nabla^{0,1}$ decomposes as $D^{0,1}+\overline{C}$ in the Hodge decomposition. Since the matrix $\mathfrak{A}$ is strictly upper-triangular, it implies that $D^{0,1}\mathfrak{b}^{\sf BT}=0$. We conclude that $\db \mathfrak{A}\cdot  \mathfrak{A}^{-1}= -\overline{C}$ as desired.
\end{proof}

%

%For each $(g,n)$ as above, we define a generating function
%\begin{equation}\label{eq:g-n-function}
%E^{X,\Omega,s}_{g,n}:= \sum_{n_0+n_1+\cdots+n_\mu=n} \langle \underbrace{f_0 \psi,\ldots, f_0 \psi }_{n_0\text{-copies}}, \underbrace{f_1, \ldots, f_1}_{n_1 \text{-copies}},\ldots,\underbrace{f_\mu, \ldots, f_{\mu}}_{n_\mu \text{-copies}}\rangle_g^{X,\Omega,s} \prod_{j=0}^\mu \frac{e_j^{n_j}}{n_j!}.
%\end{equation}
%Equivalently, $E^{X,\Omega,s}_{g,n}$ is the restriction of the function $F^{X,\Omega,s}_{g,n}$ on the subspace of $HH_\bullet(\mathfrak{X}/M)[3][[\psi]]$ spanned by $f_0\psi, f_1,\ldots,f_\mu$. Furthermore, we define a total potential function by
%\[\c\mathfrak{D}^{X,\Omega,s}:= 
%\exp\left(\sum_{g,n} E^{X,\Omega,s}_{g,n}  \cdot \hbar^{g-1} \right).\]

\begin{Proposition}
Let $\widehat{-\overline{C}}$ be the Givental quantization operator of %the infinitesimal symplectic transformation 
$-\overline{C}$.
We have
\begin{equation}\label{eq:db-potential}
    \db \mathfrak{D}^{\mathfrak{X}/M,\Omega, s^{\sf BT}}=
    \widehat{-\overline{C}}\mathfrak{D}^{\mathfrak{X}/M, \Omega,s^{\sf BT}}.\end{equation}
\end{Proposition}
\begin{proof}
Using Equation~\eqref{eq:R-action-potential} and Equation~\eqref{eq:gauge-trans} we obtain
\[ \mathfrak{D}^{\mathfrak{X}/M,\Omega, s^{\sf BT}} = \widehat{\mathfrak{A}} \;\mathfrak{D}^{\mathfrak{X}/M,\Omega,s^{\sf hol}}.\]
Differentiating by $\db$ yields 
\begin{align*}
%\label{eq:hae-1st}
 \db \mathfrak{D}^{\mathfrak{X}/M,\Omega, s^{\sf BT}} 
 &=\db \left(\widehat{\mathfrak{A}} \;\mathfrak{D}^{\mathfrak{X}/M,\Omega,s^{\sf hol}}\right)\\
 &= (\db\widehat{\mathfrak{A}}) \;\mathfrak{D}^{\mathfrak{X}/M,\Omega,s^{\sf hol}}+0\\
 &= (\db\widehat{\mathfrak{A}}) \widehat{\mathfrak{A}^{-1}} \mathfrak{D}^{\mathfrak{X}/M,\Omega, s^{\sf BT}} \\
&= \left(\db \mathfrak{A}\cdot \mathfrak{A}^{-1}\right)^{\widehat{}}\mathfrak{D}^{\mathfrak{X}/M,\Omega, s^{\sf BT}}\\
 &= \widehat{-\overline{C}}\mathfrak{D}^{\mathfrak{X}/M, \Omega,s^{\sf BT}}.
\end{align*}
Here, %we use Equation~\eqref{C-bar-operator} and 
$\left(\db \mathfrak{A}\cdot \mathfrak{A}^{-1}\right)^{\widehat{}}$ is the quantization of the operator $\db \mathfrak{A}\cdot \mathfrak{A}^{-1}$ and %the last equality follows from 
we use Equation~\eqref{C-bar-operator}.
\end{proof}

\subsubsection{Proof of the holomorphic anomaly equation for genus $g\geq 2$} With the preparations above, in this subsection we prove the holomorphic anomaly equation as stated in Theorem~\ref{thm:main}. We begin by writing Equation~\eqref{eq:db-potential} in the case without insertions, i.e., in components $(g,0)$ for $g\geq 2$. Indeed, let us choose local coordinates $t_1,\ldots,t_\mu$ on $M$. Denote by $f_0:=\Omega$ in the line bundle $\cH^{3,0}$. By the miniversality condition, the Kodaira-Spencer map 
\[\KS: T_M \stackrel{\cong}{\longrightarrow} R^1p_*T_{\mathfrak{X}/M}\]
is an isomorphism. This induces an isomorphism 
$R^1p_*T_{\mathfrak{X}/M}\cong \mathbb{H}^{2,1}$
defined by contraction with the volume form $\Omega$, i.e., $\partial_i \mapsto \KS(\partial_i)\lrcorner\Omega$. Using this isomorphism, we obtain a local holomorphic frame
$$f_1:=\KS(\partial_1)\lrcorner\Omega,\quad \ldots,\quad f_\mu:=\KS(\partial_\mu)\lrcorner\Omega$$ of the smooth vector bundle $\cH^{2,1}$. Since $\cH^{1,2}\cong \overline{\cH^{2,1}}$, we obtain a frame 
$\overline{f_1},\ldots,\overline{f_\mu}$
for the bundle $\cH^{1,2}$ by complex-conjugation. Similarly, denote by $\overline{f_0}=\overline{\Omega}$ in the line bundle $\cH^{0,3}$. By the Hodge-Riemann relation, we may scale the intersection pairing $Q$ in the middle cohomology by
\[g_{0\overline{0}}:= \sqrt{-1}Q(f_0,\overline{f_0}),\;\;\; g_{i\overline{j}}:= -\sqrt{-1}Q(f_i,\overline{f_j}),\]
to obtain a Hermitian form $g$. In the following, we use the metric tensor $g_{i\overline{j}}$ and its inverse $g^{i\overline{j}}$ to raise or lower the indices. 
For each $1\leq i\leq \mu$, the infinitesimal symplectic transformation $-\overline{C}_{\overline{i}}$ acts by
\begin{align}\label{eq:C-formula}
%\begin{split}
-\overline{C}_{\overline{i}} (f_j)  =  -G_{j\overline{i}} f_0, \quad 
-\overline{C}_{\overline{i}}(\overline{f_j})  = - \overline{C}^k_{\overline{i},\overline{j}} f_k, \quad 
-\overline{C}_{\overline{i}}(\overline{f_0})=-\overline{f_i}.
%\end{split}
\end{align}
One can show that $G_{j\overline{i}}=g^{0\overline{0}}  g_{j\overline{i}}$, see, for example~\cite[Section 2.2]{KZ}. The Hermitian form $G_{j\overline{i}}$ is known as the {\em Weil-Peterson metric} on $M$.

\begin{Proposition}
For each $g\geq 2$, we have
\begin{equation}\label{HAE-no-insertions} \db_{\overline{i}} F^{\mathfrak{X}/M,\Omega,s^{\sf BT}}_g = \frac{1}{2}\sum_{1\leq j,k\leq \mu} \overline{C}_{\overline{i}}^{j,k} \bigg( \langle f_j,f_k\rangle_{g-1,2}^{\mathfrak{X}/M,\Omega,s^{\sf BT}} + \sum_{r=1}^{g-1}  \langle f_j \rangle_{r,1}^{\mathfrak{X}/M,\Omega,s^{\sf BT}}\langle f_k \rangle_{g-r,1}^{\mathfrak{X}/M,\Omega,s^{\sf BT}}\bigg).
\end{equation}
\end{Proposition}
\begin{proof}
Observe from Equation~\eqref{eq:C-formula}, the operator $-\overline{C}_{\overline{i}}$ is of the form
\[-\overline{C}_{\overline{i}}=\begin{bmatrix}
    0 & * & 0 & 0 \\
    0 & 0 & * & 0\\
    0 & 0 & 0 & *\\
    0 & 0 & 0 & 0
\end{bmatrix}\]
in the block decomposition $\cH^{3,0}\oplus \cH^{2,1} \oplus \cH^{1,2}\oplus \cH^{0,3}$. Thus, identifying it with an element in the Lie algebra of the Givental group, we obtain $-\overline{C}_{\overline{i}} \cdot u$. Then computing with Equation~\eqref{eq:giv-propagater} yields the infinitesimal edge contribution given by
\[ \overline{C}^j_{\overline{i},\overline{l}}g^{k\overline{l}} f_j\otimes f_k + g^{0\overline{0}}\Omega\otimes \overline{f_i} + g^{0\overline{0}}\overline{f_i}\otimes \Omega. \]
Thus, writing down Equation~\eqref{eq:db-potential} in the component $(g,0)$ yields
\begin{align*}
    \db_{\overline{i}} F^{\mathfrak{X}/M,\Omega,s^{\sf BT}}_g = & \frac{1}{2} \sum_{1\leq j,k\leq \mu} \overline{C}_{\overline{i}}^{j,k} \bigg( \langle f_j,f_k\rangle_{g-1,2}^{\mathfrak{X}/M,\Omega,s^{\sf BT}} + \sum_{r=1}^{g-1}  \langle f_j \rangle_{r,1}^{\mathfrak{X}/M,\Omega,s^{\sf BT}}\langle f_k \rangle_{g-r,1}^{\mathfrak{X}/M,\Omega,s^{\sf BT}}\bigg) \\
    & + g^{0\overline{0}}\bigg( \langle \Omega,\overline{f_i}\rangle_{g-1,2}^{\mathfrak{X}/M,\Omega,s^{\sf BT}} +  \sum_{r=1}^{g-1}  \langle \Omega\rangle_{r,1}^{\mathfrak{X}/M,\Omega,s^{\sf BT}}\langle \overline{f_i} \rangle_{g-r,1}^{\mathfrak{X}/M,\Omega,s^{\sf BT}}\bigg).
\end{align*}
In the second line of this equation, the term $\langle \Omega,\overline{f_i}\rangle_{g-1,2}^{\mathfrak{X}/M,\Omega,s^{\sf BT}}$ vanishes by the string equation~\eqref{eq:string-equation}, while the term $\langle \overline{f_i} \rangle_{g-r,1}^{\mathfrak{X}/M,\Omega,s^{\sf BT}}$ vanishes by the dimension property.
\end{proof}

%Finally, to prove our main Theorem~\ref{thm:main}, we need to 
Now we express the right hand side of Equation~\eqref{HAE-no-insertions} by covariant derivatives. 
We continue to use the notation from Equation~\eqref{eq:gm-components}. In particular, we have a $(1,0)$-type connection $D$ on the vacuum line bundle $\mathbb{L}$. In local coordinates $t_1,\ldots,t_\mu$ on $M$, 
there exist local $C^\infty$-functions $a_i$ on $M$, such that
\begin{equation}
\label{i-component-of-D}
D_i \Omega = a_i \Omega.
\end{equation}
By the comparison Theorem~\ref{conj:comparison}, there are two terms in $\nabla_{i}^{s^{\sf BT},+}(\Omega\psi)$: the diagonal component given by $a_i\Omega \psi$, and a Kodaira-Spencer term given by $-\KS(\partial_i)\lrcorner\Omega=-f_i$, i.e. 
\begin{equation}\label{comparison-proof-main}
\nabla_{i}^{s^{\sf BT},+}(\Omega\psi)= -f_i + a_i\Omega \psi.
\end{equation}
%it acts on section $f_0=\Omega$ by $$D_i f_0 = a_i \Omega,$$
%for some local $C^\infty$-function $a_i$ on $M$.
The following result is a consequence of the dilaton equation and the divisor equation.
\begin{Lemma}
We have 
\begin{equation}\label{derivative-Fg}\partial_i F_g^{\mathfrak{X}/M,\Omega,s^{\sf BT}}
=\langle f_i \rangle_{g,1}^{\mathfrak{X}/M,\Omega,s^{\sf BT}}-\langle a_i\Omega \psi\rangle_{g,1}^{\mathfrak{X}/M,\Omega,s^{\sf BT}}.
\end{equation}
\end{Lemma}
\begin{proof}
By the dilaton equation~\eqref{eq:dilaton} and the divisor equation~\eqref{eq:divisor}, we have 
\begin{align*}
\langle \nabla_{i}^{s^{\sf BT},+}(-\Omega\psi), \Omega\psi \rangle_{g,1}^{\mathfrak{X}/M,\Omega,s}
&=(2g-2+1)\cdot \langle\nabla_{i}^{s^{\sf BT},+}(-\Omega\psi)\rangle_{g,1}^{\mathfrak{X}/M,\Omega,s},\\
\langle \nabla_{i}^{s^{\sf BT},+}(-\Omega\psi), \Omega\psi \rangle_{g,1}^{\mathfrak{X}/M,\Omega,s}
&=\partial_i \langle \Omega\psi \rangle_{g,1}^{\mathfrak{X}/M,\Omega,s}-\langle \nabla_{i}^{s^{\sf BT},+}(\Omega\psi)\rangle_{g,1}^{\mathfrak{X}/M,\Omega,s}.
\end{align*}
Combining these two equations, we obtain  
$$%\partial_i F_g^{\mathfrak{X}/M,\Omega,s^{\sf BT}}:= 
\partial_i \left\langle \Omega\psi \right\rangle_{g,1}^{\mathfrak{X}/M,\Omega,s}
=(2-2g)\langle \nabla_{i}^{s^{\sf BT},+}(\Omega\psi)\rangle_{g,1}^{\mathfrak{X}/M,\Omega,s}.$$
Now the result follows from the definition of $F_g^{\mathfrak{X}/M,\Omega,s^{\sf BT}}$ in~\eqref{eq:cei-no-insertion} and Equation~\eqref{comparison-proof-main}.
\end{proof}
Recall from Section~\ref{para:intro-1.1} that for each $g\geq 2$, we have a $C^\infty$-section
$$F_g^{\mathfrak{X}/M}:=F_g^{\mathfrak{X}/M,\Omega,s^{\sf BT}}\cdot \Omega^{2g-2}\in C^\infty(\mathbb{L}^{2g-2}).$$ 
\begin{Lemma}\label{lem:covariant-derivatives}
For each $1\leq i,j \leq \mu$, we have
\begin{equation*}
D_iF_g^{\mathfrak{X}/M} = \langle f_i \rangle_{g,1}^{\mathfrak{X}/M,\Omega,s^{\sf BT}}  \Omega^{2g-2} \quad \text{and}\quad
D_jD_i F^{\mathfrak{X}/M}_g =\langle f_i, f_j\rangle_{g,2}^{\mathfrak{X}/M,\Omega,s^{\sf BT}} \Omega^{2g-2}.
\end{equation*}
\end{Lemma}
\begin{proof}
For the first equation, using Equation~\eqref{i-component-of-D} and~\eqref{derivative-Fg}, we have 
\begin{align*}
D_iF_g^{\mathfrak{X}/M} 
&=\big(\partial_i F_{g}^{\mathfrak{X}/M,\Omega,s^{\sf BT}}\big) \Omega^{2g-2}+ F_g^{\mathfrak{X}/M,\Omega,s^{\sf BT}}D_i(\Omega^{2g-2})\\
&=\langle f_i \rangle_{g,1}^{\mathfrak{X}/M,\Omega,s^{\sf BT}}-\langle a_i\Omega \psi\rangle_{g,1}^{\mathfrak{X}/M,\Omega,s^{\sf BT}}+
{1\over 2g-2}\left\langle \Omega\psi \right\rangle_{g,1}^{\mathfrak{X}/M,\Omega,s}  (2g-2)a_i \Omega^{2g-2}\\
&= \langle f_i\rangle_{g,1}^{\mathfrak{X}/M,\Omega,s^{\sf BT}} \Omega^{2g-2}.
\end{align*}
The second equation is proved in a similar way.
\end{proof}

We define 
\begin{align}\label{eq:tensors}
\begin{dcases}
DF^{\mathfrak{X}/M}_r := D_j F_r^{\mathfrak{X}/M} dt_j, 
&
DDF^{\mathfrak{X}/M}_{g-1}:= D_jD_k F^{\mathfrak{X}/M}_{g-1} dt_j\otimes dt_k;\\
\overline{C}_{\overline{i}}:= \overline{C}_{\overline{i}}^{j,k} \partial_j\otimes \partial_k\otimes\Omega^2, &
    \overline{C} :=\overline{C}_{\overline{i}} d\overline{t_i}\in \Omega_M^{0,1}(T_M^{\otimes 2}\otimes \bbL^{\otimes 2}).
\end{dcases}
\end{align}

\begin{Theorem}
For each $g\geq 2$, the holomorphic anomaly equation~\eqref{eq:intro-hae-identity} holds.
\end{Theorem}
\begin{proof}
It remains to put everything together. From Equation~\eqref{HAE-no-insertions} and Lemma~\ref{lem:covariant-derivatives} we have
\begin{align*}
\db_{\overline{i}} F^{\mathfrak{X}/M}_g & = \frac{1}{2}\sum_{1\leq j,k\leq \mu} \overline{C}_{\overline{i}}^{j,k} \Big( \langle f_j,f_k\rangle_{g-1,2}^{\mathfrak{X}/M,\Omega,s^{\sf BT}} + \sum_{r=1}^{g-1}  \langle f_j \rangle_{r,1}^{\mathfrak{X}/M,\Omega,s^{\sf BT}}\langle f_k \rangle_{g-r,1}^{\mathfrak{X}/M,\Omega,s^{\sf BT}}\Big)\cdot \Omega^{2g-2}\\
& = \frac{1}{2} \sum_{1\leq j,k\leq \mu}\overline{C}_{\overline{i}}^{j,k} \Big( D_jD_kF_{g-1}^{\mathfrak{X}/M} + \sum_{r=1}^{g-1}  D_j F_r^{\mathfrak{X}/M}D_kF_{g-r}^{\mathfrak{X}/M}\Big)\cdot \Omega^{2}\\
& = \frac{1}{2} \overline{C}_{\overline{i}} \diamond \Big(DDF^{\mathfrak{X}/M}_{g-1}+\sum_{r=1}^{g-1} D F_r^{\mathfrak{X}/M}DF_{g-r}^{\mathfrak{X}/M}\Big).
\end{align*}
Here we use the notations in Equation~\eqref{eq:tensors} in the last equality and the diamond product $\diamond$ is natural contraction between dual tensors between $T_M$ and $\Omega_M$ (following the notations used in~\cite[Equation 4.2]{Liu}). 
%This implies the coordinate-free form of the holomorphic anomaly equation
%\[\db F^{\mathfrak{X}/M}_g = \frac{1}{2} \overline{C} \diamond \Big(DDF^{\mathfrak{X}/M}_{g-1}+\sum_{r=1}^{g-1} D F_r^{\mathfrak{X}/M}DF_{g-r}^{\mathfrak{X}/M}\Big).\]
 %This finishes the proof of Equation~\eqref{eq:intro-hae-identity}.
The coordinate-free form of the %holomorphic anomaly equation
HAE~\eqref{eq:intro-hae-identity} follows from the last equation in~\eqref{eq:tensors}.
\end{proof}

\subsubsection{Genus one HAE}
%When $(g, n)=(1,1)$, 
We also have a genus one holomorphic anomaly equation as predicted in~\cite{BCOV93}.
\begin{Theorem}\label{thm:genus-one-hae}
Let $\chi$ be the Euler characteristic of a fiber in the family $\mathfrak{X}\to M$. %When $g=1$, 
We have
\begin{equation}\label{eq:genus-one-hae}
\overline{\partial}_{\overline{i}}\Big( F_{1,1}^{\mathfrak{X}/M,\Omega,s^{\sf BT}}(f_j)\Big) 
=\left(\frac{\chi}{24}+1\right) G_{j\overline{i}} +
 \frac{1}{2}\sum_{1\leq k,l\leq \mu} \overline{C}^{kl}_{\overline{i}}C_{jkl}.
 \end{equation}
\end{Theorem}

\begin{proof}
Writing down Equation~\eqref{eq:db-potential} in the component $(g=1,n=1)$ with insertion $f_j$ yields
\begin{equation}\label{eq:hae-110}
\overline{\partial}_{\overline{i}}\Big( F_{1,1}^{\mathfrak{X}/M,\Omega,s^{\sf BT}}(f_j)\Big) = \begin{tikzpicture}[baseline={(current bounding
box.center)},scale=0.6] \draw [thick,directed] (3.4,4) to (3.4,2);
\node at (3.4,2) {$\bullet$}; \node at (3.4,1.5) {\tiny $g=1$};\end{tikzpicture}
+\frac{1}{2}\begin{tikzpicture}[baseline={(current bounding
box.center)},scale=0.6] \draw [thick,directed] (3.4,4) to (3.4,2);
\node at (3.4,2) {$\bullet$}; \draw [thick] (3.4,1.5) circle
[radius=.5]; 
\end{tikzpicture}
\end{equation}
where 
\begin{enumerate}
    \item the vertex contribution is given by $F_{g(v),n(v)}^{\mathfrak{X}/M,\Omega,s^{\sf BT}}$,
    \item the edge contribution is given by the tensor 
    $ \overline{C}^{kl}_{\overline{i}} f_k\otimes f_l + g^{0\overline{0}}\Omega\otimes \overline{f_i} + g^{0\overline{0}}\overline{f_i}\otimes \Omega, $
    \item the leaf contribution is given by $-\overline{C}_{\overline{i}}(f_j) \cdot u$ in the first graph; and by $f_j$ in the second graph.
\end{enumerate}
Computing these contributions using Equation~\eqref{eq:C-formula} yields
\begin{align}\label{eq:expansion-hae-11}
    \overline{\partial}_{\overline{i}}\Big( F_{1,1}^{\mathfrak{X}/M,\Omega,s^{\sf BT}}(f_j)\Big) & = - G_{j\overline{i}} \langle \Omega\psi\rangle_{1,1}^{\mathfrak{X}/M,\Omega,s^{\sf BT}} + \frac{1}{2}\overline{C}^{kl}_{\overline{i}} \langle f_j,f_k,f_l\rangle_{0,3}^{\mathfrak{X}/M,\Omega,s^{\sf BT}} + g^{0\overline{0}}  \langle f_j,\Omega,\overline{f_i}\rangle_{0,3}^{\mathfrak{X}/M,\Omega,s^{\sf BT}}.
\end{align}
The minus in the first term $- G_{j\overline{i}} \langle \Omega\psi\rangle_{1,1}^{\mathfrak{X}/M,\Omega,s^{\sf BT}} $ comes from $\psi=-u$ when applying the Givental action. Using the formula of the string vertex
\[ \hcV_{1,1,0}=\frac{1}{24}\;\;\; \begin{tikzpicture}[baseline={([yshift=-.4ex]current bounding box.center)},scale=0.3]
\draw [thick] (0,2) circle [radius=2];
\draw [thick] (-2,2) to (-0.6,2);
\draw [thick] (-0.8,2.2) to (-0.4,1.8);
\draw [thick] (-0.8, 1.8) to (-0.4, 2.2);
\draw [thick] (-1.4142, 0.5858) to [out=40, in=140] (1, 0.5);
\draw [thick] (1.25, 0.3) to [out=-45, in=225] (2, 0.2);
\draw [thick] (2,0.2) to [out=45, in=-50] (1.732, 1);
\node at (0.7, 2.2) {$\scriptstyle{u^{-1}}$};
\end{tikzpicture}
\;\;-\;\;\frac{1}{4}\;\;\; 
\begin{tikzpicture}[baseline={([yshift=-.4ex]current bounding box.center)},scale=0.3]
\draw [thick] (0,2) circle [radius=2];
\draw [thick] (0,0) to (0,1.4);
\draw [thick] (-0.2, 1.2) to (0.2, 1.6);
\draw [thick] (-0.2, 1.6) to (0.2, 1.2);
\draw [thick] (0,0) to [out=80, in=180] (0.5, 1);
\draw [thick] (0.5,1) to [out=0, in=100] (0.9, 0.4);
\draw [thick] (0,0) to [out=-80, in=180] (0.5, -1);
\draw [thick] (0.5, -1) to [out=0, in=-100] (0.9, 0);
\end{tikzpicture}\]
in~\cite[Section 8]{CalTu2}, 
we may evaluate
\[ \langle \Omega\psi\rangle_1^{\mathfrak{X}/M,\Omega,s^{\sf BT}} = -\frac{\chi}{24}.\]
Note that the negative sign here is due to the fact that we use the super-trace of identity on the $[3]$-shifted Hochschild homology. For the three point functions, the comparison theorem~\ref{conj:comparison} yields
\[ \langle f_j,f_k,f_l\rangle_{0,3}^{\mathfrak{X}/M,\Omega,s^{\sf BT}}= C_{jkl},\;\;\;  \langle f_j,\Omega,\overline{f_i}\rangle_{0,3}^{\mathfrak{X}/M,\Omega,s^{\sf BT}}= g_{j\overline{i}}.\]
Putting these into Equation~\eqref{eq:expansion-hae-11} yields
\[ \overline{\partial}_{\overline{i}}\Big( F_{1,1}^{\mathfrak{X}/M,\Omega,s^{\sf BT}}(f_j)\Big) = 
\frac{\chi}{24} G_{j\overline{i}}  + \frac{1}{2}\overline{C}^{kl}_{\overline{i}} C_{jkl} + g^{0\overline{0}}  g_{j\overline{i}}.\]
The desired equation follows from noting that $G_{j\overline{i}}=g^{0\overline{0}}  g_{j\overline{i}}$.
\end{proof}

\subsubsection{HAE for the mirror quintic family} 
Consider the famous example of the mirror quintic family 
$$\mathfrak{X}\to M=\mathbb{C}-\{1,\zeta,\zeta^2,\zeta^3,\zeta^4\mid \zeta=\exp(2\pi \sqrt{-1}/5)\},$$
where the fiber $X_t$ over $t\in M$ is given by the resolution of the global quotient orbifold
\[ [\big(W:=x_0^5+x_1^5+x_2^5+x_3^5+x_4^5 - 5t  x_0x_1x_2x_3x_4=0\big)/(\Z/5\Z)^3].\]
For this family, $\mu=1$, the Euler characteristic $\chi= -200$, and 
there exist holomorphic top forms 
$$\Omega:= 5t \frac{x_4dx_0dx_1dx_2}{\partial W/ \partial x_3}.$$
Thus, for $f:=\KS(\partial_t)\lrcorner\Omega$, Equation~\eqref{eq:genus-one-hae} gives the following:
\begin{equation}\label{eq:genus-one-quintic}
\overline{\partial}_{\overline{t}}\Big( F_{1,1}^{\mathfrak{X}/M,\Omega,s^{\sf BT}}(f)\Big) =\frac{28}{3}G_{t\overline{t}}+
 \frac{1}{2}\overline{C}_{\overline{t}\overline{t}\overline{t}}C_{ttt}g^{t\overline{t}}g^{t\overline{t}}.
 \end{equation}
Furthermore, it is known that in this case we have the following $C_{ttt}= \left(\frac{2\pi \sqrt{-1}}{5}\right)^3 \frac{5t^2}{1-t^5}$. The HAE~\eqref{eq:genus-one-quintic} is solved in \cite{BCOV93} and the solution determines the genus one CEI $F_{1,1}^{\mathfrak{X}/M,\Omega,s^{\sf BT}}(f)$ up to a holomorphic function on $M$, i.e., for some  $a(t)\in \Gamma(\mathcal{O}_M)$, 
\[ F_{1,1}^{\mathfrak{X}/M,\Omega,s^{\sf BT}}(f) 
= \frac{d}{dt}\bigg( \frac{31}{3}K - \frac{1}{2} \ln{G_{t\overline{t}}}\bigg) + a(t),\]
where $K$ is the K\"ahler potential of the metric $G_{t\overline{t}}$ given by $K:=-\ln{ \left(\sqrt{-1} \int_{\mathfrak{X}/M} \Omega\wedge\overline{\Omega}\right)}$. Conjecturally, in view of mirror symmetry~\cite{EFM,EFM2}, the holomorphic function $a(t)$ should be  $\frac{5t^{4}}{12(1-t^{5})}$.
%Recall that $$F_g:=F_g^{\mathfrak{X}/M,\Omega,s^{\sf BT}}\cdot \Omega^{2g-2}\in C^\infty(\mathbb{L}^{2g-2}),$$ 

Meanwhile, for each $g\geq 2$, we have the holomorphic anomaly equation
\[ \db_{\overline{t}} F_g^{\mathfrak{X}/M} = \frac{1}{2} \overline{C}_{\overline{t}\overline{t}\overline{t}}g^{t\overline{t}}g^{t\overline{t}} \Big( D_t D_t F_{g-1} + \sum_{r=1}^{g-1}  D_t F_r \cdot D_t F_{g-r}\Big)\cdot \Omega^{2}.\]
We refer to \cite{BCOV} for a discussion of solving the holomorphic anomaly equation. 

\section{Proofs of string equation and divisor equation in higher genus%Theorem~\ref{thm:string} and Theorem~\ref{thm:divisor}
}
\label{app:proof-string}

In this section, we prove the string and divisor equations. The first two subsections provide with a unified algebraic framework to prove both equations. The last two subsections deal with the proofs separately.

\subsection{The algebraic setup} 
\label{sec-alg-setup}
We denote the Chevalley-Eilenberg (cochain) complex of a DGLA  $\mathfrak{h}$ by
\[{\sf CE}_*(\mathfrak{h}):=\Hom\big(\sym^{\bullet}(\mathfrak{h}[1]),\mathfrak{h}[1]\big)=\bigg(\mathfrak{h}[1]\to \Hom(\mathfrak{h}[1],\mathfrak{h}[1]) \to \Hom(\sym^2(\mathfrak{h}[1]),\mathfrak{h}[1])\to\cdots\bigg).
\]
Recall that the Chevalley-Eilenberg differential is constructed as follows. By the (co)Leibniz rule, the space $\Hom\big(\sym^\bullet(\mathfrak{h}[1]),\mathfrak{h}[1]\big)$ can be identified with the space ${\sf Coder}\big(\sym^\bullet(\mathfrak{h}[1])\big)$ of coderivations of the symmetric coalgebra. Explicitly, this identification is given by 
\begin{itemize}
    \item Given an element $\phi=\prod_{k\geq 0} \phi_k$ with $\phi_k\in \Hom(\sym^k(\mathfrak{h}[1]),\mathfrak{h}[1])$, we obtain a coderivation $\widetilde{\phi}: \sym^\bullet(\mathfrak{h}[1]) \to \sym^\bullet(\mathfrak{h}[1])$ by setting 
    \begin{equation}\label{eq:id-1} \widetilde{\phi} (\alpha_1\odot\cdots\odot\alpha_n)= \sum_{k=0}^n\sum_{\sigma\in {\sf Sh(k,n-k)}} (-1)^@ \phi_k(\alpha_{\sigma(1)}\odot\cdots\odot \alpha_{\sigma(k)})\odot\alpha_{\sigma(k+1)}\odot\cdots\odot \alpha_{\sigma(n)}. \end{equation}
    The notation ${\sf Sh(k,n-k)}$ stands for the group of $(k,n-k)$-shuffles.
    \item In the reverse direction, given a coderivation $\Phi\in {\sf Coder}\big(\sym^\bullet(\mathfrak{h}[1])\big)$, we simply post-compose with the canonical projection map to yield an element in ${\sf CE}_*(\mathfrak{h})$ given by 
    \begin{equation}\label{eq:id-2}\sym^\bullet(\mathfrak{h}[1]) \stackrel{\Phi}{\longrightarrow} \sym^\bullet(\mathfrak{h}[1]) \twoheadrightarrow \mathfrak{h}[1].\end{equation}
     
\end{itemize}
Through this identification, we obtain a Lie bracket structure on ${\sf CE}_*(\mathfrak{h})$ induced from the commutator Lie bracket on the space of coderivations. For example, the Lie bracket of $[\phi_1,\phi_0]$ lies inside $\Hom\big(\sym^{0}(\mathfrak{h}[1]),\mathfrak{h}[1]\big)$ and is equal to $\phi_1(\phi_0)$. Here $\phi_0(\phi_1)=0$ because according to Equation~\eqref{eq:id-1}, $\widetilde{\phi_1}$ acts trivially on $\sym^{0}(\mathfrak{h}[1])$. Similarly, the Lie bracket of $[\phi_2,\phi_0]$ lies inside $\Hom\big(\sym^{1}(\mathfrak{h}[1]),\mathfrak{h}[1]\big)$ and is equal to the map $\phi_2(\phi_0,-)$.

Furthermore, the differential $l_1: \mathfrak{h}[1]\to \mathfrak{h}[1]$ and the Lie bracket $l_2: \sym^2\mathfrak{h}[1]\to \mathfrak{h}[1]$ can be viewed as coderivations of $\sym^\bullet(\mathfrak{h}[1])$. Moreover, one can show that the structure maps $l_1$ and $l_2$ define a DGLA structure on $\mathfrak{h}$ if and only if 
\[ [l_1+l_2,l_1+l_2]=0,\]
where $[-,-]$ is the commutator in the space of coderivations. 
For simplicity, we will %shall also 
call the element $l_1+l_2\in {\sf CE}_*(\mathfrak{h})$ a DGLA structure on $\mathfrak{h}$. The Chevalley-Eilenberg differential $d_{\sf CE}(\mathfrak{h})$ is given by
\begin{equation}
\label{CE-differential}
d_{\sf CE}(\mathfrak{h})(-)=[l_1+l_2, -].
\end{equation}
%where the bracket on the right hand side is the Lie bracket on the space of coderivations. 
We refer the details of this construction to~\cite{CE}.
\begin{Lemma}\label{lem:ce-element}
Let $\epsilon$ be an even element such that $\epsilon^2=0$. 
Then the even element $f \in {\sf CE}_*(\mathfrak{h})$ is closed if and only if 
$\exp(f\epsilon)=\id+f\epsilon\in {\rm End}(\mathfrak{h}[\epsilon])$
is an $L_\infty$ isomorphism (linear over $\mathbb{C}[\epsilon]$).
%Then to give an even closed element $f \in {\sf CE}_*(\mathfrak{h})$ is equivalent to give an $L_\infty$ isomorphism (linear over $\mathbb{C}[\epsilon]$) of the form:
\end{Lemma}

\begin{proof}
Indeed, for $f=f_0+f_1+\cdots$ with each $f_i\in \Hom(\sym^i(\mathfrak{h}[1]),\mathfrak{h}[1])$, the $i$-component map of $\exp(f\epsilon)$ is given by the map $f_i$ in the $\epsilon$-component. 
Under this correspondence, one can verify that the $L_\infty$ morphism identity is equivalent to that of $f$ being closed in the complex ${\sf CE}_*(\mathfrak{h})$. 
\end{proof}

Sometimes, we also write the $L_\infty$ isomorphism $\exp(f\epsilon)$ as an $L_\infty$ morphism 
$\exp(f\epsilon): \mathfrak{h} \to \mathfrak{h}[\epsilon]$
by extending this morphism $\epsilon$-linearly one obtains the previous $L_\infty$ isomorphism.

Let us recall the notion of pseudo-isotopies of DGLA's. A pseudo-isotopy on $\mathfrak h$ is a DGLA structure on $\mathfrak h\otimes\Omega^\bullet_{[0,1]}$, where $\Omega^\bullet_{[0,1]}$ is the algebraic de Rham complex of the simplex. Write its coderivation as $d_{\rm DR}+\mu(\tau)+d\tau\nu(\tau)$, where $d_{\rm DR}=d\tau\,\partial_\tau$. Note that in this notation, we place forms ($d\tau$) on the left. For a fixed value $\tau_{0}$, its restriction at $\tau_{0}$ gives a DGLA structure $\mu(\tau_{0})$ on $\mathfrak h$. Denote this DGLA by $\mathfrak{h}(\tau_{0})$. Furthermore, under certain finiteness conditions, we may integrate the pseudo-isotopy to obtain a family of $L_{\infty}$-isomorphisms $K(\tau):\mathfrak h(\tau)\to\mathfrak h(0)$ by solving the following ODE:
\begin{equation}
\label{eq-initial-value-K}\begin{dcases}
\frac{d}{d\tau} K(\tau) = K(\tau) \circ \nu(\tau),\\
K(0)=\id.
\end{dcases}
\end{equation}

%with initial condition $K(0)=\id$. 
%We denote by $K:=K(1)$ and consider the following. 
\begin{Lemma}\label{lem:algebraic-setup}
Let $f(\tau)+d\tau\,g(\tau)$ be an even closed element in the Chevalley-Eilenberg complex of $\mathfrak{h}\otimes\Omega^{\bullet}_{[0,1]}$. 
Then the following diagram is commutative up to homotopy
\begin{equation}\label{eq:ce-comm-diagram} 
\begin{CD}
\mathfrak{h}(1) @>\exp(f(1)\epsilon)>> \mathfrak{h}(1)[\epsilon]\\
@V K(1) VV      @V K(1) VV\\
\mathfrak{h}(0) @> \exp(f(0)\epsilon) >> \mathfrak{h}(0)[\epsilon]
\end{CD}\end{equation}
\end{Lemma}

\begin{proof}
Consider the family of $L_\infty$ morphisms
\[ K(\tau)\circ \exp(f(\tau)\epsilon) \circ K(\tau)^{-1} : \mathfrak{h}(0) \ra \mathfrak{h}(0)[\epsilon].\]
Differentiating yields
\begin{align*}
 \frac{d}{d\tau} \big( K(\tau) \exp(f(\tau)\epsilon) K(\tau)^{-1} \big) 
= & K(\tau) [ \nu(\tau), \exp(f(\tau)\epsilon) ] K(\tau)^{-1} + K(\tau) \frac{d}{d\tau}\exp( f(\tau)\epsilon ) K(\tau)^{-1}\\
= & K(\tau) \Big( ([\nu(\tau),f(\tau)] + \frac{d}{d\tau}f(\tau)) \epsilon \Big) K(\tau)^{-1}\\
= & K(\tau) \big( [\mu(\tau), g(\tau)] \epsilon\big) K(\tau)^{-1} 
\;\;\; (\mbox{By closedness of } f(\tau)+d\tau\,g(\tau))\\
= & [\mu(0), K(\tau)g(\tau)\epsilon K(\tau)^{-1}]
\end{align*}
This calculation shows that
\[K(\tau) \exp(f(\tau)\epsilon) K(\tau)^{-1} + d\tau\,K(\tau)g(\tau)\epsilon K(\tau)^{-1}\]
is an isotopic family of $L_\infty$ morphisms from $\mathfrak{h}(0)$ to $\mathfrak{h}(0)[\epsilon]$. In other words, it is an $L_\infty$ morphism from $\mathfrak{h}(0)$ to $\mathfrak{h}(0)[\epsilon] \otimes \Omega_{[0,1]}$. This implies that we have
\[ K(1) \exp(f(1)\epsilon) K(1)^{-1} \cong \exp(f(0)\epsilon).\]
This implies that $K(1)\exp(f(1)\epsilon)  \cong \exp(f(0)\epsilon) K(1)$, as desired.
\end{proof}

\subsection{An alternative construction of pseudo-isotopies}\label{sec-K-construction}

In this subsection, we prove a useful lemma to produce pseudo-isotopies of DGLA's. In particular, it reproduces the pseudo-isotopy constructed in~\cite[Section 8]{CalTu1} which integrates to the $L_{\infty}$-isomorphism $\cK$ in Equation~\eqref{eq:Triv-map}. Indeed, let us consider a chain complex $(V,D)$ over a commutative differential graded algebra $\mathcal{R}$. Assume it is equipped with an even chain map
\[ \omega: V\otimes V \to \mathcal{R},\]
which is skew-symmetric. Suppose we are given a direct sum decomposition  $$V=V_1\oplus V_2$$ as graded modules, where $V_1$ is the ``input'' summand and $V_2$ is the ``output'' summand. For $i,j\in\{1,2\}$ write $D_{ij}$ for the components of $D$. We also assume that these components and $\omega$ are continuous if a topology is involved in the construction from the beginning. Write $s:V_1\ra V_1[1]$ for the suspension map. Adding even formal parameters $\hbar$ and $\lambda$, we define
\[
\widehat{\mathfrak{d}}:=\left(\bigoplus_{k\geq 1,l\geq0}
\Hom^\cont_{\mathcal{R}}\bigl(\sym^k(V_1[1]),\sym^l V_2\bigr)\right)[[\hbar,\lambda]].
\]
We construct a {{\em curved DGLA}} structure on $\widehat{\mathfrak{d}}$ as follows. 
\begin{enumerate}
\item The curvature is $l_0=\phi+\chi$, with components $(k=1,l=1)$ and $(k=2,l=0)$ given by
\begin{equation}\label{eq:normal-symbol-curvature}
 \phi(s p):=D_{12}p,\qquad
 \chi(s p,s p'):=(-1)^{|p|}\omega(p,p').
\end{equation}
Note that if $V_{1}\subset V$ is an isotropic sub-complex, then the curvature term $l_{0}$ vanishes.
\item The differential is $l_1=D^{\dagger}+\iota+\hbar\Delta$, where $D^{\dagger}$ is the differential induced from the diagonal components $D_{11}$ and $D_{22}$. Similar to Equation~\eqref{eq-iota-tcft}, the map $\iota$ is defined by
\[
(\iota q)(s p):=(-1)^{|q|}\omega(q,p).
\]
Let $\Theta_D:=s\circ D_{21}$ be the twisted-sewing map. Similar to Equation~\eqref{eq-delta-tcft}, we define $\Delta$ by the symmetrization of 
\begin{equation}\label{eq:defi-homotopy-D}
 \Delta(q\odot q'):=(-1)^{|q|}\omega(q,\Theta_{D}q').
\end{equation}

\item   As in Equation~\eqref{eq:def-r-bracket-algebraic} and Equation~\eqref{eq:def-bracket-hbar-algebraic} we set $l_{2}$ to be
\[
 \{\alpha,\beta\}_\hbar=(-1)^{|\alpha|}[\alpha,\beta]_\hbar,\qquad
 [\alpha,\beta]_\hbar=
 \sum_{r\geq1}\hbar^{r-1}\Bigl(\alpha\underset{[r]}{\circ}\beta
 -(-1)^{|\alpha||\beta|}\beta\underset{[r]}{\circ}\alpha\Bigr).
\]
Here in the composition $\alpha\underset{[r]}{\circ}\beta$ for $r\geq 1$, a directed sewing connects an output of $\beta$ to an input of $\alpha$ via $\Theta_D$.
\end{enumerate}

\begin{Lemma}\label{lem:curved-normal-symbols}
The operations $l_0,l_1,l_2$ defined above form a curved DGLA on $\mathfrak{d}$, in the shifted convention for the structure operations. Furthermore, this curved DGLA structure restricts to the part of $\mathfrak{d}$ with at least one input.
\end{Lemma}

\begin{proof}
We deal with identities involving the curvature term. We must verify
\[
  l_{1}(\phi+\chi)=0,\qquad l_{1}^2=-[\phi+\chi,-]_\hbar.
\]
For the first equation, the $(k=1,l=1)$ component of $ l_{1}(\phi+\chi)$ is given by
\[ sp \mapsto (D_{22}D_{12}+D_{{21}}D_{11})p=0.\]
The $(k=2,l=0)$ component of it is given by
\[ sp\odot sp' \mapsto  (-1)^{|p|}\bigl(d\omega(p,p')-\omega(Dp,p')-(-1)^{|p|}\omega(p,Dp')\bigr)=0.\]
The vanishing of the right hand side is due to that $\omega$ is a chain map.

Next, let us prove the identity $l_{1}^2=-[\phi+\chi,-]_\hbar$:
\begin{itemize}
\item On an output vector $q\in V_{2}$, we have $l_{1}q=D_{22}q+\iota q$. Then $l_{1}^{2}q$ has two components. Its $V_{2}$ component is given by
\[ D_{22}D_{22}q= -D_{12}D_{21}q,\]
 Since we have $D^2=0$. The other component lies in $\Hom^\cont_{\mathcal{R}}\bigl(V_1[1], \mathfrak{R}\bigr)$. If we evaluate this component on $s p\in V_{1}[1]$, we get
 \[ \big(\iota (D_{{22}} q) + D^\dagger(\iota q)\big)(sp)= (-1)^{|q|}\omega(D_{21}q,p)+\omega(q,D_{12}p),\]
 by closedness of $\omega$. These terms are exactly $-[\phi+\chi,q]_{\hbar}$.
\item For an element  $\xi\in \Hom^\cont_{\mathcal{R}}\bigl(V_1[1], \mathfrak{R}\bigr)$, one can check that
\[(l_{1}^{2}\xi)(s p)=\xi(s D_{21}D_{12}p),\]
which is $-[\phi+\chi,\xi]_{\hbar}(sp)= -[\phi,\xi]_{\hbar}(sp)$. 
\item For the two outgoing tensors $q\odot q'$, one may compute the coefficient of $\hbar$ in $l_{1}^{2}(q\odot q')$ is given by
\[
 (-1)^{|q|}\omega(D_{21}q,D_{21}q')
 +\tfrac12\bigl(\omega(q,D_{12}D_{21}q')-\omega(D_{12}D_{21}q,q')\bigr),
\]
using the closedness of $\omega$ and $D_{11}D_{21}+D_{21}D_{22}=0$. Now we observe the first term is minus the double twisted sewing with $\chi$. And the second is minus the twisted sewing with $\phi$. 
\end{itemize}
The proof of the Jacobi and Leibniz identities follow from~\cite[Theorem 4.3]{CalTu1}. 
\end{proof}

As an illustration of this curved DGLA construction, let us take $V= L^{A}((u))$ and $D=b+uB$. And take $V_1=uL^A_+$ and $V_2=L^A_-$. We also implicitly use the identification of $V_1[1]$ with $L^A_+[1]$ by removing $u$. The pairing $\omega$ is given by the residue pairing
\begin{equation}\label{eq:omega-residue-pairing}
 \omega(u^i a,u^j b)=(-1)^i\delta_{i+j,1}\langle a,b\rangle_{\Muk}.
\end{equation}
In this case, one can check that the curvature vanishes since $V_{1}\subset V$ is an isotropic sub-complex, and the DGLA construction above recovers $\hh_A$ by taking the part with at least one input.

Let us give another example, which reconstructs the pseudo-isotopy in~\cite[Section 8.4]{CalTu1}. In this case, we work over the base ring $\mathcal{R}=\Omega_{\Delta^{1}}^{\bullet}$. We then take
\[ \mathbb{V} = L^{A}((u))\otimes \mathcal{R},\qquad D= b+uB+d_{{DR}}.\]
Set $\mathbb{V}_1=uL_{{+}}^{A}\otimes \mathcal{R}$, and replace the constant Lagrangian complement $L^{A}_{-}$ by the graph of the map $(1-\tau)u F: L^{A}_{-} \to V_{1}$, where $\tau$ is the coordinate on the $1$-simplex, and $F$ is the homotopy operator defined in Equation~\eqref{eq:homotopy-F}. Explicitly, we set $\mathbb{V}_2$ be the image of the map
\[ \big(\id,(1-\tau)u F\big): L^{A}_{-}\otimes\mathcal{R} \to (L^{A}_{-} \oplus uL_{{+}}^{A})\otimes \mathcal{R}=\mathbb{V},\]
which gives us a direct sum decomposition $\mathbb{V}=\mathbb{V}_1\oplus \mathbb{V}_2$. Observe that we have restrictions $\mathbb{V}_2(\tau=1)= L^A_-$ and $\mathbb{V}_2(\tau=0)= \imag (S)$. 

For the skew symmetric form on $\mathbb{V}$. Observe that $\omega$, when restricted to $\mathbb{V}_2\otimes \mathbb{V}_2$, is given by \[\omega(q,q')=2(1-\tau)(H-H^\sym)(q,q'),\]
by definition of the operator $H$ in Equation~\eqref{eq:homotopy-H}. Then, we choose $\delta$ as in Equation~\eqref{eq:homotopy-delta} modify $\omega$ to the following 
\begin{equation}\label{eq:normal-pairing-correction}
 \omega-2(1-\tau)(H-H^\sym)+2d\tau\;\delta.
\end{equation}

\begin{Lemma}\label{lem:closedness-res-pairing}
    The modified pairing $\omega-2(1-\tau)(H-H^\sym)+2d\tau\;\delta$, as a map $\mathbb{V}^{\otimes 2} \to \mathcal{R}$, is $D$-closed.
\end{Lemma}

\begin{proof}
    By the defining Equation~\eqref{eq:homotopy-delta} of $\delta$, we may write the modified pairing as
    \[ \omega - D\big(2(1-\tau)\delta\big).\]
    Thus, it suffices to prove $D\omega=0$, which is clear since the residue pairing is $(b+uB)$-closed and is independent of $\tau$. 
\end{proof}

Apply Lemma~\ref{lem:curved-normal-symbols} to the setting above and restricting to the positive input part yields a DGLA structure on $\hh_A\otimes \Omega_{\Delta^{1}}^{\bullet}$. Note that the curvature again vanishes here because $\mathbb{V}_1$ is still an isotropic sub-complex. The differential and Lie bracket reconstruct the pseudo-isotopy in~\cite[Section 8.4]{CalTu1}:
\begin{equation}\label{eq:family-dgla}
\begin{split}
 l_1&=b+uB+\iota+d_{\rm DR}+\hbar\tau\Delta+d\tau\,\hbar\Delta^{H^\sym},\\
 l_2&=\sum_{r\geq1}\hbar^{r-1}\tau^r\{-,-\}_r
  +d\tau\sum_{r\geq1}\hbar^{r-1}\tau^{r-1}\bigl(\{-,-\}_r^F+\{-,-\}_r^\delta\bigr).
\end{split}
\end{equation}
As shown in {\em Loc. Cit.}, integrating this pseudo-isotopy yields an $L_\infty$-morphism denoted by $K: \hh_A \to \hh_A(\tau=0)$. Its relationship with the $L_\infty$-morphism $\cK: \hh_A \to \hh_A^{\sf TRIV}$ in Equation~\eqref{eq:Triv-map} is given by the following factorization:
\[ \cK= \mathfrak{S}\circ K: \hh_A \stackrel{K}{\longrightarrow} \hh_A(\tau=0) \stackrel{\mathfrak{S}}{\longrightarrow} \hh_A^{\sf TRIV}. \]
Here $\mathfrak{S}$ is a strict morphism simply given by applying $S$ to each inputs and $R$ to each outputs. In the following, denote the middle DGLA $\hh_A(\tau=0)$ by $\hh_A^\triv$. 

\subsection{Proof of the string equation}\label{subsec:proof-string}

We continue to use the notation in Section~\ref{para:string-genus-zero}. Our goal is to prove the identity~\eqref{string-main-identity} as $b$-homology classes, i.e.,
\[[M_{u^{-1}},\lbA_{g,1,n-1}]=  [(\cK_*\rho^{A,\tw,\omega_A}_*\mc_2)_{g,1,n-1}^\epsilon]\in H_\bullet\big( \Hom^c(L_+^A[1],\sym^{n-1}L_-^A)\big).\] 
Let us consider the following diagram of $L_\infty$-morphisms between DGLA's:
\begin{equation}\label{diagram:right-string}
\begin{CD}
\hh_A @>\exp\big( (\gamma_{0,1,1},M_{u^{-1}})\epsilon \big)>> \hh_A[\epsilon]\\
@V\cK VV   @V\cK VV\\
\hh_A^{\sf TRIV} @>\exp\big( (\gamma_{0,1,1},M_{u^{-1}})\epsilon \big)>> \hh_A^{\sf TRIV}[\epsilon]
\end{CD}
\end{equation}
The key observation is that, under the push-forward by the lower left composition, the Maurer-Cartan element $\hbA$ is mapped to the left hand side of Equation~\eqref{string-main-identity}. While the push-forward of $\hbA$ by the upper right composition yields the right hand side of Equation~\eqref{string-main-identity}. However, this diagram is general not commutative, even up to homotopy. With this in mind, the proof of Equation~\eqref{string-main-identity} then consists of two main steps.
\begin{enumerate}
\item We use Lemma~\ref{lem:curved-normal-symbols} to construct a diagram which factors the Diagram~\eqref{diagram:right-string} into two squares:
\begin{equation}\label{diagram:factor-diag-string}
\begin{tikzcd}
  \hh_A \arrow[d,"K"] \arrow{r}[name=U]{\exp\left( (\gamma_{0,1,1},M_{u^{-1}})\epsilon \right)}  &[10em] \hh_A[\epsilon] \arrow[d,"K"] \\
\hh_A^\triv \arrow[r,dashed,"\cJ"] \arrow[d,"\mathfrak{S}"] & \hh_A^\triv[\epsilon] \arrow[d,"\mathfrak{S}"]\\
\hh_A^{\sf TRIV} \arrow{r}[name=U]{\exp\left( (\gamma_{0,1,1},M_{u^{-1}})\epsilon \right)} & \hh_A^{\sf TRIV}[\epsilon].
\end{tikzcd}  
\end{equation}  
We prove the top square is indeed commutative up to homotopy. 
\item The bottom square in Diagram~\eqref{diagram:factor-diag-string} is not commutative. However, we prove for all stable pair $(g,n)$, the induced map in homology is commutative at the component with only one input:
\begin{equation}\label{eq:diagram:bottom-string}
  \begin{tikzcd}
H_\bullet\big( \Hom^\cont(L_+^A[1], \sym^{n-1} L_-^A), b+uB\big) \arrow[r,dashed,"(\cJ)_{g,1,n-1}^\epsilon"] \arrow[d,"\mathfrak{S}"] &[3em] H_\bullet\big( \Hom^\cont(L_+^A[1], \sym^{n-1} L_-^A), b+uB\big) \arrow[d,"\mathfrak{S}"]\\
H_\bullet\big( \Hom^\cont(L_+^A[1], \sym^{n-1} L_-^A), b\big) \arrow{r}[name=U]{M_{u^{-1}}} & H_\bullet\big( \Hom^\cont(L_+^A[1], \sym^{n-1} L_-^A), b\big).
\end{tikzcd}    
\end{equation}
For our purpose, this weaker commutativity is enough.
\end{enumerate} 

We proceed to apply Lemma~\ref{lem:curved-normal-symbols} to construct the top square in Diagram~\eqref{diagram:factor-diag-string} above. Let us work with the base differential graded commutative ring $\mathcal{R}=\Omega_{\Omega^1}^\bullet[\xi]$, where $\xi$ is a closed odd generator with $\xi^2=0$. We continue to work with $\mathbb{V}$ and the decomposition $\mathbb{V}=\mathbb{V}_1\oplus \mathbb{V}_2$ as in the previous section, both understood as extension by scalars to $\Omega_{\Omega^1}^\bullet[\xi]$. We endow it with the following differential
\[ D(\tau,\xi):= b+uB+d_{\rm DR}+\xi M_{u^{-1}}.\]
As for the skew pairing $\omega$, we also add a correction term to Equation~\eqref{eq:normal-pairing-correction}:
\[\omega(\tau,\xi):= \omega-2(1-\tau)(H-H^\sym)+2\,d\tau\,\delta+2(1-\tau)\xi(\delta M_{u^{-1}}).
\]
where $M_{u^{-1}}$ acts on $L^A((u))\otimes L^A((u))$ by $M_{u^{-1}}\otimes \id + \id\otimes M_{u^{-1}}$. 
\begin{Lemma}
    The modified pairing $\omega(\tau,\xi)$, as a map $\mathbb{V}^{\otimes 2} \to \mathcal{R}$, is $D(\tau,\xi)$-closed.
\end{Lemma}

\begin{proof}
    The proof is similar to Lemma~\ref{lem:closedness-res-pairing}. Observe by construction, we may write the modified pairing as
    \begin{equation}\label{eq:modified-omega-string}
    \omega(\tau,\xi)= \omega - D(\tau,\xi)\big(2(1-\tau)\delta\big).\end{equation}
    Thus, it suffices to prove $D(\tau,\xi)\omega=0$. Besides that the residue pairing is $(b+uB)$-closed and is independent of $\tau$, we also need to verify that $M_{u^{-1}}\omega=0$, i.e. we have
    \[\omega(u^{-1}a,b)+\omega(a,u^{-1}b)=0.\]
    This is a direct check using Equation~\eqref{eq:omega-residue-pairing}.
\end{proof}

Applying Lemma~\ref{lem:curved-normal-symbols} to these data yields a curved DGLA structure on $\hh_A\otimes \Omega_{\Omega^1}^\bullet[\xi]$. The coefficient of $\xi$ defines an element
\[ f(\tau)+ d\tau g(\tau) \in \Hom\big(\sym^{\leq 2}(\hh_A[1]),\hh_A[1]\big)\otimes \Omega_{\Omega^1}^\bullet.\]
\begin{Lemma}\label{lem:closedness-ce-element}
    The cochain $f(\tau)+d\tau\,g(\tau)$ is closed for the pseudo-isotopy~\eqref{eq:family-dgla}.
\end{Lemma}

\begin{proof}
Observe that $D(\tau,\xi)|_{\xi=0}$ and $\omega(\tau,\xi)|_{\xi=0}$ match with the data used to construct the pseudo-isotopy~\eqref{eq:family-dgla}. Now, in the $\xi$-extended family, let us write the DGLA structure maps as
\[ \mathfrak{l}(\tau) + \big(f(\tau)+d\tau\,g(\tau)\big)\xi,\]
with $\mathfrak{l}(\tau)\in \Hom\big(\sym^{\leq 2}(\hh_A[1]),\hh_A[1]\big)\otimes \Omega_{\Omega^1}^\bullet$ the cochain associated with the pseudo-isotopy~\eqref{eq:family-dgla}. The DGLA structure equation is then equivalent to 
\[[ \mathfrak{l}(\tau) + \big(f(\tau)+d\tau\,g(\tau)\big)\xi,  \mathfrak{l}(\tau) + \big(f(\tau)+d\tau\,g(\tau)\big)\xi]=0.\]
Since $\xi^2=0$, we obtain
\[ [ \mathfrak{l}(\tau), \big(f(\tau)+d\tau\,g(\tau)\big)]=0.\]
This proves the lemma.
\end{proof}

Apply Lemma~\ref{lem:algebraic-setup} to this cocycle to obtain a commutative (up to homotopy) diagram of the form:
\begin{equation}
\begin{tikzcd}
  \hh_A \arrow[d,"K"] \arrow{r}[name=U]{\exp\left(f(1)\epsilon \right)}  &[10em] \hh_A[\epsilon] \arrow[d,"K"] \\
\hh_A^\triv \arrow{r}[name=U]{\cJ:=\exp\left(f(0)\epsilon \right)}  &[10em] \hh_A^\triv[\epsilon].\\
\end{tikzcd}  
\end{equation} 
We need to compute maps in this Diagram.
\begin{Lemma}\label{lem:string-transport}
We have the following evaluations:
\begin{enumerate}
\item At $\tau=1$, we have $f(1)= (\gamma_{0,1,1}, M_{u^{-1}})\in \Hom\big(\sym^{\leq 1}(\hh_A[1]),\hh_A[1]\big)$. Thus its action on $\hbA$ is
\[\big(\exp(\epsilon f(1))\big)_*\hbA=\hbA+\epsilon\big(\gamma_{0,1,1}+[M_{u^{-1}},\hbA]\big)=\rho^{A,\tw,\omega_A}_*\mc_2.\]
\item  At $\tau=0$, for each stable $(g,n)$ we have
\[ (\cJ)_{g,1,n-1}^\epsilon = M_{u^{-1}} +\ad^F(\gamma_{0,1,1}).\]
The second term $\ad^F(\gamma_{0,1,1})$ is defined in Equation~\eqref{eq:defi-ad-F-operator} below.
\end{enumerate}
\end{Lemma}

\begin{proof}
At $\tau=1$, we have $D(\tau=1,\xi)= b+uB+\xi M_{u^{-1}}$ and $\omega(\tau=1,\xi)=\omega$. By construction in Lemma~\ref{lem:curved-normal-symbols}, we obtain the $\xi$-component of the curved DGLA is given by $(\gamma_{0,1,1}, [M_{u^{-1}},-])$. 

At $\tau=0$, we have $D(\tau=0,\xi)= b+uB+\xi M_{u^{-1}}$ and 
$$\omega(\tau=0,\xi)=\omega-2(H-H^\sym)+2\xi(\delta M_{u^{-1}}).$$
The construction of Lemma~\ref{lem:curved-normal-symbols} yields
\begin{equation}
\label{string-symmetric-f}
f(\tau=0)_n:=
\begin{dcases}
\gamma_{0,1,1}, & \text{if} \quad n=0;\\
M_{u^{-1}} + \ad^F(\gamma_{0,1,1}) + \iota^{\delta,\gamma_{0,1,1}} &\text{if} \quad n=1;\\
0, &\text{if} \quad n\geq 2.
\end{dcases}
\end{equation}
Here the operator $\ad^F(\gamma_{0,1,1})\in \Hom(\hh_A[1],\hh_A[1])$ is defined by
\begin{equation}\label{eq:defi-ad-F-operator} 
\ad^F(\gamma_{0,1,1})(\alpha):= \sum_{j=1}^l(\gamma_{0,1,1}\circ F) \circ_j \alpha 
-\sum_{i=1}^k\alpha \prescript{}{i}{\circ} (F\circ \gamma_{0,1,1}).
\end{equation}
For example, the term $(\gamma_{0,1,1}\circ F) \circ_j \alpha$ can be depicted as
\[\begin{tikzpicture}[baseline={(current bounding
   	box.center)},scale=0.7] 
    \draw [thick,directed] (1.4,4) to (3.4,2);
    \draw [thick,directed] (5.4,4) to (3.4,2);
    \node at (3.4,3) {$\cdots\cdots$};
    \node at (3.4,2) {$\bullet$};
    \node at (3,2) {$\alpha$};
    \draw[thick,directed] (3.4,2) to (1.4,0);
    \draw[thick,directed] (3.4,2) to (5.4,0);
    \draw[thick,directed] (3.4,2) to (3.4,.5);
    \node at (3.7,1.2) {$F$};
    \node at (3.4,.5) {$\bullet$};
    \node at (4.1,.5) {$\gamma_{0,1,1}$};
    \node at (3.4,-.8) {$j$-th};
    \draw[thick,directed] (3.4,.5) to (3.4,-.5);
    \node at (2.5,0) {$\cdots$};
    \node at (4.5,0) {$\cdots$};
    \end{tikzpicture}\]
The other operator $\iota^{\delta,\gamma_{0,1,1}}$ is only a placeholder, since it won't contribute to the component $(\cJ)_{g,1,n-1}^\epsilon$ because its image has at least two inputs.
\end{proof}

For the second step, we need to prove the commutativity of Diagram~\eqref{eq:diagram:bottom-string}, i.e. to verify
\[ \mathfrak{S}\circ \big(M_{u^{-1}}+\ad^F(\gamma_{0,1,1})\big) = M_{u^{-1}}\circ \mathfrak{S},\]
as linear maps $H_\bullet\big( \Hom^\cont(L_+^A[1], \sym^{n-1} L_-^A), b+uB\big) \to H_\bullet\big( \Hom^\cont(L_+^A[1], \sym^{n-1} L_-^A), b\big)$. Here we have used the fact that $(\cJ)^\epsilon_{g,1,n-1}= \big(M_{u^{-1}}+\ad^F(\gamma_{0,1,1})\big)$. Using Lemma~\ref{lem:string-leaf-cont}, we may deduce the equation above as follows:
\begin{align*}
    [M_{u^{-1}}, \mathfrak{S}] (\alpha) & = \sum_{j=1}^{n-1}(R\circ\gamma_{0,1,1}\circ F) \circ_j \alpha - \alpha \circ (F\circ \gamma_{0,1,1}\circ S)\\
    &= \mathfrak{S}\Big(\sum_{j=1}^l(\gamma_{0,1,1}\circ F) \circ_j \alpha - \sum_{i=1}^k\alpha \prescript{}{i}{\circ} (F\circ \gamma_{0,1,1})\Big)\\
    &= \mathfrak{S}\big( \ad^F(\gamma_{0,1,1})(\alpha)\big).
\end{align*}
Putting the three steps above together, the desired identity
\[[M_{u^{-1}},\lbA_{g,1,n-1}]=  [(\cK_*\rho^{A,\tw,\omega_A}_*\mc_2)_{g,1,n-1}^\epsilon]\in H_\bullet\big( \Hom^c(L_+^A[1],\sym^{n-1}L_-^A)\big)\] 
from Equation~\eqref{string-main-identity} can be easily deduced:
\begin{align*}
    [M_{u^{-1}},\lbA_{g,1,n-1}] & = M_{u^{-1}} \mathfrak{S} [(K_*\hbA)_{g,1,n-1}]\\
    & = \mathfrak{S} \big(M_{u^{-1}}+\ad^F(\gamma_{0,1,1})\big) [(K_*\hbA)_{g,1,n-1}]\\
    & = \mathfrak{S}[\big(\cJ_*K_*\hbA\big)_{g,1,n-1}^\epsilon]\\
    & = \mathfrak{S}[\Big(K_*\exp\left( (\gamma_{0,1,1},M_{u^{-1}})\epsilon \right)_* \hbA\Big)_{g,1,n-1}^\epsilon]\\
    & = [(\cK_*\rho^{A,\tw,\omega_A}_*\mc_2)_{g,1,n-1}^\epsilon].
\end{align*}

\subsection{Proof of the divisor equation}

Recall the definition of $\eta_{g,k,l}$ in~\eqref{eta-definition} and the formulas in  Lemma~\ref{lem:divisor-evaluation}. For simplicity, we shall simply use $\nabla_{\partial_t}^{\sf Get,\pm}$ to denote the commutator operator that acts on an element in $\hh_\cA$. Also, we use $u^{-1}\eta_{0,2,0}$ to denote the operator on $\hh_\cA$ that acts by
\[ \phi \mapsto \sum_{j=1}^{l+1} u_j^{-1}\cdot \eta_{0,2,0}\circ_j \phi,\]
for an element $\phi\in \hh_\cA$ with $l+1$ outputs. To prove the divisor equation, following the discussions in Section~\ref{subsec:divisor}, it remains to prove the following equality
\begin{equation}\label{eq:divisor-left-to-prove}
[\nabla^{S,\pm}_{\partial_t},\overline{\beta}^\cA_{g,1,n-1}] 
= [\big(\cK_*\rho^{\cA,\tw,\zeta}_*\mc_2\big)_{g,1,n-1}^\epsilon]\in H_\bullet\left( {\sf Hom}^\cont \big( L^\cA_+[1] , \sym^{n-1} L^\cA_- \big), b\right)
\end{equation} 
from Equation~\eqref{eq:identify-rhs}.  The proof of this equation is in complete parallel to that of Equation~\eqref{string-main-identity} proved in Section~\ref{subsec:proof-string}. We shall only sketch the proof here which again consists of two main steps.

In the first step, we apply Lemma~\ref{lem:curved-normal-symbols} to construct a diagram of the form:
\begin{equation}\label{diagram:right-divisor-split}
\begin{tikzcd}
  \hh_\cA \arrow[d,"K"] \arrow{r}[name=U]{\exp\left( (\eta,\nabla_{\partial_t}^{\sf Get,\pm} + u^{-1} \eta_{0,2,0})\epsilon \right)}  &[10em] \hh_\cA[\epsilon] \arrow[d,"K"] \\
\hh_\cA^\triv \arrow[r,dashed,"\cI"] \arrow[d,"\mathfrak{S}"] & \hh_\cA^\triv[\epsilon] \arrow[d,"\mathfrak{S}"]\\
\hh_\cA^{\sf TRIV} \arrow{r}[name=U]{\exp\left( \big(\eta,\nabla_{\partial_t}^{\sf S,\pm} + u^{-1} \eta_{0,2,0}\big)\epsilon \right)} & \hh_\cA^{\sf TRIV}[\epsilon].
\end{tikzcd}  
\end{equation}   
To apply Lemma~\ref{lem:curved-normal-symbols}, in the case of the divisor equation, we shall work with the base differential graded commutative ring $\mathcal{R}=\Omega_{\Delta^1}^\bullet[[t,dt]]$, where $t$ is formal coordinate of the family $\cA$. We continue to work with the notations $\mathbb{V}$ the decomposition $\mathbb{V}=\mathbb{V}_1\oplus \mathbb{V}_2$ as in the previous section, both understood as extension by scalars to $\Omega_{\Omega^1}^\bullet[[t,dt]]$. The chain complex $\big(\mathbb{V},D(\tau,t)\big)$ is given by
 \[ \mathbb{V} = L^{\cA}((u))\otimes \mathcal{R},\qquad D(\tau,t):= b+uB+d_{\rm DR}+dt \nabla_{\partial_t}^{\sf Get}.\]
For the skew pairing, the residue pairing $\omega$ in Equation~\eqref{eq:omega-residue-pairing} is not $D(\tau,t)$ flat on the chain level. So we need to add a homotopy operator in the coefficient of $dt$ to compensate this. Such a homotopy operator was already constructed by Sheridan~\cite[Definition 5.37]{She}. Remarkably, Sheridan's homotopy operator is exactly $\eta_{0,2,0}$ in Lemma~\ref{lem:divisor-evaluation}. More precisely, we define $$\omega(t)=\omega+dt\,\omega_1$$ as follows. For homogeneous $p,p'\in \mathbb{V}_1$ and $q,q'\in \mathbb{V}_2$, write
\[
 p=\sum_{r\geq1}p_r u^r,\qquad
 q=\sum_{i\geq0}q_{-i}u^{-i},
\]
and similarly for $p',q'$. Then $\omega_1$, as a block matrix on $\mathbb{V}_1\oplus \mathbb{V}_2$, is given by:
\begin{equation}\label{eq:divisor-pairing-matrix}
 \omega_1 :=
 \begin{bmatrix}
 (-1)^{|p|}\eta_{0,2,0}(p_1,p'_1)
 &\displaystyle\sum_{i\geq0}(-1)^{i+1+|p|}\eta_{0,2,0}(p_{i+2},q'_{-i})\\[2mm]
 \displaystyle\sum_{i\geq0}(-1)^{i+1+|q|}\eta_{0,2,0}(q_{-i},p'_{i+2})
 &0
 \end{bmatrix}.
\end{equation}

\begin{Lemma}\label{lem:divisor-pairing-homotopy}
The pairing $\omega(t)=\omega+dt \omega_1$ is a chain map with respect to the differential $D(\tau,t):= b+uB+d_{\rm DR}+dt \nabla_{\partial_t}^{\sf Get}$.
\end{Lemma}

\begin{proof}
This lemma follows from~\cite[Lemma 5.38]{She}. Independently, one may also prove this lemma by explicit computation and using the following identities:
\begin{align*}
    [b+uB,\eta_{0,2,0}]+\iota\eta_{0,1,1} & =0,\\
    [\iota,D_t^\pm]+[b+uB,u^{-1}\eta_{0,2,0}]+[\{-,-\}_1,\eta_{0,2,0}]&=0.
\end{align*}
The first identity follows from the defining Maurer-Cartan property of the corresponding string vertices. The second identity follows from
Lemma~\ref{lem:divisor-evaluation} and the commutator identity in~\eqref{commutator-eth-S}.
\end{proof}

As in Equation~\eqref{eq:modified-omega-string}, we modify $\omega(t)$ using $\delta$. Define a map $\omega(\tau,t): \mathbb{V}^{\otimes 2} \to \mathcal{R}$ by setting
\[ \omega(\tau,t):= \omega(t)- D(\tau,t)\big(2(1-\tau)\delta\big).\]
It follows from the previous lemma that $\omega(\tau,t)$ is also a chain map. Applying Lemma~\ref{lem:curved-normal-symbols} to these data yields a curved DGLA structure on $\hh_\cA\otimes\mathcal{R}$. The coefficient of $dt$ defines an element
\[ f(\tau,t)+ d\tau g(\tau,t) \in \Hom\big(\sym^{\leq 2}(\hh_\cA[1]),\hh_\cA[1]\big)\otimes \Omega_{\Omega^1}^\bullet.\]
The following lemma is proved in the same way as Lemma~\ref{lem:closedness-ce-element}.
\begin{Lemma}
    The cochain $f(\tau,t)+d\tau\,g(\tau,t)$ is closed in the Chevalley-Eilenberg cochain complex of the DGLA $\hh_\cA$.
\end{Lemma}
We then apply Lemma~\ref{lem:algebraic-setup} to the cocycle $f(\tau,t)+ d\tau g(\tau,t)$ to obtain a commutative (up to homotopy) diagram:
\begin{equation}
\begin{tikzcd}
  \hh_\cA \arrow[d,"K"] \arrow{r}[name=U]{\exp\left(f(1)\epsilon \right)}  &[10em] \hh_\cA[\epsilon] \arrow[d,"K"] \\
\hh_\cA^\triv \arrow{r}[name=U]{\cI:=\exp\left(f(0)\epsilon \right)}  &[10em] \hh_\cA^\triv[\epsilon].\\
\end{tikzcd}  
\end{equation}
In the diagram above $f(1):=f(\tau,t)|_{\tau=1}$ and similarly $f(0):=f(\tau,t)|_{\tau=0}$. We need to compute maps in this Diagram. The following lemma is a calculation similar to Lemma~\ref{lem:string-transport}. 

\begin{Lemma}\label{lem:divisor-transport}
We have the following evaluations:
\begin{enumerate}
\item At $\tau=1$, the element $f(1)\in  \Hom\big(\sym^{\leq 1}(\hh_\cA[1]),\hh_\cA[1]\big)$ is given by
\[ \big( \eta:=\eta_{0,1,1}+\eta_{0,2,0}, \nabla_{\partial_t}^{\sf Get,\pm} + u^{-1}\eta_{0,2,0}\big) \in {\sf CE}^*(\hh_\cA).\]
\item  At $\tau=0$, for each stable $(g,n)$ we have
\[ (\cI)_{g,1,n-1}^\epsilon = \nabla_{\partial_t}^{\sf Get,\pm} +\ad^F(\eta_{0,1,1}).\]
The second term $\ad^F(\eta_{0,1,1})$ is defined by Equation~\eqref{eq:defi-ad-F-operator} replacing $\gamma_{0,1,1}$ by $\eta_{0,1,1}$.
\end{enumerate}
\end{Lemma}
 
In the second step, we consider the partial commutativity of the lower square in~\eqref{diagram:right-divisor-split}, after taking cohomology, i.e. the following diagram:
\begin{equation}\label{eq:diagram:bottom-divisor}
\begin{tikzcd}
H_\bullet\big( \Hom^\cont(L_+^\cA[1], \sym^{n-1} L_-^\cA), b+uB\big) \arrow[r,dashed,"(\cI)_{g,1,n-1}^\epsilon"] \arrow[d,"\mathfrak{S}"] &[3em] H_\bullet\big( \Hom^\cont(L_+^\cA[1], \sym^{n-1} L_-^\cA), b+uB\big) \arrow[d,"\mathfrak{S}"]\\
H_\bullet\big( \Hom^\cont(L_+^\cA[1], \sym^{n-1} L_-^\cA), b\big) \arrow{r}[name=U]{\nabla_{\partial_t}^{\sf S,\pm}} & H_\bullet\big( \Hom^\cont(L_+^\cA[1], \sym^{n-1} L_-^\cA), b\big).
\end{tikzcd}    
\end{equation}
The proof of its commutativity follows from Lemma~\ref{lem:divisor-leaf-cont}. With the two steps above ready, the main identity~\eqref{eq:identify-rhs} follows:
\begin{align*}
    [\nabla_{\partial_t}^{s,\pm},\overline{\beta}^\cA_{g,1,n-1}] & = \nabla_{\partial_t}^{s,\pm} \mathfrak{S} [(K_*\hbcA)_{g,1,n-1}]\\
    & = \mathfrak{S} \big(\nabla_{\partial_t}^{{\sf Get},\pm}+\ad^F(\eta)\big) [(K_*\hbcA)_{g,1,n-1}]\\
    & = \mathfrak{S}[\big(\cI_*K_*\hbcA\big)_{g,1,n-1}^\epsilon]\\
    & = \mathfrak{S}[\Big(K_*\exp\big( (\eta,\nabla_{\partial_t}^{\sf Get,\pm} + u^{-1} \eta_{0,2,0})\epsilon \big)_* \hbcA\Big)_{g,1,n-1}^\epsilon]\\
    & = [\big(\cK_*\rho^{\cA,\tw,\zeta}_*\mc_2\big)_{g,1,n-1}^\epsilon].
\end{align*}
This proves the desired identity~\eqref{eq:divisor-left-to-prove}, which completes the proof of the divisor equation.

\end{document}